\documentclass[10pt,a4paper,reqno]{amsart}
\allowdisplaybreaks
\numberwithin{equation}{section}
\usepackage[normalem]{ulem}
\makeatletter 
\def\@tocline#1#2#3#4#5#6#7{\relax
  \ifnum #1>\c@tocdepth 
  \else
    \par \addpenalty\@secpenalty\addvspace{#2}%
    \begingroup \hyphenpenalty\@M
    \@ifempty{#4}{%
      \@tempdima\csname r@tocindent\number#1\endcsname\relax
    }{%
      \@tempdima#4\relax
    }%
    \parindent\z@ \leftskip#3\relax \advance\leftskip\@tempdima\relax
    \rightskip\@pnumwidth plus4em \parfillskip-\@pnumwidth
    #5\leavevmode\hskip-\@tempdima
      \ifcase #1
       \or\or \hskip 1em \or \hskip 2em \else \hskip 3em \fi%
      #6\nobreak\relax
    \dotfill\hbox to\@pnumwidth{\@tocpagenum{#7}}\par
    \nobreak
    \endgroup
  \fi}
\makeatother
\usepackage{transparent}
\usepackage[margin=1.3in]{geometry}
\usepackage{amsmath,amsthm,amssymb}
\usepackage[initials,msc-links]{amsrefs}
\usepackage{url}
\usepackage[dvipsnames]{xcolor}
\usepackage[english=american]{csquotes}
\usepackage{enumerate}
\usepackage{textcomp}
\usepackage{mathrsfs}
\usepackage{mathtools}
\usepackage{color, colortbl}
\definecolor{Gray}{gray}{0.9}
\usepackage{bbm}
\usepackage{hyperref}
\usepackage{stmaryrd}
\hypersetup{
    unicode=false,          
    pdftoolbar=true,        
    pdfmenubar=true,        
    pdffitwindow=false,     
    pdfstartview={FitH},    
    pdftitle={My title},    
    pdfauthor={Author},     
    pdfsubject={Subject},   
    pdfcreator={Creator},   
    pdfproducer={Producer}, 
    pdfkeywords={keyword1, key2, key3}, 
    pdfnewwindow=true,      
    colorlinks=true,       
    linkcolor=blue ,          
    citecolor=blue ,        
    filecolor=magenta,      
    urlcolor=Aquamarine           
}
\usepackage{xcolor}
\usepackage{epsfig,color,graphicx,graphics}
\usepackage{caption}
\usepackage{amsfonts}
\usepackage{tikz}
\usepackage{siunitx}
\usepackage{booktabs,colortbl, array}
\usepackage{pgfplotstable}
\pgfplotsset{compat=1.8}

\definecolor{rulecolor}{RGB}{0,71,171}
\definecolor{tableheadcolor}{gray}{0.92}

\newtheorem{theorem}{Theorem}[section]
\newtheorem{lemma}[theorem]{Lemma}
\newtheorem{proposition}[theorem]{Proposition}
\newtheorem{corollary}[theorem]{Corollary}

\theoremstyle{definition}
\newtheorem{definition}[theorem]{Definition}
\newtheorem{remark}[theorem]{Remark}

\newtheorem{assumption}[theorem]{Assumption}

\newcommand{\Irm}{\mathrm{I}}

\newcommand{\Trm}{\mathrm{T}}

\newcommand{\Acal}{\mathcal{A}}

\newcommand{\Gcal}{\mathcal{G}}
\newcommand{\Hcal}{\mathcal{H}}

\newcommand{\Scal}{\mathcal{S}}

\newcommand{\Cscr}{\mathscr{C}}

\newcommand{\Fscr}{\mathscr{F}}

\newcommand{\Lscr}{\mathscr{L}}

\newcommand{\Nscr}{\mathscr{N}}

\newcommand{\Pscr}{\mathscr{P}}

\newcommand{\Abf}{\mathbf{A}}
\newcommand{\Bbf}{\mathbf{B}}
\newcommand{\Cbf}{\mathbf{C}}

\newcommand{\Ebf}{\mathbf{E}}

\newcommand{\Gbf}{\mathbf{G}}

\newcommand{\Sbf}{\mathbf{S}}

\newcommand{\Nbb}{\mathbb{N}}

\newcommand{\Rbb}{\mathbb{R}}

\DeclareMathOperator{\id}{id}

\DeclareMathOperator{\diverg}{div}

\DeclareMathOperator{\dist}{dist}

\newcommand{\set}[2]{\left\{\, #1 \  \textup{\textbf{:}}\  #2 \,\right\}}

\newcommand{\N}{\mathbb{N}}
\newcommand{\R}{\mathbb{R}}

\newcommand{\spt}{\mathrm{spt}}

\newcommand{\Sing}{\mathrm{Sing}}

\newcommand{\eps}{\epsilon}

\newcommand{\Ebb}{\mathbb{E}}

\renewcommand{\eps}{\varepsilon}

\makeatletter
\renewcommand*\env@matrix[1][*\c@MaxMatrixCols c]{%
    \hskip -\arraycolsep
    \let\@ifnextchar\new@ifnextchar
    \array{#1}}
\makeatother

\DeclareMathOperator{\Lip}{Lip}

\newcommand{\mres}{\mathbin{\vrule height 1.6ex depth 0pt width
        0.13ex\vrule height 0.13ex depth 0pt width 1.3ex}}

\def\vint_#1{\mathchoice%
    {\mathop{\kern 0.2em\vrule width 0.6em height 0.69678ex depth -0.58065ex
            \kern -0.8em \intop}\nolimits_{\kern -0.4em#1}}%
    {\mathop{\kern 0.1em\vrule width 0.5em height 0.69678ex depth -0.60387ex
            \kern -0.6em \intop}\nolimits_{#1}}%
    {\mathop{\kern 0.1em\vrule width 0.5em height 0.69678ex depth -0.60387ex
            \kern -0.6em \intop}\nolimits_{#1}}%
    {\mathop{\kern 0.1em\vrule width 0.5em height 0.69678ex depth -0.60387ex
            \kern -0.6em \intop}\nolimits_{#1}}}
\makeatletter
\newcommand*{\RangeX}{%
    {%
        \mathpalette\@RangeOf{X}%
    }%
}
\newcommand*{\@RangeOf}[2]{%
    \sbox0{$\m@th#1\mathsf{#2}$}%
    \mathsf{#2}%
    \kern-\wd0 %
    \mkern2.75mu\relax
    \nonscript\mkern.25mu\relax
    \mathsf{#2}%
}
\makeatother

\newcommand{\aveint}[2]{\mathchoice%
    {\mathop{\kern 0.2em\vrule width 0.6em height 0.69678ex depth -0.58065ex
            \kern -0.8em \intop}\nolimits_{\kern -0.45em#1}^{#2}}%
    {\mathop{\kern 0.1em\vrule width 0.5em height 0.69678ex depth -0.60387ex
            \kern -0.6em \intop}\nolimits_{#1}^{#2}}%
    {\mathop{\kern 0.1em\vrule width 0.5em height 0.69678ex depth -0.60387ex
            \kern -0.6em \intop}\nolimits_{#1}^{#2}}%
    {\mathop{\kern 0.1em\vrule width 0.5em height 0.69678ex depth -0.60387ex
            \kern -0.6em \intop}\nolimits_{#1}^{#2}}}
\newcommand{\flatS}{\mathfrak{F}}

\newcommand\res{\mathop{\hbox{\vrule height 7pt width .3pt depth 0pt\vrule height .3pt width 5pt depth 0pt}}\nolimits}

\title[Frequency 1 Flat singular points and $\Hcal^{m-2}$-a.e uniqueness of tangent cones]{The fine structure of the singular set of area-minimizing integral currents III: Frequency 1 flat singular points and $\Hcal^{m-2}$-a.e uniqueness of tangent cones}

\author[C. De Lellis]{Camillo De Lellis}
\address{School of Mathematics, Institute for Advanced Study, 1 Einstein Dr., Princeton NJ 08540, USA}
\email{camillo.delellis@ias.edu}

\author[P. Minter]{Paul Minter}
\address{Department of Mathematics, Fine Hall, Princeton University, Washington Road, Princeton, NJ, 08540, USA; School of Mathematics, Institute for Advanced Study, 1 Einstein Dr., Princeton, NJ, 08540, USA.}
\email{pm6978@princeton.edu\textnormal{,} pminter@ias.edu}

\author[A. Skorobogatova]{Anna Skorobogatova}
\address{Department of Mathematics, Fine Hall, Princeton University, Washington Road, Princeton, NJ 08540, USA}
\email{as110@princeton.edu}

\begin{document}

\maketitle

\begin{abstract}
    We consider an area-minimizing integral current $T$ of codimension higher than $1$ in a smooth Riemannian manifold $\Sigma$. We prove that $T$ has a unique tangent cone, which is a superposition of planes, at $\Hcal^{m-2}$-a.e. point in its support. In combination with works of the first and third authors, we conclude that the singular set of $T$ is countably $(m-2)$-rectifiable. The techniques in the present work can be seen as a counterpart for area-minimizers, in arbitrary codimension, to those developed by Simon (\cite{Simon_cylindrical}) for multiplicity one classes of minimal surfaces and Wickramasekera (\cite{W14_annals}) for stable minimal hypersurfaces.
\end{abstract}

\tableofcontents


\section{Introduction and main results}
Let $T$ be an $m$-dimensional integral current in a complete smooth $(m+\bar n)$-dimensional Riemannian manifold $\Sigma$. We assume that $T$ is area-minimizing in some (relatively) open $\Omega\subset \Sigma$, i.e. 
\[
 \mathbf{M} (T) \leq \mathbf{M} (T+\partial S)
\]
for any $(m+1)$-dimensional integral current $S$ supported in $\Omega$. The (\textit{interior}) \textit{regular set} ${\rm Reg} (T)$ is the set of points $p\in {\rm spt} (T)\cap \Omega\setminus \spt (\partial T)$ for which there is an open ball $\Bbf$ containing $p$, a regular orientable minimal surface $\Lambda\subset \Bbf$ without boundary in $\Bbf$, and an integer $Q\in \mathbb N$ such that $T \res \Bbf = Q \llbracket \Lambda \rrbracket$. The (\textit{interior}) \textit{singular set} ${\rm Sing} (T)$ is then given by the complement in $\Omega \cap {\rm spt}\, (T)$ of ${\rm Reg} (T)$. This is the third of three papers (the others being \cites{DLSk1,DLSk2}) devoted to proving the following theorem.

\begin{theorem}\label{t:big-one}
Let $T$ be an $m$-dimensional area-minimizing integral current in a $C^{3,\kappa_0}$ complete Riemannian manifold of dimension $m+\bar{n}\geq m+2$, with $\kappa_0>0$. Then, $\Sing (T)$ is $(m-2)$-rectifiable and $T$ has a unique tangent cone at $\mathcal{H}^{m-2}$-a.e. $q\in \Sing (T)$. 
\end{theorem}

We refer to the first work \cite{DLSk1} for the historical context and the motivation of our study, we just recall here that the regularity of ${\rm Sing}(T)$ given by Theorem \ref{t:big-one} is close to optimal due to the recent work \cite{Liu}, which shows that $\Sing(T)$ can be precised to be any fractal subset of a smooth oriented $(m-2)$-dimensional manifold. The only foreseeable improvement of Theorem \ref{t:big-one} is to show that the $(m-2)$-dimensional Hausdorff measure of $\Sing (T)$ is locally finite in $\spt (T)\setminus \spt (\partial T)$. We also recall that this is known in the special case $m=2$ by \cite{SXChang} (cf. also \cites{DLSS1,DLSS2,DLSS3}) and that the statement about the uniqueness of tangent cones is covered, in that case, by \cite{White}. 

Recall that, following Almgren's stratification theorem, we can subdivide $\Sing (T)$ into the disjoint union of 
\begin{itemize}
    \item the subset $\mathcal{S}^{(m-2)} (T)$ of points $p$ at which any tangent cone to $T$ has at most $m-2$ linearly independent directions of translation invariance;
    \item the remaining set $\Sing (T)\setminus \mathcal{S}^{(m-2)} (T)$ of those singular points at which at least one tangent cone is a flat plane (counted with some integer multiplicity $Q\geq 1$). 
\end{itemize}
Consistently with \cite{DLSk1} we use the notation $\flatS (T)$ for the latter set and we will call its elements {\em flat singular points}.
As a consequence of the general theory of Naber and Valtorta (\cite{NV_varifolds}) it follows that $\mathcal{S}^{(m-2)} (T)$ is $(m-2)$-rectifiable, therefore the main novelty of our work is to establish the rectifiability of $\flatS (T)$. In fact the techniques of the present paper, which can be seen as a counterpart in the higher codimension case (and for area-minimizing integral currents) of the seminal works by L.~Simon (\cite{Simon_cylindrical}) and N.~Wickramasekera (\cite{W14_annals}), can be used to derive an independent proof that $\mathcal{S}^{(m-2)} (T)$ is rectifiable. However the ideas in \cites{NV_Annals,NV_varifolds} are still a crucial ingredient in \cite{DLSk2}. 

\medskip

By the constancy theorem, the density $\Theta (T, p)$ at any $p\in \flatS (T)$ is a positive integer $Q$, which moreover obeys $Q>1$ by Allard's Regularity Theorem. We can therefore stratify $\flatS (T)$ as $\bigcup_{Q=2}^\infty \flatS_Q(T)$, where 
\[
    \flatS_Q (T) :=\{p\in \flatS (T): \Theta (T, p)=Q\}\, .
\]
According to \cite{DLSk1} we can further subdivide each $\flatS_Q (T)$ by introducing a suitable function 
\[
\flatS (T) \ni p\mapsto {\Irm} (T, p)\in [1, \infty)\, .
\]
Loosely speaking, $\Irm(T,p)$ detects the infinitesimal homogeneity of the ``singular behavior of $T$'' around $p$. A good illustration of this number is given in \cite{DLSk1}*{Example 1.2} in the case of classical holomorphic curves of $\mathbb{C}^2$, which by Federer's theorem are area-minimizing integral $2$-dimensional currents in $\mathbb R^4$. Consider $\Lambda := \{(w-h(z))^Q = z^p k (z) : (z,w)\in \mathbb C^2\}$ and require that
\begin{itemize}
\item $p>Q\geq 2$ are coprime integers;
\item $h$ and $k$ are holomorphic functions;
\item $k (0)\neq 0$.
\end{itemize}
If $T=\llbracket \Lambda \rrbracket$ in $\mathbb R^4 \cong \mathbb C^2$, then $\Irm (T, 0)= p/Q$. Given however the lack of precise information about the singular behavior of a general $m$-dimensional area-minimizing integral current, the actual definition of $\Irm (T, p)$ is rather involved: note for instance that we do not know that the tangent cone to $T$ at $p\in \flatS$ is unique, or even that all tangent cones are flat.   
In the papers \cites{DLSk1,DLSk2}, the first and third authors prove that:
\begin{theorem}\label{thm:I>1}
Let $T$ be an $m$-dimensional area-minimizing integral current in a $C^{3,\kappa_0}$ complete Riemannian manifold of dimension $m+\bar{n}\geq m+2$, with $\kappa_0>0$. Then $\flatS_{Q, >1} := \flatS_Q (T) \cap \{\Irm (T, \cdot) >1\}$ is $(m-2)$-rectifiable and the tangent cone is unique at every $p\in \flatS_{Q,>1}$.
\end{theorem}

In the present paper, we handle the rectifiability question for the remaining part of $\flatS_Q(T)$, as well as the $\mathcal{H}^{m-2}$-a.e. uniqueness of tangent cone question in the remaining portion of the singular set. More precisely, we prove the following:

\begin{theorem}\label{con:stronger}
Let $T$ be as in Theorem \ref{thm:I>1}. The following holds:
\begin{itemize}
    \item[(i)] The set $\flatS_{Q,1}(T):= \flatS_Q(T) \cap \{ \Irm(T,\cdot) = 1\}$ is $\Hcal^{m-2}$-null;
    \item[(ii)] $T$ has a unique tangent cone at $\mathcal{H}^{m-2}$-a.e. $p\in \mathcal{S}^{(m-2)} (T)$.
\end{itemize}
\end{theorem}

\subsection{Comparison with the works of Krummel and Wickramasekera} While we were completing this and the two works \cites{DLSk1,DLSk2} leading to our proof of Theorem \ref{t:big-one}, we have learned that in the works \cites{KW1,KW2,KW3}, Krummel and Wickramasekera arrived independently at a program that shows the same final result; we refer to the introduction of \cite{DLSk1} for a more general comparison between the two programs.

The present paper and \cite{KW2} are in fact the two works which are most similar. Both \cite{KW2} and this paper rely in an essential way on an new height bound: we refer to Theorem \ref{thm:main-estimate} in the first part of this paper for our precise statement. Indeed, the only difference between the two bounds seems to be that ours is stated in the more general setting of an arbitrary ambient manifold which satisfies some mild regularity assumptions, a setting which we believe can be reached equally well in \cite{KW2} at the price of some more technical work.

Building upon this height bound, both the present work and \cite{KW2} rely in a fundamental way on the estimates of \cite{Simon_cylindrical}, which are in turn used to perform a suitable blow-up analysis to prove the decay theorem which is the subject of the second part of this work. To handle the situation in which the current is very close to a plane, but much closer to a cone with $m-2$ linearly independent directions of translational invariance, both \cite{KW2} and the present work also rely on important ideas introduced by Wickramasekera in \cite{W14_annals}. 

Finally, the ideas of \cite{Simon_cylindrical} are also used in a substantial way in both papers to prove that the flat singular points at which the sheets of the surface meet with order of contact 1, namely $\flatS_{Q,1} (T)$, is $\mathcal{H}^{m-2}$-negligible. The result achieved in \cites{KW1,KW2} is in fact stronger: using our notation they actually show that $\flatS_{Q,<1+\delta}$ is $\mathcal{H}^{m-2}$-negligible for a sufficiently small $\delta= \delta (Q,m,n)$, while $\flatS_{Q,\geq 1+\delta}$ is (relatively) closed in any ball $\Bbf_r (x)$ where the current is sufficiently close to a multiplicity-$Q$ flat plane. 

While we do not pursue such a finer result here, we also believe with some additional work we can deliver these very conclusions. In fact the closedness of $\flatS_{Q,\geq 1+\delta}$ is a consequence, in \cite{KW1}, of the almost-monotonicity of the planar frequency function and a similar almost-monotonicity holds also for the Almgren frequency function relative to the center manifolds as used in \cites{DLSk1,DLSk2}. 
As for the fact that $\flatS_{Q,<1+\delta}$ is $\mathcal{H}^{m-2}$-negligible, we believe that it can be achieved with appropriate modifications of the arguments in the final part of this paper, given that at these points all the coarse blow-ups will be homogeneous with degree $d\in [1, 1+\delta]$, and thus, for a sufficiently small $\delta$, close to $1$-homogeneous Dir-minimizers. Indeed, it is an interesting question whether there is a frequency gap (depending on the number of values of the Dirichlet minimizer) for Dirichlet minimizers. Namely, if there exists $\delta = \delta(m,n,Q)>0$ such that any homogeneous $Q$-valued Dirichlet minimizer of degree $\alpha \in [1,1+\delta)$ is in fact homogeneous of degree 1. Such a result would then immediately imply that in fact the sets $\flatS_{Q,1}$ and $\flatS_{Q,<1+\delta}$ coincide.

\section{Preliminaries and notation}
\subsection{Notation}
Throughout this work, $C, C_0, C_1, \dots$ will denote constants which depend only on $m,n,\bar n, Q$. Constants depending on other parameters will typically be denoted by 
\[
\bar{C}, \bar{C}_1,\bar{C}_2,\dots \, ,
\]
with dependencies given. For $q\in \spt(T)$, the currents $T_{q,r}$ will denote the dilations $(\iota_{q,r})_\sharp T$, where $\iota_{q,r} (x):= \frac{x-q}{r}$. For $p\in\spt(T)$, $\Bbf_r(p)$ denotes the open $(m+n)$-dimensional Euclidean ball of radius $r$ centered at $p$, while $B_r(p,\pi)$ denotes the (open) $m$-dimensional disk $\Bbf_r(p)\cap \pi$ of radius $r$ centered at $p$ in the $m$-dimensional plane $\pi\subset\R^{m+\bar n}$ passing through $p$. $\Cbf_r(p,\pi)$ denotes the $(m+n)$-dimensional cylinder $B_r(p,\pi)\times \pi^\perp$ of radius $r$ centered at $p$. We let $\mathbf{p}_{\pi}: \mathbb R^{m+n}\to \pi$ denote the orthogonal projection onto $\pi$, while $\mathbf{p}_{\pi}^\perp$ denotes the orthogonal projection onto $\pi^\perp$. The plane $\pi$ is omitted if clear from the context; if the center $p$ is omitted, then it is assumed to be the origin. For example, $\Cbf_r := \Cbf_r(0,\pi)$ for $r>0$ if the choice of $\pi$ is clear from the context. $\|T\|$ denotes the mass measure induced by $T$, while $\omega_m$ denotes the $m$-dimensional Hausdorff measure of the $m$-dimensional unit disk $B_1(\pi)$. The Hausdorff distance between two subsets $A$ and $B$ of $\R^{m+\bar n}$ will be denoted by $\dist (A,B)$. 

The scalar product between vectors $v,w\in \mathbb R^{m+n}$ is denoted by $v\cdot w$ and likewise for product of matrices we will use $A\cdot B$. $|\cdot|$ will denote the Euclidean norm of vectors and the Hilbert-Schmidt norm of matrices.  
$\mathcal{H}^s$ will denote the Hausdorff $s$-dimensional measure, while in the particular instance of subsets $E$ of $\mathbb R^m$ we will use the shorthand $|E|$ for their Lebesgue measure. This convention will often also be used for the $\mathcal{H}^m$ measures of subsets $E$ of affine $m$-dimensional subspaces of $\mathbb R^{m+n}$. 

We collect here a table of additional symbols used repeatedly throughout the paper, for the reader's convenience. 
{\Small
\begin{align*}
 	& r, \rho, s, t && \text{typically denote radii}\\
    & i,j,k && \text{indices}\\
 	& \alpha, \beta, \pi && \text{$m$-dimensional planes}\\
    & \varpi && \text{$(m+\bar{n})$-dimensional plane} \\
        &\varepsilon, \delta, \eta && \text{small numbers, with $\varepsilon$ the smallest in hierarchy}\\
        &\gamma,\kappa,\mu && \text{exponents}\\
        &\varsigma, \sigma, \tau, \varkappa && \text{parameters}\\
        &\phi,\theta,\vartheta && \text{angles}\\
        &\varphi,\psi,\chi && \text{test functions}\\
        & f,g,h,u,v,w && \text{functions, with $f$, $u$, $v$ and $w$ typically denoting multi-valued approximations}\\
        & \Psi, \Sigma && \text{$\Psi$ the parameterization of the ambient manifold $\Sigma$}\\
        & \Sigma_{p,r} &&\text{the rescaled manifold $\iota_{p,r}(\Sigma)$} \\
        & S,T && \text{currents}\\
        & p,q && \text{points in $\R^{m+n}$} \\
        & x,y,z,\xi,\zeta && \text{variables (typically in $m$-dimensional subspaces of $\R^{m+n}$)} \\
        &\mathbf{p}, \ \mathbf{p}^\perp &&\text{orthogonal projection, projection to orthogonal complement, respectively} \\        
        & \mathbf{1}_E && \text{indicator function of the set $E$}\\
        & \mathbf{A} && \text{the $L^\infty$ norm of the second fundamental form}\\
        &\Theta(T,p) &&\text{the $m$-dimensional Hausdorff density of $T$ at a point $p$;} \\
        & \Sing(T), \flatS (T) && \text{singular sets of $T$, with $\flatS (T)$ the flat singularities}\\
        & \flatS_Q(T) && \text{flat singularities of $T$ where the density of $T$ is $Q$}\\
        & \flatS_{Q,1}(T) && \text{points in $\flatS(T)$ with $\Irm(T,\cdot)=1$}\\
        & L, \ell(L) && L \text{ a cube, $\ell(L)$ half the side length}\\
        & A,B && \text{linear maps}\\
        & M && \text{balancing constant (c.f. Definition \ref{d:balanced})}\\
        & X && \text{vector field}\\
	& \mathbf{S} && \text{$(m-2)$-invariant cones that are superpositions of $m$-planes} \\
    & N && \text{natural number, typically denoting the number of planes in $\Sbf$}\\
        & V && \text{spines of cones} \\
    & \mathscr{P} && \text{set of $m$-dimensional planes}\\
    & \mathscr{C} && \text{set of $(m-2)$-invariant cones that are superpositions of $m$-planes}\\
 	& \mathbf{E}^p(T,\Bbf) && \text{planar excess of $T$ in the $(m+n)$-dimensional ball $\Bbf$} \\          
	& \hat{\mathbf{E}}(T,\Sbf,\Bbf), \hat{\mathbf{E}}(\Sbf,T,\Bbf) && \text{one-sided $L^2$ conical excess in $\Bbf$ ($T$ close to $\Sbf$, $\Sbf$ close to $T$, resp.)} \\
	& \mathbb{E}(T,\Sbf,\Bbf) && \text{double-sided conical excess in $\Bbf$} \\ 
 	& B_a(V) && \text{fixed tubular neighbourhood of radius $a$ of the spine $V$ being removed from $\Bbf_1$} \\  
 	& \boldsymbol{\sigma}(\Sbf) && \text{minimal pairwise Hausdorff distance between the planes in $\Sbf$ in $\Bbf_1$} \\ 
 	& \boldsymbol{\mu}(\Sbf) && \text{maximal pairwise Hausdorff distance between the planes in $\Sbf$ in $\Bbf_1$} \\
	& \Acal_Q(\R^n) && \text{the space of $Q$-tuples of vectors in $\R^n$ (c.f. \cite{DLS_MAMS})} \\
\end{align*}
}

Since our statements are local and invariant under dilations and translations, we will work with the following underlying assumption throughout.
\begin{assumption}\label{a:main} $m\geq 3$ and $n\geq \bar n \geq 2$ are integers.
	$T$ is an $m$-dimensional integral current in $\Sigma\cap \Bbf_{7\sqrt{m}}$ with $\partial T\mres \Bbf_{7\sqrt{m}} = 0$, where $\Sigma$ is an $(m + \bar{n})$-dimensional embedded submanifold of $\mathbb R^{m+n} = \mathbb R^{m+\bar n + l}$ of class $C^{3,\kappa_0}$ with $\kappa_0>0$. $T$ is area-minimizing in $\Sigma\cap \Bbf_{7\sqrt{m}}$. $\Sigma \cap \Bbf_{7\sqrt{m}}(p)$ is the graph of a $C^{3,\kappa_0}$ function $\Psi_p : \Trm_p\Sigma \cap \Bbf_{7\sqrt{m}}(p) \to \Trm_p\Sigma^\perp$ for every $p \in \Sigma\cap\Bbf_{7\sqrt{m}}$. Moreover
    \[
        \boldsymbol{c}(\Sigma)\coloneqq\sup_{p \in \Sigma \cap \Bbf_{7\sqrt{m}}}\|D\Psi_p\|_{C^{2,\kappa_0}} \leq \bar{\eps},
    \]
    where $\bar\eps \leq 1$ is a small positive constant which will be specified later.
\end{assumption}
    
Following the notation of \cite{DLSk1}, we let $\mathbf{A}_\Sigma$ denote the $C^0$ norm of the second fundamental form $A_\Sigma$ of $\Sigma$ in $\Bbf_{7\sqrt{m}}$. In particular, under Assumption \ref{a:main}, we have
\[
        \Abf_\Sigma \coloneqq \|A_\Sigma\|_{C^0(\Sigma\cap \Bbf_{7\sqrt{m}})} \leq C_0\boldsymbol{c} (\Sigma) \leq C_0 \bar{\eps}.
\]
We will often drop the subscript as the underlying ambient manifold will mostly be clear from the context. 

We recall that the \emph{oriented tilt-excess} of $T$ in $\Cbf_r(p,\pi_0)$ relative to an $m$-dimensional oriented plane $\pi$ is defined by
\[
\mathbf{E} (T, \mathbf{C}_r (p,\pi_0), \pi) := \frac{1}{2\omega_m r^m} \int_{\mathbf{C}_r (p,\pi_0)} |\vec{T} (x)- \vec{\pi} (x)|^2\, d\|T\| (x),
\]
while
\[
\mathbf{E} (T, \mathbf{C}_r (p,\pi_0)) := \min_{\pi\subset T_p \Sigma} \mathbf{E} (T, \mathbf{C}_r (p,\pi_0), \pi)\, 
\]
where the minimum is taken over all $m$-dimensional oriented planes $\pi\subset T_p\Sigma$ (identified with their corresponding planes in $\mathbb{R}^{n+m}$). The oriented tilt-excess of $T$ in $\Bbf_r(p)$ relative to an $m$-dimensional oriented plane $\pi$ and the optimal oriented tilt-excess in $\Bbf_r(p)$, denoted respectively by $\Ebf(T,\Bbf_r(p),\pi)$ and $\Ebf(T,\Bbf_r(p))$, are defined analogously.
The tilt-excess is morally a planar $L^2$ ``gradient" excess. We will shortly also introduce a notion of planar $L^2$ ``height" excess.

We also refer to \cite{DLSk1}*{Section 3.1} for the notion of a \emph{coarse blow-up} of $T$ at $p$, which, roughly speaking, under the assumption that $r^2 \mathbf{A}^2 \ll\mathbf{E} (T, \mathbf{B}_{6\sqrt{m} r}(p))\ll 1$ along a given family of scales $r$, gives rise to a Dir-minimizing $Q$-valued function over $B_1 (0, \pi_0)$ for some $m$-dimensional plane $\pi_0$ as a subsequential normalized limit. The graph of this Dir-minimizer approximates the rescaled current $T_{p, \rho}$ in the domain $\mathbf{C}_1 (0, \pi_0)\cap \mathbf{B}_{4}$ with a mass error which is $o (\mathbf{E} (T, \mathbf{B}_{6\sqrt{m}}))$, along with other related error estimates which are detailed in \cite{DLSk1}*{Section 3.1} and \cite{DLS14Lp}. We additionally refer the reader to \cite{DLSk1} for all other relevant notation and terminology relating to coarse blow-ups.

The properties of $\flatS_{Q,1} (T)$ that will be most useful to us here are contained within the following proposition, which is a consequence of the analysis in \cite{DLSk1} (more precisely, see (9), Proposition 4.1 and Corollary 4.3 therein).

\begin{proposition}\label{p:coarse-blow-ups}
Let $T$, $\Sigma$ and $\Abf$ be as in Assumption \ref{a:main}. For every $p\in \flatS_{Q,1} (T)$ and any sequence $r_k\downarrow 0$ with the property that $\mathbf{E}_k := \mathbf{E} (T, \mathbf{B}_{6\sqrt{m} r_k}(p))\downarrow 0$ the following holds:
\begin{itemize}
\item[(i)] For $\Abf_k\coloneqq \Abf_{\Sigma_{p,r_k}}$ we have $\lim_{k\to \infty} \mathbf{E}_k^{-1} r_k^{2-2\delta_2} \mathbf{A}_k^2  = 0$;
\item[(ii)] Any coarse blow-up $\bar f$ generated by a subsequence of $\{r_k\}$ has positive Dirichlet energy, is $1$-homogeneous, and satisfies $\boldsymbol{\eta}\circ \bar f \equiv 0$.
\end{itemize}
\end{proposition}

\subsection{Main decay theorem}
To prove Theorem \ref{t:big-one}, we will show that a  key excess decay theorem holds under the assumption that $T$ is much closer to a cone with exactly $m-2$ directions of translation invariance than it is to any $m$-dimensional plane. Before coming to the statement of this theorem, let us first introduce suitable notions of $L^2$ height excess of $T$ relative to such cones. We begin by defining the cones of interest.

\begin{definition}\label{def:cones}
For every fixed integer $Q\geq 2$ we denote by $\mathscr{C} (Q)$ those subsets of $\mathbb R^{m+n}$ which are unions of $N\leq Q$ $m$-dimensional planes $\pi_1, \ldots, \pi_N$ satisfying the following properties:
\begin{itemize}
    \item[(i)] $\pi_i \cap \pi_j$ is the same $(m-2)$-dimensional plane $V$ for every pair $(i,j)$ with $i<j$;
    \item[(ii)] Each plane $\pi_i$ is contained in the same $(m+\bar n)$-dimensional plane $\varpi$.
\end{itemize}
If $p\in \Sigma$, then $\mathscr{C} (Q, p)$ will denote the subset of $\mathscr{C} (Q)$ for which $\varpi = T_p \Sigma$.

$\mathscr{P}$ and $\mathscr{P} (p)$ will denote the subset of those elements of $\mathscr{C} (Q)$ and $\mathscr{C} (Q,p)$ respectively which consist of a single plane; namely, with $N=1$. For $\mathbf{S}\in \mathscr{C} (Q)\setminus \mathscr{P}$, the $(m-2)$-dimensional plane $V$ described in (i) above is referred to as the {\em spine} of $\mathbf{S}$ and will often be denoted by $V (\mathbf{S})$.
\end{definition}

We are now in a position to introduce the \emph{conical $L^2$ height excess} between $T$ and elements in $\Cscr(Q)$.
\begin{definition}\label{def:L2_height_excess}
	Given a ball $\Bbf_r(q) \subset \R^{m+n}$ and a cone $\mathbf{S}\in \mathscr{C} (Q)$, we define the \emph{one-sided conical $L^2$ height excess of $T$ relative to $\Sbf$ in $\Bbf_r(q)$}, denoted $\hat{\Ebf}(T, \mathbf{S}, \Bbf_r(q))$, by
	\[
		\hat{\Ebf}(T, \mathbf{S}, \Bbf_r(q)) \coloneqq \frac{1}{r^{m+2}} \int_{\Bbf_r (q)} \dist^2 (p, \mathbf{S})\, d\|T\|(p).
	\]
At the risk of abusing notation, we further define the corresponding \emph{reverse one-sided excess} as
 \[
\hat{\Ebf} (\mathbf{S}, T, \Bbf_r (q)) \coloneqq \frac{1}{r^{m+2}}\int_{\Bbf_r (q)\cap \mathbf{S}\setminus \Bbf_{ar} (V (\mathbf{S}))}
\dist^2 (x, {\rm spt}\, (T))\, d\mathcal{H}^m (x)\, ,
\]
where $a=a(m)$ is a dimensional constant, to be specified later. We subsequently define the \emph{two-sided conical $L^2$ height excess} as 
\[
    \mathbb{E} (T, \mathbf{S}, \Bbf_r (q)) :=
\hat{\Ebf} (T, \mathbf{S}, \Bbf_r (q)) + \hat{\Ebf} (\mathbf{S}, T, \Bbf_r (q))\, .
\]
We finally introduce the \emph{planar $L^2$ height excess} which is given by
\[
\Ebf^p (T, \Bbf_r (q)) = \min_{\pi\in \mathscr{P} (q)} \hat{\Ebf} (T, \pi, \Bbf_r (q))\, .
\]
\end{definition}

Let us now state our main excess decay theorem. This is similar in spirit to the excess decay theorem originally seen in \cite{Simon_cylindrical}*{Lemma 1}, however a crucial difference with the present setting is that in \cite{Simon_cylindrical}*{Lemma 1} there is a built-in multiplicity one assumption which in particular rules out branch point singularities a priori. Our excess decay theorem in contrast is in a higher multiplicity setting, and therefore is closer -- and indeed our proof follows a similar pattern -- to the excess decay theorems first seen in \cite{W14_annals}*{Section 13} (see also \cite{KW}*{Lemma 5.6 and Lemma 12.1} and \cite{MW}*{Theorem 2.1}).

\begin{theorem}[Excess Decay Theorem]\label{c:decay}
For every $Q,m,n$, $\bar n$, and $\varsigma>0$, there are positive constants $\varepsilon_0 = \varepsilon_0(Q,m,n,\bar n, \varsigma) \leq \frac{1}{2}$, $r_0 = r_0(Q,m,n,\bar n, \varsigma) \leq \frac{1}{2}$ and $C = C(Q,m,n,\bar n)$ with the following property. Assume that 
\begin{itemize}
\item[(i)] $T$ and $\Sigma$ are as in Assumption \ref{a:main};
\item[(ii)] $\|T\| (\Bbf_1) \leq (Q+\frac{1}{2}) \omega_m$;
\item[(iii)] There is $\mathbf{S}\in \mathscr{C} (Q, 0)\setminus\Pscr(0)$ such that 
\begin{equation}\label{e:smallness}
\mathbb{E} (T, \mathbf{S}, \Bbf_1) \leq \varepsilon_0^2 \mathbf{E}^p (T, \Bbf_1)\, 
\end{equation}
and 
\begin{equation}\label{e:no-gaps}
\Bbf_{ \varepsilon_0} (\xi) \cap \{p: \Theta (T,p)\geq Q\}\neq \emptyset \qquad \forall \xi \in V (\mathbf{S})\cap \Bbf_{1/2}\, ;
\end{equation}
\item[(iv)] $\mathbf{A}^2 \leq \varepsilon_0^2 \mathbb{E} (T, \mathbf{S}', \Bbf_1)$ for {\em any} $\mathbf{S}'\in \mathscr{C} (Q, 0)$.
\end{itemize}
Then there is a $\mathbf{S}'\in \mathscr{C} (Q,0) \setminus \mathscr{P} (0)$ such that 
\begin{enumerate}
    \item [\textnormal{(a)}] $\mathbb{E} (T, \mathbf{S}', \Bbf_{r_0}) \leq \varsigma \mathbb{E} (T, \mathbf{S}, \Bbf_1)\,$ \label{e:decay} \\
    \item [\textnormal{(b)}] $\dfrac{\mathbb{E} (T, \mathbf{S}', \Bbf_{r_0})}{\mathbf{E}^p (T, \Bbf_{r_0})} 
\leq 2 \varsigma \dfrac{\mathbb{E} (T, \mathbf{S}, \Bbf_1)}{\mathbf{E}^p (T, \Bbf_1)}$ \\
    \item [\textnormal{(c)}] $\dist^2 (\Sbf^\prime \cap \Bbf_1,\Sbf\cap \Bbf_1) \leq C \mathbb{E} (T, \mathbf{S}, \Bbf_1)$\label{e:cone-change}
    \item[\textnormal{(d)}] $\dist^2 (V (\mathbf{S}) \cap \Bbf_1, V (\mathbf{S}')\cap \Bbf_1) \leq C \dfrac{\mathbb{E}(T,\mathbf{S},\Bbf_1)}{\Ebf^p(T,\Bbf_1)}$\, .\label{e:spine-change}
    \end{enumerate}
\end{theorem}

\subsection{Structure of the paper}
The majority of this paper will be dedicated to the proof of Theorem \ref{c:decay}. We begin with a refined $L^2-L^\infty$ height bound in Part \ref{p:heightbd}, which will be a key tool for the estimates in the remainder of the paper. However, we believe that this height bound is of interest in itself, and would have a number of other applications.

Part \ref{p:decay-proof} will contain the proof of the Excess Decay Theorem \ref{c:decay}. The proof consists of several key parts, modeled off the foundational works of Simon (\cite{Simon_cylindrical}) and Wickramasekera (\cite{W14_annals}); the reader familiar with these works will find many of the arguments in this part similar in spirit to those seen in these, and we will make the effort to point out the similarities and differences throughout the work to aid the reader. Section \ref{p:planes} contains some key results regarding the relative positioning and angles between pairs of planes in cones in $\Cscr(Q)\setminus\Pscr$, followed by Section \ref{p:approx}, which introduces effective graphical approximations for $T$ over such cones. In Section \ref{s:balancing}, we then use the preceding two sections to demonstrate that we may replace the initial cone $\Sbf$ in Theorem \ref{c:decay} with a \emph{balanced} cone (cf. Definition \ref{d:balanced}). In Section \ref{s:reduction} we reduce the proof of Theorem \ref{c:decay} to an a priori much weaker decay statement, in particular one where we can assume that the two-sided excess is much smaller than the \textit{minimal} angle in the cone (as opposed to the maximal angle, which is morally what \eqref{e:smallness} says). Section \ref{p:estimates} is then dedicated to the \emph{Simon estimates}, including the non-concentration estimates, at the spine of the cone in $\Cscr(Q)\setminus\Pscr$, after which, in Section \ref{p:linear} we demonstrate the excess decay conclusion at the level of the \emph{linearized problem} of multiple-valued Dirichlet minimizers. Part \ref{p:decay-proof} is then concluded with Section \ref{p:blowup}, in which we put together everything from the previous sections to conclude the proof of Theorem \ref{c:decay}.

In Part \ref{p:rect} we use Theorem \ref{c:decay} combined with a covering procedure, analogously to that done by Simon (\cite{Simon_cylindrical}), to prove Theorem \ref{t:big-one}.

\subsection*{Acknowledgments}
C.D.L. and A.S. acknowledge the support of the National Science Foundation through the grant FRG-1854147.

\part{\texorpdfstring{$L^2-L^\infty$}{L2-Linfty} Height Bound}\label{p:heightbd}
The aim of this part is to prove a generalization of Allard's tilt-excess and $L^\infty$ estimates (see \cite{Allard_72}*{Section 8}). Allard's original work, in the context of stationary varifolds, bounds the $L^\infty$ distance and $L^2$ tilt excess from a single plane with the $L^2$ distance to it. Here, we will restrict ourselves to the much smaller class of area-minimizing currents, with the additional benefit being that are able to control the $L^\infty$ distance and $L^2$ tilt excess from a \emph{finite} collection of mutually disjoint planes by the $L^2$ distance to the union of the planes. 

\section{Main statements}
For the remainder of this part, we make the following additional assumption.

\begin{assumption}\label{a:height-main}
    $T$, $\Sigma$, and $\Abf$ are as in Assumption \ref{a:main}. For some oriented $m$-dimensional plane $\pi_0\subset\R^{m+\bar n}$ passing through the origin and some positive integer $Q$, we have 
    \[
    (\mathbf{p}_{\pi_0})_\sharp T\res \Cbf_2 (0, \pi_0) = Q \llbracket B_2(\pi_0)\rrbracket\, ,
    \]
    and $\|T\|(\Cbf_2) \leq (Q+\frac{1}{2})\omega_m 2^m$.
\end{assumption}

The main result of this part is the following. 

\begin{theorem}[$L^\infty$ and tilt-excess estimates]\label{thm:main-estimate} For every $1\leq r < 2$, $Q$, and $N$, there is a positive constant $\bar{C} = \bar{C} (Q,m,n,\bar n,N,r)>0$ with the following property. Suppose that $T$, $\Sigma$, $\Abf$ and $\pi_0$ are as in Assumption \ref{a:height-main}, let $p_1, \ldots , p_N \in \pi_0^\perp$ be distinct points, and set $\boldsymbol{\pi}:= \bigcup_i p_i+\pi_0$. 
Let
\begin{equation}\label{e:L2-excess}
E := \int_{\mathbf{C}_2} \dist^2 (p, \boldsymbol{\pi}) \, d\|T\| (p)\, .
\end{equation}
Then
\begin{equation}\label{e:tilt-estimate}
\Ebf(T,\Cbf_r, \pi_0)\leq \bar{C} (E + \mathbf{A}^2)\, 
\end{equation}
and, if $E\leq 1$,
\begin{equation}\label{e:Linfty-estimate}
\spt (T) \cap \mathbf{C}_r \subset \{p : \dist (p, \boldsymbol{\pi})\leq \bar{C} (E^{1/2} + \mathbf{A})\}\, .
\end{equation}
\end{theorem}

A translation followed by a simple scaling argument clearly gives corresponding estimates when $0$ is replaced by an arbitrary center and the scales $r$ and $2$ are replaced by two arbitrary radii $\rho<R$. We additionally record the following consequence of Theorem \ref{thm:main-estimate}, which will be used frequently in the rest of the paper.

\begin{corollary}\label{c:splitting-0}
For each pair of positive integers $Q$ and $N$, there is a positive constant $\delta = \delta (Q,m,n, \bar n, N)$ with the following properties. Assume that:
\begin{itemize}
    \item[(i)] $T$, $\Sigma$, and $\Abf$ are as in Assumption \ref{a:main};
    \item[(ii)] $T$ is area-minimizing in $\Sigma$ and for some positive $r \leq \frac{1}{4}$ and $q\in \spt(T)\cap\Bbf_1$ we have $\partial T \res \mathbf{C}_{4r} (q) = 0$, $(\mathbf{p}_{\pi_0})_\sharp T = Q \llbracket B_{4r} (q)\rrbracket$, and $\|T\| (\mathbf{C}_{2r} (q)) \leq \omega_m (Q+\frac{1}{2}) (2r)^m$;
    \item[(iii)] $p_1, \ldots, p_N\in \R^{m+n}$ are distinct points with $\mathbf{p}_{\pi_0} (p_i)= q$ and $\varkappa:=\min \{|p_i-p_j|: i<j\}$;
    \item[(iv)] $\pi_1, \ldots, \pi_N$ are oriented planes passing through the origin with 
    \begin{equation}\label{e:smallness-tilted-pi-1}
    \tau \coloneqq \max_i |\pi_i-\pi_0|\leq \delta \min\{1, r^{-1} \varkappa\}\, ;
    \end{equation}
    \item[(v)] Upon setting $\boldsymbol{\pi} = \bigcup_i (p_i+ \pi_i)$, we have 
    \begin{align}
    & (r\mathbf{A})^2 + (2r)^{-m-2} \int_{\mathbf{C}_{2r} (q)} \dist^2 (p, \boldsymbol{\pi}) d\|T\| 
    \leq \delta^2 \min \{1, r^{-2}\varkappa^2\}\, .\label{e:smallness-tilted-pi-2}
    \end{align}
\end{itemize}
Then $T\res \mathbf{C}_r (q) = \sum_{i=1}^N T_i$ where
\begin{itemize}
\item[(a)] Each $T_i$ is an integral current with $\partial T_i \res \mathbf{C}_r (q) = 0$;
\item[(b)] $\dist (q, \boldsymbol{\pi}) = \dist (q, p_i+\pi_i)$ for each $q\in \spt (T_i)$;
\item[(c)] $(\mathbf{p}_{\pi_0})_\sharp T_i = Q_i \llbracket B_r (q)\rrbracket$ for some non-negative integer $Q_i$.
\end{itemize}
\end{corollary}

\subsection{Proof of Corollary \ref{c:splitting-0}}
First of all, by scaling we can assume that $r=1$ and up to translation we can assume $q=0$. We then introduce
\[
E := \int_{\mathbf{C}_2} \dist^2 (p, \boldsymbol{\pi}) d\|T\| (p)\, .
\]
Next let $\bar{\boldsymbol{\pi}}:= \bigcup_i (p_i+\pi_0)$ and observe that 
\[
\bar{E}:= \int_{\mathbf{C}_2} \dist^2 (p, \bar{\boldsymbol{\pi}}) d\|T\| (p) \leq C E + C \tau^2 \leq C \delta^2 \min \{1, \varkappa^2\}\, ,
\]
for a constant $C = C(m,n,\bar{n},Q)$ which we stress is independent of $\delta$ and $\varkappa$. Since we also have $\mathbf{A}^2 \leq \delta^2 \varkappa^2$, we can apply Theorem \ref{thm:main-estimate} to conclude that 
\[
\spt (T) \cap \mathbf{C}_{3/2} \subset \{p:\dist (p, \bar{\boldsymbol{\pi}}) \leq C_1 \delta \varkappa\}\, ,
\]
where $C_1 = C_1 (m,n,\bar{n},Q,N)$ is again independent of $\delta$ and $\varkappa$. In particular if we choose $\delta = \delta(m,n,\bar{n},Q,N)>0$ so that $C_1\delta < 1/4$, we have
\[
\spt (T) \cap \mathbf{C}_{3/2} \subset \{p:\dist (p, \bar{\boldsymbol{\pi}}) < 4^{-1} \varkappa\}\, .
\]
Due to the definition of $\varkappa$, $\{p:\dist (p, \bar{\boldsymbol{\pi}}) < 4^{-1} \varkappa\}$ consists of $N$ disjoint open sets $U_i:=\{p: \dist (p-p_i, \pi_0) < \varkappa/4\}$. We then set $T_i := T\res \mathbf{C}_{3/2}\cap U_i$; clearly each $T_i$ is integral, has no boundary in $\mathbf{C}_{3/2}$, and $\sum_i T_i = T \res \mathbf{C}_{3/2}$. It follows also that the $T_i$'s have disjoint support and so each of them is area-minimizing. If we further take $\delta< 1/8$ such that $C_1\delta < 1/8$, we can in addition ensure that for each point $p\in U_i$, 
\begin{align}
\dist (p, p_i+\pi_i) &\leq \frac{\varkappa}{4}, \label{e:close}\\ 
\dist (p, p_j+\pi_j) &\geq \frac{\varkappa}{2} \qquad \forall j\neq i \label{e:far}\, .
\end{align}
and thus (b) is certainly satisfied. Moreover, by \cite{DLS16centermfld}*{Lemma 1.6}, observe that $(\mathbf{p}_{\pi_0})_\sharp T_i = Q_i \llbracket B_{1}\rrbracket$ for some $Q_i \in \mathbb Z$; clearly we must have $\sum_i Q_i = Q$. Again applying Theorem \ref{thm:main-estimate} we have 
\[
\Ebf(T,\Cbf_1,\pi_0) \leq C \delta^2\, .
\]
In particular, if $\delta$ is chosen small enough, by \cite{DLS14Lp}*{Theorem 2.4} we can ensure the existence of a subset $K\subset B_1(\pi_0)$ of positive measure with the property that for each $x\in K$, the slice $\langle T, \mathbf{p}_{\pi_0}, x\rangle$ (see \cite{Simon_GMT} for the definition and properties of the slicing map) given by $\sum_j k_j \delta_{\xi_j}$ for some finite collection of positive integers $k_j$ and points $\xi_j(x)\in \pi_0^\perp$. Fix any such point $x$ and observe that
\[
\langle T_i, \mathbf{p}_{\pi_0}, x\rangle = \sum_{\xi_j\in U_i} k_j \delta_{\xi_j} 
\]
while
\[
Q_i = \sum_{\xi_j \in U_i} k_j\, .
\]
It follows immediately that $Q_i$ cannot be negative, which completes the proof.
\qed

\section{Preliminaries}
In this section, we collect all the required preliminary results for the proof of Theorem \ref{thm:main-estimate}.

\subsection{Oriented and non-oriented tilt-excess} Given an $m$-dimensional plane $\pi$ and a cylinder $\mathbf{C} = \mathbf{C}_r (q, \pi)$, recall that the \textit{non-oriented tilt excess} is given by
\begin{equation}\label{e:nonoriented}
\mathbf{E}^{no} (T, \mathbf{C}):= \frac{1}{2\omega_m r^m} \int_{\mathbf{C}} |\mathbf{p}_{T (x)} - \mathbf{p}_{\pi}|^2\, d\|T\| (x)\, ,
\end{equation}
where $T (x)$ denotes the (approximate) tangent plane to $T$ at $x$. This is more generally defined for integral varifolds and does not take into consideration the orientation of $T$ and $\pi$, in contrast with the oriented tilt excess $\Ebf(T, \mathbf{C})$ defined previously.

It is obvious that $|\mathbf{p}_{\alpha} - \mathbf{p}_{\beta}|\leq C |\vec{\alpha}-\vec{\beta}|$ for a geometric constant $C=C(m,n)$ and for every pair of oriented planes, so $\Ebf^{no}(T,\Cbf) \leq C\Ebf(T,\Cbf)$. On the other hand, the opposite inequality is only true if $|\vec{\alpha}-\vec{\beta}|$ is sufficiently small, for instance, if it is no larger than $1$. In particular $\vec{\alpha}$ and $\vec{\beta}$ could be opposite orientations for the same linear subspace: in that case $|\vec{\alpha}-\vec{\beta}|=2$ while $|\mathbf{p}_{\alpha} - \mathbf{p}_{\beta}|=0$. 

Nonetheless, for area-minimizing currents as in Assumption \ref{a:main}, the nonoriented excess controls the oriented excess. The idea of the argument is borrowed from \cite{DLHMS}*{Theorem 16.1}, but we repeat it here for clarity. A simple scaling and translation argument, which is left to the reader, gives a corresponding estimate on any pair of parallel concentric cylinders. 

\begin{proposition}\label{p:o<no}
For every $1\leq r<2$ there is a constant $\bar{C}= \bar{C}(Q,m,n,\bar n, 2-r)$ such that, if $T$ , $\Sigma$ and $\Abf$ are as in Assumption \ref{a:height-main}, then
\begin{equation}\label{e:o<no}
\mathbf{E} (T, \mathbf{C}_r) \leq 
\bar{C} (\mathbf{E}^{no} (T, \mathbf{C}_2) + \mathbf{A}^2)\, .
\end{equation}
\end{proposition}

\begin{proof}
Since $\mathbf{E} (T, \mathbf{C}_r) \leq \frac{1}{\omega_m r^m} \|T\| (\mathbf{C}_r)$, we can without loss of generality assume that 
$$\mathbf{E}^{no} (T, \mathbf{C}_2) + \mathbf{A}^2 \leq \delta$$
for some fixed constant $\delta$, as long as in the end we choose $\delta = \delta(Q,m,n,\bar n, 2-r)>0$. 

Moreover, for any $\eta>0$, if $\delta$ is chosen to be sufficiently small depending on $\eta$, we can also assume without loss of generality that $\mathbf{E} (T, \mathbf{C}_{1+r/2}) \leq \eta$
for another fixed constant $\eta$: indeed, arguing by contradiction, if this were not true for some $\eta>0$, we could find a sequence of $\delta_k\downarrow 0$ and a sequence of $T_k$ with $\Ebf^{no}(T_k,\Cbf_2) \leq \delta_k\downarrow 0$ yet $\Ebf(T_k,\Cbf_{1+r/2})>\eta$ for all $k$. If the supports of the currents are equibounded, then this would give that the $T_k$ converge, locally in mass in $\Cbf_2$, to a union of planes which are parallel to $\pi_0$ (counted with a suitable multiplicity), which in turn would imply $\Ebf(T_k,\Cbf_{1+r/2})\to 0$, giving the desired contradiction. If the supports are not equibounded, one can resort to the height bound \cite{DLS16centermfld}*{Theorem A.1} to decompose each $T_k$ into a disjoint finite sum of area-minimizing currents $T_k^j$, each of which have equibounded supports (after translation). One can then argue as above for $T_k^j$, for each $j$.

We thus assume 
\begin{equation}\label{e:start-small}
\Lambda_0 := \mathbf{E} (T, \mathbf{C}_{1+r/2}) + \mathbf{A}^2 \leq \eta \, ,
\end{equation}
for some $\eta$ whose choice will be given later (depending only on $Q,m,n,\bar n,2-r$). We also assume
\begin{equation}\label{e:E-dominates-A^2}
\mathbf{A}^2 \leq  \mathbf{E} (T, \mathbf{C}_{1+r/2})\,  ,  
\end{equation} 
otherwise the estimate we are looking for is trivially true. In particular $\Lambda_0 \leq 2 \mathbf{E} (T, \mathbf{C}_{1+r/2})$.

Set $r_0 := 1+\frac{r}{2}$. Recalling Almgren's strong Lipschitz approximation (\cite{DLS14Lp}*{Theorem 1.4}), provided $\eta$ is sufficiently small (depending on $Q,m,n,\bar{n},2-r$), there is a set $K\subset B_{r_0/4} (\pi_0)$ and a Lipschitz $Q$-valued map $f: B_{r_0/4} \to \mathcal{A}_Q (\pi_0^\perp)$ such that 
\begin{itemize}
\item $\|T\| ((B_{r_0/4}\setminus K)\times \pi_0^\perp) \leq C \Lambda_0^{1+2\gamma}$;
\item ${\rm Lip}\, (f) \leq C \Lambda_0^{2\gamma}$;
\item $\mathbf{G}_f \res (K\times \pi_0^\perp) = T \res (K\times \pi_0^\perp)$.
\end{itemize}
$\gamma$ and $C$ are fixed positive constants depending on $Q, m, n, \bar{n}$. At the price of making $C$ larger, we can achieve the same estimates with $B_{r_0-C\Lambda_0^{2\gamma}}$ in place of $B_{r_0/4}$; this is achieved with a simple covering argument, applying the Lipschitz approximation theorem in a collection of $C\Lambda_0^{-m\beta}$ cylinders whose cross-sections are disks in $\pi_0$ of radius $\Lambda_0^\beta$ which cover $B_{r_0-C\Lambda_0^\gamma}$, for a suitable choice of $\beta = \beta(\gamma,m)$ (for a detailed argument, see \cite{DLHMS}*{Proposition 16.2}).

Moreover, if we assume that $\eta$ is small enough, by halving the exponents $2\gamma$ to $\gamma$ we can assume that the constant $C$ in all the above estimates (including in the radius of $B_{r_0-C\Lambda_0^{2\gamma}}$) is at most $\frac{1}{4}$. More precisely, for that fixed $C = C(Q,m,n,\bar{n})$ we have $C\Lambda_0^{2\gamma} = C\Lambda_0^{\gamma}\cdot\Lambda_0^{\gamma} \leq (C\eta^{\gamma})\Lambda_0^{\gamma}$, so if we choose $\eta$ small enough so that $C\eta^{\gamma/2}\leq 1/4$, we guarantee this. 

We now set 
\[
r_1 := r_0 - \Lambda_0^\gamma\, .
\]
Summarizing we have an approximation on $B_{r_1}$, which we still denote by $f$, such that 
\begin{itemize}
\item $\|T\| ((B_{r_1}\setminus K)\times \pi_0^\perp) \leq \frac{1}{4} \Lambda_0^{1+\gamma}$;
\item ${\rm Lip}\, (f) \leq \Lambda_0^\gamma$;
\item $\mathbf{G}_f \res (K\times \pi_0^\perp) = T \res (K\times \pi_0^\perp)$.
\end{itemize}
If $\eta$ is small enough (which guarantees smallness of $\text{Lip}(f)$, and hence comparability between the oriented and non-oriented tilt-excess of $\Gbf_f$), we can estimate
\[
\mathbf{E} (\mathbf{G}_f \res (K\times \pi_0^\perp) , \mathbf{C}_{r_1}) 
\leq C  \mathbf{E}^{no} (\mathbf{G}_f \res (K\times \pi_0^\perp) , \mathbf{C}_{r_1}) 
\leq C \mathbf{E}^{no} (T, \mathbf{C}_2)
\]
so that, in particular
\[
\mathbf{E} (T, \mathbf{C}_{r_1}) \leq C \mathbf{E}^{no} (T, \mathbf{C}_2) + \frac{1}{2} \Lambda_0^{1+\gamma}\, .
\]
If $\Lambda_0^{1+\gamma} \leq \mathbf{E} (T, \mathbf{C}_{r_1})$ we are then done, provided that $r_1>r$, which may be achieved by taking $\eta$ sufficiently small. Otherwise we must have 
\[
\mathbf{E} (T, \mathbf{C}_{r_1}) \leq \Lambda_0^{1+\gamma}\, .
\]
We now iterate the above argument, inductively setting
\begin{align}
r_k &:= r_{k-1} - \Lambda_{k-1}^\gamma\\
\Lambda_k & :=  \mathbf{E} (T, \mathbf{C}_{r_k}) \leq \Lambda_{k-1}^{1+\gamma}\, .
\end{align}
We stop at a certain step if we have the desired estimate and $r_k\geq r$.

Since $\Lambda_k \leq \Lambda_0^{(1+\gamma)^k} \leq \eta^{(1+\gamma)^k}$ and $r_0-r>0$, provided that $\eta$ is small enough, the inequality $r_k \geq r$ is always guaranteed, no matter how large $k$ is. Therefore, if the procedure never stops, we conclude that $\mathbf{E} (T, \mathbf{C}_r) =0$. But of course in this case the sought-after inequality is trivially true. This completes the proof.
\end{proof}

\subsection{The case \texorpdfstring{$N=1$}{N=1}} Observe that the difficulty of Theorem \ref{thm:main-estimate} lies in the case when $N \geq 2$. Indeed, when $Q$ is arbitrary and $N=1$, the $L^2$ tilt-excess estimate \eqref{e:tilt-estimate} is an immediate consequence of the work of Allard in \cite{Allard_72} (see also \cite{DL-All}*{Proposition 4.1}), combined with Proposition \ref{p:o<no}. Regarding the $L^\infty$ estimate \eqref{e:Linfty-estimate}, this is also contained within \cite{Allard_72} when the ambient space is Euclidean (and thus $\Abf=0$). More generally the general techniques in \cite{Allard_72} would give a linear dependence of the estimate in $\Abf$: to see that in our case this can be improved to a quadratic the reader can consult, for instance, \cite{Spolaor_15}. From now on we will therefore assume the following.
\begin{proposition}\label{p:step1}
The conclusions of Theorem \ref{thm:main-estimate} holds for $N=1$ (and $Q$ arbitrary).
\end{proposition}

\subsection{Combinatorial lemmas}\label{s:comb-lemmas} We will make use of the following combinatorial lemmas. The first one is essentially a consequence of \cite{DLS_MAMS}*{Lemma 3.8}, but since we need additional information we provide a self-contained proof. All of the proofs are deferred to Appendix \ref{ap:combinatorics}. 

\begin{lemma}\label{l:clusters-1}
Fix any positive $\bar \delta\leq\frac{1}{2}$ and a natural number $N\geq 2$. There is a constant $\bar C = \bar C (\bar \delta, N)$ with the following property. Given a set (of distinct points) $P\subset \mathbb R^n$ with cardinality $N$, there is a subset $P'\subset P$ of cardinality at least two such that:
\begin{itemize}
\item[(i)] $\max \{|q-p|:q,p\in P'\} \leq \bar{C} \min \{|q-p|:q,p \in P', \ q\neq p\}$;
\item[(ii)] $\dist (p,P') \leq \bar \delta \min \{|q_1-q_2|: q_1,q_2\in P', \ q_1\neq q_2\}$ for every $p \in P$;
\item[(iii)] the points in $P\setminus P'$ can be ordered as $\{p_1, \ldots , p_J\}$ so that, setting $P_0=P$ and $P_j:= P\setminus \{p_1, \ldots, p_j\}$, the following property holds: 
\[
\dist (p_j, P_{j}) = \min\{|p-q|: p, q \in P_{j-1}, \  p\neq q\}\, .
\]
\end{itemize}
\end{lemma}

\begin{lemma}\label{l:clusters-2}
Fix any $0<\delta\leq \frac{1}{2}$ and $\varepsilon>0$. Given a set (of distinct points) $P\subset \mathbb R^n$ with cardinality $N\geq 2$, we can find a nonempty subset $\tilde P\subset P$ such that 
\begin{itemize}
\item[(i)] $\dist (p, \tilde P) \leq \delta^{-1}(1+\delta^{-1})^{N-2} \varepsilon$ for every $p\in P$;
\item[(ii)] Either $\tilde P$ is a singleton, or $\max \{\varepsilon, \dist (p, \tilde P)\} \leq \delta \min\set{|q_1-q_2|}{q_1,q_2\in \tilde P, \ q_1\neq q_2}$ for any $p\in P$.
\end{itemize}
\end{lemma}

 \begin{lemma}\label{l:clusters-3}
Consider a set $P = \{p_1, \ldots, p_N\}\subset \mathbb R^n$ of $N\geq 2$ distinct points and set $M:= \max_{i,j} \{|p_i-p_j|\}$. Then we can decompose $P= P_1\cup P_2$ into two disjoint non-empty sets such that 
 \[
    \min \{|p_1-p_2|: p_1\in P_1, p_2\in P_2\} \geq \frac{M}{2^{N-2}}\, .
 \]
 \end{lemma}

\subsection{Well-separated case}
Let us first demonstrate the validity of Theorem \ref{thm:main-estimate} under the assumption that the planes in $\boldsymbol{\pi}$ are well-separated.

\begin{lemma}\label{l:simple}
For every $1\leq \bar r<2$ there is a constant $\sigma_1 = \sigma_1 (Q,m,n,\bar n,2-\bar r)>0$ with the following property. Let $T$ be as in Assumption \ref{a:height-main}, while $p_1, \ldots, p_N\in \pi_0^\perp$ are distinct points and $\boldsymbol{\pi}:=\bigcup_i p_i+\pi_0$. Assume that
\begin{equation}\label{e:scaling-broken}
E \leq \sigma_1 \min \{H,1\}^{m+2}\, ,
\end{equation}
where $E$ is as in \eqref{e:L2-excess} and $H:= \min \{|p_i-p_j|:i \neq j\}$. 
Then 
\begin{equation}\label{e:separated}
{\rm spt}\, (T) \cap \mathbf{C}_{\bar r} \subset \{q: \dist (q, \boldsymbol{\pi})\leq \textstyle{\frac{H}{4}}\},
\end{equation}
and, in particular, all the conclusions of Theorem \ref{thm:main-estimate} hold. 
\end{lemma}

First of all, observe that it suffices to prove \eqref{e:separated} in order to conclude all the conclusions of Theorem \ref{thm:main-estimate}. Indeed, if we set $T_i := T\res \mathbf{C}_{\bar r} \cap \{q: \dist (q, p_i+\pi_0)\leq \textstyle{\frac{H}{4}}\}$, we obtain the decomposition $T\res \mathbf{C}_{\bar r} = T_1+\cdots +T_N$ with 
\begin{itemize}
\item[(i)] $\partial T_i \res \Cbf_{\bar r} = 0$;
\item[(ii)] ${\rm spt}\, (T_i) \cap {\rm spt}\, (T_j) = \emptyset$ for $i\neq j$;
\item[(iii)] $\dist (q, \boldsymbol{\pi}) = \dist (q, p_i+\pi_0)$ for every $q\in {\rm spt}\, (T_i)$. 
\end{itemize}
In particular, given $r \in [1,2)$, for $\bar r \in (r,2)$ the conclusions of Theorem \ref{thm:main-estimate} can be drawn by using the case $N=1$ applied to each $T_i$.

The proof of Lemma \ref{l:simple} is based on the following yet simpler lemma. 

\begin{lemma}\label{l:simpler}
For every $1\leq \bar r<2$ there is a constant $\sigma_2 = \sigma_2 (Q,m,n,\bar{n},2-\bar r)>0$ with the following property. Let $T$, $\boldsymbol{\pi}$, $E$, and $H$ be as in Lemma \ref{l:simple}, but instead of \eqref{e:scaling-broken} assume that
\begin{equation}\label{e:scaling-not-broken}
E  \leq \sigma_2 \qquad \mbox{and} \qquad H \geq 1\, . 
\end{equation}
Then \eqref{e:separated} and all the conclusions of Theorem \ref{thm:main-estimate} hold in $\Cbf_{\bar r}$.
\end{lemma}

\begin{proof} Observe that since $H\geq 1$, by Chebyshev's inequality we have
\[
\|T\| (\mathbf{C}_{\bar r}\cap \{\dist (\cdot, \boldsymbol{\pi})\geq \textstyle{\frac{H}{8}}\}) 
\leq 64 \sigma_2 \, . 
\]
On the other hand, for $\rho:= \min\{2-\bar r, \frac{H}{8}\}$, if ${\rm spt}\, (T) \cap  \{\dist (\cdot, \boldsymbol{\pi})\geq \textstyle{\frac{H}{4}}\}$ is not empty the monotonicity formula would give 
\[
\|T\| (\mathbf{C}_{\bar r}\cap \{\dist (\cdot, \boldsymbol{\pi})\geq \textstyle{\frac{H}{8}}\}) \geq C^{-1} \rho^m
\]
for some dimensional constant $C>0$, which yields a contradiction for a sufficiently small choice of $\sigma_2$.
\end{proof}

\begin{proof}[Proof of Lemma \ref{l:simple}]
Recall that we just need to prove \eqref{e:separated}. If $H\geq 1$ the claim follows from Lemma \ref{l:simpler}. If $H\leq 1$, we just need to show that 
\[
{\rm spt} (T) \cap \mathbf{C}_{\bar{r}H} (q) \cap \{\dist (\cdot, \boldsymbol{\pi})>\textstyle{\frac{H}{4}}\}
=\emptyset
\]
whenever the cylinder $\mathbf{C}_{2H} (q)$ is contained in the original cylinder $\mathbf{C}_2$. But then it suffices to consider the current $T_{q,H} := (\lambda_{q,H})_\sharp T$ for the map $\lambda_{q,H} (x) := \frac{x-q}{H}$ and to apply Lemma \ref{l:simpler} to $T_{q, H}$ and $\lambda_{q,H} (\boldsymbol{\pi})$: the pair falls under the assumptions provided $\sigma_1\leq \sigma_2$, after an obvious scaling argument. 
\end{proof}

We record another two simple observations which will be useful in the sequel. They are both proven in exactly the same way as Lemma \ref{l:simpler}.


\begin{lemma}\label{l:simple-2}
For every $1\leq r<2$, there is a constant $\sigma_3 = \sigma_3(Q,m,n,\bar{n},2-r)>0$ with the following property. Let $T$ and 
$\boldsymbol{\pi}$ be as in Lemma \ref{l:simpler}, but instead of \eqref{e:scaling-not-broken} assume only that 
\begin{equation}\label{e:absolute}
E \leq \sigma_3\, . 
\end{equation}
Then ${\rm spt}\, (T) \cap \mathbf{C}_r \subset \{x: \dist (x, \boldsymbol{\pi})\leq 1\}$.
\end{lemma}


\begin{lemma}\label{l:simpler-2}
For every $1\leq r<2$, there is a constant $\bar\sigma = \bar\sigma (Q,m,n,\bar{n},N,2-r)>0$ such that the following holds. 
Let $T$ and 
$\boldsymbol{\pi}$ be as in Lemma \ref{l:simpler}, but instead of \eqref{e:scaling-not-broken} assume that $M:=\max_{i,j} \{|p_i-p_j|\}\geq 1$ and
\begin{equation}\label{e:scaling-broken-2}
E\leq \bar\sigma M^2\, .
\end{equation}
Then ${\rm spt}\, (T) \cap \mathbf{C}_r \subset \{x: \dist (x, \boldsymbol{\pi}) \leq 2^{-N} M\}$.
\end{lemma}

Combining Lemma \ref{l:simpler-2} with the combinatorial Lemma \ref{l:clusters-3}, we get the following separation lemma.

 \begin{lemma}\label{l:simple-3}
 Under the assumptions of Lemma \ref{l:simpler-2}, there is a decomposition $\boldsymbol{\pi}=\boldsymbol{\pi}^1\cup \boldsymbol{\pi}^2$ and $T\res \Cbf_{r} = T_1 + T_2$ such that 
\begin{itemize}
\item[(i)] The sets $\boldsymbol{\pi}^1, \boldsymbol{\pi}^2$ are disjoint, non-empty, and unions of a subset of the planes in $\boldsymbol{\pi}$;
\item[(ii)] $\partial T_i \res \Cbf_r = 0$ for $i=1,2$;
\item[(iii)] ${\rm spt}\, (T_1) \cap {\rm spt}\, (T_2) = \emptyset$;
\item[(iv)] $\dist (q, \boldsymbol{\pi}) = \dist (q, \boldsymbol{\pi}^i)$ for every $q\in {\rm spt}\, (T_i)$ and each $i=1,2$. 
\end{itemize}
 \end{lemma}

Arguing as in the proof of Lemma \ref{l:simple}, we can draw the following further conclusion.

 \begin{lemma}\label{l:simple-4}
For every $1\leq r <2$, there is a constant $\tilde\sigma = \tilde\sigma (Q,m,n,\bar{n},N,2-r)>0$ such that the following holds. 
Let $T$ and 
$\boldsymbol{\pi}$ be as in Lemma \ref{l:simple}, but instead of \eqref{e:scaling-broken}, for $M\coloneqq\max_{i,j}\{|p_i-p_j|\}$ assume that
\[
E\leq \tilde{\sigma} \min\{M,1\}^{m+2}\, .
\]
Then there is a decomposition $\boldsymbol{\pi}=\boldsymbol{\pi}^1\cup \boldsymbol{\pi}^2$ and $T\res \Cbf_{r} = T_1 + T_2$, as in Lemma \ref{l:simple-3}.
\end{lemma}

\section{Proof of tilt-excess estimate}
The aim of this section is to prove \eqref{e:tilt-estimate}. This will be done independently to the proof of $L^\infty$ estimate \eqref{e:Linfty-estimate}; we defer the latter to the next section. We begin with two ``approximate" estimates on the oriented tilt-excess.

\subsection{First lemma} In this lemma we derive a first approximate estimate.

\begin{lemma}\label{l:tilt-1}
For every pair of radii $1\leq r<R \leq 2$ there are constants $\bar{C}=\bar{C}(Q,m,n,\bar{n},R-r)>0$ and $\gamma = \gamma (Q,m,n,\bar n)>0$ such that the following holds. Let $T$, $\Sigma$ and $\Abf$ be as in Assumption \ref{a:height-main}, suppose that $p_1, \ldots, p_N\in \pi_0^\perp$ are distinct points, and let $\boldsymbol{\pi}:=\bigcup_i p_i+\pi_0$. Assume that $E$ is as in \eqref{e:L2-excess} and let $H:= \min \{|p_i-p_j|:i \neq j\}$. Then
\begin{equation}\label{e:tilt-est-1}
\mathbf{E} (T, \mathbf{C}_r)\leq \bar{C} (E+ \mathbf{A}^2) + \bar{C} \left(\frac{E}{H^2}\right)^\gamma \mathbf{E} (T, \mathbf{C}_R) + \bar{C}
\mathbf{E} (T, \mathbf{C}_R)^{1+\gamma}\, .
\end{equation}
\end{lemma}

\begin{proof} We start by defining a suitable vector field $X: \mathbb R^{m+n} \to \pi_0^\perp$. In fact, we will actually define $\bar X : \pi_0^\perp \to \pi_0^\perp$, and then set $X (x) := \bar X (\mathbf{p}_{\pi_0}^\perp (x))$. Firstly, define $\tilde{X}$ on the union of the disks $B_{H/4} (p_i, \pi_0^\perp)$ by taking $\tilde{X} (y) := (y-p_i)$ on each $B_{H/4} (p_i,\pi_0^\perp)$. It is simple to check that the Lipschitz constant of $\tilde{X}$ can be bounded by $3$; indeed, if $y_1,y_2$ lie in disks $B_{H/4}(p_{i_1},\pi_0^\perp)$ and $B_{H/4}(p_{i_2},\pi_0^\perp)$ respectively, then
$$|p_{i_1}-p_{i_2}| \leq |y_1-y_2| + H/2 \leq |y_1-y_2| + |p_{i_1}-p_{i_2}|/2\ \ \Longrightarrow\ \ |p_{i_1}-p_{i_2}|\leq 2|y_1-y_2|.$$
The desired Lipschitz bound of $3$ follows immediately. We can thus use the Kirszbraun theorem to extend $\tilde{X}$ to a vector field $\bar X:\pi_0^\perp\to \pi_0^\perp$ which has the same Lipschitz constant; this determines $\bar X$, and in turn determines $X$.
In particular, we deduce that $X$ has the following properties: 
\begin{itemize}
\item[(i)] $X$ takes values in $\pi_0^\perp$;
\item[(ii)] $|X(x)|\leq 3 \dist (x, \boldsymbol{\pi})$ and $|\nabla_v X (x)|\leq 3 |v|$;
\item[(iii)] $\nabla_v X =0$ for every $v\in \pi_0$;
\item[(iv)] $X (x) = \mathbf{p}_{\pi_0}^{\perp} (x-p_i)$ if $\dist (x, p_i + \pi_0)\leq \frac{H}{4}$.
\end{itemize}
Note that by regularizing $\bar X$ via convolution, we can obtain a smooth vector field with the same properties, except that (iv) will hold in a slightly smaller tubular neighborhood of $\boldsymbol{\pi}$. We will thus ignore the regularity issues and use $X$ as a test in the first variation formula for $T$. 

In order to simplify our notation we introduce the set 
\[
G:= \left\{x\in\Cbf_2: \dist (x, \boldsymbol{\pi})\leq \frac{H}{4} \right\} =
\bigcup_i \left\{x\in\Cbf_2: \dist (x, p_i + \pi_0)\leq \frac{H}{4}\right\}
\]
and write $G^c$ for its complement in $\Cbf_2$.
Because of (iv), (see e.g. \cite{DL-All}*{Lemma 4.2}) we have 
\begin{equation}\label{e:divergence-correct}
{\rm div}_{\vec{T} (x)} X(x) = \frac{1}{2} |\mathbf{p}_{\pi_0} - \mathbf{p}_{T(x)}|^2 \qquad \forall x\in G\cap {\rm spt}\, (T),
\end{equation}
while, because of (i), (iii), and the Lipschitz regularity of $X$,  
\begin{equation}\label{e:divergence-estimate}
|{\rm div}_{\vec{T} (x)}X(x)| \leq C |\mathbf{p}_{\pi_0} - \mathbf{p}_{T(x)}|^2 \qquad \forall x\in G^c\cap {\rm spt}\, (T)\, .
\end{equation}
To see the latter, let $e_1, \ldots, e_m$ be an orthonormal base of the tangent plane to $T$ at $x$ and compute
\begin{align*}
|{\rm div}_{\vec{T} (x)} X (x)| & \leq
\sum_i |\nabla_{e_i} X (x) \cdot e_i| \stackrel{\text{(i)}}{=}
\sum_i |\nabla_{e_i} X (x) \cdot (e_i - \mathbf{p}_{\pi_0} (e_i))|\\
&\stackrel{\text{(iii)}}{=} \sum_i |(\nabla_{e_i} X (x)-\nabla_{\mathbf{p}_{\pi_0} (e_i)} X (x)) \cdot (e_i-\mathbf{p}_{\pi_0} (e_i))|\\
&\leq \sum_i |e_i - \mathbf{p}_{\pi_0} (e_i)|
|\nabla_{e_i-\mathbf{p}_{\pi_0} (e_i)} X (x)|\\
&\stackrel{\text{(ii)}}{\leq} 3 \sum_i |e_i - \mathbf{p}_{\pi_0} (e_i)|^2 \leq C |\mathbf{p}_{\vec{T} (x)} - \mathbf{p}_{\pi_0}|^2\, .
\end{align*}
Let us now choose radii $r_1 < r_2$ so that $r_1-r=r_2-r_1=R-r_2$ (i.e. $r_1 = (R+2r)/3$, $r_2 = (2R+r)/3$). Let $\tilde{\chi}:\pi_0\to \R$ be a smooth cut-off function which is identically 1 on $B_{r_1}(0,\pi_0)$, vanishes outside $\bar{B}_{r_2}(0,\pi_0)$, and satisfies $|D\tilde{\chi}|\leq 2/(r_2-r_1)$. Extend $\tilde{\chi}$ vertically to get a function $\chi:\R^{m+n}\to \R$, namely $\chi(x) = \tilde{\chi}(\mathbf{p}_{\pi_0}(x))$. We now want to take $\chi^2 X$ as a test function in the first variation formula for (the stationary varifold associated to) $T$; since $\chi^2 X$ need not have compact support in the directions orthogonal to $\pi$, to justify this simply note that since $\chi$ is supported in the cylinder $\bar{\Cbf}_{r_2}$, and since the support of $T$ is compact on any cylinder which is slightly smaller than $\Cbf_2$, we can multiply $\chi^2 X$ by another test function which is translation invariant in the $\pi_0$ variables, is $1$ in the ball of radius $2\sup\{\mathbf{p}_{\pi_0}^\perp(x): x\in \spt\|T\|\cap \bar{\Cbf}_{r_2}\}$ in the $\pi_0^\perp$ variables, and vanishes outside the ball of radius $3\sup\{\mathbf{p}_{\pi_0}^\perp(x): x\in \spt\|T\|\cap \bar{\Cbf}_{r_2}\}$; this then has compact support in all of $\Cbf_2$ and agrees with $\chi^2 X$ on $\Cbf_2\cap\spt(T)$, therefore allowing us to take $\chi^2 X$ as a test function. Hence, we have
\begin{align*}
\int_{\mathbf{C}_2} \chi^2 {\rm div}_{\vec{T} (x)} X (x) d\|T\| (x) 
= & -\int_{\mathbf{C}_2} \chi^2 X (x) \cdot \vec{H}_T (x)\, d\|T\| (x)\\
& - 2 \int_{\mathbf{C}_2} \chi (x)
\nabla_{\vec{T} (x)} \chi (x)  \cdot X(x)\, d\|T\| (x)\, ,
\end{align*}
where $\vec{H}_T$ is the generalized mean curvature of $T$. Combining this with \eqref{e:divergence-correct}, \eqref{e:divergence-estimate}, (ii), and the fact that $|\vec{H}_T (x)|\leq C\mathbf{A}$, we get
\begin{align*}
\int_{\mathbf{C}_2} \chi^2 |\mathbf{p}_{T(x)} - \mathbf{p}_{\pi_0}|^2\, d\|T\| (x)
&\leq C \int_{\mathbf{C}_{r_2}\cap G^c} |\mathbf{p}_{T(x)} - \mathbf{p}_{\pi_0}|^2\, d\|T\| (x)\\
&\quad
+ C (E+\mathbf A^2) + 2\left| \int_{\mathbf{C}_2} \chi (x)
\nabla_{\vec{T} (x)} \chi (x)  \cdot X(x)\, d\|T\| (x)\right|\, .
\end{align*}
Observe next that in light of (i), $\nabla_{\vec{T}} \chi  \cdot X = (\nabla_{\vec{T}} \chi - \nabla_{\vec{\pi}_0}\chi)  \cdot X(x)$. We can thus estimate further
\begin{align*}
2 \left|\int_{\mathbf{C}_2} \chi (x)\nabla_{\vec{T} (x)} \chi (x)  \cdot X(x)\, d\|T\| (x)\right|
&\leq \frac{C}{R-r} \int_{\mathbf{C}_{r_2}} \chi (x) |\mathbf{p}_{\pi_0}-\mathbf{p}_{T(x)}| |X(x)|\, d\|T\| (x)\\
&\leq \frac{CE }{(R-r)^2} +\frac{1}{4} \int_{\mathbf{C}_2} \chi^2 (x) |\mathbf{p}_{\pi_0}-\mathbf{p}_{T(x)}|^2\, d\|T\| (x),
\end{align*}
where in the latter inequality we have used that $ab \leq \frac{a}{2\eps} +\frac{\eps b}{2}$ for any two non-negative numbers $a$ and $b$ and a suitably small choice of $\eps >0$, combined with (ii).
In particular, we conclude that
\[
\mathbf{E}^{no} (T, \mathbf{C}_{r_1}) \leq \bar{C} (E+\mathbf{A}^2) + C\int_{\mathbf{C}_{r_2}\cap G^c} |\mathbf{p}_{T(x)} - \mathbf{p}_{\pi_0}|^2\, d\|T\| (x)\, ,
\]
for $\bar{C}$ now also dependent on $R-r$. When combined with Proposition \ref{p:o<no}, we arrive at
\[
\mathbf{E} (T, \mathbf{C}_r) \leq \bar{C} (E+\mathbf{A}^2) + C\int_{\mathbf{C}_{r_2}\cap G^c} |\mathbf{p}_{T(x)} - \mathbf{p}_{\pi_0}|^2\, d\|T\| (x)\, .
\]
where again $\bar{C} = \bar{C}(Q,m,n,\bar{n},R-r)>0$ (we recall that $r_1-r = (R-r)/3$ here).
We now use (a scaled version of) Almgren's estimate \cite{DLS14Lp}*{Theorem 7.1} (the first version of this estimate is in \cite{Almgren_regularity}*{Sections 3.24–3.26, Section 3.30(8)}), applied within $\Cbf_{r_2}$, with the choice $A = \mathbf{p}_{\pi_0}(\Cbf_{r_2}\cap G^c\cap\spt(T))\equiv B_{r_2}(\pi_0)\cap \mathbf{p}_{\pi_0}(G^c\cap\spt(T))$, to conclude that
\begin{align*}
\mathbf{E} (T, \mathbf{C}_r) &\leq \bar{C} (E+\mathbf{A}^2) + \bar{C} (\mathbf{E} (T, \mathbf{C}_R)+\mathbf{A}^2)^{1+\gamma} + \bar{C} |A|^\gamma (\mathbf{E} (T, \mathbf{C}_R)+\mathbf{A}^2)\\
&\leq \bar{C} (E+\mathbf A^2) + \bar{C} \mathbf{E} (T, \mathbf{C}_R)^{1+\gamma} + \bar{C} 
 \|T\| (G^c\cap\mathbf{C}_R)^\gamma \mathbf{E} (T, \mathbf{C}_R)\, ,
\end{align*}
where we have used that $\mathbf{A}\leq 1$ and $|A|\leq \|T\| (G^c \cap \mathbf{C}_R)$ (here, $|A|$ denotes the measure of the set $A$).

However, observe that $\dist (x, \boldsymbol{\pi}) \geq \frac{H}{4}$ for all $x\in G^c$ and thus by Chebyshev we get 
\[
\|T\| (G^c\cap \mathbf{C}_R) \leq \frac{C E}{H^2}\, ,
\]
which concludes the proof.
\end{proof}

\subsection{Second lemma} Now, we iterate Lemma \ref{l:tilt-1} to achieve a closer approximation of \eqref{e:tilt-estimate}, but under the additional assumption that the height excess of $T$ relative to the family of planes is much smaller than the minimal separation of the planes. We will remove this assumption in the next section.

\begin{proposition}\label{p:tilt-2}
For every pair of scales $1\leq r<r_0 < 2$, there are constants $\bar{C}=\bar{C}(Q,m,n,\bar n,N,r_0-r,2-r_0)>0$ and $\sigma_4=\sigma_4 (Q,m,n,\bar n,N,r_0-r,2-r_0)>0$ with the following properties.
Assume $T$, $\boldsymbol{\pi}$ and $H$ are as Lemma \ref{l:tilt-1}. If in addition we have
\begin{equation}\label{e:very-small}
E \leq \sigma_4 \min\{H^2, 1\} \, ,
\end{equation}
then
\begin{equation}\label{e:tilt-estimate-2}
\mathbf{E} (T, \mathbf{C}_r) \leq \bar{C} (E + \mathbf{A}^2) + \bar{C} \left(\frac{E}{H^2}\right) \mathbf{E} (T, \mathbf{C}_{r_0})\, .
\end{equation}
\end{proposition}

\begin{proof} This estimate will be proved by a two-variable induction over $Q$ and $N$. Our inductive assumption is that
\begin{itemize}
    \item[(IH)] \emph{The estimate is valid for any pair $Q'\leq Q$ and $N'\leq N$ with $Q'+N' < Q+N$.}
\end{itemize}
By Proposition \ref{p:step1}, the case $N=1$ and arbitrary $Q$ is always valid. Moreover, due to \eqref{e:very-small}, we can assume without loss of generality that $H\leq 1$, since otherwise we can apply Lemma \ref{l:simpler}, yielding the conclusions of Theorem \ref{thm:main-estimate} in $\Cbf_r$, which in particular trivially implies \eqref{e:tilt-estimate-2}.

We can similarly assume that $M\coloneqq \max\{|p_i-p_j|\} \leq 1$ and
\begin{equation}\label{e:small-E-M}
\frac{E}{M^{m+2}}\geq \tilde{\sigma}.
\end{equation}
Indeed, if $M>1$ then we can apply Lemma \ref{l:simple-3} to decompose $T$ into $T_1,T_2$ (as in the statement of the lemma) and appeal to the induction assumption (IH). Note that the validity of the hypothesis \eqref{e:very-small} remains unchanged under such a decomposition (as $E$ can only decrease while $H$ can only increase). We can apply (IH) since if both $T_1,T_2\neq 0$, then necessarily $\Theta(T_i,\cdot)\leq Q-1$ (and so $Q$ is decreased), while if, say, $T_2 = 0$, then $Q$ remains the same for $T_1$, but $N$ decreases by at least 1. Thus, we may indeed assume that $M\leq 1$. Given this, one can additionally assume that \eqref{e:small-E-M} indeed holds via the same argument, except now invoking Lemma \ref{l:simple-4} instead of Lemma \ref{l:simple-3}.

Now observe that for every given $\pi_i\in \boldsymbol{\pi}$ we can easily estimate
\[
\int_{\mathbf{C}_2} \dist^2 (x, \pi_i)
\, d\|T\| (x) \leq 2 M^2 + 2 E\, ;
\]
Thus, using Proposition \ref{p:step1} (for any fixed plane $\pi_i\in \boldsymbol{\pi}$), we have for any $1\leq R<2$,
\begin{equation}\label{e:silly-upper-bound}
\mathbf{E} (T, \Cbf_R) \leq \bar{C} (2M^2 + 2E + \mathbf{A}^2)\, .
\end{equation}
where $\bar{C}$ also depends on $2-R$. Under our current assumptions, the right-hand side of \eqref{e:silly-upper-bound} will be at most $\bar{C}$, provided we take $\sigma_4$ sufficiently small (depending only on allowed parameters). In particular, combining this with \eqref{e:small-E-M}, and the assumption $H \leq 1$ (so $E\leq 1$ also) gives
\[
\mathbf{E} (T, \Cbf_R) \leq \bar{C} \mathbf{A}^2 + \bar{C} E^{\frac{2}{m+2}}
\leq \tilde{C} \mathbf{A}^2 + \bar{C} \left(\frac{E}{H^2}\right)^{\frac{2}{m+2}}\, ,
\]
where $\bar{C}$ now also depends on $\tilde{\sigma}$. Now taking $\gamma = \gamma(Q,m,n,\bar{n})>0$ as in Lemma \ref{l:tilt-1}, the above bound clearly implies
\[
    \Ebf(T,\Cbf_R)^\gamma \leq \bar{C}\mathbf{A}^{2\gamma} + \bar{C} \left(\frac{E}{H^2}\right)^{\frac{2\gamma}{m+2}}\, ,
\]
where now $\bar{C}$ also depends on $\gamma$.
Using Lemma \ref{l:tilt-1} and \eqref{e:very-small} we now infer that for any $1\leq \rho < R < 2$ we have
\begin{align}
\mathbf{E} (T, \mathbf{C}_\rho) &\leq \bar{C} (E+\mathbf{A}^2) + \tilde{C} \left[\left(\frac{E}{H^2}\right)^{\frac{2\gamma}{m+2}} + \mathbf{A}^{2\gamma}\right] \mathbf{E} (T, \mathbf{C}_R)\label{e:preliminary-bound-p} \\
&\leq \bar C_1 (E+\mathbf{A}^2) + \bar C_1 \Big(\sigma_4^{\frac{\gamma}{m+2}} + \mathbf{A}^{\frac{(m+3)\gamma}{m+2}}\Big)\left(\frac{E}{H^2}+\mathbf{A}^2\right)^{\frac{\gamma}{m+2}} \mathbf{E} (T, \mathbf{C}_R)\, \nonumber
\end{align}
where $\bar C_1 =\bar C_1 C(Q,m,n,\bar n, \gamma,\tilde\sigma,R-\rho,2-R)$. By selecting $\sigma_4$ small enough we can ensure that
\[
\bar C_1 \sigma_4^{\frac{\gamma}{m+2}} \leq \frac{1}{4}\, .
\]
Next, if $\bar C_1 \mathbf{A}^{\frac{(m+3)\gamma}{m+2}} \leq \frac{1}{4}$
we infer that
\begin{equation}\label{e:second-iteration}
\mathbf{E} (T, \mathbf{C}_\rho) \leq \bar C_1 (E+ \mathbf{A}^2) + \frac{1}{2} \left(\frac{E}{H^2}+\mathbf{A}^2\right)^{\frac{\gamma}{m+2}} \mathbf{E} (T, \mathbf{C}_R)\, .
\end{equation}
Otherwise, if $\bar C_1 \mathbf{A}^{\frac{(m+3)\gamma}{m+2}} \geq \frac{1}{4}$, we can estimate $\mathbf{A}^{2\gamma}\leq \bar{C}_2 \mathbf{A}^2$ for some constant $\bar{C}_2$ with the same dependencies as $\bar{C}_1$. Then, noting that \eqref{e:silly-upper-bound} in particular gives that $\Ebf(T,\Cbf_R)\leq \bar{C}$ (as remarked previously), from \eqref{e:preliminary-bound-p} we have (with the same choice of $\sigma_4$ mentioned previously)
\[
\mathbf{E} (T, \mathbf{C}_\rho) \leq \bar C_3 (E+\mathbf{A}^2) + \frac{1}{4} \left(\frac{E}{H^2}\right)^{\frac{\gamma}{m+2}} \mathbf{E} (T, \mathbf{C}_R)
\]
with a worse constant $\bar C_3$ (with the same dependencies as $\bar{C}$). We can therefore assume \eqref{e:second-iteration} to be valid irrespective of the value of $\mathbf{A}$. 

We are now in a position to iterate \eqref{e:second-iteration}. Fix the smallest natural number $J=J(Q,m,n,\bar n)$ such that $\frac{J\gamma}{m+2} \geq 1$ and apply \eqref{e:second-iteration} with $\rho= r_i$, $R= r_{i-1}$ for a sequence of radii $r_i$, where $r_{J} = r$, $r_{i-1} = r_{i} +\frac{r_0-r}{J}$. This, in particular, fixes the size of the difference between radii used in all the inequalities used above, and in particular it gives a uniform bound for the constants above, hence fixing the choice of $\sigma_4 = \sigma_4(Q,m,n,\bar{n},r_0-r,2-r_0)>0$. 

If for some $i$ we have 
\[
\left(\frac{E}{H^2} + \mathbf{A}^2\right)^{\frac{\gamma}{m+2}}  \mathbf{E} (T, \mathbf{C}_{r_{i-1}})\leq 
\mathbf{E} (T, \mathbf{C}_{r_i})\, 
\] 
we then can absorb the second term in the right hand side of \eqref{e:second-iteration} into the left hand side to conclude that
\[
\mathbf{E} (T, \mathbf{C}_{r_i}) \leq \bar{C} (E+ \mathbf{A}^2)\, .
\]
Given that $1\leq r \leq r_i<2$ we achieve the desired estimate \eqref{e:tilt-estimate-2} in this case. Otherwise, we must have
\[
\mathbf{E} (T, \mathbf{C}_{r_i}) \leq \left(\frac{E}{H^2} + \mathbf{A}^2\right)^{\frac{\gamma}{m+2}}  \mathbf{E} (T, \mathbf{C}_{r_{i-1}})
\]
for all $i$, which leads us to 
\begin{equation}\label{e:end-of-iteration}
\mathbf{E} (T, \mathbf{C}_r) = \mathbf{E} (T, \mathbf{C}_{r_J}) \leq \left(\frac{E}{H^2} + \mathbf{A}^2 \right)^{\frac{J \gamma}{m+2}} \mathbf{E} (T, \mathbf{C}_{r_0})\, .
\end{equation}
Given that $\frac{J\gamma}{m+2}\geq 1$ (so we may write $\frac{J\gamma}{m+2} = 1+\tilde{\gamma}$ for some $\tilde{\gamma}\geq 0$), and since $\frac{E}{H^2}\leq \sigma_4 \leq 1$, we again arrive at the desired estimates \eqref{e:tilt-estimate-2} (using again that $\Ebf(T,\Cbf_{r_0})\leq \bar{C}$ from \eqref{e:silly-upper-bound} to absorb the $\Abf^2$ factor from the parenthesis into the first term on the right-hand side of \eqref{e:tilt-estimate-2}).
\end{proof}

\subsection{Proof of the tilt-excess estimate}
We are finally in a position to prove \eqref{e:tilt-estimate}. The goal is to exploit the combinatorial results in Section \ref{s:comb-lemmas} to remove the hypothesis \eqref{e:very-small} in Proposition \ref{p:tilt-2} by possibly replacing $\boldsymbol{\pi}$ with a smaller, refined sub-collection of planes $\bar{\boldsymbol{\pi}}$, and in turn conclude the tilt-excess estimate for $\bar{\boldsymbol{\pi}}$.

Let $\boldsymbol{\pi}:= \bigcup_i (p_i+\pi_0)$ be as in the statement of Theorem \ref{thm:main-estimate}. Fix a positive parameter $0<\delta\leq 1/2$ (whose choice will be determined later) and apply Lemma \ref{l:clusters-2} to $P=\{p_1,\dots,p_N\}$ with this choice of $\delta$ and $\varepsilon = E^{1/2}$, yielding a subset $\tilde P \subset P$ obeying the conclusions of Lemma \ref{l:clusters-2}. Let $\tilde{\boldsymbol{\pi}}= \bigcup_{\tilde p \in \tilde P}(\tilde p + \pi_0) \subset \boldsymbol{\pi}$ be the corresponding union of parallel planes. By property (i) of Lemma \ref{l:clusters-2} we then have 
\begin{equation}\label{e:upper-bound-bar-1}
\tilde E := \int_{\mathbf{C}_2} \dist^2 (q, \tilde{\boldsymbol{\pi}})\, d\|T\| (q) \leq 
2E + C \max_i \dist(p_i, \tilde{P})^2
\end{equation}
and hence
\begin{equation}\label{e:upper-bound-bar}
\tilde E \leq C(1+\delta^{-2} (1+\delta^{-1})^{2N-4}) E\, .
\end{equation}
If $\tilde{P}$ is a singleton, then we can apply Proposition \ref{p:step1} to $T$ and $\bar{\boldsymbol{\pi}}$ to conclude \eqref{e:tilt-estimate} from \eqref{e:upper-bound-bar}, since we are then in the case $N=1$ of Theorem \ref{thm:main-estimate}.

We may thus henceforth assume that $\tilde P$ consists of at least two distinct points. 
In this case, property (ii) of Lemma \ref{l:clusters-2} gives that $E^{1/2}\leq \delta \tilde H$ and $\max_i \dist(p_i,\tilde P) \leq \delta \tilde H$, where $\tilde H := \min \{|\tilde p-\tilde q|: \tilde p, \tilde q \in \tilde P, \tilde p \neq \tilde q\}$, and thus combining this with \eqref{e:upper-bound-bar-1}, we get
\begin{equation}\label{e:upper-bound-bar-2}
\tilde E \leq C \delta^2 \tilde H^2\, ,
\end{equation}
where 
the constant $C = C(Q,m,n,\bar{n})$ is independent of $\delta$.

Fix another constant $0<\bar\delta\leq 1/2$ (to be determined later). We now apply Lemma \ref{l:clusters-1} with this $\bar\delta$ to further refine $\tilde{\boldsymbol{\pi}}$, 
finding a second subset $P_*\subset \tilde P$ with the properties listed in Lemma \ref{l:clusters-1}. By property (i) of Lemma \ref{l:clusters-1}, if we denote by $M_*:=\max \{|p-q|: p,q\in P'\}$ and $H_*:= \min \{|p'-q'|: p', q'\in P', \ p'\neq q'\}$, we achieve that
\begin{equation}\label{e:M<H}
M_* \leq \bar C(\bar \delta, N) H_*
\end{equation}
and, combining (ii) of Lemma \ref{l:clusters-1} with \eqref{e:upper-bound-bar-2}, we have
\begin{equation}\label{e:-upper-bound-prime}
E_*:=  \int_{\mathbf{C}_2} \dist^2 (q, \boldsymbol{\pi}_*)\, d\|T\| (q) \leq 2\tilde E + C\bar{\delta}^2 H_*^2
\leq C(\delta^2+\bar{\delta}^2)H_*^2\, ,
\end{equation}
since $\tilde H\leq H_*$. Here, $C=C(Q,m,n,\bar{n})$ is a new constant which we stress is independent of $\bar\delta$, and $\boldsymbol{\pi}_*:= \bigcup_{p_*\in P_*} (p_*+\pi_0)$ is the union of parallel planes corresponding to $P_*$. Let $J:=\# (\tilde P \setminus P_*)$ and let $P_j\subset\tilde P$ be the collections of points given by property (iii) of Lemma \ref{l:clusters-1}; in particular, $P_0 = \tilde{P}$, $P_J = P_*$. Define radii $(r_j)_{j=0}^{J+2}$ by $r_{J+2} = 2$, $r_0 = r$, and $r_j - r_{j-1} = \frac{2-r}{J+2}$.
We will apply various estimates on the tilt-excess between the radii $r_j<r_{j-1}$; note that the parameter $\sigma_4$ in Proposition \ref{p:tilt-2} when applied at such scales obeys $\sigma_4 = \sigma_4\big(Q,m,n,\bar{n},N,(2-r)/(J+2)\big)$ and so is now fixed (independent of $j$). One should note however that the dependence of $\sigma_4$ in the radius variable, $(2-r)/(J+2)$, only actually depends on a lower bound on the radius, which here is $(2-r)/N$, meaning $\sigma_4 = \sigma_4(Q,m,n,\bar{n},N,2-r)$. Recalling \eqref{e:upper-bound-bar-2} and \eqref{e:-upper-bound-prime}, we now choose $\delta$ and $\bar{\delta}$, depending only on $Q,m,n,\bar{n}$, so that
\begin{equation}\label{e:condition-to-apply-tilt-estimate}
E_*\leq \sigma_4 H_*^2 \qquad \mbox{and}\qquad \tilde E \leq \sigma_4 \tilde H^2\, .
\end{equation}
We may further assume that $H_*, \tilde H \leq 1$, since otherwise we may apply Lemma \ref{l:simpler} to reach the desired conclusion. Having fixed all the parameters, we may henceforth treat all constants depending on them as just $\bar{C}=\bar{C}(Q,m,n,\bar n,N, 2-r)$.

Now, applying Proposition \ref{p:step1} (or, more precisely, the $N=1$ case of \eqref{e:tilt-estimate}) to $T$ with $r=r_{J+1}$ and $p+\pi_0$ in place of $\pi_0$ for any $p\in P_*$, we get
\[
\Ebf(T,\Cbf_{r_{J+1}}) \leq \bar {C}(E_* + M_*^2 + \Abf^2) \leq \tilde{C}(E_* + H_*^2 + \Abf^2)
\]
using \eqref{e:M<H} in the last inequality (we stress that this is the only time we need to apply Proposition \ref{p:step1}, namely to $P_* =P_J$ as it is the only time we have comparability between the maximum and minimum distances between the planes in $\boldsymbol{\pi}_j$, as defined below). If we now apply Proposition \ref{p:tilt-2} with $\boldsymbol{\pi}_*$, $r_J$, $r_{J+1}$ in place of $\boldsymbol{\pi}$, $r$, $r_0$ respectively, we get
\begin{align}
\Ebf(T,\Cbf_{r_J}) \leq \bar{C}(E_* + \Abf^2) + \bar{C}\left(\frac{E_*}{H^2_*}\right)\Ebf(T,\Cbf_{r_{J+1}}) & \leq \bar{C}(E_* + \Abf^2) + \bar{C}\left(\frac{E_*}{H_*^2}\right)(E_* + H_*^2 + \Abf^2)\nonumber\\
& \leq \bar{C}(E_* + \Abf^2) \label{e:first-tilt-bound}
\end{align}
where in the last inequality we have used the inequality $E_*\leq \sigma_4 H_*^2$ from \eqref{e:condition-to-apply-tilt-estimate}.

Next, for $j=0,1,\dotsc,J$, define the collections $\boldsymbol{\pi}_j := \bigcup_{p\in P_j}(p+\pi_0)$ of parallel planes associated to the sets $P_j$; we remark that $P_J = P_*$. For each such $J$ set
\[
E_j:= \int_{\Cbf_{2}}\dist^2 (q,\boldsymbol{\pi}_j)\ \|T\|(q)
\]
and
\[
H_j := \min\{|p-q|:p,q\in P_j,\, p\neq q\}\, .
\]
Now observe that, for each $j=0,1,\dotsc,J$, by property (iii) of Lemma \ref{l:clusters-1}, we have that
\begin{equation}\label{e:height-recurrence}
E_j \leq 2E_{j-1} + CH_{j-1}^2.
\end{equation}
Now combining \eqref{e:first-tilt-bound}, \eqref{e:height-recurrence} with $j=J$, and the fact that $E_* = E_J$, we have
\[
\Ebf(T,\Cbf_{r_J}) \leq \bar{C}(E_{J-1} + H_{J-1}^2 + \Abf^2).
\]
We now distinguish two possibilities. If $E_{J-1}\leq \sigma_4 H^2_{J-1}$, then we can apply Proposition \ref{p:tilt-2} with $\boldsymbol{\pi}_{J-1}$, $r_{J-1}$, $r_J$ in place of $\boldsymbol{\pi}$, $r$, $r_0$ to get
\[
\Ebf(T,\Cbf_{r_{J-1}}) \leq \bar{C}(E_{J-1}+\Abf^2) + \bar{C}\left(\frac{E_{J-1}}{H_{J-1}^2}\right)\Ebf(T,\Cbf_{r_J}) \leq \bar{C}(E_{J-1}+\Abf^2).
\]
On the other hand, if the opposite inequality holds, namely $H_{J-1}^2 \leq \sigma_4^{-1}E_{J-1}$, then we can estimate directly and see that
\[
\Ebf(T,\Cbf_{r_{J-1}}) \leq \bar{C}\Ebf(T,\Cbf_{r_J}) \leq \bar{C}(E_{J-1}+\Abf^2)
\]
where we remark that we have used in the first inequality that $r\geq 1$ here. So, we see that in either situation, this inequality holds.

Now iterate this argument, namely the one beginning after \eqref{e:height-recurrence}, we see that we get
\[
\Ebf(T,\Cbf_{r_j}) \leq \bar{C}(E_j + \Abf^2)
\]
for each $j=J,J-1,\dotsc,0$. In particular, taking $j=0$, we get
\[
\Ebf(T,\Cbf_r) \leq \bar{C}(E_0 + \Abf^2).
\]
However, since $E_0 = \tilde{E}$ and $\tilde{E}\leq CE$ from \eqref{e:upper-bound-bar}, we reach the desired conclusion.
\qed

\section{Proof of \texorpdfstring{$L^\infty$}{Linfty} height bound}
It remains to prove \eqref{e:Linfty-estimate} of Theorem \ref{thm:main-estimate}, which will be achieved by induction on $N$. As already observed, if $N=1$ we know that Theorem \ref{thm:main-estimate} holds (for all $Q$) by Proposition \ref{p:step1}. The core of the inductive argument is the following proposition.

\begin{proposition}\label{p:inductive}
Fix $N\geq 2$ and assume that, under the assumptions of Theorem \ref{thm:main-estimate}, \eqref{e:Linfty-estimate} holds for any $N'<N$ and any $Q'\leq Q$. Then it holds for $N$ and $Q$. 
\end{proposition}
Clearly once we have shown this, Theorem \ref{thm:main-estimate} follows by induction.

\subsection{A decay lemma} The crucial ingredient for Proposition \ref{p:inductive} is the following $L^2$ height excess decay, which crucially relies on the tilt-excess estimate \eqref{e:tilt-estimate} that we have already established.

\begin{lemma}\label{l:decay}
There are constants $\rho_0 = \rho_0 (m,n,Q)>0$ and $C = C(Q,m,n,\bar{n})>0$ such that, for every fixed $0<\rho\leq \rho_0$, there are
$\sigma_5 = \sigma_5 (Q,m,n,\bar{n},N, \rho)>0$ and $0 < \beta_0 = \beta_0(Q,m,n,\bar{n}) < 1$ such that the following holds. 
Assume $T$, $E$, and $\boldsymbol{\pi}$ are as in Theorem \ref{thm:main-estimate} with $P= \{p_1, \ldots, p_N\}$ and that
\begin{equation}\label{e:very-small-again}
E + \mathbf{A}^2 \leq \sigma_5.
\end{equation} 
Then there is another set of points $P':= \{q_1, \ldots, q_{N'}\}$ with $N'\leq Q$ such that:
\begin{itemize}
\item[(A)] $\dist (q_i, P) \leq C (E+\mathbf{A}^2)^{1/2}$ for each $i$;
\item[(B)] If we set $\boldsymbol{\pi}':= \bigcup (q_i+\pi_0)$, then
\begin{equation}\label{e:decay-estimate}
\int_{\mathbf{C}_{2\rho}} \dist^2 (x, \boldsymbol{\pi}')\, d\|T\| (x) 
\leq \rho^{m+2\beta_0} (E+\mathbf{A}^2)\, .
\end{equation}
\end{itemize}
\end{lemma}

\begin{remark}
    Note that we need not have $N'\leq N$ in the conclusion of Lemma \ref{l:decay}; the number of planes in the new collection $\boldsymbol{\pi}'$ may increase.
\end{remark}

\begin{proof} As usual constants denoted by $C$ will depend only upon $Q,m,n,\bar{n}$ (their dependence on $N$ can be reduced to a dependence on $Q$ given that $N\in \{1, \ldots, Q\}$). Recall that we have just shown the validity of the tilt-excess estimate \eqref{e:tilt-estimate}, and thus (taking $r=1$ in \eqref{e:tilt-estimate}), we have
\begin{equation}\label{e:tilt-estimate-repeated}
\mathbf{E} := \mathbf{E} (T, \mathbf{C}_1) \leq C (E+ \mathbf{A}^2)\, .
\end{equation}
By choosing $\sigma_5 = \sigma_5(Q,m,n,\bar{n},N)>0$ sufficiently small, we can therefore ensure that both $\mathbf{E}$ and $\mathbf{A}$ are as small as we wish. We now subdivide into two cases, depending on the relative sizes of $\Ebf$ and $\Abf$.

\medskip

{\bf Case 1:} $\mathbf{A}^3 \leq \mathbf{E}$. Here we apply the strong Lipschitz approximation theorem \cite{DLS14Lp}*{Theorem 2.4}. By translating, we may without loss of generality assume that the origin belongs to ${\spt}\, (T)$ (and hence to the manifold $\Sigma$), that $\pi_0=\mathbb R^m\times \{0\}$ (by rotating) and that $\Psi \equiv \Psi_0: \mathbb{R}^{m+\bar n} \to \mathbb R^{n-\bar n}$ is the map parametrizing $\Sigma$ graphically over $\mathbb{R}^{m+\bar n}$ in $\Cbf_2$, as in Assumption \ref{a:main}. By \cite{DLS14Lp}*{Remark 2.5}, we have the estimates
\begin{equation}\label{e:estimates_psi}
\Psi (0)=0 \quad\mbox{and}\quad \|D\Psi\|_{C^2} \leq C (\mathbf{E}^{1/2}+\Abf)\, ;
\end{equation}
here, $C = C(m,n,\bar{n}).$ By \cite{DLS14Lp}*{Theorem 2.4} there exists constants $\gamma = \gamma(Q,m,n,\bar{n}) > 0$ and $\varepsilon = \varepsilon(Q,m,n,\bar{n})>0$ such that if $\Ebf<\varepsilon$ (which can be guaranteed provided $\sigma_5$ is sufficiently small), then there is a multi-valued map $u: B_{1/4} (0, \pi_0) \to \mathbb R^{\bar n}$ such that $f = (u, \Psi (x, u))$ (using the notation of \cite{DLS14Lp}) is a good approximation of $T\res \mathbf{C}_{1/4}$, in the following sense:
\begin{itemize}
\item[(i)] ${\rm Lip}\, (f) \leq C (\mathbf{E}+\mathbf{A}^2)^\gamma \leq C(E+\Abf^2)^\gamma$;
\item[(ii)] There is a closed set $K\subset B_{1/4}$ of measure at least $\frac{1}{2}|B_{1/4}|$ such that $\mathbf{G}_f \res (K\times \mathbb R^n)=T\res (K\times \mathbb R^n)$ and 
\[
\|T\| ((B_{1/4}\setminus K)\times \mathbb R^n) \leq C (\mathbf{E}+\mathbf{A}^2)^{1+\gamma} \leq C(E+\mathbf{A}^2)^{1+\gamma} \, ;
\]
\end{itemize}
where $C = C(Q,m,n,\bar{n})$. Notice that in the above estimates, we have used \eqref{e:tilt-estimate-repeated} to get the improved control in terms of the $L^2$ height excess. From the above properties, the estimate \cite{DLS14Lp}*{Theorem 2.4(2.6)} and \eqref{e:estimates_psi} we also see that 
\begin{equation}\label{e:Dirichlet}
\int_{B_{1/4}} |Df|^2 \leq C\int_K |Df|^2 + C (\mathbf{E}+\mathbf{A}^2)^{1+\gamma} \leq C \Ebf + C(\mathbf{E}+\mathbf{A}^2)^{1+\gamma} \overset{\eqref{e:very-small-again}}{\leq} C (\mathbf{E} + \mathbf{A}^2)\, .
\end{equation}
Moreover, for every fixed $\eta$, if $\mathbf{E}$ is sufficiently small (depending on $\eta$), by \cite{DLS14Lp}*{Theorem 2.6} there is a Dir-minimizing function $v:B_{1/4}\to \mathcal{A}_Q (\mathbb R^{\bar n})$ such that, if we set $g = (v, \Psi (x,v))$, then
\begin{align}
&\int_{B_{1/4}} \mathcal{G} (f,g)^2 \leq \eta \mathbf{E} \label{e:f,g-bound-1}\\
&\int_{B_{1/4}} |Dg|^2 \leq \int_{B_{1/4}} |Df|^2 + \eta \mathbf{E}\overset{\eqref{e:Dirichlet}}{\leq} C (\mathbf{E} + \mathbf{A}^2) \overset{\eqref{e:tilt-estimate-repeated}}{\leq} C(E+\Abf^2)\, . \label{e:f,g-bound-2}
\end{align}
Note that the condition $\Abf\leq \Ebf^{\frac{1}{4}+\bar\delta}$ in \cite{DLS14Lp}*{Theorem 2.6} is satisfied, with $\bar\delta = 1/12 $ in this case, since we are assuming $\Abf^3\leq \Ebf$.

Observe now that, by a simple Chebyshev argument, for at least half of the points $x\in K\cap B_{1/8}$ we have the following property:
\begin{itemize} 
\item[(a)] if $f(x)= \sum_i Q_i \llbracket f_i (x)\rrbracket$ (with $f_i(x)$ distinct), for each $i$ there is a $j(i)$ such that $|f_i(x) - p_{j(i)}|\leq C E^{1/2}$\, .
\end{itemize}
Hence, by another Chebyshev argument, combined with \eqref{e:f,g-bound-1}, we may find at least one such point $x\in K\cap B_{1/8}$ for which we have the corresponding property for $g(x)$ (in fact, we can find a set of positive measure on which this holds):
\begin{itemize}
    \item[(b)]  if $g(x)= \sum_i \tilde{Q}_i \llbracket g_i (x)\rrbracket$ (with $g_i(x)$ distinct), for every $i$ there is a $j(i)$ such that $|g_i (x)-p_{j(i)}|\leq C E^{1/2}$.
\end{itemize}
Now, combining \eqref{e:f,g-bound-2} with the H\"older estimate \cite{DLS_MAMS}*{Theorem 3.9} for Dir-minimizing functions, we also conclude an analogous estimate at $0$:
\begin{itemize}
    \item[(c)] if $g_i (0) = \sum_i Q^*_i \llbracket g_i (0)\rrbracket$ (with $g_i(0)$ distinct), for every $i$ there is a $j(i)$ such that $|g_i (0)-p_{j(i)}|\leq C (E^{1/2} + \mathbf{A})$
\end{itemize}
Notice that since $x\in B_{1/8}$, $\dist(x,\partial B_{1/4})\geq C^{-1}$ and so the choice of $\delta$ in \cite{DLS_MAMS}*{Theorem 3.9} only depends on $n,m,Q$). We now set $q_i := g_i (0)$; obviously the number of distinct points $q_i$ is at most $Q$. Clearly, by (c) above, this choice of $q_i$ obeys conclusion (A) of the lemma.

Observe that, by all the estimates carried over so far and by Lemma \ref{l:simple-2}, for each $\rho\leq \frac{1}{8}$ and some $\alpha = \alpha(m,Q)\in \ ]0,1[$ we have
\begin{align*}
\int_{\mathbf{C}_{2\rho}} \dist^2 (x, \boldsymbol{\pi}')\, d\|T\| (x) &
\overset{\text{(ii)}}{\leq} C (E + \mathbf{A}^2)^{1+\gamma} + C \int_{B_{2\rho}} \mathcal{G} (f(x), g (0))^2\, dx\\
&\leq C \sigma_5^{\gamma} (E+\mathbf{A}^2) + C \eta (E+\mathbf{A}^2) + C \rho^{m+2\alpha} (E+\mathbf{A}^2)\, ;
\end{align*}
where Lemma \ref{l:simple-2} was used in the first inequality, while in the second inequality we use \eqref{e:very-small-again}, \eqref{e:f,g-bound-1}, the $\alpha$-H\"older continuity of $g$ from \cite{DLS_MAMS}*{Theorem 3.9}, and \eqref{e:f,g-bound-2}.

Now choose $\beta_0\in \ ]0, \alpha/2]$, and choose $\rho_0>0$ with $C\rho_0^\alpha\leq\frac{1}{3}$. Then choose $\eta=\eta(\rho)$ such that $C\eta\leq\frac{1}{3}\rho^{m+2\beta_0}$. Then, given this choice of $\eta$, provided that we choose $\sigma_5$ sufficiently small so that our application of \cite{DLS14Lp}*{Theorem 2.6} was valid with this $\eta$, and so that $C\sigma_5^\gamma \leq \frac{1}{3}\rho^{m+2\beta_0}$, the above estimate clearly then gives (B) of the lemma. This completes the proof of Case 1. 

\medskip

{\bf Case 2:} In this case we have $\mathbf{E} \leq \mathbf{A}^3$. As in the previous case, using the same notation, we introduce the Lipschitz approximation $f = (u, \Psi (x,u))$ on $B_{1/4}(0,\pi_0)$. Observe that, since the graph of $f$ coincides with the current $T$ on $K\times \mathbb R^n\subset B_{1/4}\times\R^n$, we again have the first two inequalities of \eqref{e:Dirichlet}, only now we further estimate this as follows:
\[
\int_{B_{1/4}} |Df|^2 \leq C \Ebf + C (\mathbf{E} + \mathbf{A}^2)^{1+\gamma}
\leq C \mathbf{A}^{2+2\gamma} \leq C \sigma_5^{\gamma} \mathbf{A}^2\, .
\]
We then use the Poincar\'e inequality for $Q$-valued functions (see e.g. \cite{DLS_MAMS}*{Proposition 2.12 \& Proposition 4.9}) and H\"older's inequality to find a point $Y\in \mathcal{A}_Q$ such that
\[
\int_{B_{1/4}} \mathcal{G} (f,Y)^2 \leq C \left(\int_{B_{1/4}} \mathcal{G} (f, Y)^{2^*}\right)^{2/2^*}
\leq C \int_{B_{1/4}} |Df|^2\, .
\]
In particular we reach
\begin{equation}
\int_{B_{1/4}} \mathcal{G} (f,Y)^2\leq C \sigma_5^{\gamma} \mathbf{A}^2\, .\label{e:latter-inequality}
\end{equation}
Write $Y = \sum_i Q_i \llbracket q_i \rrbracket$, where the $q_i$ are distinct, and set $\boldsymbol{\pi}':= \bigcup_i (q_i+\pi_0$). We then find that,
for every $0<\rho \leq \frac{1}{8}$, by Lemma \ref{l:simple-2}, \eqref{e:latter-inequality}, and (ii), we have
\[
\int_{\mathbf{C}_{2\rho}} \dist^2 (x, \boldsymbol{\pi}')\, d\|T\| (x) 
\leq C \sigma_5^{\gamma} \mathbf{A}^2 + \|T\| ((B_{1/4}\setminus K)\times \mathbb R^n)
\leq C \sigma_5^\gamma \mathbf{A}^2\, .
\]
In particular, for every fixed $\rho$, we may fix $\beta_0$ to be an arbitrary positive dimensional constant and choose $\sigma_5$ small enough (dependent on $\beta_0$) to guarantee that $C\sigma_5^\gamma \leq \rho^{m+2\beta_0}$
\[
\int_{\mathbf{C}_{2\rho}} \dist^2 (x, \boldsymbol{\pi}')\, d\|T\| (x) \leq \rho^{m+2\beta_0} \mathbf{A}^2\, .
\]
This proves conclusion (B) of the lemma, in this case.

We are left with proving the estimate (A). Recall that, in light of (ii), we have $|B_{1/4}\setminus K|\leq \frac{1}{2} |B_{1/4}|$, and hence, via a similar Chebyshev argument as in Case 1, now using \eqref{e:latter-inequality}, we may estimate
\[
\dist^2 (q_i, {\rm spt}\, f (x))\leq C \mathbf{A}^2\, 
\qquad \forall i \, ,
\]
for all $x$ in a subset $K'\subset K$ whose measure is at least $\frac{1}{4} |B_{1/4}|$. On the other hand, on at least half of the set $K'$ we also have by another Chebyshev argument
\[
\max_{p\in {\rm spt}\, f (x)} \dist^2 (p, \boldsymbol{\pi})\leq C E\, . 
\]
Therefore, choosing a point $p\in\spt f(x)$ for $x$ in this latter set and using the triangle inequality we get
\[
\dist^2 (q_i, \boldsymbol{\pi}) \leq C (E+\mathbf{A}^2)\,  \qquad \forall i\, ,
\] 
which proves (A) and thus completes Case 2, which also completes the proof.
\end{proof}

\subsection{Proof of Proposition \ref{p:inductive}} 
Having proved the $L^2$ height excess decay lemma, we are now in a position to prove the inductive step given in Proposition \ref{p:inductive}. Fix $1\leq r<2$ as in the statement of Theorem \ref{thm:main-estimate} and set
\[
M := \max_{i,j} \{|p_i-p_j|\}\, .
\]
We will show the existence of a constant $\sigma_6 = \sigma_6(Q,m,n,\bar{n},N,r)>0$ such that, if 
\begin{equation}\label{e:small-again}
E+\mathbf{A}^2 \leq \sigma_6 M^2\, ,
\end{equation} 
then there is decomposition $\boldsymbol{\pi}=\boldsymbol{\pi}_1\cup \boldsymbol{\pi}_2$ and $T\res \Cbf_{1+r/2} = T_1 + T_2$ such that 
\begin{itemize}
\item[(i)] The sets $\boldsymbol{\pi}_1,\boldsymbol{\pi}_2$ are disjoint and nonempty;
\item[(ii)] $\partial T_i \res \Cbf_{1+r/2} = 0$;
\item[(iii)] ${\rm spt}\, (T_1) \cap {\rm spt}\, (T_2) = \emptyset$;
\item[(iv)] $\dist (q, \boldsymbol{\pi}) = \dist (q, \boldsymbol{\pi}_i)$ for every $q\in {\rm spt}\, (T_i)$, for each $i=1,2$.
\end{itemize}
In particular, assuming such a decomposition, after rescaling $1+r/2$ to scale $2$ we can apply the inductive assumption to $T_1$ and $T_2$ (note that whilst we could have, e.g. $T_2 = 0$, i.e. $Q$ remains fixed, however by (i) we then know that $N$ strictly decreases) to conclude
\[
{\rm spt} (T_i)\cap \mathbf{C}_r \subset \{\dist (\cdot, \boldsymbol{\pi}_i)\leq \bar{C} (E^{(i)}+\mathbf{A}^2)^{1/2}\}
\]
where $\bar{C} = \bar{C}(Q,m,n,\bar{n},N,r)$ and 
\[
E^{(i)} := \int_{\mathbf{C}_{1+r/2}} \dist^2 (x, \boldsymbol{\pi}_i)\, d\|T_i\| (x)\, .
\]
Given that obviously $E^{(1)},E^{(2)}\leq CE$, we would have reached the desired estimate \eqref{e:Linfty-estimate}, provided that \eqref{e:small-again} holds. On the other hand, if the estimate \eqref{e:small-again} does not hold, i.e. $M^2 \leq \sigma_6^{-1} (E+\mathbf{A}^2)$, we can take an arbitrary plane $p_i + \pi_0$ from $\boldsymbol{\pi}$ and conclude that
\[
\int_{\mathbf{C}_2} \dist^2 (x, p_i+\pi_0)\, d\|T\| (x) \leq \bar{C} (E+\mathbf{A}^2)\, .
\]
It would then follow from the case $N=1$ of Theorem \ref{thm:main-estimate}, namely Proposition \ref{p:step1}, that
\[
{\rm spt}\, (T) \cap \mathbf{C}_r \subset \{\dist (\cdot, p_i+\pi_0) \leq \bar{C} (E+\mathbf{A}^2)^{1/2}\}
\subset \{\dist (\cdot, \boldsymbol{\pi})\leq \bar{C} (E+\mathbf{A}^2)^{1/2}\}\, ,
\]
for any fixed plane $p_i+\pi_0\in \boldsymbol{\pi}$, proving \eqref{e:Linfty-estimate}.

We are thus left to prove the existence of the decomposition satisfying (i)-(iv) under the assumption that \eqref{e:small-again} holds for a sufficiently small $\sigma_6$. We first use the combinatorial Lemma \ref{l:clusters-3} to decompose $P \equiv \{p_1,\dotsc,p_N\}= P_1\cup P_2$, and consequently $\boldsymbol{\pi} = \boldsymbol{\pi}^*_1\cup \boldsymbol{\pi}^*_2$ for $\boldsymbol{\pi}_i^* =\bigcup_{p\in P_i}(p+\pi_0)$, with the property that $P_1$ and $P_2$ are disjoint non-empty sets satisfying 
\begin{equation}\label{e:plane-separation}
{\rm sep}\, (P_1,P_2) \equiv \min \{|p_1-p_2|: p_1\in P_1,\, p_2\in P_2\} \geq \frac{M}{2^{N-2}}\, .
\end{equation}
Our claim is that the desired decomposition will be achieved with the latter choice. In light of \eqref{e:plane-separation}, this will follow immediately if we have that 
\begin{equation}\label{e:close-enough}
{\rm spt}\, (T) \cap \mathbf{C}_{1+r/2} \subset \{x:\dist (x, \boldsymbol{\pi}) \leq 2^{-N} M\}\, .
\end{equation}
Let us thus show that \eqref{e:close-enough} indeed holds. Fix a point $x\in B_{1+r/2} (0, \pi_0)$ and the scaling factor $1-r/2$. Let $T^0$ be the rescaled current $(\iota_{x,1-r/2})_\sharp T$. Likewise, introduce the sets $\boldsymbol{\pi}_0 := \iota_{x,1-r/2} (\boldsymbol{\pi})$, $\Sigma^0 = \iota_{x, 1-r/2} (\Sigma)$, the shorthand notation $\mathbf{A}_0$ for the supremum norm of the second fundamental form of $\Sigma^0$, and 
\[
E_0 := \int_{\mathbf{C}_2} \dist^2 (y, \boldsymbol{\pi}_0)\, d\|T^0\| (y)\, .
\]
Observe that $E_0 \leq (1-r/2)^{-m-2} E$, $\mathbf{A}_0^2 = (1-r/2)^2 \mathbf{A}^2$. Setting $M_0:= (1-r/2)^{-1} M$, since $1-\frac{r}{2} \leq 1$ we see immediately that 
\[
E_0 + \mathbf{A}_0^2 \leq (1-r/2)^{-m} \sigma_6 M_0^2\, .
\]
Let now $\bar \sigma$ be the threshold of Lemma \ref{l:simpler-2}. We start by choosing $\sigma_6$ small enough so that 
\begin{equation}\label{e:one-condition-for-sigma}
\sigma_7 := (1-r/2)^{-m} \sigma_6 \leq \bar \sigma\, .
\end{equation}
In particular, if $M_0\geq 1$ we could apply Lemma \ref{l:simpler-2} and achieve the estimate
\[
{\rm spt}\, (T^0) \cap \mathbf{C}_1 \subset  \left\{\dist (\cdot , \boldsymbol{\pi}_0) \leq 2^{-N}M_0\right\}\, .
\]
Given that this holds for every $x\in B_{1+r/2} (0, \pi_0)$, it would then imply \eqref{e:close-enough} and conclude the proof. 

Otherwise, if $M_0 <1$, we may fix $\rho = \rho_0<1$, with $\rho_0$ as in Lemma \ref{l:decay}, and let $\sigma_5$ be the corresponding threshold from Lemma \ref{l:decay}. We may further choose $\sigma_7$ (as defined in \eqref{e:one-condition-for-sigma}) small enough such that $\sigma_7 \leq \sigma_5$. In particular, since $M_0 < 1$, we can apply Lemma \ref{l:decay} to $T^0$ and $\boldsymbol{\pi}_0$. Denote by $\boldsymbol{\pi}'_0$ the corresponding new group of planes and let 
\begin{itemize}
\item $T^1 = (\iota_{0, \rho})_\sharp T^0$;
\item $\Sigma^1 = \iota_{0, \rho} (\Sigma^0)$, with $\Abf_1$ the supremum norm of the second fundamental form of $\Sigma^1$;
\item $\bar{\boldsymbol{\pi}}_1 = \iota_{0, \rho} (\boldsymbol{\pi}'_0)$;
\item $M_1 = \rho^{-1} M_0$;
\item $\bar E_1 = \int_{\mathbf{C}_{2}} \dist^2 (x, \bar{\boldsymbol{\pi}}_1)\, d\|T^1\| (x)$.
\end{itemize}
Observe that from \eqref{e:decay-estimate} we have the inequality (with $\beta_0<1$ as in Lemma \ref{l:decay})
\begin{equation}\label{e:scaled-decay}
\bar E_1 + \mathbf{A}_1^2 \leq \rho^{2\beta_0-2}(E_0+\Abf_0^2) \leq \rho^{2\beta_0-2}\sigma_7 M_0^2 =  \rho^{2\beta_0}\sigma_7 M_1^2\, .
\end{equation}
As $\rho<1$, we can thus keep iterating this procedure with the same $\rho$, and along the iteration each $\bar{\boldsymbol{\pi}}_k$ consists of some number of planes, which could be larger than the initial number $N$, but it never exceeds $Q$. This in particular means that the parameter $\sigma_7$ remains fixed, dependent on only $\rho$, $Q$, $N$ and dimensional constants.  generating currents $T^j$, ambient manifolds $\Sigma^j$, families of parallel planes $\boldsymbol{\pi}'_j$ and $\bar{\boldsymbol{\pi}}_j = \iota_{0, \rho} (\boldsymbol{\pi}'_{j})$ until we reach the first index $k$ such that 
\[
    M_k = \rho^{-k} M_0 \geq 1.
\]

Let us now see how this iteration affects our original current $T$. Setting
\[
\bar E_j := \int_{\mathbf{C}_{2}} \dist^2 (x, \bar{\boldsymbol{\pi}}_j)\, d\|T^j\| (x)\, 
\] 
we may conclude, analogously to \eqref{e:scaled-decay}, that 
\[
\bar E_j + \mathbf{A}_j^2 \leq \rho^{2\beta_0j} \sigma_7 M_j^2\, .
\]
Now define $\boldsymbol{\pi}_j := \iota_{0, \rho^j} (\boldsymbol{\pi}_0) =
\iota_{0, \rho^j} (\iota_{x,(1-r/2)} (\boldsymbol{\pi}))$ and set
\[
E_j := \int_{\mathbf{C}_2} \dist^2 (x, \boldsymbol{\pi}_j)\, d\|T^j\| (x)\, .
\]
We wish to estimate $E_j$ and the distance of the new family of planes $\boldsymbol{\pi}'_j$ from the one $\boldsymbol{\pi}_j$ obtained by rescaling $\boldsymbol{\pi}_0$. We therefore introduce
 \[
d (j):=  \max \{\dist (\pi, \boldsymbol{\pi}_j): \pi \in \boldsymbol{\pi}'_j\}
 \]
 Note that from (A) from Lemma \ref{l:decay} we have
 \[
 d (0) \leq C (E_0 + \mathbf{A}_0^2)^{1/2} \leq C \sigma_7^{1/2} M_0\, .
 \]
 Subsequently, using the triangle inequality (to estimate the distance from $\boldsymbol{\pi}_1$ to $\bar{\boldsymbol{\pi}}_1$ and then from $\bar{\boldsymbol{\pi}}_1$ to $\boldsymbol{\pi}^\prime_1$)
 \[
 d(1) \leq \left(C \sigma_7^{1/2} M_0\right) \rho^{-1} + C (E_1+\mathbf{A}_1^2)^{1/2} 
 \leq C \sigma_7^{1/2} M_1 (1+\rho^\beta)\, 
 \]
 where we have used our previous bounds and the fact $M_1 = \rho^{-1}M_0$. Inductively, we achieve 
 \[
 d (j) \leq C \sigma_7^{1/2} M_j \sum_{i=0}^j \rho^{j\beta}\, ;
 \]
 we stress that $C$ is the same each time we apply the lemma and so it is constant in $j$. In particular, 
 \[
 d (k) \leq C \sigma_7^{1/2} M_k\, .
 \]
 Since
 \[
 E_k \leq C d(k)^2 + C \bar E_k\, ,
 \]
 we conclude that 
 \[
 E_k \leq C \sigma_7 M_k^2\, ,
 \]
 where the constant $C$ depends only on $m,n,\bar{n}$, and $Q$. In particular, due to the fact that $M_k\geq 1$ by construction (and $M_k$ is the maximal distance between the planes in $\boldsymbol{\pi}_k$), we may choose $\sigma_7$ smaller than a suitable geometric constant in order to apply Lemma \ref{l:simpler-2}, concluding that
 \[
 {\rm spt}\, (T^k) \cap \mathbf{C}_1 \subset \{\dist (\cdot, \boldsymbol{\pi}_k) \leq 2^{-N} M_k\}\, .
 \]
 Rescaling this information in the original system of coordinates gives 
 \[
 {\rm spt}\, (T) \cap \mathbf{C}_{\rho^k (1-r/2)} (x) \subset \{\dist (\cdot, \boldsymbol{\pi})\leq 2^{-N} M\}\, .
 \]
 The arbitrariness of $x\in B_{1+r/2} (0, \pi_0)$ gives then \eqref{e:close-enough} and hence concludes the proof.
 \qed

\medskip

 The proof of Theorem \ref{thm:main-estimate} is therefore complete.



\part{Proof of the Main Excess Decay Theorem}\label{p:decay-proof}

In this part, we prove Theorem \ref{c:decay}. The overall structure is as follows. In Section \ref{p:planes}, we begin with a few technical preliminaries regarding the relative positioning of planes, and some combinatorial preliminaries which will be needed for the graphical parameterizations in the following section. In Section \ref{p:approx}, we set up efficient graphical parameterizations for $T$ over \emph{balanced} cones in $\Cscr(Q)$ (as defined in the preceding section), away from the spines of the cones and when the two-sided excess of $T$ relative to $\Sbf$ is much smaller than the \emph{minimal angle} of the cone (such an assumption being used in this context is equivalent to that used by Wickramasekera \cite{W14_annals}*{Section 10, Hypothesis $(\star\star)$}). This construction heavily relies on the height bound from Part \ref{p:heightbd}, and given this our construction is a suitable adaptation of those seen in the work of Simon (\cite{Simon_cylindrical}) and Wickramasekera (\cite{W14_annals}), with added technical complications from working in arbitrary codimension and with area-minimizing currents. Section \ref{s:balancing} is then dedicated to showing that we may assume that the cone $\Sbf$ in Theorem \ref{c:decay} is balanced, allowing us to approximate $T$ with a multi-valued graph over this cone away from its spine, as demonstrated in the previous section. In Section \ref{s:reduction} we reduce the proof of Theorem \ref{c:decay} to an a priori much weaker weaker decay theorem which has a much stronger assumption on the size of the $L^2$ height excess of $T$ relative to $\Sbf$; the idea is to use an excess decay statement with multiple scales, analogous to that first seen in \cite{W14_annals}*{Section 13}. In Section \ref{p:estimates}, we then show the non-concentration of excess and the Simon estimates at the spine of the cone $\Sbf$; these are the key estimates first used in Simon's work (\cite{Simon_cylindrical}) and then adapted to a degenerate setting, as we similar face here, by Wickramasekera (\cite{W14_annals}*{Section 10}). We then display the analogue of Theorem \ref{c:decay} at the linearized level (for multi-valued Dir-minimizers) in Section \ref{p:linear}. Finally, in Section \ref{p:blowup}, we conclude with a blow-up procedure, collecting all the previous arguments to yield a limiting Dir-minimizer that contradicts the results in Section \ref{p:linear} if we assume the failure of Theorem \ref{c:decay}.

\section{Relative position of pairs of planes}\label{p:planes}
We begin with some elementary results on the relative positions of pairs of planes. These will be useful for results appearing later in this section regarding comparability of angle parameters for cones in $\Cscr(Q)$ and multi-valued Dir-minimizers formed from superpositions of linear maps.

We start with a simple lemma in linear algebra.

\begin{lemma}\label{l:symmetry}
Consider two $m$-dimensional linear subspaces $\alpha$, $\beta$ of $\mathbb R^N$ and the two quadratic forms $Q_1:\alpha\to \R$, $Q_2:\beta\to \R$ given by $Q_1(v) := \dist^2(v,\beta)$ and $Q_2(w):= \dist^2(w,\alpha)$. Then:
\begin{itemize}
\item[(a)] $Q_1$ and $Q_2$ have the same eigenvalues with the same multiplicities.
\item[(b)] For every orthonormal base $v_1, \ldots , v_m$ of eigenvectors of $Q_1$ there is a corresponding orthonormal base $w_1, \ldots , w_m$ of eigenvectors of $Q_2$ with the property that
\[
Q_1 (v_i) = Q_2 (w_i) = 1 -  (w_i\cdot v_i)^2\, ,
\]
and $v_i \cdot w_j =0$ for $i\neq j$.
\item[(c)] If $\alpha \cap \beta^\perp = \{0\}$ a choice of $w_1, \ldots , w_m$ is given by $w_i = \frac{\mathbf{p}_\beta (v_i)}{|\mathbf{p}_\beta (v_i)|}$.
\end{itemize}
\end{lemma}
Before proving this, let us note that this lemma motivates the following definition:

\begin{definition}\label{d:frank}
Given two $m$-dimensional linear subspaces $\alpha, \beta$ of $\mathbb R^N$ whose intersection has dimension $m-k$, we order the $k$ positive eigenvalues $\lambda_1 \leq \lambda_2 \leq \cdots \leq \lambda_k$ of $Q_1$ as in Lemma \ref{l:symmetry}, with the convention that the number of occurrences of the same real $\lambda$ in the list equals its multiplicity as eigenvalue of $Q_1$. The \textit{Morgan angles} of the pair $\alpha$ and $\beta$ are the numbers $\theta_i (\alpha, \beta) := \arcsin \sqrt{\lambda_i}$ for $i=1,\dots,k$.
\end{definition}

Of course we can define the Morgan angles between two intersecting affine planes of the same dimension by simply translating an intersection point to the origin.

The following is then an immediate corollary of Lemma \ref{l:symmetry} and basic linear algebra:

\begin{corollary}\label{c:growth}
Suppose that $\alpha, \beta$ are two $m$-dimensional linear subspaces of $\R^N$. Denote by $V$ their $(m-k)$-dimensional intersection and let $\theta_1 \leq \cdots \leq \theta_k$ be their Morgan angles. Then 
\begin{equation}\label{e:Morgan-max-min}
|v| \sin \theta_1 \leq \dist (v, \beta) \leq |v| \sin \theta_k
\qquad \forall v\in V^\perp\cap \alpha\, ,
\end{equation}
and
\begin{equation}\label{e:Haus=max_eigen}
\dist (\alpha \cap \Bbf_1, \beta \cap \Bbf_1) = \sup \{ \dist (v, \beta): v\in \alpha\cap \Bbf_1\} = \sin \theta_k\, .
\end{equation}
\end{corollary}

Indeed, this simply follows because $V^\perp\cap \alpha$ is the eigenspace of the quadratic form $Q_1 = \dist^2(\cdot,\beta)$ spanned by those eigenvectors which have positive eigenvalues, combined with the fact that the extremal eigenvalues of a quadratic form are realized through constrained optimization over the unit sphere. Note that the first equality in \eqref{e:Haus=max_eigen} is simply due to the fact that $\alpha$ and $\beta$ are $m$-dimensional subspaces.

\begin{proof}[Proof of Lemma \ref{l:symmetry}]
We prove (b) and (a) at the same time. Moreover, since we give an explicit construction for the vectors $w_i$, the proof will also show (c). We start by observing that 
\begin{equation}\label{e:clarifies_q_form_1}
Q_1 (v) = |\mathbf{p}_\beta^\perp (v)|^2 = v \cdot (\mathbf{p}_\beta^\perp (v))\, .
\end{equation}
Now pick an orthonormal basis $\{v_1,\dotsc,v_m\}$ of $\alpha$ which diagonalizes $Q_1$. Moreover, let us remark that we can extend $Q_1$ to a quadratic form $\bar{Q}_1$ on $\mathbb R^N$ by setting it identically $0$ on $\alpha^\perp$. $\bar{Q}_1$ can be computed explicitly as 
\begin{equation}\label{e:clarifies_q_form_2}
\bar{Q}_1 (z) = z \cdot (\mathbf{p}_\alpha \circ \mathbf{p}_\beta^\perp  \circ \mathbf{p}_\alpha (z)) 
\end{equation}
for a general $z\in \mathbb R^N$. In particular, each $v_i$ must be an eigenvector of the symmetric matrix $\mathbf{p}_\alpha \circ \mathbf{p}_\beta^\perp  \circ \mathbf{p}_\alpha$, which in turn implies that 
\begin{equation}\label{e:eigenvalue_eq}
\mathbf{p}_\alpha (\mathbf{p}^\perp_\beta(v_i)) = \lambda_i v_i
\end{equation}
for some constant $0\leq \lambda_i \leq 1$. Observe, moreover, that in case $\lambda_i = 1$, then necessarily
\[
\mathbf{p}^\perp_\beta (v_i) = v_i\, .
\]
Without loss of generality let us order the $v_i$ to assume that $0\leq Q_1(v_i) \leq Q_1(v_j)\leq 1$ for all $i\leq j$. Let us write $k$ for the largest element of $\{1,\dotsc,m\}$ for which $Q_1(v_k)<1$. Observe that $\mathbf{p}_\beta(v_j) = 0$ for all $j>k$. For $i\leq k$, define $w_i:= \frac{\mathbf{p}_\beta(v_i)}{|\mathbf{p}_\beta(v_i)|}$, and write $W = \mathrm{span}\{w_1,\dotsc,w_k\}$. Select any orthonormal basis $\{w_{k+1},\dotsc,w_m\}$ of $W^\perp\cap \beta$. Since $\mathbf{p}_\beta(v_j) = 0$ for $j>k$, it is obvious that $W = \mathbf{p}_\beta(\alpha)$. In particular, it follows that $w_j\perp \alpha$ for every $j>k$, and so for $j>k$ we have $1 = Q_1(v_j) = Q_2(w_j)$ and $ w_j\cdot v_i = 0$ for all $i$. 

To handle the case $j\leq k$, note that, if $i\neq j$, then
\begin{align*}
v_i\cdot w_j &= |\mathbf{p}_\beta (v_j)|^{-1} (v_i \cdot \mathbf{p}_\beta (v_j)) = - |\mathbf{p}_\beta (v_j)|^{-1} (v_i \cdot \mathbf{p}_\beta^\perp (v_j))\\
&= - |\mathbf{p}_\beta (v_j)|^{-1} (v_i \cdot \mathbf{p}_\alpha (\mathbf{p}_\beta^\perp (v_j))) \stackrel{\eqref{e:eigenvalue_eq}}{=}
- \lambda_j |\mathbf{p}_\beta (v_j)|^{-1} v_i\cdot v_j = 0\, .
\end{align*}
Next consider $v= \sum_{i=1}^k a_i v_i$ and compute
\begin{align*}
    Q_1(v) &= |\mathbf{p}^\perp_\beta(v)|^2 = |v|^2 - |\mathbf{p}_\beta(v)|^2\\
    & = \sum_i a_i^2 - \left|\sum_{i=1}^k (w_i\cdot v) w_i \right|^2
 = \sum_i a_i^2 - \left|\sum^k_{i=1}a_i (w_i\cdot v_i) w_i\right|^2\\
    & = \sum_i a_i^2\left(1-(w_i\cdot v_i)^2\right) - 2\sum_{i<j\leq k} a_ia_j ( w_i \cdot w_j) (w_i\cdot v_i) (w_j\cdot v_j)\, .
\end{align*}
Given that $Q_1(v) = \sum_i a_i^2 Q_1(v_i) = \sum_i a_i^2(1-(v_i\cdot w_i)^2)$ and the coefficients $a_i$ can be chosen arbitrarily, we must then have
\[
(w_i\cdot w_j) ( w_i\cdot v_i) (w_j \cdot v_j) = 0
\]
for every $1\leq i<j\leq k$. In particular, since $(w_i\cdot v_i) (w_j \cdot v_j) \neq 0$ when $i<j\leq k$ we immediately infer that $w_i\cdot w_j =0$, hence concluding that $w_1,\dotsc,w_m$ is an orthonormal basis of $\beta$. Moreover, we have already seen that $Q_1(v_i,v_i) = 1- (v_i\cdot w_i)^2$.

Next, fix $j\leq k$ and observe that, by construction, $w_j$ is orthogonal to $v_\ell$ when $\ell>k$. We also already know that for $i,j\leq k$, $i\neq j$, that $w_j$ is orthogonal to $v_i$. In particular, this gives that $\mathbf{p}_\alpha(w_j) =  (v_j\cdot w_j) v_j$ for every $j$. Hence, given any $w = \sum_i b_i w_i\in \beta$, we have
$$
Q_2(w) = |w|^2 - |\mathbf{p}_\alpha(w)|^2 = \sum_i \mu_i^2(1-(w_i\cdot v_i)^2).
$$
This proves that $Q_1(v_i) = Q_2(w_i)$ and that $w_i$ is an orthonormal basis of eigenvectors of $Q_2$, which completes the proof of (b) and also implies at the same time the statement (a).
\end{proof}

{\subsection{Rotating planes}\label{s:rotations}

In this section we use the material of the previous one to define ``canonical
rotations'' which map equidimensional linear subspaces onto each other. These objects are not strictly necessary for our considerations, but they help streamlining some arguments.

Consider two linear subspaces $\alpha$ and $\beta$ of $\mathbb R^{m+n}$ of the same dimension $k$ and assume that
\begin{equation}\label{e:not-orthogonal}
\alpha \cap \beta^\perp = \{0\}\, .
\end{equation}
We then define a canonical element $R (\alpha, \beta)\in O (\mathbb R^{m+n})$ mapping $\alpha$ onto $\beta$ in the following fashion.  First of all consider the quadratic form $Q (v) = \dist^2 (v, \beta)$ on $\alpha$ introduced in the previous section and let 
$v_1, \ldots , v_k$ be an orthonormal base which diagonalizes $Q$. Likewise consider the quadratic form $Q^\perp$ on $\alpha^\perp$ defined by $\dist^2 (v, \beta^\perp)$ and let $v_{k+1}, \ldots , v_{m+n}$ be an orthonormal base which diagonalizes $Q^\perp$. We then define $w_i := \frac{\mathbf{p}_\beta (v_i)}{|\mathbf{p}_\beta (v_i)|}$ for $1\leq i \leq k$ and $w_i :=  \frac{\mathbf{p}_\beta^\perp (v_i)}{|\mathbf{p}_\beta^\perp (v_i)|}$ for $k+1\leq i \leq m+n$. By Lemma \ref{l:symmetry}, $w_1, \ldots , w_{m+n}$ is an orthonormal base of $\mathbb R^{m+n}$ and if we define $R (v_i)=w_i$ and extend $R$ by linearity we clearly have an element of $O (\mathbb R^{m+n})$ with the property that $O (\alpha)= \beta$.

The properties of this map is then given in the following

\begin{lemma}\label{l:rotations}
Let $\alpha$ and $\beta$ be two equidimensional linear subspaces such that $\alpha \cap \beta^\perp = \{0\}$. Then
\begin{itemize}
    \item[(a)] $R = R (\alpha, \beta)$ is well-defined, i.e. its definition does not depend on the choice of the diagonalizing orthonormal bases of $Q$ and $Q^\perp$;
    \item[(b)] $R$ is an element of $SO (m+n)$;
    \item[(c)] $R (\alpha, \beta) = (R (\beta, \alpha))^{-1}$ and $R$ is the identity on $\alpha \cap \beta$ and $\alpha^\perp \cap \beta^\perp$; in particular $R (\alpha, \alpha)$ is the identity;
    \item[(d)] $R$ depends continuously on $\alpha$ and $\beta$.
\end{itemize}
\end{lemma}
\begin{proof} Concerning (a), the only ambiguity in the definition stems from the fact that the vectors $v_1, \ldots, v_{m+n}$ are not uniquely defined. Assume thus $v'_1, \ldots v'_m$ and $v'_{m+1}, \ldots , v'_{m+n}$ form two other orthonormal bases of $Q$ and $Q^\perp$. Define $w'_i$ accordingly. Moreover, without loss of generality, assume that the eigenvalues are the same for $v_i$ and $v'_i$. If the eigenvalues of $Q$ are all distinct and those of $Q^\perp$ are also all distinct, then $v_i = \pm v'_i$ and hence $w_i = \pm w'_i$ and we see immediately that the definition of $R$ does not depend on the choice of the bases. Assume otherwise that there are some eigenvalues with higher multiplicity. To fix ideas assume therefore that $v_j, \ldots, v_\ell$ form a base of a maximal eigenspace $W$ for a fixed eigenvalue $\lambda$ of $Q$. Then $v'_j, \ldots, v'_\ell$ is also a base of the same eigenspace. Note however that, by the very definition of $Q$, $|\mathbf{p}_\beta (v)| = \sqrt{1-\lambda} |v|$ for every $v\in W$. In particular we see immediately that, computing $v'_k$ as a linear combination $\sum_{j\leq i \leq \ell} a_{\ell i} v_i$, then $w'_k = \sum_{j\leq i\leq \ell} a_{\ell i} w_i$. Therefore $R$ is well defined in this case as well.

Concerning (b), observe that $v_i \cdot w_i > 0$ for all $i$, and thus $v_1, \ldots , v_{m+n}$ and $w_1, \ldots, w_{m+n}$ have the same orientation. Concerning (c) observe first that $\alpha \cap \beta$ is the maximal eigenspace of $Q$ for the eigenvalue $0$, while $\alpha^\perp \cap \beta^\perp$ is the maximal eigenspace of $Q$ for the eigenvalue $0$.  Thus $v_1, \ldots , v_{m+n}$ contains an orthonormal base of the former and an orthonormal base for the latter and $R$ is by definition the identity on these elements. Moreover, it turns out that, by Lemma \ref{l:symmetry}, when defining $R (\beta, \alpha)$, we can choose $w_1, \ldots w_k$ and $w_{k+1}, \ldots, w_{m+n}$ as orthonormal bases diagonalizing $\dist^2 (\cdot, \alpha)$ and $\dist^2 (\cdot, \alpha^\perp)$. But then Lemma \ref{l:symmetry}(b) implies that $v_i = \frac{\mathbf{p}_{\alpha} (w_i)}{|\mathbf{p}_{\alpha} (w_i)|}$ for $1\leq i \leq m+n$, from which it immediately follows that $R (\beta, \alpha)$ maps $w_i$ in $v_i$, and it is thus the inverse of $R (\alpha, \beta)$.

As for the continuous dependence, observe that if $\alpha_k \to \alpha$ and $\beta_k \to \beta$ and we fix $v_1^k, \ldots, v^k_{m+n}$ with $v^k_1, \ldots , v^k_m$ diagonalizing $\dist^2 (\cdot, \beta_k)$ on $\alpha_k$ and $v^k_{m+1}, \ldots , v^k_{m+n}$ diagonalizing $\dist^2 (\cdot, \beta_k^\perp)$ on $\alpha_k^\perp$, then we can extract a subsequence for which $v^k_i$ converge to $v_i$. It follows immediately that the orthonormal vectors $v_1, \ldots , v_m$ diagonalize $Q$ and $v_{m+1}, \ldots , v_{m+n}$ diagonalize $Q^\perp$. Hence the algorithm given to determine $R (\alpha_k, \beta_k)$ shows immediately that the corresponding subsequence must converge to $R (\alpha, \beta)$. This completes the proof.
\end{proof}
}

\subsection{Area-minimizers and Dir-minimizers}

Of particular interest for us is the following consequence of F.~Morgan's work \cite{Morgan}*{Theorem 2}.

\begin{lemma}\label{l:frank-distance}
Let $\mathbf{S} \subset \mathbb R^{m+n}$ be the union of $N$ distinct $m$-dimensional planes $\alpha_1, \ldots , \alpha_N$ with the property that, for every $i<j$, $\alpha_i\cap \alpha_j$ is the same $(m-2)$-dimensional plane $V$. If $T$ is an $m$-dimensional integral area-minimizing current such that $\spt (T)=\mathbf{S}$, then the (two) Morgan angles of any pair $\alpha_i, \alpha_j$, $i\neq j$, coincide. 
\end{lemma}

This follows from \cite{Morgan}, since Corollary \ref{c:growth} gives the equivalence of $\theta_1$ and $\theta_2$ with the angles in \cite{Morgan} once the planes are oriented accordingly.

We complement this with the following counterpart for Dir-minimizers:

\begin{proposition}\label{p:two-sided-bound}
Let $m, n\geq 2$. 
There are positive absolute constants $c_1$ and $c_2$ depending on $n$ such that the following holds. Assume $u: \mathbb R^m \to \mathcal{A}_Q (\mathbb R^n)$ is a Dir-minimizing map such that $u(x) = 
\sum_i k_i \llbracket L_i (x)\rrbracket$ for distinct linear maps $L_i :\mathbb R^m \to \mathbb R^n$ and non-zero integers $k_i$ with the property that $\max_i |L_i| \leq c_1$ and that, for every $i< j$, the kernel of $L_i-L_j$ is the same $(m-2)$-dimensional subspace $V$ of $\mathbb R^m$ while all $L_i$ vanish on $V$. Let $\pi_i$ be the $m$-dimensional planes given by the graphs of the $L_i$'s. Then, for every pair $i<j$ the Morgan angles $\theta_k= \theta_k (\pi_i, \pi_j)$, $k=1,2$, satisfy the inequality $\theta_2 \leq c_2 \theta_1$.
\end{proposition}

We stress here that as the subspace $V$ in Proposition \ref{p:two-sided-bound} has dimension $m-2$, there are exactly two Morgan angles for each pair of planes $\pi_i,\pi_j$, $i\neq j$, which in the above are denoted $\theta_1(\pi_i,\pi_j)$, $\theta_2(\pi_i,\pi_j)$, which by definition obey $\theta_1(\pi_i,\pi_j)\leq \theta_2(\pi_i,\pi_j)$. In particular, Proposition \ref{p:two-sided-bound} tells us that they are comparable to each other, while Lemma \ref{l:frank-distance} tells us in the case the planes are area-minimizing, these two angles must in fact coincide. It is possible that this stronger result could hold for Dir-minimizing unions of planes also, however we will not need this here and so do not pursue this.

Proposition \ref{p:two-sided-bound} will be derived from the following special case and some elementary linear algebra.

\begin{lemma}\label{l:quasi-conformal}
There is a geometric constant $\mu \geq 1$ with the following property. Let $u: \mathbb R^2 \to \mathcal{A}_2 (\mathbb R^2)$ be a Dir-minimizing map such that $u(x) = \llbracket B (x)\rrbracket + \llbracket - B (x)\rrbracket$ for some linear $B:\mathbb R^2 \to \mathbb R^2$. Then $B$ is quasiconformal, in the sense that 
\begin{equation}\label{e:quasi-conformal}
\max \{|B (x)|: |x|=1\}\leq \mu \min \{|B(x)|: |x|=1\}\, .
\end{equation}
\end{lemma}

\begin{proof}
    Without loss of generality we can assume $|B|=1$, otherwise normalize $B$ to ensure this. If the lemma were false, then there would be a sequence of Dir-minimizing maps $u_k (x) = \llbracket B_k (x)\rrbracket + \llbracket - B_k (x) \rrbracket$, for some sequence $B_k:\mathbb R^2\to \mathbb R^2$ of linear maps with $|B_k|=1$, for which there is a sequence $e_k$ of unit vectors with $B (e_k)\to 0$. Extracting a subsequence, we can assume that $e_k\to e$ and that $B_k \to B$. $B$ is thus a linear map with $|B|=1$, and so in particular if we set $u (x) := \llbracket B (x)\rrbracket + \llbracket - B (x) \rrbracket$, then $u$ does not vanish identically. However, the line $\{\lambda e : \lambda \in \mathbb R\}$ would be a line of singularities, while from \cite{DLS_MAMS}*{Proposition~3.20} we know that $u$ is Dir-minimizing. In particular we would contradict Almgren's partial regularity theorem for Dir-minimizers, cf. \cite{DLS_MAMS}*{Theorem~0.11}. This contradiction proves the result.
\end{proof}

Given Lemma \ref{l:quasi-conformal}, the elementary linear algebra needed to prove Proposition \ref{p:two-sided-bound} is the following:

\begin{lemma}\label{l:elementary-planes}
There is a geometric constant $c = c(n)>0$ with the following property. Consider two linear maps $L_1, L_2: \mathbb R^2 \to \mathbb R^n$ with $|L_i|\leq c (n)$ and ${\rm rank}\, (L_1-L_2)=2$. If 
\[
    \sigma_1 = \min\{|B(e)|: |e|=1\} \leq \sigma_2 = \max \{|B (e)|:|e|=1\}
\]
denote the two singular values of $B=\frac{1}{2}(L_1-L_2)$, and $\pi_1$, $\pi_2$ are the two planes in $\mathbb R^{2+n}$ given by the graphs of $L_1$ and $L_2$ respectively, then 
\begin{align*}
&\frac{\sigma_2}{4} \leq  \theta_2 (\pi_1, \pi_2) \leq 4\sigma_2\\
&\frac{\sigma_1}{4} \leq \theta_1 (\pi_1, \pi_2) \leq 4 \sigma_1\, 
\end{align*}
where $\theta_1(\pi_1,\pi_2)\leq \theta_2(\pi_1,\pi_2)$ are the Morgan angles of the planes $\pi_1,\pi_2$.
\end{lemma}

\begin{proof}
    Let $A:=\frac{1}{2}(L_1+L_2)$ and $B:=\frac{1}{2}(L_1-L_2)$. Fix $p\in \partial \mathbf{B}_1 \cap \pi_1$. Since in particular $p\in \pi_1$, we may find $x\in\R^2$ such that $p = (x, (A+B) (x))$. Note that $|x|\leq 1$ as $|p|=1$. Since $(x, (A-B) (x))\in \pi_2$ we clearly have 
\begin{equation}\label{e:from-above}
\dist (p, \pi_2) \leq |2B (x)|\leq 2 \max \{|B (e)|: |e|=1\}\, .
\end{equation}
Next, let $q\in \pi_2$ be the point of minimum distance in $\pi_2$ from $p$ and let $\xi\in \mathbb R^2$ be so that $q = (\xi,L_2(\xi)) = (\xi,(A-B)(\xi))$. Then,
\begin{equation}\label{e:brute-force}
\dist^2 (p, \pi_2) = |p-q|^2 = |x-\xi|^2 + |(A+B) (x)- (A-B) (\xi)|^2\, .
\end{equation}
Set $\eta:=x-\xi$ and rewrite the expression in \eqref{e:brute-force} as
\[
|\eta|^2+|(A+B)(\eta+\xi)-(A-B)(\xi)|^2 = |\eta|^2 + |2B(\xi) + (A+B)(\eta)|^2.
\]
Since $\eta$ is the minimum point of the latter quadratic expression, we can differentiate it to find that $\eta$ obeys
\[
2\eta + 4(A+B)^TB (\xi) + 2(A+B)^T(A+B) (\eta) = 0.
\]
In turn, we can solve this as
\[
\eta = - 2 ({\rm Id} + (A+B)^T (A+B))^{-1} (A+B)^T (B(\xi)) = -2 ({\rm Id} + L_1^TL_1)^{-1} L_1^T (B(\xi))\, .
\]
Insert the latter in \eqref{e:brute-force} to find
\begin{equation}
\dist^2 (p, \pi_2) = 4 |({\rm Id} + L_1^TL_1)^{-1} L_1^T (B(\xi))|^2 + 4 |\left({\rm Id} - L_1 ({\rm Id} + L_1^T L_1)^{-1} L_1^T\right) (B(\xi))|^2\, .
\end{equation}
Observe now that the operator norm of $({\rm Id}+L_1^TL_1)^{-1}$ is always at most $1$: indeed, since ${\rm Id} + L_1^T L_1$ is self-adjoint and its smallest eigenvalue is at least $1$, its inverse is self-adjoint and has positive eigenvalues no larger than $1$. We can therefore estimate
\[
\left|\left({\rm Id} - L_1 ({\rm Id} + L_1^T L_1)^{-1} L_1^T\right) (B(\xi))\right|\geq |B(\xi)| - |L_1|^2 |B(\xi)|\, .
\]
In particular, if we choose the constant $c$ sufficiently small, we conclude
\begin{equation}\label{e:compare-dist-B}
\dist^2(p, \pi_2) \geq \frac{1}{2} |B(\xi)|^2\, .
\end{equation}
But, again under the assumption that $c$ is sufficiently small, and since $|q|\leq 1$ (as $\dist(p,0) =1$ and $0\in \pi_2$), we have that $|\xi|\geq \frac{\sqrt{2}}{2}$. Hence we get
\begin{equation}\label{e:from-below}
\dist (p, \pi_2) \geq \frac{1}{2} \min \{|B (e)|: |e|=1\}\, .
\end{equation}
Combining \eqref{e:from-above} and \eqref{e:from-below} and recalling that $\sigma_2 = \max \{|B (e)|:|e|=1\}$ and $\sigma_1 = \min\{|B(e)|: |e|=1\}$ are the two singular values of $B$, we conclude that 
\begin{equation}\label{e:singular_values_vs_distances}
\frac{\sigma_1}{2} \leq \min \{\dist (p, \pi_2): p \in \partial\Bbf_1\cap \pi_1\}
\leq \max \{\dist (p, \pi_2): p \in \partial\Bbf_1\cap \pi_1\} \leq 2 {\sigma_2}\, .
\end{equation}

\medskip

Next let $e\in \R^2$ be a unit vector with $|B (e)|=\sigma_2$ and consider $v= (e, L_1 (e))$ and $v' := \frac{v}{|v|} \in \pi_1 \cap \partial \Bbf_1$. Decreasing $c$ further if necessary, we have
\begin{equation}\label{e:sing_value-10}
\dist (v, \pi_2) \leq \frac{3\sqrt{2}}{4} \dist (v', \pi_2)
\leq \frac{3\sqrt{2}}{4} \max \{\dist (p, \pi_2): p\in \partial\Bbf_1\cap \pi_1\}\, .
\end{equation}
Consider $q= \mathbf{p}_{\pi_2} (v)$ and let $\xi\in \mathbb R^2$ be such that $(\xi, L_2 (\xi))=q$. Note that 
\[
|\xi-e|\leq |v-q|\leq |\mathbf{p}_{\pi_2} - \mathbf{p}_{\pi_1}| |v|\, .
\]
In particular $|\xi-e|\leq C (|L_1|+|L_2|)\leq C c$. On the other hand, recall that we have the inequality \eqref{e:compare-dist-B} but for $v$ in place of $p$ (since we did not use the fact that $p\in\partial\Bbf_1$ to achieve this). In particular, provided $c$ is sufficiently small we can write 
\begin{align}
\dist (v, \pi_2) &\geq \frac{1}{\sqrt{2}} |B (e)| - \frac{1}{\sqrt{2}} |B (e-\xi)|
\geq \frac{{\sigma_2}}{\sqrt{2}} (1-|\xi-e|) \geq \frac{3{\sigma_2}}{4\sqrt{2}} \, .
\label{e:sing_value-11}
\end{align}
We can now combine \eqref{e:sing_value-10} and \eqref{e:sing_value-11}, yielding
\begin{equation}\label{e:sing_value-12}
\sigma_2 \leq 2 \max \{\dist (p, \pi_2): p\in \partial\Bbf_1\cap \pi_1\}\, .
\end{equation}
In a similar fashion, fix now $e$ such that $|B (e)| = \sigma_1$ and consider $v$ and $v'$ as above, for this choice of $e$. Arguing analogously to \eqref{e:from-above}, it follows that 
\begin{equation}\label{e:sing_value-13}
\min \{\dist (p, \pi_2): p\in \pi_1\cap \partial \Bbf_1\} 
\leq \dist (v', \pi_2) \leq \dist (v, \pi_2) \leq 2|B (e)|\, ,
\end{equation}
because $(e, L_2 (e))\in \pi_2$. In particular we deduce that
\begin{equation}\label{e:sing_value-14}
\min \{\dist (p, \pi_2): p\in \pi_1\cap \partial \Bbf_1\} \leq 2{\sigma_1}\, .
\end{equation}

\medskip

Summarizing, if we combine \eqref{e:singular_values_vs_distances},
\eqref{e:sing_value-12}, and \eqref{e:sing_value-14}, we arrive at 
\begin{align}
&\frac{\sigma_1}{2} \leq \min \{\dist (p, \pi_2): p\in \pi_1\cap \partial \Bbf_1\} \leq 2 \sigma_1\label{e:s-value-2-sided-1}\\
&\frac{\sigma_2}{2} \leq \max \{\dist (p, \pi_2): p\in \pi_1\cap \partial \Bbf_1\} \leq 2 \sigma_2\label{e:s-value-2-sided-2}\, .
\end{align}

\medskip

Recall now that, since we are dealing with two-dimensional planes, the variational definition of the eigenvalues of quadratic forms gives  
\begin{align*}
\theta_1 & = \arcsin ({\min} \{\dist (p, \pi_2): p \in \partial \Bbf_1\cap \pi_1\})\\
\theta_2 & = \arcsin ({\max} \{\dist (p, \pi_2): p \in \partial \Bbf_1\cap \pi_1\})\, .
\end{align*}
On the other hand $\max \{\dist (p, \pi_2): p \in \partial\Bbf_1\cap \pi_1\}$ is controlled by $|L_1|+|L_2|$, hence if we choose the latter sufficiently small we get 
\begin{align*}
& \frac{1}{2} \min \{\dist (p, \pi_2): p \in \partial \Bbf_1\cap \pi_1\} \leq \theta_1 
\leq 2 \min \{\dist (p, \pi_2): p \in \partial \Bbf_1\cap \pi_1\}\\
& \frac{1}{2} \max \{\dist (p, \pi_2): p \in \partial \Bbf_1\cap \pi_1\} \leq \theta_2 
\leq 2\max \{\dist (p, \pi_2): p \in \partial \Bbf_1\cap \pi_1\} 
\end{align*}
Combining the latter inequalities with \eqref{e:s-value-2-sided-1} and \eqref{e:s-value-2-sided-2} 
we conclude the proof.
\end{proof}

We record the following corollary; it is not needed for now, but will be convenient for us to use later on.

\begin{corollary}\label{c:elementary-planes}
There is a geometric constant $c = c(n)>0$ with the following property. 
Consider two linear maps $L_1, L_2: \mathbb R^m \to \mathbb R^n$ with the property that ${\rm ker}\, (L_1-L_2)$ is $(m-2)$-dimensional, let $\lambda >0$, and denote by $\alpha_i$ and $\beta_i$ the $m$-dimensional planes given by the graphs of $L_i$ and $\lambda L_i$ respectively, for $i=1,2$. If $\max \{\lambda,1\} |L_i|\leq c(n)$, then
\[
\frac{\theta_1 (\alpha_1, \alpha_2)}{\theta_2 (\alpha_1, \alpha_2)} \leq 16^2 \frac{\theta_1 (\beta_1, \beta_2)}{\theta_2 (\beta_1, \beta_2)}\, .
\]
\end{corollary}
\begin{proof}
For $i=1,2$, letting $\sigma_i$ and $\sigma_i^{(\lambda)}$ denote the singular values of $\frac{1}{2}(L_1-L_2)$ and $\frac{1}{2}(\lambda L_1-\lambda L_2)$ respectively, we have $\sigma_i^{(\lambda)} = \lambda\sigma_i$, and so the inequality is an immediate corollary of Lemma \ref{l:elementary-planes}.
\end{proof}

We can now prove Proposition \ref{p:two-sided-bound}.

\begin{proof}[Proof of Proposition \ref{p:two-sided-bound}]
First of all observe that, for every choice of $i$ and $j$, the map $v (x) = \llbracket L_i (x)\rrbracket+\llbracket L_j (x)\rrbracket$ is Dir-minimizing. It therefore suffices to prove the proposition when $Q=2$ and we have two distinct linear maps $L_1$ and $L_2$. 

Next consider the $(m-2)$-dimensional subspace $W$ which is the kernel of $L_1-L_2$, and let
$A= \frac{1}{2}(L_1+L_2)$ and $B=\frac{1}{2}(L_1-L_2)$. Observe that since $A$ is linear, the function
\[
\tilde{v}(x) := v(x)\ominus A(x) \equiv \llbracket L_1(x)-A(x)\rrbracket + \llbracket L_2(x)-A(x)\rrbracket = \llbracket B(x)\rrbracket + \llbracket -B(x)\rrbracket
\]
is also Dir-minimizing (see \cite{DLS_MAMS}*{Lemma 3.23}). Now, for any point $z\in \R^m$, let us write $z = (x,y)\in W^\perp\times W$. Since the image of $B$ on $W^\perp$ lies in a $2$-dimensional subspace, and the kernel of $B$ is $(m-2)$-dimensional, we may quotient out the kernel of $B$ and consider it as a function $W^\perp\to \R^2$. As the domain of $B$ is then a 2-dimensional subspace, we can then apply Lemma \ref{l:quasi-conformal} to conclude that \eqref{e:quasi-conformal} holds for $B$, and thus we can control the ratio of the two singular values of $B$ by a geometric constant, $\mu$. But then observe that the two Morgan angles of the planes $\pi_1$ and $\pi_2$ coincide with the Morgan angles of the $2$-dimensional planes of $\R^4$ constructed above. We can therefore apply Lemma \ref{l:elementary-planes} to conclude the proof of Proposition \ref{p:two-sided-bound}.
\end{proof}

\subsection{Separated regions, Alignment, and Shifting}

Here we collect three lemmas of a different flavor, which all have to do with the geometry of a collection of planes which all intersect in a common $(m-2)$-dimensional subspace; they will be used later on at different stages, when proving Theorem \ref{c:decay}.

The following lemma shows that, even in the absence of a comparison estimate between the Morgan angles for a given finite collection of planes, it is possible to find a sizeable region of one of the planes $\alpha$ where the minimum distance of a point in that region to the other planes is comparable to the minimum of the Hausdorff distances of these planes to $\alpha$.

\begin{lemma}[Separated region]\label{l:sep_region}
There is a constant $0<c =c(m,N)< \frac{1}{2}$ with the following property. Suppose that $\alpha, \beta_1, \ldots, \beta_N$ are distinct $m$-dimensional subspaces of $\mathbb R^{m+n}$. Then there is a point $\xi \in \alpha\cap \partial \Bbf_{1/2}$ with the property that 
\begin{equation}\label{e:sep_region}
\min_i \inf \{\dist (\zeta, \beta_i) : \zeta\in \Bbf_{c} (\xi)\cap \alpha\}
\geq c \min_i \dist (\alpha\cap \Bbf_1, \beta_i \cap \Bbf_1)\, .
\end{equation}
\end{lemma}

\begin{proof}
Following the notation of Lemma \ref{l:symmetry} and Definition \ref{d:frank}, let us denote by $Q_i :\alpha\to \R$ the quadratic forms $Q_i (v) := \dist^2 (v, \beta_i)$ and let $e_i$ be an eigenvector corresponding to the its maximal eigenvalue $\lambda^i$ and hence also to the largest Morgan angle $\theta^i$ of the pair $(\alpha, \beta_i)$. Complete $e_i$ to an orthonormal basis which diagonalizes $Q_i$. For any $\zeta\in \alpha$ we thus have 
\[
\dist^2 (\zeta, \beta_i) \geq (e_i \cdot \zeta)^2 \lambda^i 
= (e_i\cdot \zeta)^2 \sin^2 (\theta^i)
\]
Recalling that, by Corollary \ref{c:growth}, $\sin (\theta^i) = \dist (\alpha \cap \Bbf_1, \beta_i \cap \Bbf_1)$, we conclude that 
\[
\dist (\zeta, \beta_i) \geq |e_i \cdot \zeta| \dist (\alpha \cap \Bbf_1, \beta_i \cap \Bbf_1)\, .
\] 
Therefore, we just need to find a vector $\xi \in \partial \Bbf_{1/2} \cap \alpha$ with the property that 
\[
|\zeta\cdot e_i| \geq c \qquad \forall i, \forall \zeta \in \Bbf_c (\xi) \cap \alpha\, 
\]
for $c =c(m,N)$ to be determined. Given the elementary estimate
\[
|\zeta\cdot e_i|\geq |\xi\cdot e_i| - c\, ,
\]
for any such $\zeta$ and $\xi$, it therefore suffices to find $\xi\in \partial \Bbf_{1/2} \cap \alpha$ such that 
\[
|\xi \cdot e_i|\geq 2c \qquad \forall i\, .
\]
Now for each $i=1,\dots,N$, let $S_i:= \{v\in \partial \Bbf_{1/2} \cap \alpha : |v\cdot e_i|\leq 2c\}$ and observe that $\mathcal{H}^{m-1} (S_i) \leq C (m) c$ for some constant $C(m)$. In particular, 
$\mathcal{H}^m (\bigcup_i S_i) \leq N C (m) c$, and thus if $c = c(m,N)$ is chosen sufficiently small, $\partial \Bbf_{1/2} \setminus \bigcup_i S_i$ must have positive measure, and hence it contains at least one point $\xi$ which obeys the desired conditions.
\end{proof}

As an immediate corollary of Lemma \ref{l:sep_region} above we have the following.

\begin{corollary}\label{c:close-to-one-plane}
Let $\alpha, \beta_1, \ldots, \beta_N$ be a collection of $m$-dimensional distinct planes of $\mathbb R^{m+n}$ and set $\mathbf{S} = \bigcup_i \beta_i$. Then there is a constant $\bar C = \bar{C}(N,m)>0$ with
\[
\min_i \dist (\alpha\cap \Bbf_1, \beta_i \cap \Bbf_1) \leq \bar{C} \dist (\alpha\cap \Bbf_1, \mathbf{S} \cap \Bbf_1)\, .
\]
\end{corollary}

The following lemma concerns itself with the alignment of spines of pairs of cones in $\mathscr{C}(Q)$, cf. Definition \ref{def:cones}, and will be of fundamental importance. 

\begin{lemma}\label{l:spine}
For every $M>0$ and natural numbers $m$, $n$, and $Q$, there is a constant $\bar{C} = \bar{C}(M,m,n,Q)>0$ with the following property. Assume that
\begin{itemize}
\item[(i)] $\mathbf{S}$ and $\mathbf{S}'$ consist of $2\leq N, N' \leq Q$ $m$-dimensional distinct planes $\alpha_1, \ldots, \alpha_N$ and $\beta_1, \ldots, \beta_{N'}$, respectively;
\item[(ii)] The intersection of every pair $\alpha_i\neq \alpha_j$ is a single $(m-2)$-dimensional subspace $V (\mathbf{S})$ and the intersection of every pair $\beta_i\neq \beta_j$ is a single $(m-2)$-dimensional subspace $V (\mathbf{S}')$;
\item[(iii)] For every pair $\alpha_i\neq \alpha_j$ the two Morgan angles $\theta_1(\alpha_i,\alpha_j) \leq \theta_2(\alpha_i,\alpha_j)$ satisfy $\theta_2 \leq M \theta_1$.
\end{itemize}
Then
\begin{equation}\label{e:alignment}
\dist (V (\mathbf{S}) \cap \Bbf_1, V (\mathbf{S}')\cap \Bbf_1) \leq \bar{C} 
\frac{\dist (\mathbf{S}\cap \Bbf_1, \mathbf{S}'\cap \Bbf_1)}{\min_i \dist (\mathbf{S} \cap \Bbf_1, \alpha_i \cap \Bbf_1)}\, .
\end{equation}
\end{lemma}

\begin{proof}
    We may assume without loss of generality that the planes $\alpha_i$ are ordered such that 
    \[
        \dist (\alpha_1\cap \Bbf_1, \alpha_2 \cap \Bbf_1) = \max_{i\neq j}\dist (\alpha_i \cap \Bbf_1, \alpha_j\cap \Bbf_1)\, ,
    \]
    and observe that 
\begin{equation}\label{e:distant-2}
\min_i \dist (\mathbf{S} \cap \Bbf_1, \alpha_i \cap \Bbf_1) \leq \dist (\alpha_1\cap \Bbf_1, \alpha_2 \cap \Bbf_1)\, .
\end{equation}
On the other hand, because of Corollary \ref{c:close-to-one-plane} and since $\alpha_i\subset \Sbf$ for all $i$, we can select $\beta_i$ and $\beta_j$ such that 
\begin{align}
\dist (\alpha_1 \cap \Bbf_1, \beta_i \cap \Bbf_1) &\leq \bar{C}  \dist (\mathbf{S}\cap \Bbf_1, \mathbf{S}'\cap \Bbf_1)\label{e:friend-of-1}\\
\dist (\alpha_2 \cap \Bbf_1, \beta_j \cap \Bbf_1) &\leq \bar{C}  \dist (\mathbf{S}\cap \Bbf_1, \mathbf{S}'\cap \Bbf_1)\, .\label{e:friend-of-2}
\end{align}
To simplify our notation we use $V$ and $V'$ in place of $V (\mathbf{S})$ and $V (\mathbf{S}')$. 
Because of the condition (iii) on the Morgan angles $\theta_k (\alpha_1, \alpha_2)$, using \eqref{e:Morgan-max-min} and \eqref{e:Haus=max_eigen} we have
\begin{equation}\label{e:growth-1}
\dist (w, \alpha_2) \geq \bar{C}^{-1} \dist (\alpha_1\cap \Bbf_1, \alpha_2 \cap \Bbf_1) |w- \mathbf{p}_V (w)| \qquad \forall w\in \alpha_1\, .
\end{equation}
Fix now $v'\in V'\cap \Bbf_1$ and observe that, since $v'$ belongs to both $\beta_i$ and $\beta_j$, due to \eqref{e:friend-of-1} and \eqref{e:friend-of-2} we must have
\begin{align*}
|\mathbf{p}_{\alpha_1} (v')-v'| &= \dist (v', \alpha_1) \leq \bar C \dist (\mathbf{S}\cap \Bbf_1, \mathbf{S}'\cap \Bbf_1)\\
|\mathbf{p}_{\alpha_2} (v')-v'| &= \dist (v', \alpha_2) \leq \bar{C} \dist (\mathbf{S}\cap \Bbf_1, \mathbf{S}'\cap \Bbf_1)\, .
\end{align*}
In particular, from the triangle inequality and the fact $|\mathbf{p}_{\alpha_2} (v')|\leq 1$ one can then deduce
\begin{equation}\label{e:close-10}
\dist (\mathbf{p}_{\alpha_1} (v'), \alpha_2) \leq \bar{C} \dist (\mathbf{S}\cap \Bbf_1, \mathbf{S}'\cap \Bbf_1)\, .
\end{equation}
Now, let $w= \mathbf{p}_{\alpha_1} (v')$ and observe that
\begin{equation}\label{e:est-1}
\dist (v', V) \leq |v'-w| + |w- \mathbf{p}_V (w)|
\leq \bar C \dist (\mathbf{S}\cap \Bbf_1, \mathbf{S}'\cap \Bbf_1) + |w-\mathbf{p}_V (w)|\, ,
\end{equation}
while, by \eqref{e:growth-1} and \eqref{e:close-10},
\begin{equation}\label{e:est-2}
|w- \mathbf{p}_V (w)|\leq \bar{C} \frac{\dist (w, \alpha_2)}{\dist (\alpha_1\cap \Bbf_1, \alpha_2 \cap \Bbf_1)}\leq \bar{C} \frac{\dist (\mathbf{S}\cap \Bbf_1, \mathbf{S}'\cap \Bbf_1)}{\dist (\alpha_1\cap \Bbf_1, \alpha_2 \cap \Bbf_1)}\, .
\end{equation}
Thus, given \eqref{e:distant-2} and $\dist (\alpha_1\cap \Bbf_1, \alpha_2 \cap \Bbf_1)\leq 1$, \eqref{e:est-1} and \eqref{e:est-2} lead to 
\[
\dist (v', V) \leq \bar{C} \frac{\dist (\mathbf{S}\cap \Bbf_1, \mathbf{S}'\cap \Bbf_1)}{\min_i \dist (\mathbf{S} \cap \Bbf_1, \alpha_i \cap \Bbf_1)}\, .
\]
Observe however that $v'$ is an arbitrary element of $V'\cap \Bbf_1$, and we have thus reached 
\[
\sup \{\dist (v',V) : v'\in V'\cap \Bbf_1)\}
\leq \bar{C} \frac{\dist (\mathbf{S}\cap \Bbf_1, \mathbf{S}'\cap \Bbf_1)}{\min_i \dist (\mathbf{S} \cap \Bbf_1, \alpha_i \cap \Bbf_1)}\, .
\]
On the other hand, since $V$ and $V'$ have the same dimension, by Corollary \ref{c:growth} we infer that
\[
\sup \{\dist (v',V) : v'\in V'\cap \Bbf_1)\} = \dist (V\cap \Bbf_1, V'\cap \Bbf_1)\, ,
\]
hence concluding the desired statement.
\end{proof}

Finally, the following lemma gives a lower bound for the distance of $z$ to $q+\mathbf{S}$ for many points $z\in\mathbf{S}$. Before stating it we introduce some useful terminology.

\begin{definition}\label{def:rotationally_invariant}
A set $\Omega\subset \mathbb R^{m+n}$ is said to be \textit{invariant under rotation around} a linear subspace $V$ if $R (\Omega)=\Omega$ for any rotation $R$ of $\mathbb R^{m+n}$ which fixes $V$.
\end{definition}

\begin{lemma}[Shifting lemma]\label{l:shifting}
For every $M\geq 1$ and every open set $U\subset \Bbf_1$ that is invariant under rotations around $V$, there is a constant $\bar{C} = \bar{C}(M,m,n,N,U)>0$ with the following properties. Assume that $\mathbf{S}= \alpha_1\cup\cdots \cup \alpha_N$ satisfies (i), (ii), and (iii) in Lemma \ref{l:spine}, let $q\in \Bbf_{1/2}$, and let $\boldsymbol{\mu} (\mathbf{S}) = \max_{i<j} \dist (\alpha_i\cap \Bbf_1, \alpha_j\cap \Bbf_1)$. Then there is an index $j\in\{1,\dotsc,N\}$ and a subset $\Omega\subset \alpha_j\cap U$ such that 
$\mathcal{H}^m (\Omega) \geq \bar{C}^{-1}$ and 
\begin{equation}\label{e:shifting-LB}
|\mathbf{p}_{\alpha_1}^\perp (q)| + \boldsymbol{\mu} (\mathbf{S}) |\mathbf{p}_{V^\perp\cap \alpha_1} (q)| \leq \bar{C} \dist (z, q+\mathbf{S}) \qquad \forall z\in \Omega\, .
\end{equation}
\end{lemma}

\begin{remark}\label{r:shifting-scaling-invariant}
Observe that Lemma \ref{l:shifting} can be scaled. Under the assumption that $\Bbf_r$ replaces $\Bbf_1$, $U_{0,r}\equiv \iota^{-1}_{0,r}(U)$ replaces $U$, and $q\in \Bbf_{r/2}$, we can conclude the existence of a subset $\Omega \subset U_{0,r}$ with measure larger than $\bar C^{-1} r^m$ with the property that \eqref{e:shifting-LB} holds. Under these assumptions the constant $\bar{C}$ can be taken to be the same as the one in Lemma \ref{l:shifting}.
\end{remark}

In order to prove Lemma \ref{l:shifting}, we will need the following elementary result.

\begin{lemma}\label{l:maximizer}
There is a dimensional constant $C_0 = C_0(m,n)>0$ with the following property.
Let $\mathbf{S}$, $M$, and $q$ be as in Lemma \ref{l:shifting}. Then
\begin{equation}\label{e:orthogonal-parallel}
|\mathbf{p}_{\alpha_1}^\perp (q)| + \boldsymbol{\mu} (\mathbf{S}) |\mathbf{p}_{V^\perp\cap \alpha_1} (q)| \leq C_0 M \max_i |\mathbf{p}_{\alpha_i}^\perp (q)|\, .
\end{equation}
\end{lemma}

\begin{proof} 
Let $k$ be such that 
\[
|\mathbf{p}_{\alpha_k}^\perp (q)| = \max_i |\mathbf{p}_{\alpha_i}^\perp (q)|\, .
\]
Thus, in particular, $|\mathbf{p}^\perp_{\alpha_1} (q)|\leq |\mathbf{p}^\perp_{\alpha_k} (q)|$, and so it remains to show that
\begin{equation}\label{e:parallel}
\boldsymbol{\mu} (\mathbf{S}) |\mathbf{p}_{V^\perp\cap \alpha_1} (q)| \leq C_0 M \max_i |\mathbf{p}_{\alpha_i}^\perp (q)|\, .
\end{equation}
First of all pick $j$ which maximizes $\dist (\alpha_j \cap \Bbf_1, \alpha_k \cap \Bbf_1)$, so that in particular
\[
    \dist (\alpha_j \cap \Bbf_1, \alpha_l \cap \Bbf_1) \leq \dist (\alpha_j \cap \Bbf_1, \alpha_k \cap \Bbf_1) \qquad \forall l\, .
\]
Thus, the triangle inequality yields
\[
\dist (\alpha_j \cap \Bbf_1, \alpha_k \cap \Bbf_1)\geq \frac{1}{2} \boldsymbol{\mu} (\mathbf{S})\, .
\]
Next, recalling the definition of Morgan angles, observe that for every $w\in \alpha_j \cap V^\perp$ we have (using the above line, assumption (iii) in Lemma \ref{l:spine} and \eqref{e:Haus=max_eigen})
\[
\boldsymbol{\mu} (\mathbf{S}) |w| \leq C_0 M |\mathbf{p}_{\alpha_k}^\perp (w)|\, .
\]
Now choose $w= \mathbf{p}_{V^\perp\cap \alpha_j} (q) = \mathbf{p}_V^\perp (q) - \mathbf{p}_{\alpha_j}^\perp (q)$. Since $\mathbf{p}_{\alpha_k}^\perp \circ \mathbf{p}_V^\perp =\mathbf{p}_{\alpha_k}^\perp$ (because $V\subset\alpha_k$) we may estimate
\begin{align*}
\boldsymbol{\mu} (\mathbf{S}) |\mathbf{p}_{V^\perp\cap \alpha_j} (q)|
&\leq C_0 M|\mathbf{p}_{\alpha_k}^\perp (\mathbf{p}_V^\perp (q))| + C_0 M |\mathbf{p}_{\alpha_k}^\perp (\mathbf{p}_{\alpha_j}^\perp (q))|\\
&\leq C_0 M |\mathbf{p}_{\alpha_k}^\perp (q)| + C_0 M |\mathbf{p}_{\alpha_j}^\perp (q)|\, .
\end{align*}
Using the maximality of $k$, we then reach
\begin{equation}\label{e:parallel-2}
\boldsymbol{\mu} (\mathbf{S}) |\mathbf{p}_{V^\perp\cap \alpha_j} (q)|
\leq C_0 M |\mathbf{p}_{\alpha_k}^\perp (q)|\, .
\end{equation}
It remains to replace $\alpha_j$ with $\alpha_1$ in the projection on the left-hand side of the above inequality. Since $\mathbf{p}_{V^\perp\cap \alpha_1}\circ\mathbf{p}_{\alpha_j} = \mathbf{p}_{V^\perp\cap \alpha_1}\circ\mathbf{p}_{V^\perp\cap\alpha_j}$, we have
\begin{align*}
|\mathbf{p}_{V^\perp \cap \alpha_1} (q)| &\leq |\mathbf{p}_{V^\perp \cap \alpha_1} (\mathbf{p}_{V^\perp\cap \alpha_j} (q))| + |\mathbf{p}_{V^\perp\cap \alpha_1} (\mathbf{p}_{\alpha_j}^\perp (q))|\\
&\leq |\mathbf{p}_{V^\perp\cap \alpha_j} (q)| + |\mathbf{p}_{\alpha_j}^\perp (q)|\\
&\leq |\mathbf{p}_{V^\perp\cap \alpha_j} (q)| + |\mathbf{p}_{\alpha_k}^\perp (q)|\, .
\end{align*}
Combining the latter inequality with \eqref{e:parallel-2} and using that $\boldsymbol{\mu}(\Sbf)\leq 1$, we reach \eqref{e:parallel} and hence complete the proof of the lemma.
\end{proof}

We can now prove Lemma \ref{l:shifting}.

\begin{proof}[Proof of Lemma \ref{l:shifting}]
We choose $j$ such that $|\mathbf{p}_{\alpha_j}^\perp (q)|= \max_i |\mathbf{p}_{\alpha_i}^\perp (q)|$ 
and using Lemma \ref{l:maximizer} we aim at proving that
\begin{equation}\label{e:distance_LB-2}
 |\mathbf{p}_{\alpha_j}^\perp (q)|\leq \bar C_1 \dist (z, q+\mathbf{S}) 
 \qquad \forall z\in \Omega\, ,
\end{equation}
for some set $\Omega\subset U$ with $\mathcal{H}^m (\Omega) \geq \bar C_1^{-1}$, where the constant $\bar C_1$ is allowed to depend on $U$. In fact it turns out that the constant $\bar C_1$ depends on
\begin{equation}\label{e:explicit_dependence_on_U}
\gamma:= \inf \{\mathcal{H}^m (U\cap \alpha) : V\subset \alpha \; \text{and $\alpha$ is an $m$-dimensional subspace}\}\, .
\end{equation}
It is not entirely obvious that $\gamma$ is necessarily positive. Note, however, that if a point $p$ belongs to $U$ and another point $q$ satisfies $\mathbf{p}_V (q) = \mathbf{p}_V (p)$ and $\dist (q, V)= \dist (p, V)$, then necessarily $q\in U$, in light of the rotational invariance of $U$. In particular there is an open subset $U'\subset V\times \mathbb R^+$ such that $U= \{p: (\mathbf{p}_V (p), \dist (p, V))\in U'\}$. From this we conclude that, not only is $\gamma$ positive, but in fact that $\mathcal{H}^m (U\cap \alpha)$ is exactly the same number for every $m$-dimensional plane which contains $V$.

Observe first that \eqref{e:distance_LB-2} is equivalent to
\begin{equation}\label{e:distance_LB-3}
 \min_{i} |\mathbf{p}_{\alpha_i}^\perp (z-q)|\geq \bar C_1^{-1} |\mathbf{p}_{\alpha_j}^\perp (q)| \qquad \forall z\in\Omega\, .
\end{equation}
Now assume that the claim is false, namely that \eqref{e:distance_LB-3} fails for all $z\in U\cap \alpha_j$ with the exception of a set $E$ of Hausdorff measure smaller than $\bar C_1^{-1}$. In particular by choosing $\bar C_1$ large enough, we can ensure that, for some $i\in \{1, \ldots , N\}$ the set $F_i\subset \alpha_j\cap U$ where  
\begin{equation}\label{e:too-small-0}
|\mathbf{p}_{\alpha_i}^\perp (z-q)|  \ \leq \bar C_1^{-1} |\mathbf{p}_{\alpha_j}^\perp (q)| \qquad \forall z\in F_i
\end{equation}
has measure at least $\frac{\gamma}{N+1}$.
We then claim that
\begin{align}
|\mathbf{p}_{\alpha_i}^\perp (q)| & \leq C \gamma^{-4} \bar C_1^{-1}  |\mathbf{p}_{\alpha_j}^\perp (q)|\, ,\label{e:too-small-1}\\
|\mathbf{p}_{\alpha_i}^\perp (z)| &\leq C \gamma^{-4} \bar C_1^{-1}  |\mathbf{p}_{\alpha_j}^\perp (q)|\qquad \forall
z\in \Bbf_1\cap \alpha_j \, ,\label{e:too-small-2}
\end{align}
where the constant $C$ depends only on $m$ and $N$.
Note that $i$ is the index chosen such that $\mathcal{H}^m (F_i) \geq \frac{\gamma}{N+1}$, but the second estimate is claimed for every $z\in \Bbf_1\cap \alpha_j$, and the latter will in the end be used to get a contradiction.

{ In order to prove \eqref{e:too-small-1} and \eqref{e:too-small-2} consider first the set $F'_i := \mathbf{p}_{V^\perp} (F_i)$. The latter belongs to the $2$-dimensional subspace $V^\perp \cap \alpha_j$ and the coarea formula implies immediately that 
\begin{equation}\label{e:area-lower-bound}
 \mathcal{H}^2 (F'_i) \geq C^{-1} \mathcal{H}^m (F_i) \geq C^{-1} \gamma
\end{equation}
for a positive dimensional $C= C(m,N)$. Choose next a vector $e_1\in F'_i$ such that $|e_1|\geq \frac{1}{2} \sup \{|x|:x\in F'_i\}$ and observe that $|e_1|\geq C \sqrt{\gamma}$ because $4\pi |e_1|^2 \geq \mathcal{H}^2 (F'_i)$. Hence choose $e_2\in F'_i$ such that 
\[
|e_2 - |e_1|^{-2} (e_2\cdot e_1) e_1| \geq \frac{1}{2} \sup \{|x - |e_1|^{-2} (x\cdot e_1) e_1|: x\in F'_i\}
\]
and observe that $\mathcal{H}^2 (F'_i)\leq 4 |e_2 - |e_1|^{-2} (e_1\cdot e_2) e_1|$, so that 
\[
||e_1| e_2 - |e_1|^{-1} (e_1\cdot e_2) e_1| \geq C^{-1} \gamma^{3/2}\, .
\]
We next define the linear map $\Phi: \mathbb R^2 \to V^\perp\cap \alpha_j$ by $(\lambda_1, \lambda_2)\mapsto \lambda_1 e_1 + \lambda_2 e_2$ and observe, by elementary geometry that $|\det \Phi| = ||e_1| e_2 - |e_1|^{-1} (e_1\cdot e_2)e_1|$. In particular, $|\det\, \Phi^{-1}|\leq C\gamma^{-3/2}$ and, given that $|\Phi|\leq C$,  we also have $|\Phi^{-1}| \leq C \gamma^{-3/2}$. Therefore $F''_i := \Phi^{-1} (F'_i)$ is contained in a disk of radius at most $C \gamma^{-3/2}$. Consider now the number
\[
\mu:= \sup \{|\lambda_1+\lambda_2-1|: (\lambda_1, \lambda_2)\in F''_i\} 
\]
and notice that $\mathcal{H}^2 (F''_i) \leq C \gamma^{-3/2} \mu$. Since $|\det \Phi|\leq 1$ we get $\mathcal{H}^2 (F'_i) \leq \mathcal{H}^2 (F''_i)$, and thus, when combined with \eqref{e:area-lower-bound}, we infer $\mu\geq C^{-1}\gamma^{5/2}$, and thus the existence of $(\lambda_1, \lambda_2)\in F''_i$ such that 
\begin{align}
|\lambda_1+\lambda_2-1|\geq C^{-1} \gamma^{5/2}\, .
\end{align}
Observe now that, by the very definition of $F'_i$, there are $v_1, v_2, v_3\in V$ such that 
\begin{align*}
&e_1 + v_1 \in F_i\\   
&e_2 + v_2 \in F_i\\
&\lambda_1 e_1 + \lambda_2 e_2 + v_3 \in F_i\, .
\end{align*}
}
Since $V\subset \alpha_i$ we can write 
\begin{align*}
& (\lambda_1+\lambda_2-1) \mathbf{p}_{\alpha_i}^\perp (q)
= \lambda_1 \mathbf{p}_{\alpha_i}^\perp (q-e_1) + \lambda_2 \mathbf{p}_{\alpha_i}^\perp (q-e_2) -
\mathbf{p}_{\alpha_i}^\perp (q- (\lambda_1 e_1+\lambda_2 e_2))\\
= & \lambda_1 \mathbf{p}_{\alpha_i}^\perp (q-e_1-v_1) + \lambda_2 \mathbf{p}_{\alpha_i}^\perp (q-e_2-v_2) - \mathbf{p}_{\alpha_i}^\perp (q-(\lambda_1 e_1 +\lambda_2 e_2 + v_3))\, .
\end{align*}
{ In particular we conclude
\begin{equation}\label{e:too-small-3}
|\mathbf{p}_{\alpha_i}^\perp (q)| \leq C \bar{C}_1^{-1} \gamma^{-5/2} |\mathbf{p}_{\alpha_j}^\perp (q)|\, ,
\end{equation}}
which is in fact stronger than \eqref{e:too-small-1}.
On the other hand we can also write 
\begin{align*}
|\mathbf{p}_{\alpha_i}^\perp (e_k)|\leq
|\mathbf{p}_{\alpha_i}^\perp (e_k-q)| + |\mathbf{p}_{\alpha_i}^\perp (q)| \qquad k=1,2
\end{align*}
and hence we immediately conclude
\begin{equation}\label{e:too-small-4}
|\mathbf{p}_{\alpha_i}^\perp (e_k)|\leq C \bar{C}_1^{-1} \gamma^{-5/2} |\mathbf{p}_{\alpha_j}^\perp (q)|\, .
\end{equation}
{ In turn, using again the map $\Phi$ we can express any $z\in \Bbf_1\cap \alpha_j$ as $z= \lambda_1 e_1+\lambda_2 e_2 + v$ for some coefficients $|\lambda_i|\leq C \gamma^{-3/2}$ and some vector $v\in V$. In particular we achieve \eqref{e:too-small-2} from \eqref{e:too-small-4}.

Observe however that \eqref{e:too-small-2} can be equivalently written as 
\begin{equation}\label{e:too-small-5}
\dist (z, \alpha_i) \leq C \bar{C}_1^{-1} \gamma^{-4} |\mathbf{p}_{\alpha_j}^\perp (q)|\qquad \forall z\in \Bbf_1\cap \alpha_j\, ,
\end{equation}
and thus also as
\begin{equation}\label{e:too-small-6}
\dist (\alpha_i\cap \Bbf_1, \alpha_j\cap \Bbf_1) \leq C \bar{C}_1^{-1} \gamma^{-4} |\mathbf{p}_{\alpha_j}^\perp (q)|
\end{equation}}
Since $\gamma$ is fixed, $|q|\leq \frac{1}{2}$, and $\alpha_i$ and $\alpha_j$ have the same dimension, the latter estimate implies that, by choosing $\bar{C}_1$ appropriately large, we can assume that the linear subspaces $\alpha_i$ and $\alpha_j$ are sufficiently close. In particular, given that $|q|\leq \frac{1}{2}$, for an appropriate large choice of $\bar{C}_1$ the affine subspace $q+\alpha_i^\perp$ must intersect $\Bbf_1\cap \alpha_j$ at some point $z$. But then at that point $z$ we would have 
\begin{equation}\label{e:intersection}
\mathbf{p}_{\alpha_i}^\perp (q-z) = q-z\, .
\end{equation}
Since $z\in \alpha_j$ we must have $|q-z| \geq \dist (q, \alpha_j) = |\mathbf{p}_{\alpha_j}^\perp (q)|$. { Hence 
\[
|\mathbf{p}_{\alpha_i}^\perp (q-z)| \geq |\mathbf{p}_{\alpha_j}^\perp (q)|\, .
\]
On the other hand, combining \eqref{e:too-small-1} and \eqref{e:too-small-2} we get 
\[
|\mathbf{p}_{\alpha_i}^\perp (q-z)| \leq 2 C \bar{C}_1^{-1} \gamma^{-4} |\mathbf{p}_{\alpha_j}^\perp (q)|\, .
\]
Since $C$ is a constant which depends only on $m$ and $N$ we can now choose $\bar{C}_1$ large enough, depending only on $\gamma$, $m$, and $N$ so that $C \bar{C}_1^{-1} \gamma^{-4} < \frac{1}{2}$. But then the last two inequalities would be in contradiction, unless $\mathbf{p}_{\alpha_j}^\perp (q) = \mathbf{p}_{\alpha_i}^\perp (q-z) = 0$. In particular we conclude that $q$ belongs to $\alpha_j$. In this case, however, \eqref{e:distance_LB-2} holds trivially.}
\end{proof}

\section{Graphical approximations}\label{p:approx}

In this section we start facing two of the geometric issues that complicate the proof of Theorem \ref{c:decay}. In an ideal situation the cone $\mathbf{S}$ in the statement of the theorem consists of $N$ $m$-dimensional planes $\alpha_1, \ldots , \alpha_N$ with the additional properties that, for any pair $\alpha_i\neq \alpha_j$,
\begin{itemize}
    \item[(i)] The Hausdorff distance between $\alpha_i$ and $\alpha_j$ is relatively larger than $\mathbb{E} (T, \mathbf{S}, \Bbf_1)$;
    \item[(ii)] The two Morgan angles $\theta_1(\alpha_i,\alpha_j)$, $\theta_2(\alpha_i,\alpha_j)$ formed by them are comparable.
\end{itemize}
Under these two additional assumptions we can hope to use the bounds of Theorem \ref{thm:main-estimate} and Corollary \ref{c:splitting-0} to give a good approximation of $T$ by graphs over the collection of planes $\alpha_i$, after removing a small neighborhood of the spine $V(\Sbf)$.

Although there is no reason to assume (i) and (ii) for $\Sbf$ a priori, we can hope to achieve them for a different cone $\mathbf{S}'$ without increasing the excess $\Ebb(T,\Sbf',\Bbf_1)$ too much relative to the excess $\Ebb(T,\Sbf,\Bbf_1)$. In light of this we first specify an algorithm that allows us to gain control on how much the excess increases when we discard planes of $\mathbf{S}$ until we achieve (i). This algorithm is summarized in the Pruning Lemma \ref{l:pruning}. For later use we want to iteratively apply Lemma \ref{l:pruning} and keep track of the structure of the planes which have been discarded during this process; this is accomplished in Lemma \ref{l:algorithm}. As for (ii), we will work under the assumption that it holds for now. Later, in Section \ref{p:balancing}, we will demonstrate that the new cone $\Sbf'$ achieving (i) indeed additionally satisfies (ii). 

In the remainder of this Section, we show how to gain a graphical approximation under a suitable quantification of (i) and (ii). We begin with a ``crude approximation'' in Section \ref{s:crude}, followed by a more intricate ``refined approximation" in Section \ref{s:refined}; the latter will be needed later.

\subsection{Pruning Lemma and Layer Subdivision}

The main purpose of this section is to introduce two very useful elementary combinatorial lemmas with the aim discussed above.

We will always work under the following assumption. We will often use the terminology ``plane" when referring to a linear subspace of $\R^{m+n}$.

\begin{assumption}\label{a:cones}
$\mathbf{S}\subset \mathbb R^{m+n}$ is an $m$-dimensional cone such that
\begin{itemize}
\item[(i)] $\mathbf{S}$ is a union of $N$ distinct $m$-dimensional planes $\alpha_1, \ldots, \alpha_N$;
\item[(ii)] for each pair $i\neq j$ the intersection $\alpha_i \cap \alpha_j$ is the same $(m-2)$-dimensional plane $V$, which we refer to as the \emph{spine} of $\mathbf{S}$.
\end{itemize}
\end{assumption}

Note in particular that the cones in the class of $\mathscr{C} (p, Q)$ of Definition \ref{def:cones} fall under the latter assumption.

Our first technical lemma we call the \emph{Pruning Lemma}. It has two main uses. One is to prove the second technical lemma (Lemma \ref{l:algorithm}); we will explain the meaning behind that lemma when we get to it. The other use will be to "prune" a cone, throwing away some of its planes, and ultimately get that the excess relative to the pruned cone is sufficiently small relative to the minimal angle of the pruned cone, which is a crucical assumption for our graphical approximation results later.

\begin{lemma}[Pruning lemma]\label{l:pruning}
Let $N\geq 2$, $D>0$, and $0<\delta \leq 1$. Set $\Gamma := \delta^{2-N} (N-1)!$ and $\varepsilon:= (1+\Gamma )^{-1} \delta$. If
\begin{itemize}
\item[(i)] $\mathbf{S} = \alpha_1 \cup \cdots \cup \alpha_N$ is as in Assumption \ref{a:cones};
\item[(ii)] $D\leq \varepsilon \max_{i<j} \dist (\alpha_i\cap \Bbf_1, \alpha_j\cap \Bbf_1)$;
\end{itemize}
then there is a subcollection $I\subset \{1, \ldots, N\}$ consisting of at least $2$ elements and satisfying the following requirements: 
\begin{align}
\max_j \min_{i\in I} \dist (\alpha_i\cap \Bbf_1, \alpha_j\cap \Bbf_1) &\leq \Gamma  D\label{e:pruning-1}\\
D + \max_j \min_{i\in I} \dist (\alpha_i\cap \Bbf_1, \alpha_j\cap \Bbf_1) &\leq \delta \min_{j,\ell\in I:\, j<\ell}
\dist (\alpha_j \cap \Bbf_1, \alpha_\ell \cap \Bbf_1)\, .\label{e:pruning-2}\\
\max_{i,j\in I:\, i<j} \dist (\alpha_i \cap \Bbf_1, \alpha_j\cap \Bbf_1)
&= \max_{i<j} \dist (\alpha_i \cap \Bbf_1, \alpha_j \cap \Bbf_1)\, .\label{e:pruning-3}
\end{align}
\end{lemma}

\begin{proof}
Set $I (0) = \{1, \ldots, N\}$. If either (a) $N=2$ or (b) $N\geq 3$ and 
\begin{equation}\label{e:condition-on-E}
D \leq \delta \min_{i<j} \dist (\alpha_i\cap \Bbf_1, \alpha_j\cap\Bbf_1)\, ,
\end{equation}
then we select $I= I (0)$ and the proof is complete, since the left hand side of \eqref{e:pruning-1} is zero (and hence the left hand side of \eqref{e:pruning-2} equals $D$), while \eqref{e:pruning-3} is obvious. Observe also that, since $\varepsilon < \delta$, the condition \eqref{e:condition-on-E} is implied by (ii) when $N=2$.

Otherwise, we select indices $\ell_1$ and $\ell_2$ such that 
\[
\min_{i<j} \dist (\alpha_i\cap \Bbf_1, \alpha_j\cap\Bbf_1) = \dist (\alpha_{\ell_1}\cap \Bbf_1, \alpha_{\ell_2}\cap\Bbf_1)
\]
and indices $j_1$ and $j_2$ such that 
\[
\max_{i<j} \dist (\alpha_i\cap \Bbf_1, \alpha_j\cap\Bbf_1) = \dist (\alpha_{j_1}\cap \Bbf_1, \alpha_{j_2}\cap\Bbf_1)\, .
\]
Since $N\geq 3$ we can choose them so that $\{j_1, j_2\}
\neq \{\ell_1, \ell_2\}$. In particular, we can pick $\ell (0)\in \{\ell_1, \ell_2\}\setminus \{j_1,j_2\}$. We then set $I (1) := I (0) \setminus \{\ell (0)\}$. Notice that
\[
\max_{i<j\in I (1)}  \dist (\alpha_i\cap \Bbf_1, \alpha_j\cap\Bbf_1)
= \max_{i<j\in I (0)}  \dist (\alpha_i\cap \Bbf_1, \alpha_j\cap\Bbf_1)\, ,
\]
while, by the assumption that \eqref{e:condition-on-E} fails,
\[
\min_{j\in I (1)} \dist (\alpha_{\ell (0)}\cap \Bbf_1, \alpha_j\cap\Bbf_1)  < \delta^{-1}D\, .
\]
Assuming that we have inductively selected sets $I (0), I (1), \ldots I (s)$, we use the same procedure above to select a new subset $I (s+1)\subset I (s)$ by removing one element $\ell (s)$, provided that the cardinality of $I (s)$ is at least $3$ and 
\begin{equation}\label{e:condition-to-fail}
D + \max_j \min_{i\in I (s)} \dist (\alpha_i\cap \Bbf_1, \alpha_j \cap \Bbf_1)
> \delta \min_{i<j\in I (s)} \dist (\alpha_i \cap \Bbf_1, \alpha_j \cap \Bbf_1)\, .
\end{equation}
Otherwise, we stop; clearly this process must terminate in finitely many steps. We denote by $\sigma$ the index of the stopping step; note that $\sigma \leq N-2$. We claim that $I= I (\sigma)$ satisfies the requirements of the lemma.

First of all we prove the inequality 
\begin{equation}\label{e:inductive}
\min_{j\in I (s)} \dist (\alpha_{\ell (s')}\cap \Bbf_1, \alpha_j\cap \Bbf_1)
\leq (s-s') \min_{j\in I (s)} \dist (\alpha_{\ell (s-1)}\cap \Bbf_1, \alpha_j\cap \Bbf_1) \quad \forall s'< s \leq \sigma\, . 
\end{equation}
Note that if $s=s'+1$ the inequality is in fact an obvious equality (just from how it is written). In particular, the claim holds when $s'= \sigma-1$. We now assume that the claim holds for all $s'>s_0$ and will proceed to show it when $s'=s_0$, by induction. Fix $s>s_0$ and let $j_*\in I (s_0+1)$ be such that 
\[
\min_{j\in I (s_0+1)} \dist (\alpha_{\ell (s_0)}\cap \Bbf_1, \alpha_j\cap \Bbf_1)
= \dist (\alpha_{\ell (s_0)}\cap \Bbf_1, \alpha_{j_*}\cap \Bbf_1)\, .
\]
We first observe that, by the very definition of $\ell (s_0)$ and $j_*$ we have
\begin{equation}\label{e:minimality}
\dist (\alpha_{\ell (s_0)}\cap \Bbf_1, \alpha_{j_*}\cap \Bbf_1) \leq
\min_{j\in I (s)} \dist (\alpha_{\ell (s-1)}\cap \Bbf_1, \alpha_j\cap \Bbf_1)\, .
\end{equation}
In particular if $j_* \in I (s)$, the inequality \eqref{e:inductive} is obvious.

Otherwise $j_*\not\in I(s)$ and so $j_* = \ell (s_*)$ for some $s_0 < s_*< s$. In this case, using \eqref{e:minimality} and the fact that, by the inductive assumption,\eqref{e:inductive} holds for $s=s_*$, we write
\begin{align*}
& \min_{j\in I (s)} \dist (\alpha_{\ell (s_0)}\cap \Bbf_1, \alpha_j\cap \Bbf_1)\\
\leq & \dist (\alpha_{\ell (s_0)}\cap \Bbf_1, \alpha_{j_*} \cap \Bbf_1)
+ \min_{j\in I (s)} \dist (\alpha_{\ell (s_*)}\cap \Bbf_1, \alpha_j\cap \Bbf_1)\\
\leq & \min_{j\in I (s)} \dist (\alpha_{\ell (s-1)}\cap \Bbf_1, \alpha_j\cap \Bbf_1)
+ (s-s_*) \min_{j\in I (s)} \dist (\alpha_{\ell (s-1)}\cap \Bbf_1, \alpha_j\cap \Bbf_1)\, .
\end{align*}
In particular,
since $s_*-s_0 \geq 1$, we have shown \eqref{e:inductive} for $s=s_0$. We thus conclude that \eqref{e:inductive} indeed holds by induction over $s$.

Next, note that \eqref{e:inductive} implies that 
\begin{equation}\label{e:max-min-s-bound}
\max_j \min_{i\in I (s)} \dist (\alpha_i\cap \Bbf_1, \alpha_j \cap \Bbf_1)
\leq s \min_{i\in I (s)} \dist (\alpha_{\ell (s-1)}\cap \Bbf_1, \alpha_i\cap \Bbf_1)\, 
\end{equation}
for all $s$, by simply maximizing over all $s'<s$ on the left-hand side of \eqref{e:inductive}, since $s-s'\leq s$. In particular, combined with \eqref{e:condition-to-fail}, we must have 
\[
D + t \min_{i\in I (t)} \dist (\alpha_{\ell (t-1)}\cap \Bbf_1, \alpha_i\cap \Bbf_1)
> \delta \min_{i\in I (t+1)} \dist (\alpha_{\ell (t)} \cap \Bbf_1, \alpha_i\cap \Bbf_1)\, 
\]
for $t =1, \ldots , s-1$, with $s\leq \sigma$ fixed arbitrarily (note that the right-hand side of this expression equals that of \eqref{e:inductive}, by definition of $\ell(t)$). 
Setting $d (t) := \min_{i\in I (t)} \dist (\alpha_{\ell (t-1)}\cap \Bbf_1, \alpha_i\cap \Bbf_1)$, we rewrite the above as the recursive inequality
\[
\delta^{-1} (D+ t d (t)) > d (t+1)\, ,
\]
which, setting the convention $d(0)=0$, can be assumed valid for $j=0$ as well. We can thus iterate this to get
\begin{align*}
d(s) & \leq \delta^{-1}D\left(1+\delta^{-1}(s-1) + \delta^{-2}(s-1)(s-2) + \cdots + \delta^{-(s-1)}\cdot(s-1)!\right)\\
&\leq \delta^{-1}D\left(s\cdot \delta^{-(s-1)}\cdot(s-1)!\right)\\
& = \delta^{-s}s! D\, .
\end{align*}
Since $s\leq N-2$, we get $d (s) \leq \delta^{2-N} (N-2)! D $, so combining with \eqref{e:max-min-s-bound}, for any $s\leq \sigma$ we have
\[
\max_j \min_{i\in I (s)} \dist (\alpha_i\cap \Bbf_1, \alpha_j \cap \Bbf_1)
\leq s d (s) \leq (N-2) d (s) \leq \delta^{2-N} (N-1)! D\, .
\]
It is therefore the case that \eqref{e:pruning-1} holds with $I = I(\sigma)$.

Next, \eqref{e:pruning-2} certainly holds with $I = I(\sigma)$ by construction if {$|I(\sigma)|\geq 3$}, since then the procedure stopped due to the fact that \eqref{e:condition-to-fail} fails. We thus have to show that \eqref{e:pruning-2} holds when {$|I(\sigma)|=2$}. In this case observe that our procedure guarantees that
\begin{equation}\label{e:maximum-constant}
\max_{i<j\in I(\sigma)} \dist (\alpha_i\cap \Bbf_1, \alpha_j \cap \Bbf_1)
= \max_{i<j} \dist (\alpha_i\cap \Bbf_1, \alpha_j \cap \Bbf_1)\, ,
\end{equation}
which shows that \eqref{e:pruning-3} holds in general. But then, as $|I(\sigma)|=2$ we certainly have
$$\min_{i<j\in I(\sigma)} \dist (\alpha_i\cap \Bbf_1, \alpha_j \cap \Bbf_1) = \max_{i<j\in I(\sigma)} \dist (\alpha_i\cap \Bbf_1, \alpha_j \cap \Bbf_1)$$
and so combining this with \eqref{e:maximum-constant}, using assumption (ii) and the fact that we have already proved \eqref{e:pruning-1}, we have
\begin{align*}
D + \max_j\min_{i\in I(\sigma)}\dist (\alpha_i\cap \Bbf_1,\alpha_j\cap \Bbf_1)\leq D + \Gamma D & = (1+\Gamma )D\\
&\leq (1+\Gamma )\eps\max_{i<j}\dist (\alpha_i\cap\Bbf_1,\alpha_j\cap\Bbf_1)\\
&= \delta\min_{i<j\in I(\sigma)}\dist (\alpha_i\cap\Bbf_1,\alpha_j\cap\Bbf_1)\, ,
\end{align*}
proving \eqref{e:pruning-2} in this case also (we have used that $\varepsilon = (1+\Gamma )^{-1} \delta $ here).
\end{proof}

{We now iteratively apply Lemma \ref{l:pruning} to a cone $S$ comprised of planes $\{\alpha_1, \ldots, \alpha_N\}$ as in Assumption \ref{a:cones}, getting a finite family of subcollections of the integers $\{1, \ldots , N\}$ leading to a family of simpler cones, in the following fashion. The collections of planes after each application of Lemma \ref{l:pruning} can be thought of as ``layers'', which are subcollections of a starting one.} Moreover, when we fix plane in a certain layer, the closest one among those of the previous layers is much closer than the minimum distance between any pair of planes belonging to these previous layers. {What we end up with is that each plane we throw away when moving from one layer to the next is a distance comparable to the minimal distance in the original layer, whilst for the last layer we have comparability between the maximum and minimum distances of the planes.}

\begin{lemma}[Layer subdivision]\label{l:algorithm}
For every integer $N \geq 2$ and every $0<\delta \leq 1$, there is $\eta = \eta (\delta, N)>0$ with the following properties. Let $\mathbf{S}$ and $\alpha_1, \ldots , \alpha_N$ be as in Assumption \ref{a:cones}. Then, there is $\kappa\in \N$ and subcollections 
$I (0) \supsetneq I (1) \supsetneq \cdots \supsetneq I (\kappa)$ of the integers $\{1, \ldots , N\}$, each of cardinality at least $2$ and with $I(0) = \{1,\dotsc,N\}$, so that the numbers
\begin{align}
m (s) &:= \min_{i<j \in I(s)} \dist (\alpha_i \cap \Bbf_1, \alpha_j\cap \Bbf_1)\\
d (s) &:= \max_{i \in I (0)} \min_{j\in I (s)} \dist (\alpha_i \cap \Bbf_1, \alpha_j\cap \Bbf_1)\\
M (s) &:= \max_{i<j \in I(s)} \dist (\alpha_i \cap \Bbf_1, \alpha_j\cap \Bbf_1)
\end{align}
satisfy the following requirements:
\begin{itemize}
\item[(i)] $M (\kappa)= M (0)$;
\item[(ii)]  $\eta M (\kappa) \leq m (\kappa)$;
\item[(iii)] $d (s) \leq \delta m (s)$ and $\eta d (s) \leq m (s-1)$ for every $1\leq s\leq \kappa$;
\item [(iv)] {$m(s-1)\leq \delta m(s)$ for every $1\leq s\leq \kappa$.}
\end{itemize}
\end{lemma}

\begin{proof}
Let $0<\delta\leq 1$. We fix $\eta>0$ such that $\eta\leq\varepsilon$, where $\varepsilon$ is the constant in Lemma \ref{l:pruning} corresponding to $\delta/N$ in place of $\delta$. In particular, $N\eta\leq \Gamma^{-1}$, where $\Gamma$ is again as in Lemma \ref{l:pruning} with $\delta$ replaced by $\delta/N$.

If $\eta M(0) \leq m (0)$, we set $\kappa=0$ and obviously the Lemma holds. Otherwise, we inductively apply Lemma \ref{l:pruning} with $D= m (s-1)$ to produce $I (s)$ from $I(s-1)$, as long as $\eta M (s-1) > m (s-1)$. We wish to check that the conclusions of the lemma hold for this particular choice of subcollections. The fact that the sequence of sets is strictly decreasing and each set has cardinality at least two are obvious. Property (i) is immediate from \eqref{e:pruning-3} of Lemma \ref{l:pruning}, and (iv) is immediate from \eqref{e:pruning-2}, since we are taking $D= m (s-1)$. Property (ii) holds by construction, as the process must terminate in finite time; at worst, when $\kappa = N-2$ and $|I(\kappa)|=2$. Moreover, by \eqref{e:pruning-2} of Lemma \ref{l:pruning} we have the inequality
\[
\max_{i\in I (s-1)} \min_{j\in I (s)} \dist (\alpha_i \cap \Bbf_1, \alpha_j \cap \Bbf_1)
\leq \frac{\delta}{N} m (s)\, .
\]
for every $s=1,\dotsc,\kappa$ (recall that when generating $I(s)$, we are applying Lemma \ref{l:pruning} with $I(s-1)$ in place of $\{1,\dotsc,N\}$). But the triangle inequality then gives 
\[
d (s) \leq \frac{\delta}{N} (m(1) + \cdots + m (s))\, .
\]
Since $m (t) \leq m (s)$ for all $1\leq t<s$, and $\kappa \leq N-2$, we therefore must have $d(s) \leq \delta m(s)$ for every $s$, proving the first inequality in (iii). Finally, observe that by \eqref{e:pruning-1} of Lemma \ref{l:pruning}, we also have
\[
\max_{i\in I (s-1)} \min_{j\in I (s)} \dist (\alpha_i \cap \Bbf_1, \alpha_j \cap \Bbf_1)
\leq \Gamma m (s-1)
\]
for $s=1,\dotsc,\kappa$, and thus by the same triangle inequality argument as above, we achieve $d(s)\leq N\Gamma m(s-1)$ for every $s$. Hence, recalling that $N\eta\leq \Gamma^{-1}$, we achieve the second inequality of (iii).
\end{proof}

\subsection{Graphical parameterizations for $T$ over $\Sbf$}\label{s:graphical-approx}

The aim of this part is to efficiently parameterize area-minimizing currents over cones $\mathbf{S}$ as in Definition \ref{def:cones} satisfying the additional pairwise Morgan angle comparability condition (ii) outlined in the introduction of Section \ref{p:approx}. We recall that we are working under the Assumption \ref{a:main} throughout. Moreover, recall the sets $\mathscr{C} (Q)$ and $\mathscr{P}$ as in Definition \ref{def:cones} and the $L^2$ based excesses of Definition \ref{def:L2_height_excess}. 

The important concept of a ``balanced cone'' is given in the following definition.

\begin{definition}\label{d:balanced}
Let $\mathbf{S}\in \mathscr{C} (Q)$, $M\geq 1$, and let $\alpha_1, \ldots , \alpha_N$ be the $N$ distinct $m$-dimensional planes forming $\mathbf{S}$. We say that $\mathbf{S}$ is \emph{$M$-balanced} if for every $i\neq j$ the inequality
\begin{equation}\label{e:balanced}
\theta_2 (\alpha_i, \alpha_j) \leq M \theta_1 (\alpha_i, \alpha_j)
\end{equation}
holds for the two Morgan angles of the pair $\alpha_i, \alpha_j$.
\end{definition}

Observe that the condition is empty when $\mathbf{S}\in \mathscr{P}$. Moreover, if we say that a cone $\Sbf = \alpha_1\cup\cdots\cup \alpha_N\in \mathscr{C}(Q)$ is $M$-balanced for some $M>0$, we will implicitly be assuming that all the $\alpha_i$ are distinct. Loosely speaking, we are interested in balanced cones because two planes within a balanced cone can only degenerate to a single plane, and not to two planes meeting along an $(n-1)$-dimensional axis.

We will give two different graphical approximation results: a series of ``crude approximation'' results, followed by a series of ``refined approximation'' results.

\subsection{Crude approximation statements}\label{s:crude} Here we will give the statements of the crude approximation results; in the next section, we will prove all of them. The key starting point is the following splitting lemma. We recall the notation
\[
\boldsymbol{\sigma} (\Sbf) := \min_{i<j} \dist (\alpha_i \cap \Bbf_1, \alpha_j\cap \Bbf_1).
\]

\begin{lemma}[Crude splitting]\label{l:splitting-1}
For every $Q,m,n, \bar{n}\in \N$ and every $M>0$, $\rho,\eta>0$, there are constants $\delta = \delta (Q,m,n,\bar{n},M,\rho,\eta)>0$ and $\varrho = \varrho (Q,m,n,\bar{n},M,\rho,\eta)>0$ with the following property.
Let $T$ be as in Assumption \ref{a:main} with $\|T\| (\Bbf_4)\leq (Q+\frac{1}{2}) 4^m \omega_m$. Assume that $2\leq N \leq Q$, that $\mathbf{S}= \alpha_1 \cup \cdots\cup \alpha_N\in \mathscr{C}(Q)$ is $M$-balanced, and 
\begin{equation}\label{e:small-in-crude}
\int_{\Bbf_4\setminus B_{\rho} (V)} \dist^2 (p, \mathbf{S})\, d\|T\| (p) + \mathbf{A}^2 \leq \delta^2 {\boldsymbol{\sigma} (\Sbf)^2 =: \delta^2 \sigma^2}\, ,
\end{equation}
where $V = V(\Sbf)$ is the spine of $\Sbf$. Then the following properties hold:
\begin{itemize}
\item[(a)] The sets ${W_i}:= (\Bbf_4 \setminus \overline{B}_{\rho} (V))\cap \{\dist (\cdot, \alpha_i) < \varrho \sigma\}$ are pairwise disjoint;
\item[(b)] $\spt (T) \cap \Bbf_{4-\eta} \setminus \overline{B}_{\rho+\eta} (V) \subset \bigcup_i W_i$.
\end{itemize}
\end{lemma}

From the above lemma, the tilt-excess bound from Theorem \ref{thm:main-estimate}, and Almgren's strong Lipschitz approximation over planes (\cite{DLS14Lp}*{Theorem 1.4}), we can then conclude the following, where we use heavily the notation of \cite{DLS14Lp}*{Theorem 1.4} (in particular, if $u$ is a Lipschitz multi-valued map, ${\rm gr}\, (u)$ will denote its set-theoretic graph and $\mathbf{G}_u$ the current naturally induced by it).

\begin{proposition}\label{p:Lipschitz-1}
Let $T$, $W_i$, and all associated notation be as in Lemma \ref{l:splitting-1}. Consider for each $i\in\{1,\dotsc,N\}$ the regions 
$\Omega_i := (\Bbf_{4-2\eta} \cap \alpha_i)\setminus \overline{B}_{\rho+\eta} (V)$ and 
$\boldsymbol{\Omega}_i := \Bbf_{4-\eta}\cap\mathbf{p}_{\alpha_i}^{-1} (\Omega_i)$. Set $T_i := T \res (W_i\cap \boldsymbol{\Omega}_i)$ and 
\[
E_i := \int_{\Bbf_{4}\setminus B_{\rho} (V)} \dist^2 (p, \alpha_i)\, d\|T_i\| (p)
\]
Then, there are non-negative integers $Q_1, \ldots , Q_N$ with $\sum_i Q_i \leq Q$ satisfying the following properties:
\begin{itemize}
\item[(a)] $\partial T_i \res \boldsymbol{\Omega}_i = 0$;
\item[(b)] $\alpha_i$ can be appropriately oriented so that $(\mathbf{p}_{\alpha_i})_\sharp T_i = Q_i \llbracket \Omega_i \rrbracket$;
\item[(c)] The following estimate holds
\begin{equation}\label{e:Linfty-crude}
\dist^2(q,\alpha_i) \equiv |\mathbf{p}_{\alpha_i}^\perp (q)|^2 \leq C E_i + C \mathbf{A}^2 \qquad \forall q\in \spt (T_i)\cap \boldsymbol{\Omega}_i\, ;
\end{equation}
\item[(d)] For all $i$ with $Q_i\geq 1$, there are Lipschitz multi-valued maps $u_i: \Omega_i \to \mathcal{A}_{Q_i} (\alpha_i^\perp)$ and closed sets $K_i \subset \Omega_i$ such that ${\rm gr} (u_i) \subset \Sigma$, $T_i \res \mathbf{p}_{\alpha_i}^{-1} (K_i) = \mathbf{G}_{u_i} \res \mathbf{p}_{\alpha_i}^{-1} (K_i)$, and the following estimates hold:
\begin{align}
\|u_i\|_\infty^2 + \|Du_i\|_{L^2}^2 & \leq C (E_i +\mathbf{A}^2) \label{e:graphicality-1} \\
{\rm Lip}\, (u_i) &\leq C (E_i +\mathbf{A}^2)^\gamma \label{e:graphicality-2} \\
|\Omega_i \setminus K_i| + \|T\| (\mathbf{\Omega}_i\setminus \mathbf{p}_{\alpha_i}^{-1} (K_i))&\leq C (E_i +\mathbf{A}^2)^{1+\gamma}\, ; \label{e:non-graphicality}
\end{align}
\item[(e)] $Q_i=0$ if and only if $T_i=0$;
\item[(f)] Finally, if in addition we have the ``reverse excess" estimate
\begin{equation}\label{e:other-side}
\int_{\mathbf{S} \cap \Bbf_{4-2\eta}\setminus B_{\rho+2\eta} (V)} \dist^2 (p, \spt (T))\, d\mathcal{H}^m (p)
\leq \delta^2 \sigma^2\, ,
\end{equation}
then $Q_i\geq 1$ for every $i$.
\end{itemize}
Here, $\gamma = \gamma(Q,m,n,\bar{n})>0$ and $C = C(Q,m,n,\bar{n})>0$;
\end{proposition}

In several situations we will need to ensure that $\sum_i Q_i =Q$ in Proposition \ref{p:Lipschitz-1}. This conclusion needs however an additional assumption, as well as smallness of the tubular neighborhood radius $\rho$ of the spine. Depending on the situation we can guarantee this either by using (1) the existence of a point of large density; or (2) the existence of a region with sufficiently large mass.

\begin{lemma}\label{l:matching-Q}
There exists $\rho_* = \rho_*(Q,m)\in (0,1)$ such that, if we assume that $T$ satisfies the hypotheses of Lemma \ref{l:splitting-1} with $\rho\leq \rho_*$ and either:
\begin{itemize}
    \item[(a)] $\{\Theta (T, \cdot) \geq Q\} \cap \Bbf_{\varepsilon} (0)\neq \emptyset$ for a sufficiently small $\varepsilon = \varepsilon (Q, m,n, \bar n)$; or
    \item[(b)] for some $C_*>0$, $\rho_*$ is sufficiently small also depending on $C_*$, and there is a closed set $\Omega\subset \Bbf_4$ with non-empty interior that is invariant under rotation around $V$ (c.f. Definition \ref{def:rotationally_invariant}) and for which $\|T\| (\Omega) \geq (Q-\frac{1}{2}) \mathcal{H}^m (\alpha_1 \cap \Omega)\geq C_*$.
\end{itemize}
Then, if $Q_i$ is as in Proposition \ref{p:Lipschitz-1} and $\delta$ is sufficiently small, we have $\sum_i Q_i =Q$.
\end{lemma}

We will only be applying Lemma \ref{l:matching-Q} to specific choices of $\Omega$, for which we may ensure that $\mathcal{H}^m(\alpha_1\cap \Omega)\geq C_*$ always holds for a specific choice of $C_* = C_*(Q,m)>0$. In particular a suitable choice of $\rho_* = \rho_*(Q,m)$ will work whenever we will apply Lemma \ref{l:matching-Q} in the rest of the paper, without any additional assumptions involving $C_*$ in alternative (b). We will therefore fix this choice of radius $\rho_* = \rho_*(Q,m)$ for the rest of the paper.

Finally, we remark that all the statements above can be suitably scaled and translated to analogous statements where the initial domain $\mathbf{B}_4$ in Lemma \ref{l:splitting-1} is replaced by an arbitrary ball $\mathbf{B}_{4r} (q)$. 

An approximation statement analogous to Proposition \ref{p:Lipschitz-1} also holds in the much simpler setting in which $\mathbf{S}$ consists of a single plane. This case is somewhat special because we cannot identify a unique spine $V$ and at the same time the number $\sigma$ in Lemma \ref{l:splitting-1} is $\infty$, hence the smallness condition \eqref{e:small-in-crude} would be empty. For this reason we state the proposition separately even though it could be embedded as a special case of Proposition \ref{p:Lipschitz-1}.

\begin{proposition}[Crude approximation on a single plane]\label{p:Lipschitz-2}
For every $Q,m,n,\bar{n}\in \N$ and $\rho,\eta>0$, there exist $\delta = \delta (Q,m,n,\bar{n},\rho,\eta)>0$ and $C = C(Q,m,n,\bar{n})$ with the following property. Let $T$ be as in Assumption \ref{a:main} with $\|T\| (\Bbf_4)\leq (Q+\frac{1}{2}) 4^m \omega_m$. Assume $\mathbf{S}= \alpha_1\in \mathscr{P}$, $V\subset \alpha_1$ is an $(m-2)$-dimensional subspace, and 
\begin{equation}\label{e:small-in-crude-2}
E_1+\mathbf{A}^2 := \int_{\Bbf_4\setminus B_{\rho} (V)} \dist (p, \mathbf{S})^2\, d\|T\| (p) + \mathbf{A}^2 \leq \delta^2 \, .
\end{equation}
Set $\Omega_1 := (\Bbf_{4-2\eta} \cap \alpha_1)\setminus \overline{B}_{\rho+\eta} (V)$ and 
$\boldsymbol{\Omega}_1 := \Bbf_{4-\eta} \cap \mathbf{p}_{\alpha_{1}}^{-1} (\Omega_1)$ and $T_1 := T \res \boldsymbol{\Omega}_1$
Then, there is non-negative integer $Q_1\leq Q$ such that the following holds: 
\begin{itemize}
\item[(a)] $\partial T_1 \res \mathbf{\Omega}_1 = 0$
\item[(b)] $\alpha_1$ can be appropriately oriented so that $(\mathbf{p}_{\alpha_1})_\sharp T_1 = Q_1 \llbracket \Omega_1 \rrbracket$;
\item[(c)] The following estimate holds
\begin{equation}\label{e:Linfty-crude-2}
\dist^2(q,\alpha_1)\equiv |q-\mathbf{p}_{\alpha_1} (q)|^2 \leq C E_1 + C \mathbf{A}^2 \qquad \forall q\in \spt (T_1)\cap \boldsymbol{\Omega}_1\, ;
\end{equation}
\item[(d)] There is a Lipschitz multi-valued map $u_1: \Omega_1 \to \mathcal{A}_{Q_1} (\alpha_1^\perp)$ and a closed set $K_1 \subset \Omega_1$ such that ${\rm gr} (u_1) \subset \Sigma$, $T_1 \res \mathbf{p}_{\alpha_1}^{-1} (K_1) = \mathbf{G}_u \res \mathbf{p}_{\alpha_1}^{-1} (K_1)$ and the following estimates hold:
\begin{align}
\|u_1\|_\infty^2 + \|Du_1\|_{L^2}^2 & \leq C (E_1 +\mathbf{A}^2)\\
{\rm Lip}\, (u_1) &\leq C (E_1 +\mathbf{A}^2)^\gamma\\
|\Omega_1 \setminus K_1| + \|T\| (\mathbf{\Omega}_1\setminus \mathbf{p}_{\alpha_1}^{-1} (K_1))&\leq C (E_1 +\mathbf{A}^2)^{1+\gamma}\, ;
\end{align}
\item[(e)] $Q_1=0$ if and only if $T_1=0$;
\item[(f)] If $\rho\leq\rho_*$ and one of conditions (a) and (b) in Lemma \ref{l:matching-Q} holds then $Q_1=Q$. 
\end{itemize}
Here, $\gamma = \gamma(Q,m,n,\bar{n})>0$ and $C = C(Q,m,n,\bar{n})>0$.
\end{proposition}

We also have an analogous scaled versions of the above where we replace $\mathbf{B}_4(0)$ by an arbitrary ball $\mathbf{B}_{4r}(p)$ and all quantities are scaled accordingly.

We now prove all of the above results.

\subsection{Proofs of Crude Approximation Results}
We will prove Lemma \ref{l:splitting-1}, Proposition \ref{p:Lipschitz-1}, and Lemma \ref{l:matching-Q}. Proposition \ref{p:Lipschitz-2} then follows by the same arguments, and so we will leave the details to the reader.

\begin{proof}[Proof of Lemma \ref{l:splitting-1}] We remark that statement (a) of Lemma \ref{l:splitting-1} holds as soon as $\varrho$ is smaller than an suitable constant which depends only on $M,\rho,\eta$ (and the number of planes as well as the dimensions). We fix any such $\varrho$, and claim that as soon as $\delta$ is small enough (b) holds as well. 

Indeed, to this end we may argue by contradiction: consider a sequence $T_k$ of integral currents and a sequence $\Sigma_k$ of Riemannian submanifolds of $\mathbb{R}^{m+n}$ satisfying Assumption \ref{a:main}, together with cones $\mathbf{S}_k\in\Cscr(Q)$ such that
\begin{itemize}
    \item[(i)] $\|T_k\| (\Bbf_{{4}}) \leq (Q+\frac{1}{2}) \omega_m {4}^m$;
    \item[(ii)] $\mathbf{S}_k = \alpha^k_1 \cup \cdots \cup\alpha^k_{N (k)}\in \mathscr{C}(Q)$ is $M$-balanced, where $N (k) \leq Q$;
    \item[(iii)] If we write $\sigma_k := \boldsymbol{\sigma} (\Sbf_k)$ and
    \[
    \mathbf{E}_k := \int_{\Bbf_{{4}}\setminus B_{\rho} (V (\mathbf{S}_k))}
    \dist^2 (p, \mathbf{S}_k) d\|T_k\| (p)\, ,
    \]
    we have $\sigma_k^{-2} (E_k + \mathbf{A}_k^2)\to 0$, where $\Abf_k$ corresponds to $\Sigma_k$;
    \item[(iv)] There are points $p_k\in \spt (T_k)\cap \Bbf_{4-\eta}\setminus \overline{B}_{\rho+\eta} (V (\mathbf{S}_k))$ such that $\dist (p_k, \mathbf{S}_k) \geq \varrho \sigma_k$.
\end{itemize}
Observe that $\sigma_k$ is a bounded sequence since $0\leq\sigma_k\leq 1$. Hence,
up to the extraction of a subsequence and after applying suitable rotations, we can assume that:
\begin{itemize}
    \item[(v)] $V (\mathbf{S}_k)$ is a fixed $(m-2)$-dimensional plane $V$ and $N\equiv N(k)\leq Q$ is a fixed integer; 
    \item[(vi)] $\mathbf{S}_k$ converges, locally in Hausdorff distance, to $\mathbf{S}\in \mathscr{C}(Q)$ which is the union of $N'\leq N$ distinct planes $\alpha_i$ such that $\alpha_i \cap \alpha_j = V$ for all pairs $i<j$ (note that we could have $N^\prime = 1$, in which case the latter condition here is vacuous);
    \item[(vii)] $T_k$ converges to an integral current $T$ which is area-minimizing in {$\Bbf_4$ and obeys $\partial T = 0$};
    \item[(viii)] $\spt (T) \cap \Bbf_{{4}}\setminus B_{\rho} (V) \subset \mathbf{S}$.
\end{itemize}
By the constancy theorem and the fact that $\partial T = 0$, it then follows that $T\res \Bbf_4\setminus \overline{B}_\rho(V) = \sum_i \bar Q_i\llbracket \alpha_i\rrbracket \res (\Bbf_4\setminus \overline{B}_\rho(V))$ where the $\bar Q_i$ are integers. Orienting the $\alpha_i$ suitably, we can assume that $\bar Q_i\geq 0$.

We also know that $\spt(T_k)$ converges locally in Hausdorff distance to $\spt(T)$. In particular, passing to a subsequence, the points $p_k$ converge to some point $p$ which must lie in one of the planes $\alpha_j$ which form $\Sbf$. We denote by $\pi_0$ this latter plane and observe that clearly $|p|\leq {4-\eta}$ whilst $\dist(p,V)\geq {\rho+\eta}$. For each fixed $i$, the sequence of planes $\alpha_i^k$ must converge (locally in Hausdorff distance) to some plane of $\Sbf$ and, again upon extraction of a suitable subsequence and relabelling, we may assume that there is an $N_0\leq N$ with the property that $\alpha_i^k$ converges to $\pi_0$ when $i\leq N_0$, whilst it converges to some other plane of $\Sbf$ when $i>N_0$. {Clearly we know that $N_0\geq 1$ by construction.}

Now for a fixed parameter ${\bar{\eta}>0}$, consider the currents $T^\prime_k := T_k\res ((\Bbf_{4-\eta/2}\setminus \overline{B}_{\rho+\eta}(V))\cap \{\dist(\cdot,\pi_0)<{\bar{\eta}}\})$ and the cones $\Sbf^\prime_k:= \alpha^k_1\cup\cdots\cup \alpha_{N_0}^k$. Observe that, if we choose ${\bar{\eta}}$ sufficiently small, the convergence properties outlined above imply that:

\begin{itemize}
\item $\partial T'_k =0$ in $\Bbf_{4-\eta/2}\setminus \overline{B}_{\rho+\eta} (V)$ for all $k$ sufficiently large;
\item $\dist (q, \mathbf{S}_k)= \dist (q, \mathbf{S}'_k)$ for all $q\in \spt (T'_k)$ and all $k$ sufficiently large;
\item $T'_k$ converges to $\bar Q_j \llbracket \pi_0 \cap \Bbf_{4-\eta/2}\setminus \overline{B}_{\rho+\eta} (V)\rrbracket$, where $1\leq \bar Q_j\leq Q$ is as above.
\end{itemize}

Now consider the cylinder $\mathbf{C}_{4r} (p, \pi_0)${, where $p$ is the limit point as above and $r$ is a geometric constant.} Observe that, upon choosing $r$ suitably, in light of (iii), for all $k$ sufficiently large we can apply Corollary \ref{c:splitting-0} to $T_k^\prime$ in the cylinder $\Cbf_{4r}(p,\pi_0)$ {(which we may ensure is disjoint from $B_{\rho}(V)$ provided that we take $4r\leq {\eta}$)}. We can therefore decompose $T'_k \res \mathbf{C}_r (p, \pi_0)$ as $T'_{k,1}+\cdots+ T'_{k, N_0}$ with the property that 
\begin{equation}\label{e:close-to-alpha_i}
\dist (q, \alpha^k_i) = \dist (q, \mathbf{S}'_k) = \dist (q, \mathbf{S}_k) \qquad \forall q\in \spt (T'_{k,i})\, .
\end{equation}
Moreover, for each fixed $i$, the currents $T'_{k,i}$ converge to $Q_i \llbracket B_r (p, \pi_0)\rrbracket$, for some non-negative integers $Q_i$ which obey $\sum_i Q_i = \bar Q_j$. Consider now the points $p_{k,i} := \alpha^k_i \cap \mathbf{p}_{\pi_0}^{-1} (p)$ {(which exist for all $k$ sufficiently large)} and the cylinders $\mathbf{C}^{k,i} := \mathbf{C}_{r/2} (p_{k,i}, \alpha^k_i)$. The latter cylinder is converging to $\mathbf{C}_{r/2} (p, \pi_0)$ and hence in particular, for $k$ large enough, we know that $\Cbf^{k,i} \cap \Bbf_1\subset \Cbf_r(p,\pi_0)$. Hence, $\partial T'_{k,i} \res \mathbf{C}^{k,i} =0$  for $k$ large enough because $\partial T'_{k,i} \res \mathbf{C}_r (p, \pi_0) =0$ and $\spt(T'_{k,i})$ are converging locally in Hausdorff distance to $\bar B_r (p, \pi_0)$. In particular, it follows that $(\mathbf{p}_{\alpha^k_i})_\sharp T'_{k,i} = Q'_i \llbracket B_{r/2} (p_{k,i}, \alpha^k_i)\rrbracket$ for some integers $Q^\prime_i\geq 0$ and that $\|T'_{k,i}\| (\mathbf{C}^{k,i}) \to Q_i \omega_m 2^{-m} r^m$. 

However, because of \eqref{e:close-to-alpha_i} we also have 
\[
\int_{\mathbf{C}^i} \dist^2 (q, \alpha^k_i)\ d\|T'_{k,i}\| (q) \leq \mathbf{E}_k\, .
\]
We can now apply (a scaled version of) the $L^\infty$ estimate \eqref{e:Linfty-estimate} from Theorem \ref{thm:main-estimate} to conclude that 
\[
\dist^2 (q, \alpha^k_i) \leq C r^{-m-2} \mathbf{E}_k + C r^2 \mathbf{A}^2 
\leq C (r^{-m-2} + r^2) \delta^2 \sigma_k^2 
\]
for all $q\in \spt (T'_{k,i})\cap \mathbf{C}_{r/4} (p_{k,i}, \alpha^k_i)$. But the contradiction point $p_k$ must be contained in one of the cylinders $\mathbf{C}_{r/4} (p_{k,i}, \alpha^k_i)$ for $k$ sufficiently large, and it is also contained in the support of $T'_k = \sum_i T^\prime_{k,i}${; we may pass to a subsequence to fix the $i$ for which this is true}. Hence, for $k$ large enough we must have the estimate
\[
\dist^2 (p_k, \alpha^k_i) \leq C (r^{-m-2} + r^2) \delta^2 \sigma_k^2\, 
\]
where $i$ is now fixed. But since $r$ is fixed, it suffices to choose $\delta$ small enough to ensure $C (r^{-m-2} + r^2)\delta^2 \leq \varrho^2/4$ and contradict (iv) above; notice that $\delta$ indeed has the correct dependencies.
\end{proof}

\begin{proof}[Proof of Proposition \ref{p:Lipschitz-1}]
First of all notice that the conclusions (a), (b), and the estimate (c) all follow from the arguments in the proof of Lemma \ref{l:splitting-1}. The estimates in point (d) of Proposition \ref{p:Lipschitz-1} then follow from the Lipschitz approximation of \cite{DLS14Lp} and \eqref{e:tilt-estimate} of Theorem \ref{thm:main-estimate}.

As for the conclusion (e), because for each $i$ one has the identity
\[
\|T_i\| (\mathbf{\Omega}_i) = Q_i |\Omega_i| + \frac{1}{2} \int_{\Omega_i} |\vec{T} - \vec{\alpha}_i|^2\ d\|T_i\| \, ,
\]
if we again apply the tilt-excess estimate \eqref{e:tilt-estimate} from Theorem \ref{thm:main-estimate}, when $Q_i=0$ we must have 
\[
\|T_i\| (\mathbf{\Omega}_i) \leq {C}\delta^2 \sigma^2\, .
\]
In particular, if $r$ is the radius of the cylinders considered in the proof of Lemma \ref{l:splitting-1}, once $\delta$ is sufficiently small, the monotonicity formula guarantees that $\|T_i\| (\mathbf{B}_r (q))=0$ for every ball $\mathbf{B}_r (q)$ which is contained in $\mathbf{\Omega}_i$ (indeed, as soon as the mass ratio of $T_i$ falls below $1$, we get a contradiction if $T_i$ is not zero as $\Theta(T_i,q) \geq 1$ at every point $q\in\spt(T_i)$). This then implies conclusion (e).

As for (f), we easily conclude from the argument above that, if $Q_i=0$, then the distance of any point $q\in\Omega_i$ to $\spt (T)$ must be at least $\min \{ \varrho \sigma, r\}$. We thus infer
\[
\int_{\mathbf{S}\cap \Bbf_{4}\setminus B_{\rho} (V)} {\dist^2(p,\spt(T))}\ d\mathcal{H}^m (p) \geq c \min \{\varrho^2 \sigma^2, r^2\}
\]
for some geometric constant $c$. This obviously contradicts \eqref{e:other-side} if $\delta$ is small enough. This completes the proof of Proposition \ref{p:Lipschitz-1}.
\end{proof}

\begin{proof}[Proof of Lemma \ref{l:matching-Q}]
We argue by contradiction, starting in an identical manner to that seen in the proof of Lemma \ref{l:splitting-1}. Indeed, following the notation used there, it suffices to show the limiting current $T$ which obeys
$$T\res \Bbf_4\setminus\overline{B}_\rho(V) = \sum_i \bar{Q}_i\llbracket \alpha_i\rrbracket\res (\Bbf_4\setminus\overline{B}_\rho(V))$$
satisfies $\sum_i \bar{Q}_i = Q$ for any $\rho \leq \rho_*(Q,m)$ sufficiently small, when we suppose that one of the hypotheses (a) or (b) holds along the contradiction sequence (for suitable $\eps = \eps_k\downarrow 0$ in the case of (a) and suitable sets $\Omega_k$ in the case of (b)).

First, note that from the monotonicity formula and convergence (namely that we have weak$-*$ convergence of the masses $\|T_k\|$ to $\|T\|$, where $T_k$ are the currents as in the proof of Lemma \ref{l:splitting-1}) we know $\|T\|(\Bbf_3) \leq (Q+\frac{1}{2})3^m\omega_m$, which for sufficiently small $\rho$ evidently implies that we must have $\sum_i\bar{Q}_i\leq Q$. Thus, we just need to show that $\sum_i \bar{Q}_i\geq Q$ when we additional suppose the hypotheses from (a) or (b).

Note next that we can cover $\Bbf_3\cap V$ by $C\rho^{-(m-2)}$ balls of radius $\rho$; if we double the radius of each ball, we may then without loss of generality assume that they cover $\Bbf_3\cap B_\rho(V)$ also. But then the monotonicity formula for the mass ratios of $T$ gives for any such ball $\Bbf_i$ in this cover,
\[
\|T\|(\Bbf_i) \leq C\rho^m\|T\|(\Bbf_{7/2}) \leq C(Q,m)\rho^m
\]
and so
\begin{equation}\label{e:spine-nhd}
    \|T\|(\Bbf_3\cap B_\rho(V)) \leq C\rho^m\cdot \rho^{-(m-2)} = C\rho^2.
\end{equation}
Let us first suppose that the alternative (a) holds. Taking a sequence $\eps_k\downarrow 0$, for the sequence of currents $T_k$ from the proof of Lemma \ref{l:splitting-1}, we have a sequence of points $p_k\in \Bbf_{\eps_k}$ with $\Theta(T_k,p_k)\geq Q$. Upper semi-continuity of the density guarantees that the limiting current satisfies $\Theta(T,0)\geq Q$, and so $\|T\|(\Bbf_3) \geq 3^m\cdot Q\omega_m$, and hence $\|T\|(\Bbf_3\setminus\overline{B}_\rho(V)) \geq 3^m\cdot Q\omega_m - C\rho^2$. But then in light of the structure of $T$, this directly implies that
$$\sum_i \bar{Q}_i(3^m\omega_m - 3^{m-2}\omega_{2}\rho^2) \geq 3^m\cdot Q\omega_m - C\rho^2$$
and so since the $\bar{Q}_i$ are non-negative integers if $\rho \leq \rho_* = \rho_*(Q,m)$ is sufficiently small this evidently implies that $\sum_i \bar{Q}_i\geq Q$ in this case.

If instead (b) holds, then we have $\|T_k\|(\Omega_k)\geq (Q-\frac{1}{2})\mathcal{H}^m(\alpha_1^k\cap\Omega_k)$ for each $k$ and subsets $\Omega_k$ satisfying the given assumptions in (b). Since $\mathcal{H}^m(\alpha^1_k\cap\Omega_k)\geq C_*$, we may pass to a subsequence to ensure that $\mathcal{H}^m(\alpha^1_k\cap\Omega_k)\to \tilde{C}_* \in [C_*,4^m\omega_m]$. Again from the weak$-*$ convergence of the masses $\|T_k\|$ to $\|T\|$, using the rotational invariance of $T$ around $V$ in the region $\Bbf_4\setminus B_\rho(V)$, we readily get
$$\left(Q-\tfrac{1}{2}\right)\tilde{C}_* \leq \lim_{k\to\infty}\|T_k\|(\Omega_k) \leq \lim_{k\to\infty}\|T_k\|(\Omega_k\setminus B_\rho(V)) + C\rho^2 \leq \sum_i \bar{Q}_i\tilde{C}_* + C\rho^2.$$
Note that in the second inequality we use a mass bound analogous to the one in \eqref{e:spine-nhd}, which one may observe still holds for $T_k$ in $\Omega_k\cap B_\rho (V)$ since $\Omega_k$ is closed and contained in $\Bbf_4$. Thus, provided $\rho \leq \rho_* = \rho_*(Q,m,C_*)$ is sufficiently small, as $\sum_i \bar{Q}_i$ is always an integer this gives that $\sum_i \bar{Q}_i \geq Q$, completing the proof.
\end{proof}

\subsection{Refined approximation}\label{s:refined} We now come to the second graphical approximation, which follows a much more refined procedure, using the layers introduced in Lemma \ref{l:algorithm}. For the reader acquainted with \cite{W14_annals}, this should be compared to \cite{W14_annals}*{Remark (3) in Section 8, and Section 10}.

\begin{assumption}[Assumptions for the refined approximation]\label{a:refined}
Suppose $T$ and $\Sigma$ are as in Assumption \ref{a:main} and $\|T\| (\Bbf_4) \leq 4^m (Q+\frac{1}{2}) \omega_m$. Suppose $\mathbf{S}=\alpha_1\cup \cdots \cup \alpha_N$ is a cone in $\mathscr{C} (Q)\setminus \mathscr{P}$ which is $M$-balanced, where $M>0$ is a given fixed constant, and $V$ is the spine of $\mathbf{S}$. For a sufficiently small constant $\eps=\eps(Q,m,n,\bar{n},M)$ whose choice will be fixed in Assumption \ref{a:parameters} below, suppose that $\{\Theta (T,\cdot) \geq Q\} \cap \Bbf_\varepsilon (0) \neq \emptyset$ and (cf. Proposition \ref{p:Lipschitz-1}) suppose that
\begin{equation}\label{e:smallness-alg}
\mathbb{E} (T, \mathbf{S}, \Bbf_4) + \mathbf{A}^2 \leq \varepsilon^2 \boldsymbol{\sigma} (\Sbf)^2\, .
\end{equation}
Recall that $\mathbb{E}(T,\Sbf,\Bbf_4)$ is a two-sided $L^2$ excess between $T$ and $\Sbf$.
\end{assumption}

\subsubsection{Whitney decomposition}\label{ss:whitney} Let $L_0$ be the closed cube in $V$ with side-length $\frac{2}{\sqrt{m-2}}$ centered at $0$ (if we make the identification $V= \mathbb R^{m-2}$, then we can write explicitly $L_0 =\Large[-\frac{1}{\sqrt{m-2}}, \frac{1}{\sqrt{m-2}}\Large]^{m-2}$) and consider the set 
\begin{equation}\label{e:region-R}
R :=\{p: \mathbf{p}_V (p) \in L_0\, \text{ and }\, 0<|\mathbf{p}_{V^\perp} (p)|\leq 1\}\, .
\end{equation}
We recall here that we are assuming $m\geq 3$ (cf. Assumption \ref{a:main}). 

Obviously $R$ is invariant under rotations around $V$. We next decompose $R$ into a countable family of closed sets which are also invariant under rotations around $V$. Firstly, for every $\ell\in \mathbb N$ denote by $\mathcal{G}_\ell$ the collection of $(m-2)$-dimensional cubes in the spine $V$ obtained by subdividing $L_0$ into $2^{\ell (m-2)}$ cubes of side-length $\frac{2^{1-\ell}}{\sqrt{m-2}}$, and we let $\mathcal{G} = \bigcup_\ell \mathcal{G}_\ell$. {Note that we can generate $\mathcal{G}_{\ell+1}$ from $\mathcal{G}_\ell$ by bisecting every face of every cube in $\mathcal{G}_\ell$. We write $L$ for a cube in $\mathcal{G}$, so $L\in \mathcal{G}_\ell$ for some $\ell\in \N$.} When we want to emphasize the dependence of the integer $\ell$ on $L$ we will write $\ell (L)$ and we will call it the \emph{generation of $L$}. If $L\subset L'$ and $\ell (L')= \ell (L)+1$, we then call $L'$ the \emph{parent} of $L$, and $L$ a \emph{child} of $L'$, while more generally, when $\ell (L') > \ell (L)$, we say that $L'$ is an \emph{ancestor} of $L$ and $L$ a \emph{descendant} of $L'$.

For every $L\in \mathcal{G}_\ell$ we let 
\[
R (L) := \{p: \mathbf{p}_V (p)\in L \quad \mbox{and} \quad 2^{-\ell-1} \leq |\mathbf{p}_{V^\perp} (p)|\leq 2^{-\ell}\}\, .
\]
Observe that $R (L_0)$ is not the region $R$ (as $L_0\in\mathcal{G}_0$ so $\ell(L_0)=0$), but rather $R=\bigcup_{L\in\Gcal} R(L)$.

For each $L\in \mathcal{G}_\ell$ we let $y_L\in V$ be its center and denote by $\Bbf (L)$ the ball $\mathbf{B}_{2^{2-\ell(L)}} (y_L)$ {(in $\R^{m+n}$)}
and by $\mathbf{B}^h (L)$ the set $\Bbf (L) \setminus B_{\rho_*2^{-\ell(L)}} (V)$, where $\rho_*$ is as in Lemma \ref{l:matching-Q}. 
It will be convenient to consider slight enlargements of the sets $R (L)$. More precisely, given a positive number $1\leq \lambda\leq \frac{3}{2}$ and $L\in \mathcal{G}_\ell$ we will denote by $\lambda L$ the cube concentric to $L$ {in $V$} with side-length $\frac{\lambda 2^{1-\ell}}{\sqrt{m-2}}$ and $\lambda R (L)$ the set 
\[
\lambda R (L) := \{p: \mathbf{p}_V (p)\in \lambda L \quad \mbox{and} \quad \lambda^{-1} 2^{-\ell-1} \leq |\mathbf{p}_{V^\perp} (p)|\leq \lambda 2^{-\ell}\}\, .
\]
In fact, in the rest of the paper we will use some fixed choices of $\lambda$ which depend only on the dimension and which are rather close to $1$. Finally, an important role will be played by the ``planar cross sections'' of the sets $R (L)$ and $\lambda R (L)$, namely the intersections of these sets with the planes $\alpha_i$ forming the cone $\mathbf{S}$; these intersections will be denoted by $L_i$ and $\lambda L_i$, respectively.

The following elementary lemma, whose proof is left to the reader, summarizes some important geometric properties of the sets just introduced, and verifies that $\{R(L):L\in\Gcal\}$ is indeed a Whitney decomposition towards $V$. With a slight abuse of terminology we will talk about the interior of $L$ and $L_i$ meaning their interiors in the relative topology of $V$ and $\alpha_i$, respectively. In order to help the reader visualize the content of Lemma \ref{l:whitney} we refer to Figure \ref{f:whitney-1} below. 

\begin{lemma}\label{l:whitney}
Consider the collection of cubes $\mathcal{G}$ introduced above and its elements $L$. Then the following properties hold:
\begin{itemize}
\item[(i)] Given any pair of distinct $L, L'\in \mathcal{G}$ the interiors of $R(L)$ and $R(L')$ are pairwise disjoint and $R(L)\cap R(L') \neq \emptyset$ if and only if $L\cap L' \neq \emptyset$ and $|\ell(L)-\ell(L')|\leq 1$, while the interiors of $L$ and $L'$ are disjoint if $\ell (L) \leq \ell (L')$ and $L'$ is not an ancestor of $L$.
\item[(ii)] The union of $R(L)$ ranging over all $L\in\mathcal{G}$ is the whole set $R$.
\item[(iii)] The diameters of the sets $L$, $R (L)$, $\lambda L$, $\lambda R (L)$, $L_i$, $\lambda L_i$, and $\mathbf{B}^h (L)$ are all comparable to $2^{-\ell (L)}$ and, with the exception of $L,\lambda L$, all comparable to the distance between an arbitrarily element within them and $V$; more precisely, any such diameter and distance is bounded above by $C 2^{-\ell (L)}$ and bounded below by $C^{-1} 2^{-\ell (L)}$ for some constant $C$ which depends only on $m$ and $n$.
\item[(iv)] There is a constant $C = C(m,n)$ such that, if {$\Bbf^h (L) \cap \Bbf^h (L') \neq \emptyset$, then $|{\ell (L)} - \ell (L')|\leq C$ and $\dist (L, L') \leq C 2^{-\ell (L)}$. In particular, for every $L\in \mathcal{G}$, the subset of $L'\in \mathcal{G}$ for which $\Bbf^h (L)$ and $\Bbf^h (L')$} have nonempty intersection is bounded by a constant. 
\item[(v)] $\sum_{L\in \mathcal{G}_\ell} \mathcal{H}^{m-2} (L) = C(m)$ for any $\ell$ and therefore, {for any $\kappa>0$},
\begin{equation}\label{e:geometric}
\sum_{L\in \mathcal{G}} 2^{-(m-2+\kappa) \ell (L)} \leq C (\kappa, m)\, .
\end{equation}
\end{itemize}
\end{lemma}

\begin{figure}[h]
    \def\svgwidth{0.55\columnwidth}
    \rotatebox{-90}{
    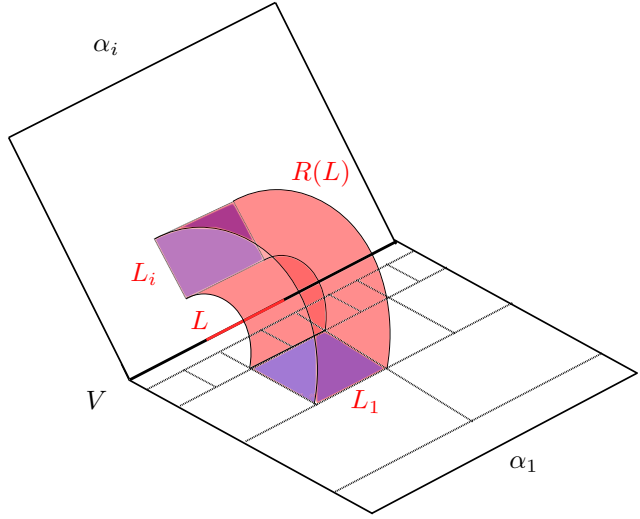}
    \vspace{-2em}
\caption{\Small An illustration of the Whitney decomposition of $R$ (illustrated on $\alpha_1$). The subspace $V$ is represented by the thick line joining the two planes, with a representation cube $L$ showing, with the corresponding sets $R(L)$ (in red) and $L_i$ (in blue). The set $R(L)$ is obtained by rotating the given cube $L_1$ around $V$ in the ambient $(m+n)$-dimensional space, a portion of which is shown.}
\label{f:whitney-1}
\end{figure}

\subsubsection{Layering and choice of the parameters}\label{s:layering-param}
We now fix a $\bar \delta>0$ (whose choice will specified below in Assumption \ref{a:parameters}) and apply the layering subdivision Lemma \ref{l:algorithm} with this $\bar{\delta}$ in place of $\delta$ therein, to identify a family of sub-cones $\mathbf{S} = \mathbf{S}_0 \supsetneq \mathbf{S}_1 \supsetneq \cdots \supsetneq \mathbf{S}_\kappa$ where $\mathbf{S}_k$ consists of the union of the planes $\alpha_i$ with $i\in I (k)$ for the set of indices $I(k)$ given by Lemma \ref{l:algorithm}. We then distinguish two cases:
\begin{itemize}
\item[(a)] if $\max_{i<j\in I (\kappa)} \dist (\alpha_i \cap \Bbf_1, \alpha_j \cap \Bbf_1) < \bar \delta$, we define an additional cone $\mathbf{S}_{\kappa+1}$ consisting of a single plane, given by the smallest index in $I (\kappa)$ and we set $\bar\kappa := \kappa+1$ and $I(\bar\kappa):= \{\min I(\kappa)\}$;
\item[(b)] otherwise, we select no smaller cone and set $\bar \kappa := \kappa$.
\end{itemize}
We next detail the choice of the various parameters involved in our discussion.
\begin{assumption}[Selection of the parameters]\label{a:parameters}
Firstly, we denote by $\delta^*$ the minimum of the parameters $\delta$ needed to ensure that Proposition \ref{p:Lipschitz-1}, Lemma \ref{l:matching-Q}, and Proposition \ref{p:Lipschitz-2} are applicable to all the cones $\mathbf{S}_k${, $k\in \{0,1,\dotsc,\bar{\kappa}\}$}: note that all the $\Sbf_k$ are $M$-balanced by construction and that therefore $\delta^* = \delta^*(m,n,\bar{n},Q,M)>0$, and in particular its choice does not depend on $\bar\delta$. Subsequently, we fix a parameter $\tau = \tau(m,n,\bar{n},Q,M)>0$ smaller than $\frac{\delta^*}{C}$ for some large constant $C = C(m,n,\bar{n},Q)>0$. The parameter $\bar\delta$ leading to the layering $\mathbf{S}_0 \supsetneq \mathbf{S}_1 \supsetneq \cdots \supsetneq \mathbf{S}_{\bar\kappa}$ is then chosen to be much smaller than $\tau$; so $\bar\delta = \bar\delta(m,n,\bar{n},Q,M)>0$. In particular, $\bar\delta \leq \delta^*$. Finally, $\varepsilon = \varepsilon(m,n,\bar{n},Q,\delta^*,\bar\delta,\tau)>0$ will be chosen even smaller than {$\bar\delta$}.
\end{assumption}

\subsubsection{Outer, central, and inner regions} We will next subdivide the cubes in $\mathcal{G}$ using the following criterion.  In order to simplify our notation we introduce the shorthand:
\[
\mathbf{E} (L,k):= 2^{(m+2) \ell (L)} \int_{\Bbf^h (L)} \dist^2(q, \mathbf{S}_k)\, d\|T\| (q)\, .
\]
For every $k\in \{0,1,\dotsc,\bar\kappa\}$, recall that $I (k)$ gives the subset of $\{1, \ldots , N\}$ such that $\mathbf{S}_k = \bigcup_{i\in I (k)} \alpha_i$. If $I (k)$ consists of more than one element, we set 
\begin{equation}\label{e:separation}
\mathbf{s} (k) := \min_{i<j \in I (k)} \dist (\alpha_i \cap \Bbf_1, \alpha_j \cap \Bbf_1)\, ,
\end{equation}
while we set $\mathbf{s} (k) := \bar\delta$ if $I (k)$ is a singleton. 
\begin{definition}\label{d:regions}
Let $L\in \mathcal{G}$. We say that:
\begin{itemize}
\item[(i)] $L$ is an \emph{outer cube} if $\mathbf{E} (L^\prime,0) \leq \tau^2 \mathbf{s} (0)^2$
for every ancestor $L^\prime$ of $L$ (including $L$).
\item[(ii)] $L$ is a \emph{central cube} if it is not an outer cube and if 
$\min_k \mathbf{E} (L^\prime, k)/{\mathbf{s}(k)^2} \leq \tau^2$
for every ancestor $L'$ of $L$ (including $L$).
\item[(iii)] $L$ is an \emph{inner cube} if it is neither an outer nor a central cube, but its parent is an outer or a central cube. 
\end{itemize}
The corresponding families of cubes will be denoted by $\mathcal{G}^o$, $\mathcal{G}^c$, and $\mathcal{G}^{in}$, respectively. Observe that any cube $L\in \mathcal{G}$ is either an outer cube, or a central cube, or an inner cube, or a descendant of inner cube.
\end{definition}
We correspondingly define three subregions of $R$:
\begin{itemize}
    \item The \emph{outer region}, denoted $R^o$, is the union of $R (L)$ for $L$ varying over elements of $\mathcal{G}^o$.
    \item The \emph{central region}, denoted $R^c$, is the union of $R (L)$ for $L$ varying over elements of $\mathcal{G}^c$.
    \item Finally, the \emph{inner region}, denoted $R^{in}$, is the union of $R(L)$ for $L$ ranging over the elements of $\mathcal{G}$ which are neither outer nor central cubes.
\end{itemize}
Alternatively, the inner region can be defined as the union of $R(L)$ for $L$ ranging over the inner cubes and their descendants. For a visual illustration of the subdivision and the corresponding cubes we refer to Figure \ref{f:whitney-2}.

\begin{figure}[htbp]
 \def\svgwidth{0.55\columnwidth}
    \rotatebox{-90}{
    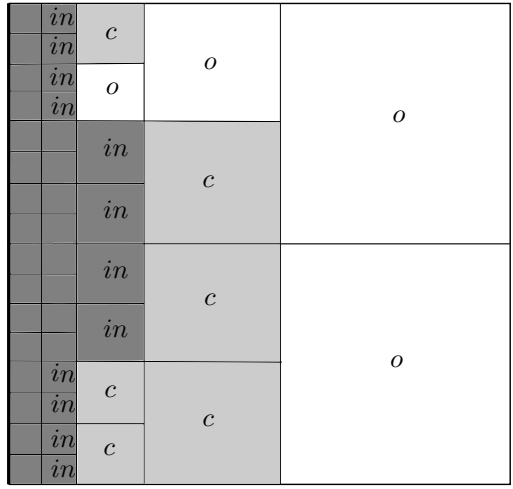}
    \vspace{-2em}
\caption{\Small An example of a possible labeling of the cubes and of a corresponding subdivision of $R$. The outer region is white, while the central region is lightly shadowed and the inner region is shadowed. The labels $o$, $c$, and $in$ identify rotationally invariant sets $R (L)$ corresponding to cubes $L$ which are, respectively, outer, central, and inner cubes. Note that descendants of inner cubes {\em are not} inner cubes, even though the corresponding rotationally invariant regions are included in the inner region.}\label{f:whitney-2}
\end{figure}

The following lemma will be pivotal to define our refined approximation.
\begin{lemma}\label{l:controls_Whitney}
Let $T$ and $\mathbf{S}$ be as in Assumption \ref{a:refined} and assume the parameters $\bar\delta$, $\delta^*$, $\tau$, and $\varepsilon$ satisfy Assumption \ref{a:parameters} and the ratios $\frac{\varepsilon}{\bar\delta},\frac{\bar\delta}{\tau}$, and $\frac{\tau}{\delta^*}$ are smaller than a constant $c =c (Q,m,n,\bar{n})>0$. Then
\begin{itemize}
    \item[(i)] ${L_0}\in \mathcal{G}^o$ and moreover, for every choice of $\tau$ and $\ell \in \mathbb N$ there is a constant $\bar c=\bar c(Q,m,n,\bar{n},\tau,\ell)>0$ such that, if $\varepsilon< \bar c$, then $\mathcal{G}_\ell \subset \mathcal{G}^o$.
    \item[(ii)] For every $L\in \mathcal{G}^c$ there is an index $k (L)\in \{0, \ldots , \bar\kappa\}$ such that 
    \begin{equation}\label{e:below-tau}
\mathbf{E} (L,k) \leq \tau^2 \mathbf{s} (k)^2 
\end{equation}
for all $k\geq k (L)$ while for all $k<k(L)$ we have 
\begin{equation}\label{e:above-tau}
\mathbf{E} (L,k) > \tau^2 \mathbf{s} (k)^2\, .
\end{equation} 
\item[(iii)] For every $L\in \mathcal{G}^o$ we have that \eqref{e:below-tau} holds for $k=0$ (and thus for every $k$), whilst for every $L\in \mathcal{G}^{in}$ we have that \eqref{e:above-tau} holds for every $k$.
\item[(iv)] There is a constant $\bar C = \bar C(Q,m,n,\bar{n},\bar\delta, \delta^*, \tau)>0$ such that
\begin{align}
\mathbf{E} (L, k(L)) &
\leq \bar C \mathbf{E} (L, 0) \hspace{3em}\forall L\in \mathcal{G}^c\label{e:0-control-central}\\
1 & \leq \bar C \mathbf{E} (L, 0) \hspace{3em}\forall L\in \mathcal{G}^{in}.\label{e:0-control-inner}
\end{align}
\item[(v)] For every $L\in \mathcal{G}^o$, Proposition \ref{p:Lipschitz-1} is applicable to the current $T_{y_L, 2^{-\ell{(L)}}}$ and the cone $\mathbf{S}_0$, whilst for every $L\in \mathcal{G}^c$ either Proposition \ref{p:Lipschitz-1} or Proposition \ref{p:Lipschitz-2} is applicable to the current $T_{y_L, 2^{-\ell{(L)}}}$ and the cone $\mathbf{S}_{k(L)}$ {(depending on whether $\Sbf_{k(L)}$ consists of at least two planes or is a single plane).}
\end{itemize}
\end{lemma}

Given Lemma \ref{l:controls_Whitney} it is convenient to introduce the convention $k (L) =0$ if $L\in \mathcal{G}^o$ and to set $\mathbf{E} (L):=\mathbf{E} (L, k(L))$. Note however that $k(L)$ might be indeed equal to $0$ for several elements of $\mathcal{G}^c$ as well. 

\begin{proof}
    First of all, for any $L\in \mathcal{G}$ we have, since $\Bbf^h (L) \subset \Bbf_4$, that
\[
\mathbf{E} (L,0) \leq C 2^{(m+2)\ell (L)} \mathbb{E} (T, \mathbf{S}, \Bbf_4) \leq C 2^{(m+2)\ell (L)} \varepsilon^2 \mathbf{s} (0)^2\, ,
\]
where the second inequality comes from \eqref{e:smallness-alg}, and where $C$ depends only on $m$. In particular statement (i) in the Lemma follows immediately, as for any $L\in \mathcal{G}_{\ell}$ we have an upper bound on $\ell(L)$ and so choosing $C2^{(m+2)\ell}\varepsilon^2 <\tau^2$ we get the desired inequality.

Regarding statement (ii), notice that 
\begin{equation}\label{e:excess-cone-bound}
\mathbf{E} (L, k) \leq C \dist^2 (\mathbf{S}_{k-1}, \mathbf{S}_k) + C \mathbf{E} (L, k-1)
\end{equation}
for each $k=0,\dots,\bar\kappa$ and for some constant $C=C(m,n)$. Recalling that $\dist (\mathbf{S}_{k-1}, \mathbf{S}_k) \leq \bar \delta \mathbf{s} (k)$ and that $\mathbf{s} (k-1) \leq \bar \delta \mathbf{s} (k)$ (by Lemma \ref{l:algorithm} (iii) and (iv)), we immediately get
\[
\mathbf{E} (L, k) \leq C \left(1+\frac{\mathbf{E} (L, k-1)}{\mathbf{s} (k-1)^2}\right) \bar\delta^2 \mathbf{s} (k)^2\, .
\]
In particular, if $\frac{\bar\delta}{\tau}$ is smaller than an appropriate geometric constant, we infer that the inequality $\mathbf{E} (L, k-1) \leq \tau^2 \mathbf{s} (k-1)^2$ implies $\mathbf{E} (L, k) \leq \tau^2 \mathbf{s} (k)^2$. But then statement (ii) holds if we let $k(L)$ be the smallest integer $k$ for which $\mathbf{E} (L, k) \leq \tau^2 \mathbf{s} (k)^2$, whose existence is guaranteed by the assumption that $L\in \mathcal{G}^c$. The same argument implies immediately the statement (iii) (as if $L\in \mathcal{G}^o$ then we know the minimum integer $k(L)$ above is $k(L) = 0$, whilst if $L\in\mathcal{G}^{in}$ the minimum over all $k$ of the ratio $\mathbf{E} (L, k)/\mathbf{s} (k)^2$ is larger than $\tau^2$).

Now recalling the second half of Lemma \ref{l:algorithm}(iii), we have 
\[
\dist (\mathbf{S}_{k-1}, \mathbf{S}_k) \leq \eta^{-1} \mathbf{s} (k-1)\, ,
\]
where $\eta$ depends on $N$ and $\bar\delta$. In particular, recalling \eqref{e:excess-cone-bound}, if $\mathbf{E} (L, k-1) > \tau^2 \mathbf{s} (k-1)^2$ we conclude the inequality
\[
\mathbf{E} (L, k) \leq \frac{C}{\eta^2} \mathbf{s} (k-1)^2 + C \mathbf{E} (L, k-1)
\leq C(1+ (\eta\tau)^{-2}) \mathbf{E} (L, k-1)\, .
\]
If $L\in \mathcal{G}^c$ we can apply the latter for $k=1, \ldots , k (L)$ to conclude \eqref{e:0-control-central}. If $L\in \mathcal{G}^{in}$ we can apply it for every $k$, combined with \eqref{e:above-tau}, to conclude that
\[
\mathbf{s} (\bar \kappa)^2 \leq C \mathbf{E} (L, 0)\, .
\]
We now distinguish two cases:
\begin{itemize}
    \item $I (\bar\kappa)$ consists of a single element; in this case $\mathbf{s} (\bar\kappa) = \bar\delta$.
    \item $I (\bar\kappa)$ consists of more than one element; in this case (from Lemma \ref{l:algorithm}(ii) and the defining property (b) of $\bar\kappa$ in Section \ref{s:layering-param})
    \[
    \mathbf{s} (\bar \kappa) \geq \eta \max_{i<j\in I (\bar \kappa)} d (\alpha_i\cap \Bbf_1, \alpha_j \cap \Bbf_1) \geq \eta \bar\delta\, .
    \]
\end{itemize}
In both cases we conclude that $\mathbf{s} (\bar\kappa)$ is bounded away from $0$ and thus \eqref{e:0-control-inner} holds, proving (iv).

As for (v), observe that for $L\in \mathcal{G}^o\cup \mathcal{G}^c$,
\[
\int_{\mathbf{B}_4\setminus B_{\rho_*} (V)} \dist^2(q, \mathbf{S}_{k(L)})\, d\|T_{y_L, 2^{-\ell (L)}}\| (q)
= \mathbf{E} (L, k (L)) \leq \tau^2 \mathbf{s} (k (L))^2\, ,
\]
where we have used (iii). Moreover, $\mathbf{s} (0) \leq \mathbf{s} (k)$ for all $k$ and $\mathbf{A}^2 \leq \varepsilon^2 \mathbf{s} (0)^2$ by \eqref{e:smallness-alg}. Hence, if we denote by $\mathbf{A}_L$ the supremum norm of the second fundamental form of the rescaled manifold $\Sigma_{y_L, 2^{-\ell (L)}}$, we then have 
\[
\int_{\mathbf{B}_4\setminus B_{\rho_*} (V)} \dist^2(q, \mathbf{S}_{k(L)})\, d\|T_{y_L, 2^{-\ell (L)}}\| (q)
+ \mathbf{A}_L^2 \leq C (\tau^2 + 2^{-2\ell (L)} \varepsilon^2) \mathbf{s} (k)^2\, .
\]
Recalling the definition of $\mathbf{s} (k)$ it suffices to guarantee that $C (\tau^2 + \varepsilon^2) \leq (\delta^*)^2$ to guarantee that Proposition \ref{p:Lipschitz-1} is applicable in case $I (k)$ consists of more than one element. Otherwise we know that $\mathbf{s} (k) = \bar\delta$ and the same smallness condition guarantees the applicability of Proposition \ref{p:Lipschitz-2}. This completes the proof.
\end{proof}

\subsubsection{Local approximations} 
For each outer and central cube $L$, from Lemma \ref{l:controls_Whitney}(v) we apply Proposition \ref{p:Lipschitz-1} or Proposition \ref{p:Lipschitz-2} to the current $T_{y_L, 2^{-\ell (L)}}$ and the cone $\Sbf_{k(L)}$ (the choice of which proposition to apply depends only on whether $|I(k(L))|>1$ or not). We thus gain a corresponding Lipschitz approximation for $T_{y_L,2^{-\ell(L)}}$ on the planar domains $(\Bbf_2\setminus \overline{B}_{2\rho_*} (V))\cap \mathbf{S}_{k (L)}$. By translating and scaling back we gain corresponding Lipschitz approximations of the current $T$ defined over the open domains $\Omega (L) := (\Bbf_{2^{1-\ell (L)}} {(y_L)} \setminus \overline{B}_{\rho_*2^{1-\ell (L)}}(V))\cap \mathbf{S}_{k (L)}$ of $\mathbf{S}_{k(L)}$.

We set $\Omega_i (L) := \Omega (L) \cap \alpha_i$ and denote by $u_{L,i}$ the corresponding multi-valued function given above and by $Q_{L,i}$ the number of values $u_{L,i}$ takes. We moreover denote by $\mathbf{\Omega}_i (L)$ the sets $2^{-\ell (L)}\mathbf{\Omega}_i + y_L$ and by $K_i (L)$ the sets $2^{-\ell (L)} K_{L,i} + y_L$, with $\mathbf{\Omega}_i$ and $K_{L,i}$ given by Proposition \ref{p:Lipschitz-1} (or Proposition \ref{p:Lipschitz-2}) for $T_{y_L, 2^{-\ell (L)}}$. These objects are defined only for the indices $i$ belonging to the collection $I (k(L))$, but we may extend our notation to allow also the case $Q_{L,i} = 0$, in which the map $u_{L,i}$ does not exist (these will be the indices of planes not contained in $\Sbf_{k(L)}$);  this results in a slight abuse of our terminology and notation. The collection of corresponding multi-valued functions (ranging over all $i$) corresponding to a given $L$ will be denoted by $u_L$, and we will call them ``local approximations of $T$ related to $L$''. The following is the key proposition detailing the refined approximation.

\begin{proposition}[Refined approximation]\label{p:refined}
Let $T$, $\Sigma$ and $\Sbf$ be as in Assumption \ref{a:refined}, and suppose that the assumptions of Lemma \ref{l:controls_Whitney} hold. There exists $1< \lambda = \lambda(m)\leq \frac{3}{2}$ and $\bar C = \bar C(Q,m,n,\bar{n},\delta^*)>0$, not depending on $\bar\delta$, $\tau$, and $\varepsilon$, such that the local approximations $u_{L,i}$ satisfy the following properties.
\begin{itemize}
    \item[(i)] For every fixed $i\in \{1,\dotsc,N\}$, $Q_{L,i}$ is the same {\emph positive} integer for every $L\in \mathcal{G}^o$. 
    \item[(ii)] $\sum_i Q_{L,i} = Q$ for every $L\in \mathcal{G}^o\cup \mathcal{G}^c$.
    \item[(iii)] For every $L\in \mathcal{G}^o\cup\mathcal{G}^c$ we have $\spt (T)\cap \lambda R(L) \subset \bigcup_i \mathbf{\Omega}_i (L)$ and
    \begin{align}
        2^{2\ell (L)} |q- \mathbf{p}_{\alpha_i} (q)|^2 &\leq \bar C (\mathbf{E} (L) + 2^{-2\ell (L)} \mathbf{A}^2) \qquad \forall q \in \spt(T)\cap\mathbf{\Omega}_i (L)\, ;\label{e:thin-slab}
    \end{align}
    \item[(iv)] For $L\in \mathcal{G}^o\cup \mathcal{G}^c$, if we set 
    \[
    T_{L,i}:= T\res \mathbf{\Omega}_i (L)\cap \big\{\dist (\cdot, \alpha_i) < \bar C 2^{-\ell (L)} (\mathbf{E} (L) + 2^{-2\ell (L)} \mathbf{A}^2)^{1/2}\big\}\, 
    \]
    (with $\bar{C}$ larger than the constant in the estimate \eqref{e:thin-slab}), 
    then, {for $K_i(L)=2^{-\ell(L)}K_i+y_L\subset \Omega_i(L)$ as above},
    \[
    T_{L,i}\res \mathbf{p}_{\alpha_i}^{-1} (K_i (L)) = \mathbf{G}_{u_{L,i}} \res \mathbf{p}_{\alpha_i}^{-1} (K_i (L))\, ,
    \]
${\rm gr}\, (u_{L,i})\subset \Sigma$, and the following estimates hold:
\begin{align}
2^{2\ell (L)}\|u_{L,i}\|_\infty^2 + 2^{m\ell (L)}\|Du_{L,i}\|_{L^2}^2 & \leq \bar C (\mathbf{E} (L) + 2^{-2\ell (L)} \mathbf{A}^2)\label{e:error-refined-1}\\
{\rm Lip}\, (u_{L,i}) &\leq C (\mathbf{E} (L) + 2^{-2\ell (L)} \mathbf{A}^2)^\gamma\label{e:error-refined-2}\\
|\Omega_i (L) \setminus K_i (L)| + \|T_{L,i}\| (\mathbf{\Omega}_i (L)\setminus \mathbf{p}_{\alpha_i}^{-1} (K_i (L)))&\leq \bar C 2^{-m \ell (L)} (\mathbf{E} (L) +2^{-2\ell (L)} \mathbf{A}^2)^{1+\gamma}\, .\label{e:error-refined-3}
\end{align}
\item[(v)] For every $L\in \mathcal{G}^o\cup \mathcal{G}^c$, $\Theta (T, \cdot) \leq \max_i Q_{L,i} + \frac{1}{2}$ on $R (L)$. In particular, $\Theta (T, \cdot) \leq Q-\frac{1}{2}$ on $R (L)$ if $L\in \mathcal{G}^o$. 
\end{itemize}
\end{proposition}
We will prove Proposition \ref{p:refined} in tandem with the next lemma. Consider a domain $U\subset \lambda R (L)$ which is invariant under rotations around $V$ and set $U_i := \alpha_i \cap U$. We wish to replace $T\res U$ with the union of the portions of the graphs of the functions $u_{L,i}$ lying over $U_i$. This will generate errors of two types. One type is due to the fact that $\spt (T)\cap U$ is not completely contained in the union of the graphs of the functions $u_{L,i}$: this will be taken care of by Proposition \ref{p:refined} above. A second type of error is due to the fact that, even though the estimate \eqref{e:thin-slab} holds, if we define  
\begin{equation}\label{e:slabs}
\tilde{U} := \bigcup_i \{q: \mathbf{p}_{\alpha_i} (q) \in U_i \quad \mbox{and} \quad |q-\mathbf{p}_{\alpha_i} (q)|
\leq \bar C 2^{-\ell (L)} (\mathbf{E} (L)+ 2^{-2\ell (L)}\mathbf{A}^2)^{1/2}\}\, 
\end{equation}
(where $\bar C$ is assumed to be the constant of estimate \eqref{e:thin-slab}), there still is a difference between $\spt (T)\cap \tilde{U}$ and $\spt (T) \cap U$. The purpose of the following lemma is to estimate errors of this second type:

\begin{lemma}\label{l:curved}
Under the assumptions of Lemma \ref{l:controls_Whitney}, consider $L\in \mathcal{G}^o \cup \mathcal{G}^c$, let $U\subset \lambda R(L)$ be a set invariant under rotations around $V$ whose cross-sections $U_i = U\cap \alpha_i$ are Lipschitz open sets or the closures of Lipschitz open sets, and define $\tilde{U}$ as in \eqref{e:slabs}. Then
\begin{align}
\|T\| (U\setminus \tilde{U}) + \|T\|(\tilde{U}\setminus U) &\leq \bar C \left(\mathbf{E} (L) + 2^{-2\ell(L)}\mathbf{A}^2\right) 2^{-\ell (L)} \mathcal{H}^{m-1} (\partial U_i)\, \nonumber\\
&\qquad \ \ \ \ + \bar C 2^{-m \ell (L)}\left(\mathbf{E} (L) + 2^{-2\ell(L)}\mathbf{A}^2\right)^{1+\gamma} \label{e:curved}
\end{align}
where the constant $\bar C$ depends on the parameters $m,n,\bar{n},Q, \bar\delta$, and the Lipschitz regularity of the boundary of the rescaled cross-section $2^{\ell (L)} U_i$. 
\end{lemma}

\begin{remark}\label{r:regularity-of-cross-section}
We will in fact only apply this Lemma to sets $U$ whose cross-sections $U_i = U\cap \alpha_i$ have a limited number of shapes, up to the rescaling factor $2^{-\ell (L)}$. In particular, in these cases the corresponding constant $\bar C$ in the estimate depends only upon $m,n,\bar{n},Q,$ and $\bar\delta$.
\end{remark}

\begin{proof}[Proof of Proposition \texorpdfstring{\ref{p:refined}}{refined} and of Lemma \texorpdfstring{\ref{l:curved}}{curved}]
    We will prove the two statements at the same time. We start by showing that the estimates of the statements (iii) and (iv) of Proposition \ref{p:refined} hold. Fix $L\in \mathcal{G}^o\cup \mathcal{G}^c$, let $k= k(L)$, and $\ell = \ell (L)$. In order to simplify our notation we write $T':= T_{y_L, 2^{{-}\ell}}$ and $\Sigma':= \Sigma_{y_L, 2^{{-}\ell}}$. The supremum norm of the second fundamental form of $\Sigma'$ is clearly $2^{-\ell} \mathbf{A}$. On the other hand, by scaling invariance,
\begin{equation}\label{e:scaling-invariant}
\int_{\Bbf_4\setminus B_{\rho_*} (V)} \dist^2(q, \mathbf{S}_k)\, d\|T'\| (q) = 
\mathbf{E} (L)\, .
\end{equation}
Let us assume that $I(k)$ consists of more than one element; in particular, by Lemma \ref{l:controls_Whitney}(v), we can apply Proposition \ref{p:Lipschitz-1} to $T'$, $\Sigma'$ and $\mathbf{S}_k$. The argument is entirely analogous when $I(k)$ consists of one element except instead we apply Proposition \ref{p:Lipschitz-2}, and so we will just focus on the former case. We consider the domains $\mathbf{\Omega}_i$ given by Proposition \ref{p:Lipschitz-1} and recall that from Lemma \ref{l:splitting-1} it follows that
\[
\spt (T')\cap \mathbf{B}_{3}\setminus B_{2\rho_*} (V) \subset \bigcup_i \mathbf{\Omega}_i\, ,
\]
In turn this rescales to the statement
\[
\spt (T) \cap \mathbf{B}_{3\cdot 2^{-\ell}} (y_L)\setminus B_{\rho_*2^{1-\ell}} (V) \subset \bigcup_i \mathbf{\Omega}_i (L)\, .
\]
For an appropriate choice of the constant $\lambda$ (i.e. sufficiently close to $1$), 
\[
\lambda R(L) \subset \mathbf{B}_{3\cdot 2^{-\ell}} (y_L)\setminus B_{\rho_*2^{1-\ell}} (V)
\]
and so the first claim of Proposition \ref{p:refined}(iii) follows. As for the height estimate in (iii), Proposition \ref{p:Lipschitz-1}(c) and \eqref{e:scaling-invariant} gives 
\[
|q-\mathbf{p}_{\alpha_i} (q)|^2 \leq C (\mathbf{E} (L) + 2^{-2\ell} \mathbf{A}^2) \qquad \forall q\in \mathbf{\Omega}_i\cap \spt (T')
\]
which scales to 
\[
|q-\mathbf{p}_{\alpha_i} (q)|^2 \leq C 2^{-2\ell} (\mathbf{E} (L) + 2^{-2\ell} \mathbf{A}^2) \qquad \forall q\in \mathbf{\Omega}_i (L)\cap \spt (T)\, ,
\]
i.e. \eqref{e:thin-slab}. This proves (iii).

The first two claims of Proposition \ref{p:refined}(iv) follow directly from the corresponding claims in Proposition \ref{p:Lipschitz-1}. As for the estimates, if we denote by $u_i$ and $K_i$ the multi-valued functions approximating {$T'_i := T' \res \mathbf{\Omega}_i\cap W_i$} and the corresponding 
``coincidence sets'' given by Proposition \ref{p:Lipschitz-1}, then, taking into account \eqref{e:scaling-invariant}, we have the estimates
\begin{align*}
\|u_{i}\|_\infty^2 + \|Du_{i}\|_{L^2}^2 & \leq C (\mathbf{E} (L) + 2^{-2\ell} \mathbf{A}^2)\\
{\rm Lip}\, (u_{i}) &\leq C (\mathbf{E} (L) + 2^{-2\ell} \mathbf{A}^2)^\gamma\\
|\Omega_i \setminus K_i| + \|T'_i\| (\mathbf{\Omega}_i \setminus \mathbf{p}_{\alpha_i}^{-1} (K_i))&\leq C (\mathbf{E} (L) +2^{-2\ell} \mathbf{A}^2)^{1+\gamma}\, .
\end{align*}
The estimates in Proposition \ref{p:refined}(iv) therefore follow from the obvious scaling relations
\begin{align*}
\|u_i\|_{\infty} &= 2^{\ell}\|u_{L,i}\|_\infty\\
\|Du_i\|_{\infty} &= \|Du_{L,i}\|_\infty\\
\|Du_i\|_{L^2} &= 2^{m\ell/2}\|Du_{L,i}\|_{L^2}\\
|\Omega_i\setminus K_i| &= 2^{m\ell} |\Omega_i (L)\setminus K_i (L)|\\
\|T'_i\| (\mathbf{p}^{-1}_{\alpha_i} (\Omega_i\setminus K_i)) &=
2^{m\ell} \|T_{L,i}\| (\mathbf{p}^{-1}_{\alpha_i} (\Omega_i (L) \setminus K_i (L)))
\end{align*}
which completes the proof of (iv).

Next, we will prove Proposition \ref{p:refined}(i). We claim that $Q_{L,i} = Q_{L',i}$ if $L, L'\in \mathcal{G}^o$ and $L'$ is the parent of $L$; notice that since $L_0\in \mathcal{G}^o$ by Lemma \ref{l:controls_Whitney}(i), this will hold for any $i\in I(0)$, and thus for every $i\in \{1,\dotsc,N\}$. First, notice that because of \eqref{e:thin-slab}, it is easy to see that $T_{L,i}$ and $T_{L',i}$ coincide over $\mathbf{\Omega}_i (L)\cap \mathbf{\Omega}_i (L')$. In particular, the two currents $\mathbf{G}_{u_{L,i}}$ and $\mathbf{G}_{u_{L',i}}$ coincide over $\mathbf{p}_{\alpha_i}^{-1} (K_i (L)\cap K_i (L'))$. Clearly by (iv), namely \eqref{e:error-refined-3}, $K_i(L)\cap K_i(L')$ has positive measure provided $\mathbf{E} (L)$, $\mathbf{E} (L')$, and $\mathbf{A}$ are smaller than a geometric constant, and all these conditions can be ensured by choosing $\tau$ and $\varepsilon$ smaller than a geometric constant (here it is crucial that the constant $C$ in the estimates of statement (iv) does not depend on $\tau$ and $\varepsilon$). But once we know that $K_i(L)\cap K_i (L')$ has positive measure and the current graphs coincide over this set we clearly conclude that $Q_{L,i}=Q_{L',i}$. 

Given this, since every parent of an element in $\mathcal{G}^o$ belongs to $\mathcal{G}^o$ (by definition of $\mathcal{G}^o$) we conclude from the above that $Q_{L,i} = Q_{L_0,i}$, which obviously implies the first claim of Proposition \ref{p:refined}(i). The fact that all of them are positive follows from Proposition \ref{p:Lipschitz-1}(f) once we assume $\varepsilon$ is sufficiently small, because it will force $Q_{L_0,i}\geq 1$ for every $i$. Thus (i) is proved.

Next, because of (i), the conclusion of point (ii) holds for $L\in \mathcal{G}^o$ once we show $\sum_i Q_{L_0,i}= Q$. For the latter we can use Lemma \ref{l:matching-Q} with assumption (a), since $y_{L_0}=0$ and $\{\Theta (T,\cdot ) \geq Q\} \cap \Bbf_\varepsilon (0) \neq \emptyset$. 

In order to prove (ii) when $L\in \mathcal{G}^c$ we are again going to argue inductively, showing that:
\begin{itemize}
    \item[(A)] If $\sum_i Q_{L',i}=Q$ for $L'\in \mathcal{G}^c\cup \mathcal{G}^o$, then $\sum_i Q_{L,i} =Q$ for every child $L\in \mathcal{G}^c\cup \mathcal{G}^o$ of $L^\prime$.
\end{itemize}
In order to show (A) we will in fact use the result of Lemma \ref{l:curved}, which we shall prove now.

Fix $L\in \mathcal{G}^c\cup \mathcal{G}^o$ and a set $U\subset \lambda R(L)$ which is invariant under rotations around $V$. Firstly, from what has been proved so far of Proposition \ref{p:refined}, namely (iii) and (iv), we can conclude that $T\res (U\cup\tilde{U}) = \sum_i T_{L,i} \res (U\cup\tilde{U})$ and $\spt (T_{L,i})\cap \spt (T_{L,j})=\emptyset$ for all $i<j$ (which follows from (iv) provided $\tau$ and $\varepsilon$ are sufficiently small, as then the neighbourhoods of the $\alpha_i$ in (iv) are all disjoint, and from the choice of $\lambda$, which ensures that $U\cap \tilde{U}\cap \spt (T)$ is contained in the union of $\boldsymbol{\Omega}_i (L)$). In particular, we have
\[
\|T\| (U\Delta\tilde{U})
=\sum_i \|T_{L,i}\| (U\Delta\tilde{U})
\]
where $U\Delta\tilde{U}:= (U\setminus \tilde{U})\cup (\tilde{U}\setminus U)$ is the symmetric difference of $U$ and $\tilde{U}$. Thus we can again use the conclusion (iv) in Proposition \ref{p:refined} to estimate further
\begin{align*}
\|T\| ({U\Delta\tilde{U}})
&\leq \sum_i \|\mathbf{G}_{u_{L,i}}\| ({U\Delta\tilde{U}})
 + C 2^{-m\ell(L)} (\mathbf{E} (L) + 2^{{-}2 \ell(L)} \mathbf{A}^2)^{1+\gamma}\, .
\end{align*}
Consider now the set 
\[
\Delta_i := \mathbf{p}_{\alpha_i} ({\rm gr}\, (u_{L,i})\cap ({U\Delta\tilde{U}}))
\]
Since by choosing $\tau$ and $\varepsilon$ small enough we can assume $\Lip (u_{L,i})\leq 1$ (by {\eqref{e:error-refined-2}}), it follows immediately from this this Lipschitz regularity that
\begin{align*}
\|\mathbf{G}_{u_{L,i}}\|{(U\Delta\tilde{U})} \leq C Q_{L,i} |\Delta_i|
\end{align*}
for some geometric constant $C$. Since $\sum_i Q_{L,i} \leq Q$ (note in particular that we do not need the equality $\sum_i Q_{L,i} =Q$ in this argument), in order to reach the estimate of Lemma \ref{l:curved} it suffices to show that 
\begin{align}
|\Delta_i|&\leq C (\mathbf{E} (L) + 2^{-2 \ell (L)}\mathbf{A}^2) 2^{-\ell (L)} \mathcal{H}^{m-1} (\partial U_i)\, ,\label{e:curved-2}
\end{align}
where $C$ has the dependencies in the statement of the lemma (recall that $U_i = U\cap \alpha_i$). We will indeed show that
\[
\dist (q, \partial U_i) \leq C 2^{-\ell (L)} (\mathbf{E} (L) + {2^{-2\ell(L)}}\mathbf{A}^2) \qquad \forall q\in \Delta_i\, . 
\]
In particular, using the Lipschitz regularity of $\partial U_i$, \eqref{e:curved-2} follows immediately {(by taking a suitable cover of $\partial U_i$ and taking the tubular neighbourhood of each set in this cover with radius $2$ times the above distance bound to cover $\Delta_i$).}
In order to show \eqref{e:curved-2}, fix a point $x\in \Delta_i$ and observe that by definition there must be a point $p\in {\rm gr}\, (u_{L,i})\cap ({U\Delta\tilde{U}})$ such that $x= \mathbf{p}_{\alpha_i} (p)$. Let $v= \mathbf{p}_V (p) = \mathbf{p}_V (x)$ and denote by $\sigma$ the half-line in $\alpha_i$ which originates at $v$ and contains $x$. On this half-line we denote by $y$ the point such that 
\begin{align}\label{e:distance-y-1}
    \dist (y, V) = |y-v| = \dist (p, V)\, ,
\end{align}
{and moreover as all points on this half-line project to $v$, we have }
\begin{align}\label{e:distance-y-2}
\mathbf{p}_V(y) = \mathbf{p}_V(p)\, .
\end{align}
Observe that $p\in U$ if and only if $y\in U$ {(using that $U$ is invariant under rotations about $V$ and \eqref{e:distance-y-1}, \eqref{e:distance-y-2})}, which happens if and only if $y\in U_i$. On the other hand $p\in \tilde{U}$ if and only if $x=\mathbf{p}_{\alpha_i} (p)\in U_i$. Therefore one of the following two alternatives hold:
\begin{itemize}
    \item $p\in U\setminus \tilde{U}$, and hence $y\in U_i$ but $x\not \in U_i$
    \item $p \in \tilde{U}\setminus U$, and hence $x\in U_i$ but $y \not\in U_i$.
\end{itemize}
In both cases the segment $\sigma$ joining $x$ and $y$ must contain a point of $\partial U_i$ and thus $\dist (x, \partial U_i) \leq |x-y|$. We therefore wish to estimate the latter.

To that end consider the triangle with vertices $v, p$, and $x$ and let $\theta$ be the angle at $v$. We then have
\begin{align}
|x-y| &= |p-v| (1- \cos \theta) \leq C 2^{-\ell (L)} (1-\cos \theta)\label{e:|x-y|}\\
\sin \theta &= \frac{|p-x|}{|p-v|} \leq C (\mathbf{E} (L) + 2^{-2\ell (L)}\mathbf{A}^2 )^{1/2}\, .\label{e:sin-theta}
\end{align}
{Here, the upper bound on $|p-v|$ comes from \eqref{e:error-refined-1} and the upper bound on the distance of $U_i$ from $V$, while the lower bound on $|p-v|$ also comes from the lower bound on the distance of $U_i$ to $V$, and the upper bound on $|p-x|$ comes from \eqref{e:error-refined-1} (this is using that $p\in {\rm gr}(u_{L,i})$, but also holds more generally when $p\in \spt(T_{L,i})$). In particular as \eqref{e:sin-theta} tells us that $\theta$ is small (e.g. $\theta<\pi/3$ suffices), we get $\sin(\theta)\geq\theta/2$; as $1-\cos(\theta) \leq \theta^2/2$ is always true, combining these with \eqref{e:|x-y|} and \eqref{e:sin-theta} we establish the desired bound on $|x-y|$, hence completing the proof of \eqref{e:curved-2} and so also the proof of Lemma \ref{l:curved}.} 

\medskip

Having proved Lemma \ref{l:curved} we now come to the proof of (A). Consider $L, L'\in \mathcal{G}^c \cup \mathcal{G}^o$, with $L$ a child of $L'$. Consider the set $U:= \lambda R (L) \cap \lambda R (L')$; clearly $U$ is invariant under rotations around $V$. Under the assumption that $\sum_i Q_{L',i} = Q$, we can use Lemma \ref{l:curved} to prove 
\begin{equation}\label{e:enough-mass}
\|T\| (U) \geq (Q-{\textstyle{\frac{1}{2}}}) \mathcal{H}^m (U_1)\, .
\end{equation}
{for $U_1 := U\cap\alpha_1$ as before. } Indeed, introduce the set $\tilde{U}$ {for this choice of $U$} as in Lemma \ref{l:curved} and note that
\[
\|T\| (\tilde{U}) = \sum_i \|T_{L',i}\| ({\tilde{U}})\geq \sum_i Q_{L',i} \mathcal{H}^m (U_i) = Q \mathcal{H}^m (U_1)\, 
\]
(observe that the set $\tilde{U}$ is formed by cylindrical domains with cross-sections $U_i$, and in particular $(\mathbf{p}_{\alpha_i})_\sharp (T_{L',i}) \res \tilde{U}) = 
Q_{L',i} \llbracket U_i \rrbracket$).

We can now use Lemma \ref{l:curved} to estimate 
\[
\|T\| (U) \geq \|T\| (\tilde{U}) - \|T\| (\tilde{U}\setminus U) \geq Q \mathcal{H}^m (U_1) - C 2^{{-}\ell (L)} (\mathbf{E} (L) + 2^{-{2}\ell (L)} \mathbf{A}^2) \mathcal{H}^{m-1} (\partial U_1)\, .
\]
Note that in principle the constant $C$ appearing in the estimate depends on the region $U$ (which in turn depends on $L$ and $L'$), however the latter is determined by the cross-section which, after rescaling by $2^{\ell (L)}$ and translating, is the intersection of two shapes ranging in a finite number of possibilities (the number of which depends on $\lambda$). In particular, by Remark \ref{r:regularity-of-cross-section}, we can assume that the constant $C$ depends only on the constants $m$, $n$, $\bar n$, $Q$, $\bar \delta$ and $\lambda$ {(note however that we have already fixed $\lambda = \lambda(m)$)}. On the other hand we also get $2^{-\ell (L)} \mathcal{H}^{m-1} (\partial U_1) \leq C \mathcal{H}^m (U_1)$, with a constant $C$ depending on $m$ and $\lambda$. We thus conclude 
\[
\|T\| (U) \geq (Q- C (\mathbf{E} (L) + 2^{-2\ell (L)} \mathbf{A}^2)) \mathcal{H}^m (U_1) \geq (Q- C (\tau^2 + \varepsilon^2)) \mathcal{H}^m (U_1)\, .
\]
Since the constant $C$ is independent of both $\tau$ and $\varepsilon$, an appropriate smallness condition on these two parameters guarantees the validity of \eqref{e:enough-mass}.

If we now consider the current $T_{y_L, 2^{-\ell (L)}}$, the corresponding rescaled set $\Omega = 2^{\ell (L)} (U-y_L)$ satisfies requirement (b) of Lemma \ref{l:matching-Q}, which implies that $\sum_i Q_{L,i} = Q$, if $\tau$ is sufficiently small. While it is true that Lemma \ref{l:matching-Q} imposes a smallness condition on $\tau$ which depends on the set $\Omega$, it is easy to see that the latter varies among a fixed number of sets (since $L$ is a child of $L^\prime$), which depend on the relative position of $L$ compared to $L'$ because once $\lambda$ is fixed, $\lambda R (L) \cap \lambda R (L')$ overlap for a fixed number of cubes $L$, $L'$ always. In particular, there is a choice of smallness condition on $\tau$ which ensures the applicability of the lemma for every pair of cubes $L$ and $L'$ as in (A). Therefore to finish the proof, note that as (ii) holds for $\mathcal{G}^o$, it follows by induction from (A) that (ii) holds for every $L\in \mathcal{G}^c\cup\mathcal{G}^o$, as every cube $L$ in this set is a descendent of a cube in $\mathcal{G}^o$, while their ancestors are all in $\mathcal{G}^c\cup \mathcal{G}^o$. So we are done with this part.

We finally prove (v). Fix a point $p\in R(L)\cap \spt (T)$ and note that it must belong to $\spt (T_{L,i})$ for some $i$. For $\rho= (\lambda-1) \frac{2^{-\ell (L)-1}}{\sqrt{m-2}}$ we estimate
\begin{align*}
\|T_i\| (\Bbf_\rho (p)) & \leq \|T_i\| (\mathbf{C}_\rho (p, \pi_i))
\leq Q_{L,i} \rho^m + C (\mathbf{E} (L) + 2^{-2\ell (L)} \Abf^2)^{1+\gamma} 2^{-m \ell (L)}\\
& \leq (Q_{L,i} + C (\mathbf{E} (L) + 2^{-2\ell (L)} \Abf^2)^{1+\gamma}) \rho^m\, .
\end{align*}
Then if $\eps$ and $\tau$ are small enough, we conclude the claim from \eqref{e:below-tau} and the monotonicity formula.
\end{proof}

\subsection{Coherent approximation on the outer region and first blow-up} 

For technical reasons, it is useful to have a single multi-valued approximation defined over the union of all $L_i$ for $L$ varying among the elements of $\mathcal{G}^o$. Proposition \ref{p:coherent} below gives a precise statement of this, but we first introduce some notation.

\begin{definition}\label{d:neighbors}
Let $T$, $\Sigma$, and $\mathbf{S}$ be as in Proposition \ref{p:refined}. {Fix $L\in \mathcal{G}^o$.} We denote by $\mathscr{N} (L)$ the set of $L'\in \mathcal{G}^o$ such that $R(L)\cap R(L')\neq \emptyset$, and let 
\[
\bar{\mathbf{E}} (L) := \max \{\mathbf{E} (L'): L'\in \mathscr{N} (L)\}\, .
\]
\end{definition}

\begin{proposition}[Coherent outer approximation]\label{p:coherent}
Let $T$ and $\mathbf{S}$ be as in Proposition \ref{p:refined}. For every $i\in \{1, \ldots, N\}$ we define 
\[
R^o_i := \bigcup_{L\in \mathcal{G}^o} L_i \equiv \alpha_i\cap\bigcup_{L\in \mathcal{G}^o} R (L)\, 
\]
and let $Q_i := Q_{L_0, i}$. Then, there are Lipschitz multi-valued maps $u_i : R^o_i \to \mathcal{A}_{Q_i} (\alpha_i^\perp)$ and closed subsets $\bar{K}_i (L)\subset L_i$ satisfying the following properties.
\begin{itemize}
\item[(i)] ${\rm gr}\, (u_i)\subset \Sigma$ and $T_{L,i} \res \mathbf{p}_{\alpha_i}^{-1} (\bar{K}_i(L))= \mathbf{G}_{u_i}\res \mathbf{p}_{\alpha_i}^{-1} (\bar{K}_i (L))$ for every $L\in \mathcal{G}^o$, where $T_{L,i}$ are defined as in Proposition \ref{p:refined};
\item[(ii)] The following estimates hold
\begin{align}
2^{2\ell (L)}\|u_i\|_{L^\infty (L_i)}^2 + 2^{m\ell (L)}\|Du_i\|_{L^2 (L_i)}^2 & \leq C (\bar{\mathbf{E}} (L) + 2^{-2\ell (L)} \mathbf{A}^2)\label{e:coherent-1}\\
\|Du_i\|_{L^\infty (L_i)} &\leq C (\bar{\mathbf{E}} (L) + 2^{-2\ell (L)} \mathbf{A}^2)^\gamma\label{e:coherent-2}\\
|L_{i} \setminus \bar{K}_i (L)| + \|T_{L,i}\| (\mathbf{p}_{\alpha_i}^{-1} (L_{i}\setminus \bar{K}_i (L)))&\leq C 2^{-m \ell (L)} (\bar{\mathbf{E}} (L) +2^{-2\ell (L)} \mathbf{A}^2)^{1+\gamma}\, ,\label{e:coherent-3}
\end{align}
where the constant $C$ depends only upon $Q,m,n,\bar n$, and $\bar\delta$.
\end{itemize}
\end{proposition}

\begin{proof}
The ideas of the proof are borrowed from \cite{DLS_MAMS}*{Section 1.2.2} and \cite{DLS16centermfld}*{Section 6.2}. On the one hand there is a slight complication compared to the arguments borrowed from \cite{DLS_MAMS}*{Section 1.2.2} due to the fact that the regions  $R (L)$ are not cubes; on other hand compared to the arguments borrowed from \cite{DLS16centermfld}*{Section 6.2} our situation is considerably simpler. 

First observe that, since $L_i = R(L)\cap \alpha_i$, the set $\mathscr{N} (L)$ is equivalently described as those cubes $L'\in \mathcal{G}^o$ such that $L_i\cap L'_i\neq \emptyset$ for some $i$. In fact, given the invariance of $R (L)$ under rotations around $V$, if this is true for some $i$ then it is true for all $i$, and so we can fix an arbitrary $i$. Notice that
\begin{itemize}
    \item[(a)] the cardinality of $\mathscr{N} (L)$ is bounded by a dimensional constant;
    \item[(b)] $|\ell (L')-\ell (L)|\leq 1$ for every $L'\in \mathscr{N} (L)$.
\end{itemize}
Indeed, (a) follows by the construction of the collection of cubes $\Gcal$ and their associated sets $R(L)$, while (b) is a consequence of Lemma \ref{l:whitney}(i). 
For each cube $L\in \Gcal^o$, consider the approximations $u_{L,i}$ and the coincidence sets $K_i (L)$ given by Proposition \ref{p:refined} and define 
\[
\bar{K}_i (L) := \bigcap_{L' \in \mathscr{N} (L)} K_i (L')\, .
\]
Observe that for every $L'\in \mathscr{N} (L)$ we have $u_{L,i}= u_{L',i}$ on $\bar{K}_i (L)$. Note that
\begin{align*}
    |L_i\setminus \bar{K}_i (L)|  & \leq \sum_{L^\prime\in \mathscr{N}(L)}|\Omega_i(L^\prime)\setminus K_i(L^\prime)|
\end{align*}
and
\begin{align*}
    \|T_{L,i}\|(\mathbf{p}_{\alpha_i}^{-1}(L_i\setminus &\bar{K}_i(L)))\\
    \leq & \|T_{L,i}\|(\mathbf{p}_{\alpha_i}^{-1}(L_i\setminus K_i(L))) + \sum_{L^\prime\in \mathscr{N}(L)\setminus\{L\}}\|T_{L,i}\|(\mathbf{p}_{\alpha_i}^{-1}({L_i}\setminus K_i(L^\prime)))\\
    \leq & \|T_{L,i}\|(\mathbf{p}_{\alpha_i}^{-1}(L_i\setminus K_i(L))) + \sum_{L^\prime\in \mathscr{N}(L)\setminus\{L\}}\| T_{L',i}\|(\mathbf{\Omega}_i(L^\prime)\setminus \mathbf{p}_{\alpha_i}^{-1}(K_i(L^\prime))) \, ,
\end{align*}
and thus \eqref{e:coherent-3} follows from Proposition \ref{p:refined}(iv) and (a) and (b) above. 

{Note that}{ we may define $u_i := u_{L,i}$ on $\bar{K}_i(L)$, {which is clearly} well-defined.} To complete the proof we need to extend $u_i$ {from $\bigcup_{L\in\mathcal{G}^o}\bar{K}_i(L)$} to a Lipschitz map on $R_i^o:= \bigcup_{L\in \mathcal{G}^o} L_i$ which satisfies ${\rm gr}\, (u_i)\subset \Sigma$ and the estimates \eqref{e:coherent-1} and \eqref{e:coherent-2}. Recall that we have the estimates
\begin{align}
\|u_i\|_{L^\infty ({\bar{K}_{i}(L)})} &\leq C 2^{-\ell (L)} (\mathbf{E} (L) + 2^{-2\ell (L)} \mathbf{A}^2)^{1/2}\label{e:extended-1}\\
\textrm{Lip} (u_i|_{\bar{K}_{i}(L)}) &\leq C  (\mathbf{E} (L) + 2^{-2\ell (L)} \mathbf{A}^2)^\gamma\, ,\label{e:extended-2}
\end{align}
{by Proposition \ref{p:refined}(iv),} just because on the set ${\bar{K}_{i}(L)}$ the map $u_i$ coincides with $u_{L,i}$. 

Observe moreover that, if we consider the larger domains $\tilde{K}_i (L) := \bigcup_{L'\in \mathscr{N} (L)} \bar{K}_i (L)$, we still have estimates analogous to \eqref{e:extended-1}, \eqref{e:extended-2}, with $\bar{\Ebf} (L)$ replacing $\Ebf (L)$, namely
\begin{align}
\|u_i\|_{L^\infty ({\tilde{K}_{i}(L)})} &\leq C 2^{-\ell (L)} (\bar{\mathbf{E}} (L) + 2^{-2\ell (L)} \mathbf{A}^2)^{1/2}\label{e:extended-11}\\
\textrm{Lip} (u_i|_{\tilde{K}_{i}(L)}) &\leq C  (\bar{\mathbf{E}} (L) + 2^{-2\ell (L)} \mathbf{A}^2)^\gamma\, ,\label{e:extended-22}
\end{align}
Our aim is to show that we can find an extension $u_i$ to $\bigcup_{L\in \mathcal{G}^o} L_i$ and use \eqref{e:extended-11}, \eqref{e:extended-22} to show that {for each $L\in\mathcal{G}^o$} this extension satisfies
\begin{align}
\|u_i\|_{L^\infty (L_i)} &\leq C 2^{-\ell (L)} (\bar{\mathbf{E}} (L) + 2^{-2\ell (L)} \mathbf{A}^2)^{1/2}\label{e:extended-111}\\
\textrm{Lip} (u_i|_{L_i}) &\leq C  (\bar{\mathbf{E}} (L) + 2^{-2\ell (L)} \mathbf{A}^2)^\gamma\, ,\label{e:extended-222}
\end{align}
The remaining claim, namely the $L^2$ bound on $Du_i$ over $L_i$ claimed in point (ii) of the proposition, is then an obvious consequence of \eqref{e:extended-2}, the bound on $|L_i\setminus {\bar{K}_{i}(L)}|$, and the fact that 
\[
\|Du_i\|_{L^2}^2 ({\bar{K}_{i}(L)}) \leq C 2^{-m \ell (L)} (\mathbf{E} (L) + 2^{-2\ell (L)} \mathbf{A}^2)\, ,
\]
the latter being again a consequence of $u_i|_{\bar{K}_i (L)} = u_{L,i}$.

In order to accomplish the latter task we first observe that we can ignore the requirement that ${\rm gr}\, (u_i)\subset \Sigma$. Indeed, fix the $\bar{n}$-dimensional subspace $\pi = \alpha_i^\perp \cap T_0 \Sigma$ which is the orthogonal complement of $\alpha_i$ in $T_0 \Sigma$ and let $\Psi: \Bbf_7 \cap T_0 \Sigma \to T_0 \Sigma^\perp$ be the map whose graph describes $\Sigma$. Recall that $\|D^2 \Psi\|_{C^0} \leq C \Abf$ and, since $D\Psi (0)=0$, we conclude $\|D\Psi\|_{C^0} \leq C \Abf$ as well. It thus suffices to find an extension of the $\pi$-component {$u_i^\pi$} of the map $u_i$ with the desired estimate and compose it with $\Psi$ {in the remaining components of $\alpha_i^\perp$} to find the {desired extension of $u_i$} {(the formula for the latter map would then be $x\mapsto (u_i^\pi(x),\Psi(x,u_i^\pi(x))) \in \pi \times T_0 \Sigma^\perp$).} 

Once we are allowed to ignore the {above} issue, we consider a cellular decomposition of the $L_i$'s into $0$-cells (the $0$-skeleton), $1$-cells attached to the $0$-cells (the $1$-skeleton), $2$-cells, and so on, with the final $m$-dimensional cells being the interiors of the $L_i$'s. This can be done canonically across all $L$, note indeed that each $L_i$ is the product, in $V\times (V^\perp \cap \alpha_i)$, of the cube $L\subset V$ with the annulus $\{y\in V^\perp\cap \alpha_i : 2^{-\ell (L)-1}\leq |y|\leq 2^{-\ell (L)}\}$.

We denote by $\mathcal{S}_i$ the collection of $i$-cells {in the above}.  We also slightly fatten each $0$-cell $p$ to open neighborhoods $U^p$ so that the separation between any $U^p$ and $U^q$ with $p,q\in L_i$ is at least $c 2^{-\ell (L)}$ {for some dimensional constant $c>0$}. {Now}, for every $1$-cell {$\sigma$} with endpoints $p,q\in \mathcal{S}_0$, we slightly fatten $\sigma \setminus (U^p\cup U^q)$ to an open set $U^\sigma$ and again we take care that the separation between two fattenings $U^\sigma$ and $U^\tau$ for distinct 1-cells $\sigma$ and $\tau$ contained in the same $L_i$ is at least $c 2^{-\ell (L)}$. We proceed in this way over all skeleta. Figure \ref{f:whitney-3} gives a graphical illustration of this on some specific $0$, $1$, and $2$-cells.

\begin{figure}[htbp]
\begin{center}
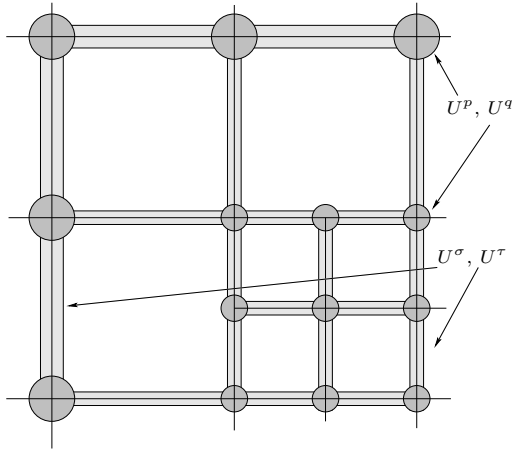
\end{center}
\caption{\Small An illustration of the fattening of some $0$-cells, $1$-cells and $2$-cells, used in the extension algorithm to find a coherent approximation.}\label{f:whitney-3}
\end{figure}
We then find a Lipschitz extension of the maps $u_i$ to the points $0$-skeleton in the following fashion:

For each point $P$ in the $0$-skeleton, we let $N (P)$ be the union of $L'_i$ for those cubes $L'\in \mathcal{G}^o$ which contain $P$. {The extension is then done using \cite{DLS_MAMS}*{Theorem 1.7}, which ensures that}, after setting $K := \bigcup_{L\in \mathcal{G}^o} \bar{K}_i (L)$ 
\begin{align}
\textrm{Lip} (u_i|_{(K\cap N(P))\cup U^p}) &\leq C \textrm{Lip} (u_i|_{K\cap N(P)})\label{e:intermediate1}\\ 
\|u_i\|_{L^\infty ((K\cap N(P))\cup U^p)} &\leq C \|u_i\|_{L^\infty (K\cap N(P))}\, .\label{e:intermediate2}  
\end{align}

To extend to the $1$-cells fix any $\sigma \in \mathcal{S}_1$ contained in some $L_i$ with endpoints $p$ and $q$, and let $N (\sigma)$ be the union of all the $L'_i$ which intersect $\sigma$. Observe that, because $U^p$ and $U^q$ are far enough apart from one another, using \eqref{e:intermediate1} and \eqref{e:intermediate2} we get the estimates
\begin{align*}
\textrm{Lip} (u_i|_{(K \cap N (\sigma))\cup (U^p\cup U^q)}) & \leq C \textrm{Lip} (u_i|_{K\cap N (\sigma)}) + C 2^{-\ell (L)} \|u_i\|_{L^\infty (K\cap N(\sigma))}\\
\|u_i\|_{L^\infty ((K \cap N (\sigma)) \cup (U^p\cup U^q))} &\leq C \|u_i\|_{L^\infty(K\cap N (\sigma))}\, .
\end{align*} 
We then proceed and define $u_i$ on each $U^\sigma$ {\em separately} so that {(again, using \cite{DLS_MAMS}*{Theorem 1.7})}
\begin{align*}
\textrm{Lip} (u_i|_{(K \cap N (\sigma))\cup (U^p\cup U^q \cup U^\sigma)}) & \leq C \textrm{Lip} (u_i|_{(K\cap N (\sigma))\cup (U^p\cup U^q)})\\
\|u_i\|_{L^\infty ((K \cap N (\sigma)) \cup (U^p\cup U^q \cup U^\sigma))} &\leq C \|u_i\|_{L^\infty((K\cap N (\sigma))\cup (U^q\cup U^p))}\, .
\end{align*}
We proceed inductively in this fashion and after going over all the skeleta we finally arrive at an extension which, defining $N (L_i)$ as the union of the $L'_i$ over $L'\in \mathscr{N} (L)$, satisfies
\begin{align*}
\textrm{Lip} (u|_{(K\cap N (L_i)) \cup L_i)}) &\leq C \textrm{Lip} (u|_{K\cap N (L_i)}) + C 2^{-\ell (L)} \|u_i\|_{L^\infty (K\cap N (L_i))}\, .\\
\|u_i\|_{L^\infty ((K\cap N (L_i)) \cup L_i)} &\leq C \|u_i\|_{L^\infty (K\cap N (L_i))}\, .
\end{align*}
Given that $K\cap N (L_i) = \tilde{K}_i (L)$, we conclude the desired bounds from \eqref{e:extended-11} and \eqref{e:extended-22}.
\end{proof}

\subsection{Blow-up}
We have thus constructed a collection of $N$ multi-valued approximations on the outer region. From the estimates of the previous proposition and a suitable covering argument it is easy to see that, away from the spine $V$, their Dirichlet energies are controlled by the conical excess $\hat{\mathbf{E}} (T, \mathbf{S}, \Bbf_4)$. Our next main goal is to show that, after we normalize them by $(\hat{\mathbf{E}}(T,\Sbf,\Bbf_4))^{1/2}$, they are close to a Dir-minimizing function. The control of the Dirichlet energy is still too crude for our final purpose, as it degenerates when we get closer to the spine: in order to gain a uniform control outside the spine we will need {to use Simon's estimates}; Section \ref{p:estimates} will be devoted to this task.  However, this ``first blow-up'' will be useful for two reasons: it is sufficient for the purpose of Section \ref{p:balancing}, where we ``balance'' the best approximating cone {(recall the discussion at the beginning of Section \ref{p:approx})}, and it will ultimately be used in the final blow-up argument to prove Dir-minimality of the blow-up limit away from the spine. In the final blow-up argument, the better control on the Dirichlet energy will then be used to prove Dir-minimality across the spine, using the fact that {$(m-2)$-dimensional subspaces} have vanishing $W^{1,2}$-capacity.

\begin{proposition}[First blow-up\label{p:Dir-minimizing}]\label{p:first-blow-up}
Let $T$ and $\mathbf{S}$ be as in Proposition \ref{p:refined} and assume the parameters $\delta^*,\tau$, and $\bar\delta$ are fixed. Then, for every $\sigma, \varsigma>0$ there are constants $C = C(m,n,Q,\delta^*, \tau, \bar\delta)>0$ and $\varepsilon = \varepsilon(m,n,Q,\delta^*, \tau, \bar\delta, \sigma, \varsigma)>0$ such that the following holds:
\begin{itemize}
    \item[(i)] $R\setminus B_\sigma (V)$ is contained in the outer region $R^o$ {(recall $R$ is as in \eqref{e:region-R})};
    \item[(ii)] If $u_i$ are the maps of Proposition \ref{p:coherent} and $R_i := (R\setminus B_\sigma (V))\cap \alpha_i$ then
    \begin{equation}\label{e:first-Dir-control}
    \int_{R_i} |Du_i|^2 \leq C \sigma^{-2} \hat{\mathbf{E}} (T, \mathbf{S}, \Bbf_4) + C \mathbf{A}^2\, .
    \end{equation}
    \item[(iii)] If additionally $\mathbf{A}^2 \leq \varepsilon^2 \hat{\mathbf{E}}(T, \mathbf{S}, \Bbf_4)$ and we set $v_i := \hat{\mathbf{E}} (T, \mathbf{S}, \Bbf_4)^{-1/2} u_i$, then there is a map $w_i: R_i \to \mathcal{A}_{Q_i} (\alpha_i^\perp)$ which is Dir-minimizing and such that 
    \begin{equation}\label{e:first-blowup}
    d_{W^{1,2}} (v_i, w_i) \leq \varsigma\, ,
    \end{equation}
    where $d_{W^{1,2}}$ is the $W^{1,2}$ distance defined in \cite{DLS_MAMS}. 
\end{itemize}
\end{proposition}

Note that, crucially, while the parameter $\varepsilon$ must be chosen suitably small depending on $\varsigma$, the constant $C$ in \eqref{e:first-Dir-control} is instead independent of it. As already mentioned, the control \eqref{e:first-Dir-control} is indeed sub-optimal and we will later be able to remove the factor $\sigma^{-2}$. 

\begin{remark}\label{r:blow-up}
We wish to spent a few words on a procedure which will be used often in the rest of the paper. Assume we have a sequence of currents $T_k$, manifolds $\Sigma_k$, and cones $\mathbf{S}_k$ satisfying the assumptions of Proposition \ref{p:first-blow-up} for some fixed choice of the parameters $\delta^*$, $\tau$, and $\bar \delta$. We assume that the cones $\mathbf{S}_k$ are converging {locally} in the sense of Hausdorff {distance} to some limiting $\mathbf{S}_\infty$, and fix some sequence of $m$-dimensional planes $\alpha^k_i\subset \mathbf{S}_k$ converging to a plane $\alpha_i^\infty\subset \mathbf{S}_\infty$. Assume moreover that $\hat{\Ebf}_k = \hat{\Ebf} (T_k, \mathbf{S}_k, \Bbf_1)$ and $\hat{\Ebf} (T_k, \mathbf{S}_k, \Bbf_1)^{-1} \mathbf{A}_k^2$ are both converging to zero. Let $u^k_i$ be the coherent outer approximations over the domains $R^o_{i,k}$ which consist of the outer regions $R^o_k$ {for each $T_k$} intersected with the planes $\alpha^k_i$. We would like to use Proposition \ref{p:first-blow-up} to extract a Dir-minimizing limit of suitable normalizations of the maps $u^k_i$, namely $v^k_i := E_k^{-1/2} u^k_i$ for some choice of normalization constants $E_k$ satisfying $\hat{\Ebf}_k \leq E_k$. One technical point is that the maps $v^k_i$ are not defined on the same plane. In order to deal with issue apply a rotation and map $\alpha_i^k$ onto $\alpha_i^\infty$. In fact, even though this is not needed, it is convenient to choose the canonical rotations $R_k = R(\alpha^k_i, \alpha_i^\infty)$ of Lemma \ref{l:rotations}. We can then extract, up to extracting a subsequence, a limit $v_i$ of $v^k_i \circ R_k^{-1}$; note that the latter maps are defined over (subdomains of) the same plane. Observe that the rotations $R_k$ do depend on $i$ and thus we do not have a single canonical rotation which works for every $i$. However, we can also see from Lemma \ref{l:rotations} that, if we consider the graphs of the maps $v^k_i$ as subsets of $\mathbb R^{m+n}$, the latter are indeed converging to the graph of $v_i$. Our limiting object is, in that sense, canonical.

In the sequel, when we are referring to ``the blow-up $v_i$ of the maps $v^k_i$'' we will assume that we have followed the above algorithm.
\end{remark}

\begin{proof}
    First of all we let $\sigma$ be given and fixed. We then select $\ell$ so that $\frac{\sigma}{2} \leq 2^{-\ell}\leq \sigma$ and define 
\[
\mathcal{G}_{\leq \ell}:= \bigcup_{j=0}^\ell \mathcal{G}_j\, .
\] 
We can then appeal to Lemma \ref{l:controls_Whitney}(i) to get that, if $\varepsilon$ is small enough, then $\mathcal{G}_{\ell+1}\subset \mathcal{G}^o$. But by definition of outer cubes the latter implies in fact ${\mathcal{G}_{\leq \ell+1}}\subset \mathcal{G}^o$, which implies the first claim of the proposition.

{Let $L\in \mathcal{G}_{\leq \ell}$, and observe that (b) in the proof of Proposition \ref{p:coherent} tells us that for any $L'\in\Nscr(L)$, we have $L'\in \mathcal{G}_{\leq \ell+1}$.} Appealing to Proposition \ref{p:coherent}{(ii) (namely, \eqref{e:coherent-1})} we have
\begin{align*}
\int_{L_i} |Du_i|^2 &\leq C 2^{-m\ell (L)} (\bar{\mathbf{E}} (L) + 2^{-2\ell (L)} \mathbf{A}^2)\\
&\leq C 2^{-(m+2)\ell (L)}{\Abf^2} + C \sum_{L'\in \mathscr{N} (L)} 2^{-m\ell (L')} \mathbf{E} (L', 0)\, ,
\end{align*}
where $\mathscr{N} (L)$ and $\bar{\mathbf{E}} (L)$ are as in Definition \ref{d:neighbors}. 
Since the cardinality of $\mathscr{N} (L)$ is bounded by a geometric constant $C = C(m,n)$, we immediately conclude
\begin{align*}
\int_{R_i} |Du_i|^2 &\leq C \sum_{L\in \mathcal{G}_{\leq \ell+1}}
2^{-m\ell (L)} (\mathbf{E} (L,0) + 2^{-2\ell (L)} \mathbf{A}^2)\\
&= C \sum_{L\in \mathcal{G}_{\leq \ell+1}} \Big(2^{2\ell (L)} \int_{\mathbf{B}^h (L)} \dist^2 (q, \mathbf{S})\, d\|T\| (q) + 2^{-(m+2)\ell (L)} \mathbf{A}^2\Big)\\
&\leq C \sigma^{-2} \sum_{L\in \mathcal{G}^o} \int_{\mathbf{B}^h (L)} \dist^2 (q, \mathbf{S})\, d\|T\| (q) + {C \mathbf{A}^2} \, ,
\end{align*}
where we have used \eqref{e:geometric} {(with $\kappa =4$)}.
We next observe that $\mathbf{B}^h (L)\subset \Bbf_4$ for every $L\in \mathcal{G}$ and that, by Lemma \ref{l:whitney}(iv), for given any point $q\in \Bbf_4$ the cardinality of elements $L\in \mathcal{G}$ for which $q\in \mathbf{B}^h (L)$ is bounded by a constant $C = C(m,n)$. We thus conclude \eqref{e:first-Dir-control}. 

All that remains to be proven is (iii). By the definition of $T_{L,i}$ as the restriction of $T$ on a suitable open set, {for each $i$} there is an integral current $T_i$ with the following properties:
\begin{itemize}
    \item $\partial T_i =0$ on $\mathbf{p}_{\alpha_i}^{-1} \left(\bigcup_{L\in \mathcal{G}^o} L_i\right)$;
    \item $T_i$ is area-minimizing;
    \item $T_i \res \mathbf{p}_{\alpha_i}^{-1} (L_i) = T_{L, i} \res \mathbf{p}_{\alpha_i}^{-1} (L_i)$.
\end{itemize}
In particular, if we define $K_i := \bigcup_{L\in \mathcal{G}^o} \bar{K}_i (L)$, the argument leading to \eqref{e:first-Dir-control} and the estimates in Proposition \ref{p:coherent} lead to the following:
\begin{align}
T_i \res \mathbf{p}_{\alpha_i}^{-1} (K_i \cap R_i) &= \mathbf{G}_{u_i} \res \mathbf{p}_{\alpha_i}^{-1}{(K_i\cap R_i)}\\
\|u_i\|^2_{L^\infty (R_i)} &\leq C (\hat{\mathbf{E}} (T, \mathbf{S}, \Bbf_4) + \mathbf{A}^2)\\
\|Du_i\|_{L^\infty} &\leq C(\sigma^{-2} \hat{\mathbf{E}} (T, \mathbf{S}, \Bbf_4) + \mathbf{A}^2)^\gamma\\
|R_i\setminus K_i| +\|T_i\| (\mathbf{p}_{\alpha_i}^{-1} (R_i\setminus K_i)) &
\leq C (\sigma^{-2} \hat{\mathbf{E}} (T, \mathbf{S}, \Bbf_4) + \mathbf{A}^2)^{1+\gamma}\, .
\end{align}
We are thus in the same position to apply the arguments of \cite{DLS14Lp} leading to \cite{DLS14Lp}*{Theorem~2.6} in order to conclude point (iii) of Proposition \ref{p:Dir-minimizing}. The reader will notice that the only obstruction to applying \cite{DLS14Lp}*{Theorem~2.6} is that the domain of the map $u_i$ given above is not a ball. However, the arguments only use the regularity of the boundary of the domain, and since the boundary of $R_i$ is Lipschitz, those arguments apply here as well.
\end{proof}

\section{Cone balancing}\label{s:balancing}

In this section we introduce a suitable procedure which allows us to pass from a possibly unbalanced cone to a balanced cone whilst only changing the excess by a constant. This procedure is done under the assumption that the two-sided $L^2$ height excess of $T$ relative to a cone $\Sbf\in \mathscr{C}(Q)$ is significantly smaller than the planar $L^2$ height excess of $T$.
To make our exposition cleaner, we recall the notation
\begin{align*}
\boldsymbol{\sigma}(\Sbf) &:= \min_{i<j}\dist(\alpha_i\cap \Bbf_1,\alpha_j\cap\Bbf_1)\\
\boldsymbol{\mu}(\Sbf) &:= \max_{i<j}\dist(\alpha_i\cap \Bbf_1,\alpha_j\cap\Bbf_1)
\end{align*}
where $i,j$ in the minimum and maximum range over the indices of the planes in $\Sbf$.

\begin{proposition}[Cone balancing]\label{p:balancing}
Assume that $T$ and $\Sigma$ are as in Assumption \ref{a:main}, $\Bbf_1= \Bbf_1 (0) \subset \Omega$, $\mathbf{S}\in \mathscr{C} (Q)$, and $\alpha_1, \ldots, \alpha_N$ are as in Definition \ref{def:cones}. Then, there are constants $C = C(Q,m,n,\bar{n})>0$ and $\varepsilon_0 = \varepsilon_0(Q,m,n,\bar{n})>0$ with the following property. Assume that
\begin{equation}\label{e:smallexcess}
    \mathbf{A}^2 \leq \varepsilon_0^2 \mathbb{E} (T, \mathbf{S}, \Bbf_1) \leq \eps_0^4 \Ebf^p(T,\Bbf_1)\, .
\end{equation}
 Then there is a subset $\{i_1, \ldots , i_k\}\subset \{1, \ldots , N\}$ with $k\geq 2$ such that, upon setting $\mathbf{S}' = \alpha_{i_1} \cup \cdots \cup \alpha_{i_k}$, the following holds:
\begin{itemize}
\item[(a)] $\mathbf{S}'$ is $C$-balanced;
\item[(b)] $\mathbb{E} (T, \mathbf{S}', \Bbf_1) \leq C \mathbb{E} (T, \mathbf{S}, \Bbf_1)$;
\item[(c)] $\dist^2(\Sbf\cap \Bbf_1,\Sbf^\prime\cap\Bbf_1) \leq C\mathbb{E}(T,\Sbf,\Bbf_1)$;
\item[(d)] $C^{-1} \Ebf^p (T, \Bbf_1) \leq \boldsymbol{\mu} (\Sbf)^2 =
\boldsymbol{\mu} (\Sbf^\prime)^2\leq C \Ebf^p (T, \Bbf_1)$.
\end{itemize}
\end{proposition}

Using the Pruning Lemma (Lemma \ref{l:pruning}), the proof of Proposition \ref{p:balancing} can be reduced to showing the following proposition, which roughly says that if the two-sided $L^2$ height excess of $T$ relative to $\Sbf$ is significantly smaller than the minimal angle in the cone $\Sbf$, then in fact $\Sbf$ is already $C$-balanced for some constant $C$. Indeed, intuitively, since $T$ is area-minimizing, if it is very close to $\Sbf$ then we would expect the union of planes in $\Sbf$ to roughly behave like an area-minimizer, and from Morgan's result (Lemma \ref{l:frank-distance}) we would expect $\Sbf$ to be balanced.

\begin{proposition}\label{p:balancing-2}
Assume that 
$T$, $\Sigma$, and $\mathbf{S}$ are as in Proposition \ref{p:balancing}. Then there are constants $C = C(m,n,\bar{n},N)$ and $\varepsilon = \varepsilon (m,n,\bar{n},N)$ with the following property. If we additionally have that
\begin{equation}\label{e:evensmallerexcess}
    N\geq 2 \quad \text{and} \quad {\mathbf{A}^2 + \mathbb{E} (T, \mathbf{S}, \Bbf_1) \leq \varepsilon^2 \boldsymbol{\sigma}(\Sbf)^2\,} ,
\end{equation}
then $\mathbf{S}$ is $C$-balanced.
\end{proposition}

Let us first show that Proposition \ref{p:balancing-2} implies Proposition \ref{p:balancing}.

\begin{proof}[Proof of Proposition \ref{p:balancing}]
Suppose that Proposition \ref{p:balancing-2} holds, and let $\varepsilon_*$ be the minimum over $N\leq Q$ of all the constants $\varepsilon = \varepsilon(m,n,\bar{n},N)$ from Proposition \ref{p:balancing-2}. Then fix $\eps_0\leq \eps_*$ to be determined later.

The hypotheses of Proposition \ref{p:balancing} in particular give that $\mathbb{E}(T,\Sbf,\Bbf_1)\leq \eps_0^2 \Ebf^p(T,\Bbf_1)$. Next we estimate
\begin{align*}
    \Ebf^p(T,\Bbf_1) &\leq \int_{\Bbf_1}\dist^2(x,\alpha_1)\, d\|T\|(x) \leq C\hat{\Ebf}(T,\Sbf,\Bbf_1) + C\max_{i<j}\dist^2(\alpha_i\cap \Bbf_1,\alpha_j\cap \Bbf_1)\\
    &\leq C\eps_0^2 \Ebf^p (T, \Bbf_1) + C \boldsymbol{\mu} (\Sbf)^2\, ,
\end{align*}
where $C = C(m,n)$. In particular for $\eps_0$ smaller than a geometric constant we get
\begin{equation}\label{e:mu-control}
\Ebf^p (T,\Bbf_1)\leq C \boldsymbol{\mu} (\Sbf)^2\, ,
\end{equation}
where $C= C(m,n)$. Thus we also conclude
\[
\mathbb{E}(T,\Sbf,\Bbf_1) \leq \underbrace{C \eps_0^2}_{=:\eta}\boldsymbol{\mu}(\Sbf)^2\, .
\]
Let us now fix $\delta>0$ (which will be determined later), and let $\Gamma$ be as in Lemma \ref{l:pruning} for this choice of $\delta$ and $N$. If we take $\eta^{1/2} = (1+\Gamma)^{-1}\delta$ in the above, we see that if $\eps_0 = \eps_0(m,n,N,\delta)>0$ is sufficiently small, and $D := \mathbb{E}(T,\Sbf,\Bbf_1)^{1/2}$, then
$$D\leq (1+\Gamma)^{-1}\delta \boldsymbol{\mu} (\Sbf)\, .$$
Thus we are in the situation to apply Lemma \ref{l:pruning} with this choice of $D$. This yields a subset $I=\{i_1,\dots,i_k\}\subset \{1,\dots,N\}$, with $k\geq 2$, such that the corresponding planes $\{\alpha_{i_1},\dots,\alpha_{i_k}\}$ satisfy \eqref{e:pruning-1}, \eqref{e:pruning-2}, \eqref{e:pruning-3} of the Pruning Lemma. Now set $\Sbf^\prime:= \alpha_{i_1}\cup \cdots \cup \alpha_{i_k}$; we claim that Proposition \ref{p:balancing} holds with this $\Sbf^\prime$. We start by observing that, since $\boldsymbol{\mu} (\Sbf) = \boldsymbol{\mu} (\Sbf^\prime)$ (by \eqref{e:pruning-3}) and \eqref{e:mu-control} holds, we just need
\begin{equation}\label{e:other-mu-control}
\boldsymbol{\mu} (\Sbf^\prime)^2 \leq C \Ebf^p (T, \Bbf_1)
\end{equation}
to complete the proof of (d). However, let us first prove (a)--(c).

{Observe that conclusion \eqref{e:pruning-1} of the Pruning Lemma gives
$$\max_{j=1,\dotsc,N}\min_{i\in I}\dist^2(\alpha_i\cap\Bbf_1,\alpha_j\cap\Bbf_1) \leq \Gamma^2 \mathbb{E}(T,\Sbf,\Bbf_1).$$
In particular, this will give condition (c) once we have chosen $\delta$ appropriately. Furthermore, observe the following consequence of this: for $q\in \spt(T)\cap \Bbf_1$, suppose $\dist(q,\Sbf^\prime) = \dist(q,\alpha_{i_j})$ for some $i_j\in I$. If $\dist(q,\Sbf) = \dist(q,\alpha_{i_j})$, then we clearly have $\dist(q,\Sbf^\prime) = \dist(q,\Sbf)$. Otherwise, we must have $\dist(q,\Sbf) = \dist(q,\alpha_\ell)$ for some $\ell\not\in I$. By the above consequence of the Pruning Lemma however, there is some $i_{j_*}\in I$ for which
$$\dist^2(\alpha_\ell\cap \Bbf_1,\alpha_{i_{j_*}}\cap \Bbf_1) \leq \Gamma^2\mathbb{E}(T,\Sbf,\Bbf_1).$$
In particular, we must have 
\begin{align*}
\dist^2(q,\Sbf^\prime) \leq \dist^2(q,\alpha_{i_{j_*}}\cap \Bbf_1) & \leq 4\dist^2(q,\alpha_{\ell}) + 4\dist^2(\alpha_{\ell}\cap \Bbf_1, \alpha_{i_{j_*}}\cap\Bbf_1)\\
& \leq 4\dist^2(q,\Sbf) + 4\Gamma^2\mathbb{E}(T,\Sbf,\Bbf_1).
\end{align*}
In either case, we see that for any $q\in \spt(T)\cap \Bbf_1$,
$$\dist^2(q,\Sbf^\prime) \leq 4\dist^2(q,\Sbf) + 4\Gamma^2 \mathbb{E}(T,\Sbf,\Bbf_1)$$
which evidently gives
$$\hat{\mathbf{E}}(T,\Sbf^\prime,\Bbf_1) \leq 4(1+\Gamma^2 (Q+1)\omega_m)\mathbb{E}(T,\Sbf,\Bbf_1).$$
This deals with controlling one term of the two-sided height excess $\mathbb{E}(T,\Sbf^\prime,\Bbf_1)$. However, controlling the other term is simple as $\Sbf^\prime\subset \Sbf$, and so $\hat{\mathbf{E}}(\Sbf^\prime,T,\Bbf_1)\leq \hat{\mathbf{E}}(\Sbf,T,\Bbf_1)$. Combining we therefore get $\mathbb{E}(T,\Sbf^\prime,\Bbf_1)\leq C\mathbb{E}(T,\Sbf,\Bbf_1)$ where $C = C(Q,m,n,\delta)$; this proves conclusion (b) provided we choose $\delta = \delta(Q,m,n,\bar{n})>0$ in the end.}

To show that conclusion (a) holds we will apply Proposition \ref{p:balancing-2} to $\Sbf^\prime$. For this we must verify that the hypothesis \eqref{e:evensmallerexcess} holds in this situation. Note that from \eqref{e:pruning-2} of the Pruning Lemma we have
$$\mathbb{E}(T,\Sbf,\Bbf_1) + \max_{j=1,\dotsc,N}\min_{i\in I}\dist^2(\alpha_i\cap \Bbf_1,\alpha_j\cap \Bbf_1) \leq 2\delta^2 \min_{i<j\in I}\dist^2(\alpha_i\cap\Bbf_1,\alpha_j\cap\Bbf_1).$$

In particular, if we perform exactly the same bounds as above when we proved (b), except replacing the estimate from \eqref{e:pruning-1} by the above, we would end up with
\begin{align*}
\hat{\mathbf{E}}(T,\Sbf^\prime,\Bbf_1) & \leq 2\hat{\mathbf{E}}(T,\Sbf,\Bbf_1) + 8(Q+1)\delta^2\min_{i<j\in I}\dist^2(\alpha_i\cap \Bbf_1,\alpha_j\cap \Bbf_1)\\
& \leq (8Q+12)\delta^2\min_{i<j\in I}\dist^2(\alpha_i\cap \Bbf_1,\alpha_j\cap \Bbf_1)
\end{align*}
where in the second inequality we have used the fact $\mathbb{E}(T,\Sbf,\Bbf_1)\leq 2\delta^2\min_{i<j\in I}\dist^2(\alpha_i\cap\Bbf_1,\alpha_j\cap \Bbf_1)$ again from the statement of \eqref{e:pruning-2} above. But again, since for the other half of the excess we have
$$\hat{\mathbf{E}}(\Sbf^\prime,T,\Bbf_1) \leq \hat{\mathbf{E}}(\Sbf,T,\Bbf_1) \leq \mathbb{E}(T,\Sbf,\Bbf_1) \leq 2\delta^2\min_{i<j\in I}\dist^2(\alpha_i\cap \Bbf_1,\alpha_j\cap \Bbf_1)$$
we see that
$$\mathbb{E}(T,\Sbf^\prime,\Bbf_1) \leq (8Q+14)\delta^2\min_{i<j\in I}\dist^2(\alpha_i\cap\Bbf_1,\alpha_j\cap\Bbf_1).$$
Hence, if we choose $\delta = \delta(Q,m,n,\bar{n})>0$ obeying $(8Q+14)\delta^2 < \epsilon_*/2$, we get that one part of the inequality in \eqref{e:evensmallerexcess} holds for $T$ and $\Sbf^\prime$. However, the other part of the inequality in \eqref{e:evensmallerexcess} evidently follows, since by assumption we have
$$\Abf^2 \leq \eps_0^2\mathbb{E}(T,\Sbf,\Bbf_1) \leq 2\delta^2\eps_0^2\min_{i<j\in I}\dist^2(\alpha_i\cap \Bbf_1,\alpha_j\cap \Bbf_1)$$
and thus we have that \eqref{e:evensmallerexcess} holds for suitably chosen $\delta = \delta(Q,m,n,\bar{n})>0$. Hence, with this choice of $\delta$ we can apply Proposition \ref{p:balancing-2} to see that $\Sbf^\prime$ is $C$-balanced for some $C = C(Q,m,n,\bar{n})>0$, which completes the proof of (a).

It remains to prove (d), namely, as already observed, \eqref{e:other-mu-control}.
Fix any plane $\pi$ and observe that there is one plane $\alpha_{i_1}$ from $\Sbf'$ with $i_1\in I$ (which without loss of generality by relabeling we can assume to be $i_1=1$) such that
\[
\dist (\alpha_1\cap \Bbf_1, \pi\cap \Bbf_1) \geq \frac{1}{2} \boldsymbol{\mu} (\Sbf^\prime)\, .
\]
Indeed, if not we would get a contradiction to the definition of $\boldsymbol{\mu}(\Sbf^\prime)$ by the triangle inequality.

We next show that there is an element $p\in \alpha_1$, a radius $r (m,n)>0$, and constant $C(m,n)>0$ with the property that $\Bbf_r (p)\subset \Bbf_{3/4} \setminus B_{1/4}(V)$ and 
\begin{equation}\label{e:dist-q-lower-bound}
\dist (q, \pi) \geq C^{-1} \boldsymbol{\mu} (\Sbf^\prime) \qquad \forall q\in  B_r (p, \alpha_1)\, .
\end{equation}

Recall first that $\dist (p, \pi) = |\mathbf{p}_{\pi}^\perp (p)|$ and that, by linearity of the map $\mathbf{p}^\perp_\pi$ we have $|\mathbf{p}_\pi^\perp(p)| = |p|\cdot|\mathbf{p}_\pi^\perp(p/|p|)| = |p|\dist(p/|p|,\pi)$, and so  
\begin{equation}\label{e:simple-linearity}
|\mathbf{p}^\perp_\pi (p)|\leq \dist (\alpha_1 \cap \Bbf_1, \pi \cap \Bbf_1) |p| \qquad \forall p\in \alpha_1\, .
\end{equation}
Next choose any $v\in V^\perp\cap \alpha_1$ with $|v|= \frac{1}{2}$. Consider then the disk $B_{1/8} (v, \alpha_1)$ and inside this disk select a base $e_1, \ldots e_m$ of $\alpha_1$ with the property that any element $x\in \alpha_1\cap \overline{\Bbf}_1$ can be written as a linear combination $\sum_i \lambda_i e_i$ with $|\lambda_i|\leq C = C(m)$. It follows that, for some element $e_i$ we must necessarily have 
\[
|\mathbf{p}^\perp_\pi (e_i)|\geq \frac{1}{m C} \dist (\alpha_1\cap \Bbf_1, \pi\cap \Bbf_1)\, .
\]
Indeed by Corollary \ref{c:growth} there is a vector $e\in \overline{\Bbf}_1 \cap \alpha_1$ with $|\mathbf{p}_\pi^\perp (e)| = \dist (\alpha_1\cap \Bbf_1, \pi\cap \Bbf_1)$. We can thus use the property above to write $e= \sum_i \lambda_i e_i$ and estimate
\begin{align*}
\dist (\alpha\cap \Bbf_1, \pi\cap \Bbf_1) &= |\mathbf{p}_\pi^\perp (e)|
\leq \sum_i |\lambda_i| |\mathbf{p}_\pi^\perp (e_i)|
\leq C \sum_i |\mathbf{p}_\pi^\perp (e_i)|\, .
\end{align*}
Set then $p=e_i$ and choose the radius $r$ to equal $\min \{\frac{1}{2mC}, \frac{1}{8}\}$. We can then use the last inequality and \eqref{e:simple-linearity} to show that for all $q\in B_r(p,\alpha_1)$, 
\[
\dist (q, \pi) = |\mathbf{p}^\perp_\pi (q)|\geq |\mathbf{p}_\pi^\perp(p)| - |\mathbf{p}^\perp_\pi(q-p)| \geq \frac{1}{2mC} \dist (\alpha_1\cap \Bbf_1, \pi\cap \Bbf_1) \geq \frac{1}{4mC} \boldsymbol{\mu} (\Sbf)\, .
\]
This establishes \eqref{e:dist-q-lower-bound}.

In particular, if the parameter {$\eps_0$} is small enough, we can apply Lemma \ref{l:splitting-1}, Proposition \ref{p:Lipschitz-1}, and Lemma \ref{l:matching-Q} to $T_{0,1/4}$ and $\mathbf{S}'$ (note that we have already established that $\Sbf^\prime$ is $C$-balanced by the above and that the hypotheses of these results do hold here). In particular, let $\varrho$ be the parameter in Lemma \ref{l:splitting-1}. If we consider the current
\[
T' := T \res \Bbf_r (p) \cap \{\dist (\cdot , \alpha_1) \leq \varrho \boldsymbol{\sigma} (\Sbf^\prime)\}
\]
we will have that $\partial T' =0$ in $\Bbf_r (p)$ and that $\|T'\| (\Bbf_r (p))\geq C^{-1} r^m$ for some constant $C(m,n)$. {Moreover, from \eqref{e:Linfty-crude} in Proposition \ref{p:Lipschitz-1}}, we have 
\[
\dist (q, \alpha_1) \leq C \hat{\Ebf} (T, \Sbf^\prime, \Bbf_1)^{1/2} + C \mathbf{A}\, \qquad \forall q\in \Bbf_r(p)\cap \spt(T^\prime)
\]
for some constant $C = C(Q,m,n,\bar{n})$. {Since $\hat{\Ebf}^2(T,\Sbf^\prime,\Bbf_1) + \Abf^2 \leq C\eps_0^2\boldsymbol{\mu}(\Sbf^\prime)^2$, if $\eps_0$ is chosen sufficiently small we then easily conclude from the above and \eqref{e:dist-q-lower-bound} that}
\[
\dist (q, \pi) \geq C^{-1} \boldsymbol{\mu} (\mathbf{S}^\prime) \qquad \forall q\in \Bbf_r(p) \cap \spt (T')\, .
\]
Squaring and integrating the latter inequality with respect to {$d\|T^\prime\|$}, and using the lower bound $\|T'\| (\Bbf_r (p))\geq C^{-1} r^m$ we reach 
\[
\boldsymbol{\mu} (\Sbf^\prime)^2 \leq C \hat{\Ebf} (T, \pi, \Bbf_1)\, .
\]
However, since $\pi$ is an arbitrary $m$-dimensional plane, this completes the proof of \eqref{e:other-mu-control}.
\end{proof}

We now come to the proof of Proposition \ref{p:balancing-2}. We first prove a key special case of it when $\boldsymbol{\mu}(\Sbf)$ and $\boldsymbol{\sigma}(\Sbf)$ are comparable.

\begin{proposition}\label{p:balancing-3}
    Assume that $T,\Sigma$, and $\Sbf$ are as in Proposition \ref{p:balancing}. Fix $\eta>0$. Then, there exist constants $C = C(m,n,\bar{n},N,\eta)$ and $\eps = \eps(m,n,\bar{n},N,\eta)>0$ with the following property. If \eqref{e:evensmallerexcess} and the assumptions of Proposition \ref{p:balancing} hold with this choice of $\eps$, and moreover if
    \begin{equation}\label{e:sep-planes}
    \eta\boldsymbol{\mu}(\Sbf) \leq \boldsymbol{\sigma}(\Sbf)\, ,
    \end{equation}
    then $\Sbf$ is $C$-balanced.
\end{proposition}

Notice in particular that this implies the $N=2$ case of Proposition \ref{p:balancing-2}, as in that situation we have $\boldsymbol{\mu}(\Sbf) = \boldsymbol{\sigma}(\Sbf)$.

\begin{proof}
    We argue by contradiction. If the proposition were false, then we could find sequences $T_k,\Sigma_k,\Abf_k$ and $\Sbf_k = \alpha_1^k\cup \dots\cup \alpha_N^k$ as in Proposition \ref{p:balancing-2} such that
$$\Abf_k^2 + \mathbb{E}(T_k,\Sbf_k,\Bbf_1) \leq \eps_k^2 \boldsymbol{\sigma}(\Sbf_k)^2$$
and for which \eqref{e:sep-planes} holds for each $k$, but (after relabelling the planes) the Morgan angles $\theta_1(\alpha_1^k,\alpha_2^k)$ and $\theta_2(\alpha^k_1,\alpha_2^k)$ obey
\begin{equation}\label{e:notbalanced1}
\frac{\theta_1(\alpha^k_1,\alpha^k_2)}{\theta_2(\alpha^k_1,\alpha^k_2)}\to 0.
\end{equation}
We have two possibilities: either (a) $\limsup_{k\to\infty}\boldsymbol{\sigma}(\Sbf_k) >0$ or (b) $\limsup_{k\to\infty}\boldsymbol{\sigma}(\Sbf_k) = 0$. In the case of (a) the situation is simple: we simply pass to a subsequence (which we do not relabel) for which $\lim_{k\to\infty} \boldsymbol{\sigma}(\Sbf_k)>0$, $T_k\to T_\infty$ as currents, where $T_\infty$ is some $m$-dimensional area-minimizing integral current in $\Bbf_1$, and $\Sbf_k\to \Sbf_\infty = \alpha_1^\infty\cup \cdots\cup \alpha_N^\infty$ for a collection of \emph{$N$ distinct} planes $\alpha_i^\infty$, locally in Hausdorff distance, {with $\alpha^k_i\to \alpha^\infty_i$ for each $i=1,\dotsc,N$}. By the assumption that $\mathbb{E}(T_k,\Sbf_k,\Bbf_1) \to 0$, we must have $\spt(T_\infty)\cap \Bbf_1 = \Sbf_\infty\cap \Bbf_1$. 
Indeed, {we first see that} $\spt(T_\infty)\cap\Bbf_1\subset \Sbf_\infty\cap\Bbf_1$ and then the constancy theorem implies $T_\infty \res \Bbf_1 = \sum_i k_i \llbracket \alpha^\infty_i\rrbracket$ for some $k_i \in \mathbb Z$. The only obstruction to the equality $\spt (T_\infty) \cap \Bbf_1 = \Sbf_\infty \cap \Bbf_1$ is then the vanishing of some of the coefficients $k_i$, which would come from orientation cancellation in the limit of the $T_k$; however, this would contradict the $T_k$ being area-minimizing. 

But then by Lemma \ref{l:frank-distance}, we would necessarily have $\theta_1(\alpha_i^\infty,\alpha_j^\infty) = \theta_2(\alpha_i^\infty,\alpha_j^\infty) > 0$ for each $i<j\in\{1,\dots,N\}$, where the fact that these angles are non-zero follows from the assumption that ${\liminf_{k}}\boldsymbol{\sigma}(\Sbf_k)>0$. But we clearly have $\theta_1(\alpha_1^k,\alpha_2^k)\to \theta_1(\alpha_1^\infty,\alpha_2^\infty)$ and $\theta_2(\alpha_1^k,\alpha_2^k)\to \theta_2(\alpha^\infty_1,\alpha^\infty_2)$ from the local Hausdorff distance convergence in $\Bbf_1$, and hence we would have $\theta_1(\alpha^k_1,\alpha^k_2)/\theta_2(\alpha^k_1,\alpha^k_2)\to 1$, contradicting \eqref{e:notbalanced1}. 

Now let us handle the case (b). Here, we can pass to a further subsequence to ensure that $\alpha_i^k\to \alpha_\infty$ for all $i=1,\dots, N$, locally in Hausdorff distance, for some $m$-dimensional plane $\alpha_\infty$. Since $\mathbb{E}(T_k,\Sbf_k,\Bbf_1)\to 0$ we therefore also have that $\hat{\mathbf{E}}(T_k,\alpha_\infty,\Bbf_1) \to 0$. Note that we may assume without loss of generality that $\alpha^k_1 = \alpha_\infty$ for all $k$: indeed, we can choose a sequence of rotations $q_k:\R^{m+n}\to \R^{m+n}$ with $q_k\to \text{id}_{\R^{m+n}}$ and obeying $q_k(\alpha_1^k) = \alpha_\infty$, and then consider $q_k(\Sbf_k)$ and $(q_k)_\sharp T_k$; note that $\boldsymbol{\sigma}(\Sbf_k) = \boldsymbol{\sigma}(q_k(\Sbf_k))$ {(see the discussion in Section \ref{s:rotations})}. Hence, for $k$ sufficiently large, we can consider the (strong) Lipschitz approximations $f_k: B_{3/4}(0,\alpha_\infty)\to \mathcal{A}_Q(\alpha_\infty^\perp)$ of Almgren (see \cite{DLS14Lp}*{Theorem 2.4}) for the $T_k$ relative to $\alpha_\infty$ (here, the assumption $\hat{\mathbf{E}}(T_k,\alpha_\infty,\Bbf_1)\to 0$ is sufficient in light of Allard's tilt-excess inequality, see for example \cite{Allard_72}*{Proposition 4.1}). Now set
\[
\bar{f}_k \coloneqq \frac{f_k}{\boldsymbol{\sigma}(\Sbf_k)}.
\]
Note that because $\alpha_1^k = \alpha_\infty$ for all $k$, we have
\[
\hat{\mathbf{E}}(T_k,\alpha_\infty,\Bbf_1) \leq 4\hat{\mathbf{E}}(T_k,\Sbf_k,\Bbf_1) + C{\boldsymbol{\mu}}(\Sbf_k)^2 \leq (4\eps_k^2+C{\eta^{-1}})\boldsymbol{\sigma}(\Sbf_k)^2
\]
and so, again using Allard's tilt-excess inequality, the estimates from Almgren's strong Lipschitz approximation (see again \cite{DLS14Lp}*{Theorem 2.4}) give that $\bar{f}_k\to \bar{f}_\infty$ strongly in $W^{1,2}_{\text{loc}}(\Bbf_{3/4})\cap L^2(\Bbf_{3/4})$ for some Dir-minimizer $\bar{f}_\infty$. Note also that, if for each $i$, we denote by $L_i^k$ the linear maps parameterizing the planes $\alpha_i^k$ over $\alpha_\infty$ (in particular, $L^k_1 = 0$ by construction, although this will not play any role in the proof), then we necessarily have that $|L_i^k| \leq {C\boldsymbol{\mu}(\Sbf_k) \leq C{\eta^{-1}}\boldsymbol{\sigma}(\Sbf_k)}$ for some dimensional constant $C$, and so if we write
$$\bar{L}_i^k := \frac{L_i^k}{\boldsymbol{\sigma}(\Sbf_k)}\, , \hspace{3em}$$
then we can pass to a further subsequence to ensure also that $\bar{L}_i^k \to \bar{L}_i$ for some linear map $\bar{L}_i$ over $\alpha_\infty$. Clearly, for each $i\neq j$ we have $|\bar{L}_i-\bar{L}_j|\geq c>0$ by definition of $\boldsymbol{\sigma}(\Sbf_k)$, and so $\bar{L}_i\neq \bar{L}_j$ whenever $i\neq j$. Noting that
$$
\int \dist^2(x,\Sbf_k)\, d\|\mathbf{G}_{f_k}\|(x) \leq C\hat{\mathbf{E}}(T_k,\Sbf_k,\Bbf_1) + C\left(\Abf_k^2 + \hat{\mathbf{E}}(T_k,\alpha_\infty,\Bbf_1)\right)^{1+\gamma} = o (\boldsymbol{\sigma}(\Sbf_k)^2)\, 
$$
{where $o(\boldsymbol{\sigma}(\Sbf_k)^2)$ denotes a term which converges to zero as $k\to\infty$ when divided by $\boldsymbol{\sigma}(\Sbf_k)^2$.} Dividing both sides of the above inequality by $\boldsymbol{\sigma}(\Sbf_k)^2$ and taking $k\to \infty$, we see that necessarily the graph of $\bar{f}_\infty$ is supported in the union of the graphs of $\bar{L}_i$, $i=1,\dots,N$, and so in particular we have
$$\bar{f}_\infty = \sum_i Q_i\llbracket \bar{L}_i\rrbracket$$
for some non-negative integers $Q_i$ such that $\sum_i Q_i = Q$. We claim that necessarily both $Q_1,Q_2$ {(and in fact all $Q_i$, but we won't need this to see the contradiction)} are strictly positive. Indeed, once we have this we can complete the proof of the claim in the following manner: for every constant $\lambda>0$ the map $\lambda \bar{f}_\infty$ is Dir-minimizing, and so if $Q_1,Q_2>0$ then, by Proposition \ref{p:two-sided-bound}, the planes $\bar{\alpha}_i$ that are the graphs of ${\lambda}\bar{L}_i$ are $c_2$-balanced for all $\lambda$ sufficiently small, where $c_2$ is an absolute constant (note that necessarily the intersection of the $\bar{\alpha}_i$ is a $(m-2)$-dimensional subspace, as the intersection of the $\alpha^k_i$ is an $(m-2)$-dimensional subspace for all $k$, and so up to passing to a subsequence this subspace will converge to an $(m-2)$-dimensional subspace which is necessarily the intersection of the planes $\bar{\alpha}_i$). In particular, we have
$$\frac{\theta_1(\bar{\alpha}_1,\bar{\alpha}_2)}{\theta_2(\bar{\alpha}_1,\bar{\alpha}_2)} \geq c_2 > 0.$$
Fix such a constant $\lambda$ for which the above is true. Now, let $\bar{\alpha}_i^k$ be the planes which are the graphs of $\lambda\bar{L}^k_i$, $i=1,2$. Shrinking $\lambda$ if necessary (which can be done independently of $k$, as $|\bar{L}^k_i|\leq C$ for all $k$ and some geometric constant $C$), we can infer from Corollary \ref{c:elementary-planes} that we have
$$\frac{\theta_1(\bar{\alpha}_1^k,\bar{\alpha}_2^k)}{\theta_2(\bar{\alpha}_1^k,\bar{\alpha}_2^k)} \leq 16^2\frac{\theta_1(\alpha_1^k,\alpha_2^k)}{\theta_2(\alpha^k_1,\alpha^k_2)}.$$
But then our assumption \eqref{e:notbalanced1} implies from this that
$$\frac{\theta_1(\bar{\alpha}_1^k,\bar{\alpha}_2^k)}{\theta_2(\bar{\alpha}_1^k,\bar{\alpha}_2^k)} \to 0\, .$$
Noting that
$$\frac{\theta_1(\bar{\alpha}_1^k,\bar{\alpha}_2^k)}{\theta_2(\bar{\alpha}_1^k,\bar{\alpha}_2^k)} \to \frac{\theta_1(\bar{\alpha}_1,\bar{\alpha}_2)}{\theta_2(\bar{\alpha}_1,\bar{\alpha}_2)} \geq c_2>0$$
this gives the desired contradiction.

Therefore, all that remains to show is that $Q_1,Q_2>0$. In order to prove this, we use
    Lemma \ref{l:sep_region} to find points $\xi_k \in {\alpha^k_1\cap \partial \Bbf_{1/2}}$ and a positive constant $c_0 = c_0(m,N)$ with the property that
    \begin{equation}\label{e:separated-disc}
        \min_{j>1}\inf \{\dist (\zeta, \alpha_j^k) : \zeta\in \Bbf_{2c_0} (\xi_k)\cap \alpha_1^k\} \geq 2c_0 {\min_{j>1}\dist} (\alpha_1^k\cap \Bbf_1, \alpha_j^k \cap \Bbf_1)\, .
    \end{equation}
    In particular, for $\delta = \delta(Q,m,n,\bar{n},N)$ as in the Splitting Corollary (Corollary \ref{c:splitting-0}) and for $r=r(\delta{,\eta})$ sufficiently small, the assumptions of Corollary \ref{c:splitting-0} are satisfied in $\Cbf_{{4}r c_0} (\xi_k)$.
    
    Indeed, conditions (i), (ii), (iii) are clear. (iv) follows from \eqref{e:separated-disc}, as it gives
    \[
    \min_{j>1} \inf\{\dist(\zeta,\alpha^k_j):\zeta\in \Bbf_{2c_0}(\xi_k)\cap \alpha^k_1\}\geq 2c_0\boldsymbol{\sigma}(\Sbf_k) \geq 2c_0\eta \boldsymbol{\mu}(\Sbf_k)
    \]
    which in turn gives (iv) if we chose $r = 2\delta c_0\eta$ with a possibly smaller constant $c_0$. Note indeed that in the left hand side of (iv) we are using the distance between oriented planes, however it is just a matter of choosing the right orientation for the planes to be able to conclude that left hand side of (iv) is bounded above by $C\boldsymbol{\mu} (\Sbf)$. Finally, condition (v) follows from the fact $\Abf^2_k + \mathbb{E}(T_k,\Sbf_k,\Bbf_1)\leq \eps^2_k\boldsymbol{\sigma}(\Sbf_k)^2$, combined with the bound $\boldsymbol{\mu}(\Sbf_k)\leq \delta\min\{1,r^{-1}\varkappa\}$ just established for the above choice of $r$. 
    
    From Corollary \ref{c:splitting-0} we may conclude that $T\res \mathbf{B}_{r c_0} (\xi_k)$ splits into {currents $T_i$, $i=1,\dotsc,N$, with the $T_i$ supported in small disjoint neighborhoods of $\alpha_i^k\cap \Cbf_{{r}c_0} (\xi_k)$;} in particular, $T_1$ and $T_2$ are supported in two small disjoint neighbourhoods of $\alpha_1^k \cap \Cbf_{{r}c_0} (\xi_k)$ and $\alpha_2^k \cap \Cbf_{{r}c_0} (\xi_k)$ respectively, with the property that $(\mathbf{p}_{\alpha_1^k})_\sharp T_i = Q_i (k) \llbracket B_{{r}c_0} (\xi_k)\rrbracket$ for \emph{positive} integers $Q_i (k){\geq 1}$. Upon extraction of a subsequence we can assume the $Q_i (k)$ equal positive integers $Q'_i$ independent of $k$ and it follows easily that $Q'_i=Q_i$. This concludes the proof.
\end{proof}

We now come to the proof of Proposition \ref{p:balancing-2}. At one point in the proof we will need the following proposition, which also plays a key role later on in our work. Its proof will be given in Section \ref{p:linear}, where we collect several important facts regarding Dir-minimizers.

\begin{proposition}\label{p:remove-spine}
Assume $\Omega\subset \mathbb R^m$ is a Lipschitz domain, $V\subset \mathbb R^m$ is an $(m-2)$-dimensional plane, and $v\in W^{1,2}(\Omega;\Acal_Q(\R^n))$ is a map with the property that the restriction of $v$ to $\Omega_\varepsilon := \Omega\setminus B_\varepsilon (V)$ is Dir-minimizing for every $\varepsilon > 0$. Then $v$ is Dir-minimizing in $\Omega$.
\end{proposition}

\begin{proof}[Proof of Proposition \ref{p:balancing-2}]
    We argue by induction on $N$. The case $N=2$ is already established by Proposition \ref{p:balancing-3}, since in that situation we have $\boldsymbol{\mu}(\Sbf) = \boldsymbol{\sigma}(\Sbf)$.

    So fix $N\geq 3$. We may assume inductively the validity of Proposition \ref{p:balancing-2} for all $N^\prime< N$. If $C(m,n,\bar{n},N^\prime)$ and $\eps = \eps(m,n,\bar{n},N^\prime)$ denote the corresponding constants, set
    $$C_1:= \max_{N^\prime\leq N-1}C(m,n,\bar{n},N^\prime)\ \ \ \ \text{and}\ \ \ \ \eps_1^*:= \min_{N^\prime\leq N-1}\eps(m,n,\bar{n},N^\prime).$$
    Suppose however that the conclusion of the proposition fails for $N$. We may therefore find sequences $T_k$, $\Sigma_k$, $\Abf_k$, $\Sbf_k = \alpha^k_1\cup\cdots\cup \alpha^k_N$ and $\eps_k\downarrow 0$ with
    \begin{equation}\label{e:seq-small-excess}
        \mathbf{A}_k^2 \leq \varepsilon_k^2 \mathbb{E} (T_k, \mathbf{S}_k, \Bbf_1) \leq \varepsilon_k^4 \boldsymbol{\sigma}(\Sbf_k)^2,
    \end{equation}
    such that, up to relabelling the planes in $\Sbf_k$, we have
    \begin{equation}\label{e:notbalanced2}
\frac{\theta_1(\alpha^k_1,\alpha^k_2)}{\theta_2(\alpha^k_1,\alpha^k_2)}\to 0.
    \end{equation}
    Now, we may assume that 
    \begin{equation}\label{e:incomparable-max-min}
        \lim_{k\to\infty}\boldsymbol{\sigma}(\Sbf_k)/\boldsymbol{\mu}(\Sbf_k) = 0\, .
    \end{equation}
    Indeed, if this were not true, then up to passing to a subsequence we may assume that $\boldsymbol{\sigma}(\Sbf_k)\geq \eta^*\boldsymbol{\mu}(\Sbf_k)$ for all $k$ and some $\eta^*>0$, at which point we may apply Proposition \ref{p:balancing-3} to get that for all $k$ sufficiently large, the cones $\Sbf_k$ are $C$-balanced for some $C>0$ independent of $k$. But then we would have $\theta_1(\alpha^k_1,\alpha^k_2)/\theta_2(\alpha^k_1,\alpha^k_2)\geq C^{-1}>0$ for all $k$, in direct contradiction to \eqref{e:seq-small-excess}.

    Now apply the layer subdivision lemma (Lemma \ref{l:algorithm}) {with $\delta=1$} to each $\Sbf_k$ {(the exact choice of $\delta$ is unimportant)}. This provides subcollections $\{1,\dotsc,N\} = I_k(0)\supsetneq I_k(1)\supsetneq \cdots\supsetneq I_k(\kappa(k))$ satisfying properties (i)--(iv) of Lemma \ref{l:algorithm}. Up to extracting a further subsequence (which we again do not relabel) we can assume that $\kappa(k)$ is independent of $k$, and moreover that all the sets of indices $I_k(l)$, $\ell=0,1,\dotsc,\kappa$ are independent of $k$; we therefore label them $I(l)$. Since we are assuming now that $\boldsymbol{\sigma}(\Sbf_k)/\boldsymbol{\mu}(\Sbf_k)\to 0$, we necessarily have that $\kappa\geq 1$, as otherwise conclusion (ii) of Lemma \ref{l:algorithm} would give that $\eta\boldsymbol{\mu}(\Sbf_k) \leq \boldsymbol{\sigma}(\Sbf_k)$ for all $k$ and for some fixed $\eta = \eta(N)>0$ independent of $k$, giving a contradiction.

    For each $k$ and $l\in \{0,1,\dotsc,\kappa\}$, write $\Sbf_k^{(l)} := \bigcup_{i\in I(\ell)}\alpha^k_i$ for the cone at the $l^{\text{th}}$ layer of $\Sbf_k$. Also write
    $$\mathbf{d}_k^{(l)}:= \max_{i\in I(0)}\min_{j\in I(l)}\dist(\alpha_i^k\cap\Bbf_1,\alpha_j^k\cap\Bbf_1).$$
    From (iii) of Lemma \ref{l:algorithm}, we know that
    \begin{equation}\label{e:balancing-induction-1}
        \frac{\boldsymbol{\sigma}(\Sbf_k^{(\ell-1)})}{\mathbf{d}_k^{(l)}}\geq\eta>0
    \end{equation}
    for all $k$ and $\ell\in \{1,\dotsc,\kappa\}$, where again $\eta = \eta(N)>0$ is a fixed constant.

    Suppose first that we have
    \begin{equation}\label{e:balancing-induction-2}
        \lim_{k\to\infty}\frac{\mathbf{d}_k^{(1)}}{\boldsymbol{\sigma}(\Sbf_k^{(1)})} =0\, .
    \end{equation}
    Then we have $\mathbf{d}_k^{(1)}\leq \delta_k\boldsymbol{\sigma}(\Sbf_k^{(1)})$ for some non-negative sequence $\delta_k\to 0$. We can thus estimate
    \begin{equation}\label{e:balancing-induction-3}
    \begin{aligned}
    \hat{\Ebf}(T_k,\Sbf_k^{(1)},\Bbf_1) & \leq \hat{\mathbf{E}}(T_k,\Sbf_k,\Bbf_1) + C(\mathbf{d}_k^{(1)})^2\\
    &\leq \eps_k^2\boldsymbol{\sigma}(\Sbf_k)^2 + C\delta_k^2 \boldsymbol{\sigma}(\Sbf_k^{(1)})^2 = (\eps_k^2 + C\delta_k^2)\boldsymbol{\sigma}(\Sbf_k^{(1)})^2\, , 
    \end{aligned}
    \end{equation}
    where $C = C(m,n,Q)$ and we have used that $\boldsymbol{\sigma}(\Sbf_k)\leq \boldsymbol{\sigma}(\Sbf_k^{(1)})$, since $\Sbf_k^{(1)}\subset \Sbf_k$; this inclusion also gives 
    $$\hat{\Ebf}(\Sbf_k^{(1)},T_k,\Bbf_1)\leq \hat{\Ebf}(\Sbf_k,T_k,\Bbf_1) \leq \eps_k^2\boldsymbol{\sigma}(\Sbf_k)^2\leq \eps_k^2\boldsymbol{\sigma}(\Sbf_k^{(1)})^2\, .$$
    Thus, for all $k$ sufficiently large, we will have that hypothesis \eqref{e:evensmallerexcess} holds with the smallness threshold $\eps_1^*$. This allows us to conclude, by the induction hypothesis, that $\Sbf_k^{(1)}$ is $C$-balanced for some $C = C(m,n,\bar{n},N)$ (in fact, $C = C_1$). 

    Now, we know from \eqref{e:balancing-induction-2} that $\mathbf{d}_k^{(1)}\to 0$. Analogously to case (b) in the proof of Proposition \ref{p:balancing-3}, we perform a blow-up. This time, however, we will be normalizing by $\mathbf{d}_k$ and building graphical approximations over the (balanced) cones $\Sbf_k^{(1)}$. Indeed, \eqref{e:balancing-induction-3} tells us that the hypotheses of Proposition \ref{p:refined} (and therefore Propositions \ref{p:coherent} and \ref{p:first-blow-up}) hold for $T_k$ relative to the balanced cone $\Sbf_k^{(1)}$ for all $k$ sufficiently large. Thus, $T_k$ splits into a sum of disjoint currents, denoted by $T_k^j$, near each plane $\alpha^k_j\subset \Sbf_k^{(1)}$. 

    Now choose any sequences $\varsigma_k\downarrow 0$ and $\sigma_k\downarrow 0$. Passing to a subsequence, we may apply Proposition \ref{p:first-blow-up} with $\sigma_k,\varsigma_k$ in place of $\sigma$, $\varsigma$ to $T_k$ and $\Sbf_k^{(1)}$ to find multi-valued Lipschitz maps $u^j_k$ over the sets $R^k_i:= (R\setminus B_{\sigma_k}(V))\cap \alpha_i^k$ for $i\in I(1)$ with the property that, if $v^i_k := \hat{\Ebf}(T_k,\Sbf_k^{(1)},\Bbf_1)^{-1/2}u_k^i$, then
    \begin{equation}\label{e:balancing-induction-4}
    d_{W^{1,2}}(v^i_k,w^i_k)\leq\varsigma_k
    \end{equation}
    where $w^i_k$ are Dir-minimizing maps on $R^k_i$. So, if we set $\bar{u}^k_i:= (\mathbf{d}_k^{(1)})^{-1}u_i^k$, then estimating similarly to \eqref{e:balancing-induction-3}, since $\boldsymbol{\sigma}(\Sbf_k)\leq \mathbf{d}_k^{(1)}$, we have
    $$\hat{\Ebf}(T_k,\Sbf_k^{(1)},\Bbf_1)\leq (\eps_k^2+C)(\mathbf{d}_k^{(1)})^2$$
    meaning that we can blow-up via rescaling by $\mathbf{d}_k^{(1)}$ instead of $\hat{\Ebf}(T_k,\Sbf_k^{(1)},\Bbf_1)^{-1/2}$ (see the remark after Proposition \ref{p:first-blow-up}). Indeed, the estimates \eqref{e:graphicality-1}--\eqref{e:non-graphicality}, combined with the conclusion (ii) of Proposition \ref{p:first-blow-up}, \eqref{e:balancing-induction-4}, and a diagonal argument give that $\bar{u}^k_i\to \bar{u}^\infty_j$ locally strongly in $W^{1,2}(\Bbf_1(0)\setminus V)$ and locally uniformly in $\Bbf_1(0)\setminus V$, for some $\bar{u}^\infty_j$ which is Dir-minimizing in $\Bbf_{1-\eps}(0)\setminus B_\rho(V)$ for every $0<\eps,\rho<1/2$. Note also that, if for each $i$, we denote by $L_i^k$ the linear maps parameterizing the planes $\alpha^k_i\in \Sbf_k$ over the closest plane in $\Sbf_k^{(1)}$, then we necessarily have that $|L^k_i|\leq C\mathbf{d}_k^{(1)}$ for some dimensional constant $C$. So, if we write
    $$\bar{L}_i^k:= (\mathbf{d}_k^{(1)})^{-1}L_i^k$$
    then we can pass to a further subsequence to ensure also that $\bar{L}_i^k\to \bar{L}_i$ for some linear map $\bar{L}_i$ over some plane $\alpha_i^\infty$ in $\Sbf^{(1)}_\infty:= \lim_{k\to\infty}\Sbf^{(1)}_k$ (note that $\Sbf^{(1)}_\infty$ need not be a single plane). However, as $|L^k_i-L^k_j|\geq c\boldsymbol{\sigma}(\Sbf_k)$ for some dimensional constant $c$ (by definition of $\boldsymbol{\sigma}(\Sbf_k)$) whenever $L^k_i,L^k_j$ are defined over the same plane, from \eqref{e:balancing-induction-1} we know that $|\bar{L}_i-\bar{L}_j|\geq c\eta>0$, i.e. $\bar{L}_i\neq\bar{L}_j$ for $i\neq j\in I(0)$. Note now that
    $$\int\dist^2(x,\Sbf_k)\, d\|\mathbf{G}_{u^k_j}\|(x)\leq C\hat{\Ebf}(T_k,\Sbf_k,\Bbf_1) + C\left(\Abf_k^2+\hat{\Ebf}(T_k,\Sbf_k,\Bbf_1)\right)^{1+\gamma} = o\left((\mathbf{d}_k^{(1)})^2\right)$$
    where $\gamma = \gamma(Q,m,n,\bar{n})>0$ is as in Proposition \ref{p:Lipschitz-1}. Thus, dividing both sides by $(\mathbf{d}_k^{(1)})^2$ and taking $k\to\infty$, we see that necessarily the graph of $\bar{u}_j^\infty$ is supported on the union of all the graphs of the maps $\bar{L}_i$ associated to the plane $\alpha_j$. Letting $J_j$ denote this collection of indices $i$ for $j\in I(1)$, on $\Omega = (\Bbf_1\cap \alpha_j)\setminus V$ we have
    $$\bar{u}_j^\infty = \sum_{i\in J_j}Q^j_i\llbracket \bar{L}_i\rrbracket$$
    for some non-negative integers $Q^j_i$ such that $\sum_{i,j}Q^j_i = Q$. Following a similar argument to the proof in Proposition \ref{p:balancing-3}, we can in fact show that $Q_i^j\geq 1$ for each $i,j$. Now, since $\bar{u}^\infty_j$ is Dir-minimizing in $\Bbf_{1-\eps}\setminus B_\rho(V)$ for every $\eps,\rho>0$, we may now apply Proposition \ref{p:remove-spine} to conclude that $\bar{u}^\infty_j$ extends to a Dir-minimizer (which we still denote by $\bar{u}^\infty_j$) on $\Bbf_{1/2}\cap \alpha_j$. If the unbalancing assumption from \eqref{e:notbalanced2} has $1,2\in J_j$ for some $j$, then we can argue as in case (b) of Proposition \ref{p:balancing-3} to arrive at a contradiction. If we have $1\in J_j$ and $2\in J_{j^*}$ for some $j\neq j^*$, we can again argue in the same way, except now using the fact that $\Sbf^{(1)}_k$ is balanced, and so the planes $\alpha^k_{j}$, $\alpha^k_{j^*}$ are balanced. Thus, in this situation we arrive at a contradiction.
    
   To summarize, we have now shown that \eqref{e:balancing-induction-2} cannot hold. From Lemma \ref{l:algorithm} (as $\delta=1$) we know that $\mathbf{d}_k^{(1)}\leq \boldsymbol{\sigma}(\Sbf_k^{(1)})$, and so we may therefore pass to a subsequence to ensure that
   \begin{equation}\label{e:balancing-induction-5}
       \mathbf{d}_k^{(1)}\geq c_1\boldsymbol{\sigma}(\Sbf^{(1)}_k)
   \end{equation}
   for all $k$ and for some constant $c_1>0$.

   To progress, we now move to the next layer, namely the cones $\Sbf^{(2)}_k$. We know from \eqref{e:balancing-induction-1} that $\boldsymbol{\sigma}(\Sbf_k^{(1)})\geq\eta\mathbf{d}^{(2)}_k$. Now, if we have
   \begin{equation}\label{e:balancing-induction-6}
       \lim_{k\to\infty}\frac{\mathbf{d}^{(2)}_k}{\boldsymbol{\sigma}(\Sbf_k^{(2)})} = 0
   \end{equation}
   we argue that we can follow an analogous argument to that above when \eqref{e:balancing-induction-2} held, except now performing a blow-up of $T_k$ relative to $\Sbf^{(2)}_k$. Indeed, if \eqref{e:balancing-induction-6} holds, then we have $\mathbf{d}^{(2)}_k\leq \delta_k^\prime \boldsymbol{\sigma}(\Sbf_k^{(2)})$, for some non-negative sequence $\delta_k^\prime\to 0$. We can estimate as in \eqref{e:balancing-induction-3} to find
   \begin{align*}
    \hat{\Ebf}(T_k,\Sbf_k^{(2)},\Bbf_1) &\leq \hat{\Ebf}(T_k,\Sbf_k,\Bbf_1) + C(\mathbf{d}^{(2)}_k)^2 \leq \eps_k^2\boldsymbol{\sigma}(\Sbf_k)^2 + C(\delta_k^\prime)^2\boldsymbol{\sigma}(\Sbf_k^{(2)})^2\\
    &\leq (\eps_k^2+C(\delta_k^\prime)^2)\boldsymbol{\sigma}(\Sbf_k^{(2)})^2
    \end{align*}
   where again we have used $\boldsymbol{\sigma}(\Sbf_k)\leq \boldsymbol{\sigma}(\Sbf_k^{(2)})$ since $\Sbf_k^{(2)}\subset\Sbf_k$. Therefore, once again we get that the hypothesis \eqref{e:evensmallerexcess} holds with the smallness threshold $\eps_1^*$ for all $k$ sufficiently large, and so by the induction hypothesis we have that $\Sbf^{(2)}_k$ is $C$-balanced for some $C = C(m,n,\bar{n},N)$. The only aspect of the above blow-up procedure which we need to check in this case is the separation of all the rescaled linear functions of the planes in $\Sbf_k$ over $\Sbf_k^{(2)}$. As we are rescaling our approximations by $\mathbf{d}_k^{(2)}$, we therefore need to show that $\frac{\boldsymbol{\sigma}(\Sbf_k)}{\mathbf{d}_k^{(2)}}$ has a uniform lower bound. Indeed, combining \eqref{e:balancing-induction-1} and \eqref{e:balancing-induction-5}, we have
   \begin{equation}\label{e:balancing-induction-7}
   \frac{\boldsymbol{\sigma}(\Sbf_k)}{\mathbf{d}^{(2)}_k} = \frac{\boldsymbol{\sigma}(\Sbf_k)}{\mathbf{d}_k^{(1)}}\cdot \frac{\mathbf{d}_k^{(1)}}{\boldsymbol{\sigma}(\Sbf_k^{(1)})}\cdot\frac{\boldsymbol{\sigma}(\Sbf^{(1)}_k)}{\mathbf{d}_k^{(2)}} \geq \eta \cdot c_1\cdot \eta = \eta^2 c_1>0\, .
   \end{equation}
   This tells us that we may repeat the blow-up argument as above, with $\mathbf{d}_k^{(1)}$ replaced by $\mathbf{d}_k^{(2)}$ and $\Sbf^{(1)}_k$ replaced with $\Sbf^{(2)}_k$, to contradict the lack of balancing \eqref{e:notbalanced2} of $\Sbf_k$ when \eqref{e:balancing-induction-6} holds.

   Proceeding now inductively, we see that we must be able to find a subsequence with
   $$\mathbf{d}_k^{(l)}\geq c_l \boldsymbol{\sigma}(\Sbf_k^{(l)})$$
   for all $l=1,\dotsc,\kappa$ and some constants $c_l>0$. However, we can now show that this is in direct contradiction to the assumption \eqref{e:incomparable-max-min}. Indeed, an identical calculation from \eqref{e:balancing-induction-7} gives 
   $$\frac{\boldsymbol{\sigma}(\Sbf_k)}{\boldsymbol{\mu}(\Sbf_k)} = \frac{\boldsymbol{\sigma}(\Sbf^{(\kappa)}_k)}{\boldsymbol{\mu}(\Sbf_k)}\cdot\prod_{\ell=0}^{\kappa-1} \left(\frac{\boldsymbol{\sigma}(\Sbf_k^{(\ell)})}{\mathbf{d}_k^{(\ell+1)}}\cdot\frac{\mathbf{d}_k^{(\ell+1)}}{\boldsymbol{\sigma}(\Sbf_k^{(\ell+1)})}\right) \geq \frac{\boldsymbol{\sigma}(\Sbf^{(\kappa)}_k)}{\boldsymbol{\mu}(\Sbf_k)}\cdot \eta^{\kappa}\prod_{l=1}^\kappa c_l\, ,$$
   where we recall $\Sbf_k^{(0)}=\Sbf_k$. But we know from Lemma \ref{l:algorithm}(i) and (ii) that $\boldsymbol{\mu}(\Sbf_k) = \boldsymbol{\mu}(\Sbf_k^{(\kappa)})$ and $\eta\boldsymbol{\mu}(\Sbf_k^{(\kappa)})\leq \boldsymbol{\sigma}(\Sbf_k^{\kappa})$. Hence, the above gives
   $$\frac{\boldsymbol{\sigma}(\Sbf_k)}{\boldsymbol{\mu}(\Sbf_k)} \geq \eta^{\kappa+1}\prod^\kappa_{l=1}c_l > 0\, ,$$
   thus reaching the desired contradiction to \eqref{e:incomparable-max-min}, and completing the proof.
\end{proof} 



\section{Reduction of the main decay theorem}\label{s:reduction}

In this section, we reduce the proof of Theorem \ref{c:decay} to an a priori much weaker decay statement; Theorem \ref{t:weak-decay} below. For this we will utilize the balancing result from Proposition \ref{p:balancing}. Recall that for a cone $\Sbf\in \mathscr{C}(Q)$, we write
$$\boldsymbol{\sigma}(\Sbf) := \min_{i<j}\dist(\alpha_i\cap\Bbf_1,\alpha_j\cap\Bbf_1),\ \ \ \ \boldsymbol{\mu}(\Sbf) := \max_{i<j}\dist(\alpha_i\cap\Bbf_1,\alpha_j\cap \Bbf_1).$$
From now on, we make the following assumption regarding the (balancing) constant $M$:

\begin{assumption}\label{a:choose-M}
The constant $M\geq 1$ is chosen so that the application Proposition \ref{p:balancing} yields an $M$-balanced cone. Thus, $M = M(Q,m,n,\bar{n})$.
\end{assumption}

\begin{theorem}[Weak Excess Decay Theorem]\label{t:weak-decay}
    Fix $Q,m,n,\bar{n}$ as before, and let $M\geq 1$ be as in Assumption \ref{a:choose-M}. Fix also $\varsigma_1>0$. Then, there are constants $\varepsilon_1 = \varepsilon_1(Q,m,n,\bar{n},\varsigma_1)\in (0,1/2]$, $r^1_1 = r^1_1(Q,m,n,\bar{n},\varsigma_1)\in (0,1/2]$ and $r^2_1 = r^2_1(Q,m,n,\bar{n},\varsigma_1)\in (0,1/2]$, such that the following holds. Suppose that
    \begin{enumerate}
        \item [\textnormal{(i)}] $T$ and $\Sigma$ are as in Assumption \ref{a:main};
        \item [\textnormal{(ii)}] $\|T\|(\Bbf_1)\leq (Q+\frac{1}{2})\omega_m$;
        \item [\textnormal{(iii)}] There is $\Sbf\in \mathscr{C}(Q,0)$ which is $M$-balanced, such that
        \begin{equation}\label{e:smallness-weak}
        \mathbb{E}(T,\Sbf,\Bbf_1)\leq \eps_1^2\boldsymbol{\sigma}(\Sbf)^2
        \end{equation}
        and
        \begin{equation}\label{e:no-gaps-weak}
            \Bbf_{\eps_1}(\xi)\cap \{p:\Theta(T,p){\geq\ } Q\}\neq\emptyset \qquad \forall\xi\in V(\Sbf)\cap \Bbf_{1/2}\, ;
        \end{equation}
        \item [\textnormal{(iv)}] $\mathbf{A}^2 \leq \eps_1^2 \mathbb{E} (T, \tilde{\Sbf}, \Bbf_1)$ for every $\tilde{\Sbf} \in \mathscr{C} (Q,0)$.
    \end{enumerate}
    Then, there is $\Sbf^\prime\in \mathscr{C}(Q,0)\setminus\mathscr{P}(0)$ such that for some $i\in\{1,2\}$ we have
    \begin{equation}\label{e:weak-decay}
    \mathbb{E}(T,\Sbf^\prime,\Bbf_{r_1^i}) \leq \varsigma_1\mathbb{E}(T,\Sbf,\Bbf_1)\, .
    \end{equation}
\end{theorem}

The above theorem would appear to be considerably weaker than Theorem \ref{c:decay}; not only are we assuming the cone $\Sbf$ is balanced, but there is also a significant difference between the smallness assumption \eqref{e:smallness} and \eqref{e:smallness-weak}. In fact, \eqref{e:smallness-weak} implies the second inequality in \eqref{e:smallness} for $\eps_1$ suitably small {(indeed, Proposition \ref{p:Lipschitz-1} then gives $\boldsymbol{\sigma}(\Sbf)^2\leq C\mathbf{E}^p(T,\Bbf_1)$)}, whilst {when the cone $\Sbf$ arises from Proposition \ref{p:balancing} (as it will)} the second inequality in \eqref{e:smallness} {is} equivalent, up to constants, to {(from Proposition \ref{p:balancing}(d))}
\begin{equation}\label{e:smallness-strong}
\mathbb{E} (T, \mathbf{S}, \Bbf_1) \leq \varepsilon_0^2 \boldsymbol{\mu} (\mathbf{S})^2\, .
\end{equation}

In order to show that in fact the above seemingly weaker statement implies Theorem \ref{c:decay}, the idea is to first show that Theorem \ref{t:weak-decay} implies a multiple radii decay version of Theorem \ref{c:decay}, by first removing the necessary planes in the cone $\Sbf$ to reach a balanced cone, and then inductively removing additional planes to obtain a cone (that is still balanced) for which assumption \eqref{e:smallness-weak} holds.
In doing so we reach an intermediate statement, namely an excess decay with finitely many scales, from which we then derive Theorem \ref{c:decay}. This is the following, the idea of which is similar to that seen in \cite{W14_annals}*{Section 13}:

\begin{proposition}[Multiple Radii Excess Decay]\label{p:decay-multiple}
    Fix $Q,m,n,\bar{n}$ as before, and let $M\geq 1$ be as in Assumption \ref{a:choose-M}. Let $\bar{N}:= Q(Q-1)$. Fix $\varsigma_2>0$. Then, there exist positive constants $\varepsilon_2,r_1,\dotsc,r_{\bar{N}}\leq \frac{1}{2}$, depending only on $Q,m,n,\bar{n}$, and $\varsigma_2$, such that the following holds. Suppose that $T,\Sigma,\Sbf$ are as in Theorem \ref{c:decay}, i.e.:
    \begin{enumerate}
        \item [\textnormal{(i)}] $T$ and $\Sigma$ are as in Assumption \ref{a:main} and $\Bbf_1\subset\Omega$;
        \item [\textnormal{(ii)}] $\|T\|(\Bbf_1)\leq (Q+\frac{1}{2})\omega_m$;
        \item [\textnormal{(iii)}] There is $\Sbf\in \mathscr{C}(Q,0)\setminus\mathscr{P}(0)$ such that
        $$\mathbb{E}(T,\Sbf,\Bbf_1) \leq \eps_2^2\mathbf{E}^p(T,\Bbf_1),$$
        and that
        $$\Bbf_{\eps_2}(\xi)\cap \{p:\Theta(T,p)\geq Q\}\neq\emptyset \qquad \forall \xi\in V(\Sbf)\cap \Bbf_{1/2}\, ;$$
        \item [\textnormal{(iv)}] $\mathbf{A}^2 \leq \eps_2^2 \mathbb{E} (T, \tilde{\Sbf}, \Bbf_1)$ for every $\tilde{\Sbf} \in \mathscr{C} (Q,0)$.
    \end{enumerate}
    Then, there is a $\Sbf^\prime\in \mathscr{C}(Q,0)\setminus\mathscr{P}(0)$ and $i\in \{1,\dotsc,\bar{N}\}$ such that:
    \begin{equation}\label{e:multi-decay}
    \mathbb{E}(T,\Sbf^\prime,\Bbf_{r_i}) \leq \varsigma_2\mathbb{E}(T,\Sbf,\Bbf_1)\, .
    \end{equation}
\end{proposition}

We will then be able to show that Proposition \ref{p:decay-multiple} implies Theorem \ref{c:decay}, which therefore demonstrates that the weak excess decay from Theorem \ref{t:weak-decay} implies the stronger excess decay statement from Theorem \ref{c:decay}. Let us begin by showing that Theorem \ref{t:weak-decay} implies Proposition \ref{p:decay-multiple}.

\begin{proof}[Proof of Proposition \ref{p:decay-multiple} from Theorem \ref{t:weak-decay}]
Fix $T,\Sigma,\Sbf$, and $\varsigma_2$ as in the statement of Proposition \ref{p:decay-multiple}.

Firstly, because of Proposition \ref{p:balancing}, if we take $\eps_2 \leq \eps_0$ where $\eps_0 = \eps_0(Q,m,n,\bar{n})>0$ is as in Proposition \ref{p:balancing}, then we can find a cone $\tilde{\Sbf}$ which is $M$-balanced and obeys $\mathbb{E}(T,\tilde{\Sbf},\Bbf_1)\leq M\mathbb{E}(T,\Sbf,\Bbf_1)$, where $M = M(Q,m,n,\bar{n})$ is as in Assumption \ref{a:choose-M}. Thus, we have
$$\mathbb{E}(T,\tilde{\Sbf},\Bbf_1) \leq M\mathbb{E}(T,\Sbf,\Bbf_1) \leq M\eps_2^2\mathbf{E}^p(T,\Bbf_1)$$
In particular, this coupled with the other estimates in Proposition \ref{p:balancing} gives that if we can prove the result for $\tilde{\Sbf}$, then the result follows for $\Sbf$, up to changing $\varsigma_2$ by a factor of $M$. Thus, we may without loss of generality assume that $\Sbf$ is $M$-balanced {and $M^{-1}\Ebf^p(T,\Bbf_1)\leq\boldsymbol{\mu}(\Sbf)\leq M\Ebf^p(T,\Bbf_1)$}.

We would now like to apply Theorem \ref{t:weak-decay}. However, a priori, it may be the case that \eqref{e:smallness-weak} does not hold for $\Sbf$, as we are merely assuming that $\Ebb(T,\Sbf,\Bbf_1)\leq \eps_2^2\Ebf^p(T,\Bbf_1)$. The remaining part of the argument deals with this difficulty; the price to pay is that the decay might occur at {one of finitely many scales}, {but nonetheless the number of possible scales which are needed belong to a set of controlled cardinality.}

We start by introducing, for every integer $k\in \{2,\dotsc,Q\}$, functions $\eps_{(k)}(s)$, $r^1_{(k)}(s)$, and $r^2_{(k)}(s)$ of a parameter $s>0$ as follows:
\begin{itemize}
    \item If one takes $Q=k$ and $\varsigma_1 = s$ in Theorem \ref{t:weak-decay}, then $\eps_{{(k)}}(s):= \eps_1(k,m,n,\bar{n},s)$, where the $\eps_1$ is the constant from Theorem \ref{t:weak-decay} with this choice of $Q$, $\varsigma_1$, and $r_{(k)}^1(s):= r_1^1(k,m,n,\bar{n},s)$, $r_{(k)}^2(s):= r_1^2(k,m,n,\bar{n},s)$ are the corresponding radii, where $r^1_1$ and $r^2_1$ are as in Theorem \ref{t:weak-decay}.
\end{itemize}
In particular, $k$ here is the number of planes forming the cone $\Sbf$ to which Theorem \ref{t:weak-decay} is applied, whilst $s$ is the specified decay factor.

We now denote by $N$ the number of planes forming $\Sbf$. If $N=2$, then $\boldsymbol{\mu}(\Sbf) = \boldsymbol{\sigma}(\Sbf)$, and so from \eqref{e:smallness-strong} we get that the assumptions of Theorem \ref{t:weak-decay} hold (up to $\eps_2$ changing by a constant), and so if we take $\eps_2$ smaller than $\eps_{(2)}(\varsigma_2)$, we see that we can apply Theorem \ref{t:weak-decay}. The corresponding radii from this application of Theorem \ref{t:weak-decay} will be denoted by $r^1_{2,2}$ and $r^2_{2,2}$ {(the first subscript here denotes the number of planes, $N$, in the cone $\Sbf$, whilst the second subscript denotes the number of planes left at the point of application of Theorem \ref{t:weak-decay} -- see below)}, which are simply $r_{(2)}^1(\varsigma_2)$ and $r_{(2)}^2(\varsigma_2)$. In particular, the proof is complete in the case $Q=2$.

From now on we therefore assume $N\geq 3$. We may take $\eps_2\leq \eps_{(N)}(\varsigma_2)$. We then consider two cases. If we have
$$\mathbb{E}(T,\Sbf,\Bbf_1)\leq \eps_{(N)}(\varsigma_2)^2\boldsymbol{\sigma}(\Sbf)^2$$
then we can simply apply Theorem \ref{t:weak-decay} to get the desired decay at {one of the two radii $r_{N,N}^1:= r^1_{(N)}(\varsigma_2)$, $r_{N,N}^2:= r^2_{(N)}(\varsigma_2)$.} We can therefore assume that
\begin{equation}\label{e:second-pruning-1}
    \mathbb{E}(T,\Sbf,\Bbf_1)>\eps_{(N)}(\varsigma_2)^2\boldsymbol{\sigma}(\Sbf)^2.
\end{equation}
We now follow the same idea in the argument for the Pruning Lemma (Lemma \ref{l:pruning}). Let $\Sbf = \alpha_1\cup\cdots\cup\alpha_N$, where the $\alpha_i$ are distinct $m$-dimensional planes. We can assume, upon relabelling, that
$$\dist(\alpha_1\cap\Bbf_1,\alpha_2\cap\Bbf_1) = \boldsymbol{\sigma}(\Sbf)$$
and that there are two indices $i_*,j_*\neq 1$ with
$$\dist(\alpha_{i_*}\cap \Bbf_1,\alpha_{j_*}\cap \Bbf_1) = \boldsymbol{\mu}(\Sbf).$$
We now remove the plane $\alpha_1$ and consider $\Sbf_{N-1}:= \alpha_2\cup\cdots\cup \alpha_N$; {note that $\Sbf_{N-1}$ is still $M$-balanced, $V(\Sbf_{N-1}) = V(\Sbf)$, $\boldsymbol{\mu}(\Sbf) = \boldsymbol{\mu}(\Sbf_{N-1})$, and
$$\dist^2(\Sbf\cap\Bbf_1,\Sbf_{N-1}\cap\Bbf_1) \leq \boldsymbol{\sigma}(\Sbf)^2\leq \eps_{(N)}(\varsigma_2)^{-2}\mathbb{E}(T,\Sbf,\Bbf_1);$$
these conditions imply that if we can show the result for $\Sbf_{N-1}$, the corresponding statement for $\Sbf$ follows.}

Now observe that 
$$\mathbb{E}(T,\Sbf_{N-1},\Bbf_1) \leq C_0\mathbb{E}(T,\Sbf,\Bbf_1) + C_0\boldsymbol{\sigma}(\Sbf)^2 \leq C_0\left(1+ \eps_{(N)}(\varsigma_2)^{-2}\right)\mathbb{E}(T,\Sbf,\Bbf_1).$$
for some constant $C_0 = C_0(Q,m,n,\bar{n})\geq 1$. Let us write $C^*_1:= C_0 (1+\eps_{(N)}(\varsigma_2)^{-2})$, and so $C^*_1 = C^*_1(N,m,n,\bar{n},\varsigma_2)>1$, and so the above inequality is just
$$\mathbb{E}(T,\Sbf_{N-1},\Bbf_1) \leq C_1^*\mathbb{E}(T,\Sbf,\Bbf_1).$$
Therefore, if{
\begin{equation}\label{e:pruning-condition-2}
    \mathbb{E}(T,\Sbf,\Bbf_1) \leq \eps_{(N-1)}(\varsigma_2/C_1^*)^2\sigma(\Sbf_{N-1})^2
\end{equation}
and if we take $\eps_2 \leq \eps_{(N-1)}(\varsigma_2/C_1^*)/\sqrt{C_1^*}$ {(which of course we may)},} then the above gives that we can apply Theorem \ref{t:weak-decay} to the triple $T,\Sigma$, and $\mathbf{S}_{N-1}$ to conclude the desired decay statements in Proposition \ref{p:decay-multiple}, for one of the radii
$$r^i_{N,N-1} := r^i_{(N-1)}(\varsigma_2/C_1^*),\ \ \ \ i\in \{1,2\}.$$
If \eqref{e:pruning-condition-2} is not true, i.e. if we cannot apply Theorem \ref{t:weak-decay} as above, then we must have
\begin{equation}\label{e:second-pruning-2}
    \mathbb{E}(T,\Sbf,\Bbf_1) > \eps_{(N-1)}(\varsigma_2/C_1^*)^2\boldsymbol{\sigma}(\Sbf_{N-1})^2.
\end{equation}
We now wish to apply the above procedure inductively; indeed, suppose we have performed the above procedure $K$ times, and at each stage we have not been able to apply Theorem \ref{t:weak-decay}. Thus, upon relabelling the planes in $\Sbf$, for $k=0,1,\dotsc,K$ we have cones $\Sbf_{N-k} = \alpha_{k+1}\cup\cdots\cup \alpha_N$ which are all $M$-balanced and obey $V(\Sbf_{N-k}) = V(\Sbf)$, $\boldsymbol{\mu}(\Sbf_{N-k}) = \boldsymbol{\mu}(\Sbf)$. We also know that $\Sbf_{N-(k+1)}$ is formed from $\Sbf_{N-k}$ by removing a single plane, namely $\alpha_{k+1}$, which is a plane in $\Sbf_{N-k}$ achieving the minimal separation $\boldsymbol{\sigma}(\Sbf_{N-k})$ and such that there are two other planes in $\Sbf_{N-k}$ which achieve the maximal separation $\boldsymbol{\mu}(\Sbf_{N-k})$ (this can be done as long as $N-k\geq 3$, i.e. $0\leq k\leq N-3$). The criterion along this sequence is that
$$\mathbb{E}(T,\Sbf,\Bbf_1)>\eps_{(N-k)}(\varsigma_2/C_k^*)^2\boldsymbol{\sigma}(\Sbf_{N-k})^2$$
where $C^*_k = C_0(C^*_{k-1}+\eps_{(N-(k-1))}(\varsigma_2/C_{k-1}^*)^{-2})$ is defined inductively for $k\geq 1$, with $C^*_0 = 1$ and $C_0 = C_0(Q,m,n,\bar{n})$ is as above, and
\[
\mathbb{E}(T,\Sbf_{N-k},\Bbf_1)\leq C_k^*\mathbb{E}(T,\Sbf,\Bbf_1)\, ;
\]
for this procedure we need to assume $\eps_2\leq \eps_{(N-k)}(\varsigma_2/C^*_{k})/\sqrt{C^*_k}$ for each $k=0,1,\dotsc,K$ (which of course we may). We also therefore have
$$\dist^2(\Sbf_{N-(k+1)}\cap\Bbf_1,\Sbf_{N-k}\cap\Bbf_1)\leq \boldsymbol{\sigma}(\Sbf_{N-k})^2 \leq \eps_{(N-k)}(\varsigma_2/C_k^*)^{-2}\mathbb{E}(T,\Sbf,\Bbf_1)$$
and thus
\begin{align*}
\mathbb{E}(T,\Sbf_{N-(k+1)},\Bbf_1) & \leq C\mathbb{E}(T,\Sbf_{N-k},\Bbf_1) + C\boldsymbol{\sigma}(\Sbf_{N-k})^2\\
& \leq C_0\mathbb{E}(T,\Sbf_{N-k},\Bbf_1) + C_0\eps_{(N-k)}(\varsigma_2/C_k^*)^{-2}\mathbb{E}(T,\Sbf,\Bbf_1)
\end{align*}
which in particular gives the inductive definition of $C^*_k$. Note that this procedure can occur at most $N-2$ times, and thus all the constants and smallness assumptions here only depend on $N,m,n,\bar{n}$, and $\varsigma_2$, and so certainly we can choose $\eps_2 = \eps_2(N,m,n,\bar{n},\varsigma_2)$ small enough so that the above procedure is guaranteed whenever we cannot apply Theorem \ref{t:weak-decay}.

Hence, for this choice of $\eps_2$, if we are unable to apply Theorem \ref{t:weak-decay} at all steps to conclude, the above process must eventually terminate when the cone is formed of exactly two planes, i.e. at $\Sbf_2$. Since we are unable to apply Theorem \ref{t:weak-decay} to $\Sbf_2$, we must have
\begin{equation}\label{e:second-pruning-N-2}
    \mathbb{E}(T,\Sbf,\Bbf_1)>\eps_{(2)}(\varsigma_2/C_{N-2}^*)^2\boldsymbol{\sigma}(\Sbf_2)^2.
\end{equation}

However, as $\Sbf_2$ consists of two planes we know that $\boldsymbol{\sigma}(\Sbf_2) = \boldsymbol{\mu}(\Sbf_2)$, and moreover since by construction we know that $\boldsymbol{\mu}(\Sbf_2) = \boldsymbol{\mu}(\Sbf)$, we see that \eqref{e:second-pruning-N-2} is equivalent to
\begin{equation}\label{e:second-pruning-N-2-mu}
    \mathbb{E}(T,\Sbf,\Bbf_1)>\eps_{(2)}(\varsigma_2/C_{N-2}^*)^2\boldsymbol{\mu}(\Sbf)^2.
\end{equation}

However, we are assuming (see Proposition \ref{p:decay-multiple}(iii)) that $\mathbb{E}(T,\Sbf,\Bbf_1)\leq \eps_2^2\mathbf{E}^p(T,\Bbf_1)$, which from \eqref{e:smallness-strong} we know implies
$$\mathbb{E}(T,\Sbf,\Bbf_1)\leq C\eps_2^2\boldsymbol{\mu}(\Sbf)^2$$
for some $C = C(Q,m,n,\bar n)>0$. Thus, if we ensure that $C\eps_2^2 <\frac{1}{2}\eps_{(2)}(\varsigma_2/C_{N-2}^*)^2$ also, then this would give a contradiction to \eqref{e:second-pruning-N-2-mu}; hence, for such a choice of $\eps_2$ we see that once we reach $\Sbf_2$, we in fact \emph{must} be able to apply Theorem \ref{t:weak-decay} in the above procedure.

We hence conclude that, for any fixed $N\leq Q$, if the threshold $\eps_2$ satisfies all the (finitely many) smallness conditions above, then we have that Proposition \ref{p:decay-multiple} holds for one of the radii $r^1_{N,N}$, $r^2_{N,N}$, $r^1_{N,N-1},r^2_{N,N-1},\dotsc,r^1_{N,2},r^1_{N,2}$, of which there are $2(N-1)$ radii. In particular, if $\eps_2$ satisfies all of these smallness conditions as we range $N$ over $\{2,\dotsc,Q\}$, then the above argument is valid for any cone $\Sbf\in \mathscr{C}(Q,0)$ comprised of $N\in \{2,\dotsc,Q\}$ planes, i.e. for every $\Sbf\in \mathscr{C}(Q,0)\setminus\mathscr{P}$. Thus, Proposition \ref{p:decay-multiple} holds for every $\Sbf\in \mathcal{C}(Q,0)\setminus\mathscr{P}$ for one of the radii in the collection $\{r^i_{j,k}\}$, where $i\in \{1,2\}$ and $2\leq k\leq j\leq Q$. Thus, the number of possible decay scales is
$$\sum^Q_{N=2}2(N-1) = Q(Q-1) = \bar{N}$$
as claimed in the statement of Proposition \ref{p:decay-multiple}. This completes the proof.
\end{proof}

Now we will prove Theorem \ref{c:decay} from Proposition \ref{p:decay-multiple}, thus reducing the proof of Theorem \ref{c:decay} to proving the weaker version, Theorem \ref{t:weak-decay}. First, we observe a corollary of the Proposition \ref{p:balancing} and Proposition \ref{p:refined}, the purpose of which is to control the excess of rescalings; it will be a useful tool in the proof of the reduction.

\begin{corollary}\label{c:control-large-interval}
    Suppose $\tilde{T}$ and $\tilde{\Sigma}$ satisfy Assumption \ref{a:main}, let $\tilde{\Sbf}\in\mathscr{C}(Q,0)\setminus\mathscr{P}(0)$, and write $\tilde{\Abf}^2$ for the supremum norm of the second fundamental form of $\tilde{\Sigma}$. Fix also a radius $\bar{r}\in (0,1]$. Then, there are constants $\tilde{\eps} = \tilde{\eps}(Q,m,n,\bar{n},\bar{r})>0$ and $\tilde{C} = \tilde{C}(Q,m,n,\bar{n},\bar{r})>0$ such that the following holds. If we have
    \[
    \mathbb{E}(\tilde{T},\tilde{\Sbf},\Bbf_1) \leq \tilde{\eps}^2\mathbf{E}^p(\tilde{T},\Bbf_1)
    \]
    and 
\[
\tilde{\mathbf{A}}^2 \leq \tilde{\eps}^2 \mathbb{E} (\tilde{T}, \bar{\mathbf{S}}, \Bbf_1) \qquad \forall \bar{\Sbf}\in \mathscr{C} (Q, 0)\, ,
\]
    then there is $\Sbf^\prime\in\mathscr{C}(Q,0)\setminus\mathscr{P}(0)$ such that
    \begin{equation}\label{e:control-large-interval}
    \mathbb{E}(\tilde{T},\Sbf^\prime,\Bbf_r) \leq \tilde{C}\mathbb{E}(\tilde{T},\tilde{\Sbf},\Bbf_1)\ \ \ \ \forall r\in [\bar{r},1].
    \end{equation}
\end{corollary}

\begin{proof}
    Observe that for any $r\in [\bar{r},1]$ we clearly have
    $$\hat{\mathbf{E}}(\tilde{T},\tilde{\Sbf},\Bbf_r) = r^{-m-2}\int_{\Bbf_r}\dist^2(x,\tilde{\Sbf})\, d\|\tilde{T}\|(x) \leq \bar{r}^{-m-2}\hat{\mathbf{E}}(\tilde{T},\tilde{\Sbf},\Bbf_1).$$
    On the other hand, there is no obvious reason to have a bound of the form
    $$\hat{\mathbf{E}}(\tilde{\Sbf},\tilde{T},\Bbf_r) \leq C(\bar{r})\mathbb{E}(\tilde{T},\tilde{\Sbf},\Bbf_1)$$
    because in the integral defining the reverse excess $\hat\Ebf(\tilde\Sbf, \tilde T,\Bbf_r)$ we omit a tubular neighbourhood of the spine $V(\tilde{\Sbf})$ of radius $ar$, and thus there is no inclusion property of one domain of integration into the other when we are comparing two different scales.

    To get around this problem, first note that, provided $\tilde{\eps}\leq \eps_0$ where $\eps_0 = \eps_0(Q,m,n,\bar{n})$ is as in Proposition \ref{p:balancing}, we can apply Proposition \ref{p:balancing} to assume without loss of generality that $\tilde{\Sbf}$ is $M$-balanced ($M$ as in Assumption \ref{a:choose-M}; all that changes is the two-sided excess increases by a factor of $M = M(Q,m,n,\bar{n})$).
    
    Fix $\delta>0$ {(to be determined later)}. Now, apply the Pruning Lemma (Lemma \ref{l:pruning}) {with this choice of $\delta$ and with $D=\Ebf(\tilde{T},{\tilde{\Sbf}},\Bbf_1)^{1/2}$} to the (balanced) cone $\tilde{\Sbf}$, in exactly the same way as was done in the proof of Proposition \ref{p:balancing}, to yield a new cone $\Sbf^\prime$. The latter is the union of some subset of the planes in $\tilde{\Sbf}$, in particular it is $M$-balanced and obeys $V(\Sbf^\prime) = V(\tilde{\Sbf})$. Moreover it satisfies
    \[
    \mathbb{E}(\tilde{T},\Sbf^\prime,\Bbf_1) \leq C_\delta\mathbb{E}(\tilde{T},\tilde{\Sbf},\Bbf_1)
    \]
    and
    \[
    \mathbb{E}(\tilde{T},\Sbf^\prime,\Bbf_1) \leq C\delta^2\boldsymbol{\sigma}(\Sbf^\prime)^2
    \]
    where here $C_\delta = C_\delta(Q,m,n,\bar{n},\delta)$ and $C = C(Q,m,n,\bar{n})$. In particular, we have
    \[
    \tilde{\Abf}^2 + \mathbb{E}(\tilde{T},\Sbf^\prime,\Bbf_1) \leq C(\tilde{\eps}^2 + \delta^2)\boldsymbol{\sigma}(\Sbf^\prime)^2.
    \]
    Thus, we are in a situation to apply the refined approximation (Proposition \ref{p:refined}) {to $\tilde{T}$, $\Sbf^\prime$,} provided we take $\tilde{\eps},\delta$ sufficiently small. The region we are interested in {building the refined approximation over} is $\Bbf_1\setminus B_{a\bar{r}/2}(V({\tilde{\Sbf}}))$, and so in particular the radius of this neighbourhood of the spine is fixed only depending on $\bar{r}$ and $a$. Thus, by Lemma \ref{l:controls_Whitney}, if we take $\tilde{\eps}$ and $\delta$ smaller than a constant depending only on $Q,m,n,\bar{n},\bar{r}$, we can ensure that $\Bbf_1\setminus B_{a\bar{r}/2}(V({\tilde{\Sbf}}))\subset R^o$, (recall that $R^o$ is the outer region of the refined approximation). In particular, one can apply Proposition \ref{p:refined}; the estimates therein {(summed over the cubes $L\in\Gcal^o\cup\Gcal^c$ that intersect $\Bbf_r\setminus B_{ar}(V({\tilde{\Sbf}})$)} give that
    $$\hat{\mathbf{E}}(\Sbf^\prime,\tilde{T},\Bbf_r) \leq \tilde{C}\mathbb{E}(\tilde{T},\Sbf^\prime,\Bbf_1)$$
    for any $r\in [\bar{r},1]$, where $\tilde{C} = \tilde{C}(Q,m,n,\bar{n},\bar{r})$. Hence by choosing such an appropriate $\delta = \delta(Q,m,n,\bar{n},\bar{r})$ and $\tilde{\eps} = \tilde{\eps}(Q,m,n,\bar{n},\bar{r})$, we see that the proof is complete.
\end{proof}

The next lemma gives the final ingredient for our proof {that Proposition \ref{p:decay-multiple} implies Theorem \ref{c:decay}{(a)}. The lemma provides a} condition for which the rescalings {of Corollary \ref{c:control-large-interval}} also satisfy the conditions of Theorem \ref{c:decay}.

\begin{lemma}\label{l:large_interval}
    Fix a radius $\bar{r}\in (0,\frac{1}{2}]$ and $\eps_2\in (0,1)$. Then, there exists a positive number $\eps_0 = \eps_0(Q,m,n,\bar{n},\bar{r},\eps_2)$ such that the following holds. Suppose that $T$, $\Sigma$, and $\Sbf$ satisfy the assumptions in Theorem \ref{c:decay} with this choice of $\eps_0$. Then, for every radius $r\in [\bar{r},\frac{1}{2}]$ we have the following. If there is a cone $\Sbf_r\in \mathscr{C}(Q,0)$ which obeys
    \begin{equation}\label{e:r-condition}
        \mathbb{E}(T,\Sbf_r,\Bbf_r) \leq \mathbb{E}(T,\Sbf,\Bbf_1),
    \end{equation}
    and moreover
    \begin{equation}\label{e:A-condition}
      \mathbf{A}^2 r^2 \leq \eps_0^2 \inf \{\mathbb{E} (T, \mathbf{S}^\prime, \Bbf_r) : \Sbf^\prime\in \mathscr{C} (Q, 0)\}\, ,
    \end{equation}
    then the rescalings $T_{0,r}$, $\Sigma_{0,r}$, and $\Sbf_r$ satisfy the assumptions (i)--(iii) of Theorem \ref{c:decay} (and so also Proposition \ref{p:decay-multiple} and Corollary \ref{c:control-large-interval}) with $\eps_2$ in place of $\eps_0$, i.e.
    $$\mathbb{E}(T_{0,r},\Sbf_r,\Bbf_1)\leq \eps_2^2\mathbf{E}^p(T_{0,r},\Bbf_1)$$
    and
    $$\Bbf_{\eps_2}(\xi)\cap\{p:\Theta(T_{0,r},p)\geq Q\}\neq\emptyset \qquad \forall\xi\in V(\Sbf_r)\cap \Bbf_{1/2}.$$
\end{lemma}

\begin{proof}

We argue by contradiction. Fix $\eps_2$ and $\bar{r}$. If the conclusion of the lemma is false, then we can find a sequence of currents $T_k$, manifolds $\Sigma_k$, cones $\Sbf_k$, and $\eps_0^k\downarrow 0$ as in Theorem \ref{c:decay} with $\eps_0^k$ in place of $\eps_0$ for which the present lemma fails. Namely, there exist radii $r_k\in [\bar{r},\frac{1}{2}]$ and cones $\Sbf_{r_k} \equiv (\Sbf_k)_{r_k}$ which obey
$$(\eps_0^k)^{-2}r_k^2\Abf_k^2 \leq \mathbb{E}(T_k,\Sbf_{r_k},\Bbf_{r_k})\leq \mathbb{E}(T_k,\Sbf_{k},\Bbf_1)$$
yet either we have (as the inequality $\Abf_{0,r_k}^2\leq \eps_2^2\mathbb{E}((T_k)_{0,r_k},\Sbf_{r_k},\Bbf_1)$ holds for all $k$ sufficiently large by the above assumption)
$$\mathbb{E}(T_k,\Sbf_{r_k},\Bbf_{r_k})>\eps_2^2\mathbf{E}^p(T_k,\Bbf_{r_k})$$
or
$$\Bbf_{{r_k}\eps_2}(\xi)\cap \{p:\Theta(T_k,p)\geq Q\} = \emptyset\ \ \ \ \text{for some }\xi\in V(\Sbf_{r_k})\cap \Bbf_{r_k/2}.$$
In particular, since the assumptions of Theorem \ref{c:decay} hold for $T_k$ and $\Sbf_k$, points of density at least $Q$ in $T_k$ accumulate on the spines $V(\Sbf_k)$. Thus, if we pass to a subsequence for which $V(\Sbf_k)$ and $V(\Sbf_{r_k})$ converge, the second condition above in fact gives that we must have
\begin{equation}\label{e:non-alignment-of-spines}
\lim_{k\to \infty}\dist(V(\Sbf_{r_k})\cap \Bbf_1, V(\Sbf_k)\cap \Bbf_1) \geq {\bar{r}}\eps_2/2>0
\end{equation}
whilst the first condition above gives
\begin{equation}\label{e:no-vanishing-of-excess}
    \liminf_{k\to\infty}\frac{\mathbb{E}(T_k,\Sbf_{r_k},\Bbf_{r_k})}{\mathbf{E}^p(T_k,\Bbf_{r_k})} \geq \eps_2^2 >0.
\end{equation}

Without loss of generality, we can select planes $\pi_k$ realizing $\mathbf{E}^p(T_k,\Bbf_1)$, and, up to performing a rotation, we can assume they all coincide with the same fixed plane, $\pi_0$. We can also pass to a subsequence to ensure that the number $N_k$ of planes forming $\Sbf_k$ is a constant $N$ (obeying $N\leq Q$), and similarly the number of planes forming $\Sbf_{r_k}$ is a constant $\bar{N}$.

Let us first contradict \eqref{e:no-vanishing-of-excess}. Note that we have 
\[
\mathbb{E}(T_k,\Sbf_{r_k},\Bbf_{r_k}) \leq \mathbb{E}(T_k,\Sbf_k,\Bbf_1)\leq (\eps_0^k)^2\mathbf{E}^p(T_k,\Bbf_1)\to 0\, .
\]
and thus \eqref{e:no-vanishing-of-excess} tells us that we must have $\mathbf{E}^p(T_k,\Bbf_{r_k})\to 0$. In particular, up to subsequences, we know that the current $T_k$ is converging to a limiting current $T_\infty$ which coincides with an integer multiple of some plane $\pi_\infty$ in $\Bbf_{r_k}$. On the other hand, if we extract a converging subsequence, not relabelled, we have $\mathbf{S}_k \to \mathbf{S}_\infty$ in the local Hausdorff topology for some $\Sbf_\infty\in \mathscr{C}(Q,0)$. Hence we also conclude that $\spt (T) \cap \Bbf_1 \equiv \mathbf{S}_\infty\cap \Bbf_1$. This implies that the cone $\mathbf{S}_\infty\in \mathscr{C} (Q, 0)$ is in fact the plane $\pi_\infty$. In particular we conclude $\mathbf{E}^p (T_k, \Bbf_1) \to 0$.

We can therefore apply Almgren's (strong) Lipschitz approximation to the sequence $T_k$ and perform a blow-up procedure. We normalize the Lipschitz approximations by 
\[
E_k^{1/2} = \mathbf{E}^p(T_k,\Bbf_1)^{1/2}
\]
and, upon possibly applying suitable rotations (see Remark \ref{r:blow-up}), extract a Dir-minimizing function $f_\infty:B_{1/2}(\pi_\infty)\to\Acal_Q({\pi_\infty^\perp})$ in the blow-up limit. Moreover, since $\mathbb{E}(T_k,\Sbf_k,\Bbf_1)\leq (\eps_0^k)^2\mathbf{E}^p(T_k,\Bbf_1)$, we see that if we describe the cones $\Sbf_k$ as a union of graphs of linear maps $L_1^k,\dotsc,L_N^k:{\pi_\infty\to \pi_\infty^\perp}$ then, as we have previously seen, upon extraction of a subsequence we will have that $E_k^{-1/2}L_i^k$ converges to some linear function $L_i^\infty$ for each $i=1,\dotsc,N$. At least two of these linear functions are distinct, given that by Proposition \ref{p:balancing}(d) $\boldsymbol{\mu} (\Sbf)$ is comparable to $E_k^{1/2}$. But now, since $f_\infty=\sum_{i=1}^N \llbracket L_i^\infty \rrbracket$ and $r_k\leq \frac{1}{2}$, this would imply that $C^{-1}\leq\mathbf{E}^p(T_k,\Bbf_{r_k})/\mathbf{E}^p(T_k,\Bbf_1) \leq C$ for some geometric constant $C>0$, and hence we would have
\[
\frac{\mathbb{E}(T_k,\Sbf_{r_k},\Bbf_{r_k})}{\mathbf{E}^p(T_k,\Bbf_{r_k})} \leq C\frac{\mathbb{E}(T_k,\Sbf_{r_k},\Bbf_{r_k})}{\mathbf{E}^p(T_k,\Bbf_1)} \leq C\frac{\mathbb{E}(T_k,\Sbf_{k},\Bbf_1)}{\mathbf{E}^p(T_k,\Bbf_1)} \leq C(\eps_0^k)^2\to 0
\]
which is in direct contradiction to \eqref{e:no-vanishing-of-excess}.

Now we contradict \eqref{e:non-alignment-of-spines}. We distinguish two cases here, depending on whether $\mathbf{E}^p (T, \Bbf_1)$ converges to $0$ or not. In the latter case we can assume, after extracting a suitable subsequence, that $\mathbf{S}_k$ converges to a non-planar cone $\mathbf{S}_\infty\in \mathscr{C} (Q, 0)\setminus\mathscr{P}$ in the local Hausdorff topology which is the support of an area-minimizing current. But then it would follow that $\Sbf_{r_k}$ converges to the same cone. In particular $V (\Sbf_{r_k})\cap \Bbf_1$ and $V (\Sbf_k)\cap \Bbf_1$ both converge to $V (\Sbf_\infty)\cap \Bbf_1$, contracting \eqref{e:non-alignment-of-spines}.

In the other case note that, since $\mathbb{E}(T_k,\Sbf_{r_k},\Bbf_{r_k}) \leq (\eps_0^k)^2\mathbf{E}^p (T_k,\Bbf_1)$, if we assume that the cones $\Sbf_{r_k}$ are the union of $\bar{N}$ distinct linear maps $\bar{L}^k_i$, we see that $E_k^{-1/2}\bar{L}^k_i$ converge to a linear map $\bar{L}^\infty_i$ for each $i=1,\dotsc,\bar{N}$. Furthermore the union of the graphs of these limiting linear maps must coincide with the support of the map $f_\infty$ found previously (indeed, we get from the above inequality that they must coincide on $\Bbf_{\bar{r}}$, and hence on all of $\Bbf_1$ as they are linear). 
But this is only possible if
\[
\dist(V(\Sbf_{r_k})\cap\Bbf_1,V(\Sbf_k)\cap\Bbf_1)\to 0
\]
which contradicts \eqref{e:non-alignment-of-spines}. This completes the proof.
\end{proof}

We are now ready to complete the proof of Theorem \ref{c:decay} from Proposition \ref{p:decay-multiple}. We first address the most important conclusion, which is the decay estimate in (a). We will see after how (b), (c), and (d) can be derived from this.

\begin{proof}[Proof of Theorem \ref{c:decay}(a) from Proposition \ref{p:decay-multiple}]
We therefore fix $T$, $\Sigma$, $\Sbf$, and $\varsigma$ as in the statement of Theorem \ref{c:decay}, and fix parameters $\eps_0$ and $r_0<1/2$ to be specified later. 

Choose $\varsigma_2 = \frac{1}{2}$ in Proposition \ref{p:decay-multiple}, and denote by $\varepsilon_2 = \varepsilon_2(Q,m,n,\bar{n})$ the corresponding constant from Proposition \ref{p:decay-multiple} for this choice of $\varsigma_2$. Also denote by $r_1,\dotsc,r_{\bar{N}}\leq\frac{1}{2}$ the corresponding radii from Proposition \ref{p:decay-multiple} with this choice of $\varsigma_2$. Write $\bar{r}:= \min\{r_i: i=1,\dotsc,\bar{N}\}$, which depends only on $Q,m,n,\bar{n}$.

Define
$$\mathbb{E}(T,\Bbf_r) := \inf\{\mathbb{E}(T,\Sbf,\Bbf_r):\Sbf\in \mathscr{C}(Q,0)\setminus\mathscr{P}(0)\}$$
and
$$\mathcal{A}:= \{r\in [r_0,1/2]:\eps_0^{-2}\Abf^2r^2 \geq \mathbb{E}(T,\Bbf_r)\}.$$
We now distinguish three possibilities.

\medskip

\textbf{Case 1:} $\mathcal{A} = \emptyset$. We then have $\eps_0^{-2}r^2\Abf^2\leq \mathbb{E}(T,\Bbf_r)$ for all $r\in [r_0,1/2]$. If $\eps_0\leq \eps_2$, then at scale $1$ the assumptions of Theorem \ref{c:decay} give that we can apply Proposition \ref{p:decay-multiple} to $T,\Sigma,\Sbf$ to get that there exists $\rho_1\in [\bar{r},1/2]$ (indeed, $\rho_1 = r_i$ for some $i=1,\dotsc,\bar{N}$, but this is not relevant) and a cone $\Sbf_1\in \mathscr{C}(Q,0)\setminus \mathscr{P}(0)$ such that
$$\mathbb{E}(T,\Sbf_1,\Bbf_{\rho_1})\leq \frac{1}{2}\mathbb{E}(T,\Sbf,\Bbf_1).$$
Let us assume (as we may) that $r_0<\bar{r}$. Since $\mathcal{A} = \emptyset$, we therefore fall into the realm of Lemma \ref{l:large_interval}, taking $\eps_2$ and $\bar{r}$ to be the above constants. If $\eps_0$ is smaller than the corresponding constant from Lemma \ref{l:large_interval} with this choice of $\eps_2$ and $\bar{r}$ (and thus only depends on $Q,m,n,\bar{n})$, we see that Lemma \ref{l:large_interval} gives that $T_{0,\rho_1}$, $\Sigma_{0,\rho_1}$, and some cone $\Sbf^\prime_1$ (namely the cone achieving the infimum in $\mathbb{E}(T,\Bbf_{\rho_1})$) obey the necessary conditions to apply Proposition \ref{p:decay-multiple} again. Thus, we find a second radius $\rho_2$ obeying $\rho_2/\rho_1\in [\bar{r},1/2]$ and a cone $\Sbf_2\in \mathscr{C}(Q,0)\setminus\mathscr{P}(0)$ obeying
$$\mathbb{E}(T,\Sbf_2,\Bbf_{\rho_2})\leq \frac{1}{2}\mathbb{E}(T,\Sbf_1,\Bbf_{\rho_1})\leq \frac{1}{4}\mathbb{E}(T,\Sbf,\Bbf_1).$$

We now keep iterating this procedure, until we arrive at a radius $\rho_k$ so that the next radius $\rho_{k+1}$ is smaller than $r_0$ (in particular we no longer necessarily have that $\eps_0^{-2}\Abf^2 \rho_{k+1}^2 \leq \Ebb(T,\Bbf_{\rho_{k+1}})$); as $\rho_{k+1}/\rho_k\in [\bar{r},1/2]$, when this happens we have $\bar{r}\leq \rho_{k+1}/\rho_k < r_0/\rho_k$. Since we can apply Corollary \ref{c:control-large-interval} to $T_{0,\rho_k}$, $\Sigma_{0,\rho_k}$, and $\Sbf_k$ (again, using the fact that $\rho_k\not\in \mathcal{A} = \emptyset$ and Lemma \ref{l:large_interval}), provided $\eps_2$ is smaller than the corresponding constant from Corollary \ref{c:control-large-interval} with this choice of $\bar{r}$ (which can of course be arranged), we therefore get the existence of a cone $\Sbf^\prime\in \mathscr{C}(Q,0)\setminus\mathscr{P}(0)$ such that (choosing $r^* = r_0/\rho_k \in [\bar{r},1]$ in Corollary \ref{c:control-large-interval}, so that $r^*\rho_k = r_0$)
$$\mathbb{E}(T_{0,\rho_k},\Sbf^\prime,\Bbf_{r^*}) \leq C\mathbb{E}(T_{0,\rho_k},\Sbf_k,\Bbf_1)$$
where $C = C(Q,m,n,\bar{n})$ (indeed the constant $C$ depends on $\bar r$, but at this stage the latter has been fixed as depending only upon $Q,m,n,$ and $\bar n$). In particular, combined with the above iteration, the last inequality yields
$$\mathbb{E}(T,\Sbf^\prime,\Bbf_{r_0}) \leq C \mathbb{E}(T,\Sbf_k,\Bbf_{\rho_k}) \leq C\cdot 2^{-k}\mathbb{E}(T,\Sbf,\Bbf_1).$$
Observe however that as $\bar{r}<r_0/\rho_k < r_0/\bar{r}^k$, i.e. $r_0>\bar{r}^{k+1}$, we have $k+1\geq \log(r_0)/\log(\bar{r})$. Hence, choosing $r_0 = r_0(Q,m,n,\bar{n})$ sufficiently small we can ensure that $k\geq \frac{1}{2}\left\lfloor\frac{\log(r_0)}{\log(\bar{r})}\right\rfloor$, and hence
\begin{equation}\label{e:important-case}
    \mathbb{E}(T,\Sbf^\prime,\Bbf_{r_0}) \leq C2^{-\frac{1}{2}\left\lfloor\frac{\log(r_0)}{\log(\bar{r})}\right\rfloor}\mathbb{E}(T,\Sbf,\Bbf_1).
\end{equation}
Now we may further take $r_0$ small enough to make sure that $C2^{-\frac{1}{2}\left\lfloor\frac{\log(r_0)}{\log(\bar{r})}\right\rfloor} \leq \varsigma$ to deduce the conclusion (a) of Theorem \ref{c:decay}.
\medskip

\textbf{Case 2:} $\inf\mathcal{A} = r_0$. Note that $\mathcal{A}$ is certainly a closed set, and so in this case we have $r_0\in \mathcal{A}$. Hence, we have
$$\mathbb{E}(T,\Bbf_{r_0}) \leq \eps_0^{-2}\Abf^2 r_0^2 \leq r_0^2\mathbb{E}(T,\Sbf,\Bbf_1)$$
where the second inequality comes from our assumption in Theorem \ref{c:decay}. In particular, we can find $\Sbf^\prime\in\mathscr{C}(Q,0)\setminus \mathscr{P}(0)$ with
$$\mathbb{E}(T,\Sbf^\prime,\Bbf_{r_0}) \leq 2\mathbb{E}(T,\Bbf_{r_0}) \leq 2r_0^2\mathbb{E}(T,\Sbf,\Bbf_1)$$
and so we just need to take $2r_0^2 < \varsigma$ to deduce conclusion (a) of Theorem \ref{c:decay}. 

\medskip

\textbf{Case 3:} $\mathcal{A}\neq\emptyset$ yet $\inf\mathcal{A}>r_0$. In this case, let $r_I = \inf\mathcal{A}\in\mathcal{A}$. Let $\Sbf_I\in \mathscr{C}(Q,0)\setminus\scr{P}(0)$ be such that
\begin{equation}\label{e:A-dominates}
\mathbb{E}(T,\Sbf_I,\Bbf_{r_I})\leq 2\mathbb{E}(T,\Bbf_{r_I}) \leq 2\eps_0^{-2}\Abf^2 r_I^2.
\end{equation}
Now observe that
$$\lim_{r\uparrow r_I}\mathbb{E}(T,\Sbf_I,\Bbf_r) = \mathbb{E}(T,\Sbf_I,\Bbf_{r_I})$$
Moreover, as $[r_0,r_I)\neq\emptyset$ and $[r_0,r_I)\cap\mathcal{A} = \emptyset$, we see that for all $r\in [r_0,r_I)$,
$$\eps_0^{-2}r^2\Abf^2\leq \mathbb{E}(T,\Bbf_r) \leq \mathbb{E}(T,\Sbf_I,\Bbf_r).$$
Taking $r\uparrow r_I$ in this expression, we deduce
$$\eps_0^{-2}r^2_I \Abf^2\leq \mathbb{E}(T,\Sbf_I,\Bbf_{r_I}).$$
Moreover, {from \eqref{e:A-dominates} and our assumption \eqref{e:smallness}} we know
\begin{equation}\label{e:A-dominates-2}
\mathbb{E}(T,\Sbf_I,\Bbf_{r_I})\leq 2\eps_0^{-2}\Abf^2 r_I^2 \leq 2r_I^2\mathbb{E}(T,\Sbf,\Bbf_1).
\end{equation}
Hence, we have all the necessary assumptions to apply Lemma \ref{l:large_interval}, which guarantees that we can apply Proposition \ref{p:decay-multiple} to $T_{0,r_I}$, $\Sigma_{0,r_I}$, and $\Sbf_I$. We can now argue as in Case 1, but with $r_I$ replacing $1$ as a starting point of the iteration {and with $r_0/r_I$ as the lower endpoint}. We therefore get the existence of an $\Sbf^\prime\in \mathscr{C}(Q,0)\setminus\mathscr{P}(0)$ with the property that (see \eqref{e:important-case})
$$\mathbb{E}(T,\Sbf^\prime,\Bbf_{r_0}) \leq C2^{-\frac{1}{2}\frac{\log(r_0/r_I)}{\log(\bar{r})}}\mathbb{E}(T,\Sbf_I,\Bbf_{r_I}).$$
Combining this with \eqref{e:A-dominates-2} we arrive at
$$\mathbb{E}(T,\Sbf^\prime,\Bbf_{r_0}) \leq \left(2r_I^2C2^{\frac{1}{2}\frac{\log(r_I)}{\log(\bar{r})}}\right)2^{-\frac{1}{2}\frac{\log(r_0)}{\log(\bar{r})}}\mathbb{E}(T,\Sbf,\Bbf_{1}).$$
Noting that $2^{\frac{1}{2}\frac{\log(r_I)}{\log(\bar{r})}} = \left(\frac{1}{r_I}\right)^{\frac{\log(2)}{2\log(1/\bar{r})}} \leq r_I^{-\frac{1}{2}}$, the above becomes
$$\mathbb{E}(T,\Sbf^\prime,\Bbf_{r_0}) \leq \left(2Cr_I^{3/2}\right)2^{-\frac{1}{2}\frac{\log(r_0)}{\log(\bar{r})}}\mathbb{E}(T,\Sbf,\Bbf_{1}) \leq 2C\cdot 2^{-\frac{1}{2}\frac{\log(r_0)}{\log(\bar{r})}}\mathbb{E}(T,\Sbf,\Bbf_{1}).$$
Thus choosing $r_0$ small enough so that $2C\cdot2^{-\frac{1}{2}\frac{\log(r_0)}{\log(\bar{r})}}\leq \varsigma$, we deduce (a) from Theorem \ref{c:decay}. As the above three situations exhaust all possibilities, the proof is completed.
\end{proof}

To summarize: we have now reduced the proof of {point (a) in Theorem \ref{c:decay}} to proving the a priori much weaker result of Theorem \ref{t:weak-decay}. We will next address how points (b), (c), and (d) follow from (a). In fact, since it will prove to be useful also later when we get to the rectifiability statement in Theorem \ref{t:big-one}, we state here a more general lemma.

\begin{lemma}\label{l:a->bcd}
There is a constant $C = C(Q,m,n,\bar {n})>0$ with the following property. For every $r_0 >0$, there exists $\varepsilon = \varepsilon(Q,m,n, \bar n,r_0)>0$ such that if
\begin{itemize}
\item[(i)] $T$ and $\Sigma$ are as in Assumption \ref{a:main};
\item[(ii)] $\|T\| (\Bbf_1) \leq (Q+\frac{1}{2}) \omega_m$; and
\item[(iii)] there is $\mathbf{S}\in \mathscr{C} (Q, 0)\setminus\Pscr(0)$ such that 
\begin{equation}\label{e:smallness-abcd}
\mathbf{A}^2 + \mathbb{E} (T, \mathbf{S}, \Bbf_1) \leq \varepsilon^2 \mathbf{E}^p (T, \Bbf_1)\ ;
\end{equation}
\end{itemize}
then
\begin{equation}\label{e:no-change-E}
C^{-1} \mathbf{E}^p (T, \Bbf_1) \leq \Ebf^p (T, \Bbf_{r_0})
\leq C \Ebf^p (T, \Bbf_1)\, .
\end{equation}
Moreover, there is a constant $\bar C = \bar C (Q,m,n,\bar n, r_0)$ such that, provided $\varepsilon$ is chosen possibly smaller, if there is another cone $\Sbf'\in \mathscr{C} (Q, 0){\setminus \mathscr{P}(0)}$ that also obeys
\begin{align}\label{e:yet-another-smallness}
\mathbb{E} (T, \mathbf{S}', \Bbf_{r_0}) \leq \eps^2 \mathbf{E}^p (T, \Bbf_1)\, ,
\end{align}
then 
\begin{align}
\dist^2 (\mathbf{S}\cap \Bbf_1, \mathbf{S}^\prime \cap \Bbf_1) & \leq \bar C (\mathbf{A}^2 + \mathbb{E} (T, \Sbf, \Bbf_1) + \mathbb{E} (T, \Sbf^\prime, \Bbf_{r_0}))\label{e:near-10}\\
\dist^2 (V (\Sbf)\cap \Bbf_1, V (\Sbf')\cap \Bbf_1) 
&\leq \bar C \mathbf{E}^p (T, \Bbf_1)^{-1}
(\mathbf{A}^2 + \mathbb{E} (T, \Sbf, \Bbf_1) + \mathbb{E} (T, \Sbf^\prime, \Bbf_{r_0}))\, .\label{e:align-10}
\end{align}
\end{lemma}

\begin{proof}
First of all we argue for \eqref{e:no-change-E}. We start by {demonstrating} that
\begin{equation}\label{e:obvious}
\boldsymbol{\mu} (\Sbf)^2 \leq C \mathbf{E}^p (T, \Bbf_1)\, .
\end{equation}
In fact, fix $\pi$ such that $\mathbf{E}^p (T, \Bbf_1) = \hat{\mathbf{E}} (T, \pi, \Bbf_1)$ and use Allard's $L^2$--$L^\infty$ height bound to infer that 
\begin{equation}\label{e:Allard-1}
\dist^2 (q, \pi)\leq C \mathbf{E}^p (T, \Bbf_1) + C \Abf^2\,  \qquad \forall q\in \spt (T)\cap \Bbf_{1/2}\, .
\end{equation}
Next, let $\alpha_i$ be one of the planes which form $\Sbf$ and recall that 
\[
\int_{\alpha_i \cap \Bbf_1 \setminus B_{a} (V)} \dist^2 (q, \spt (T))\, d\mathcal{H}^m (q) 
\leq \mathbb{E} (T, \Sbf, \Bbf_1)\, .
\]
Using Chebyshev's inequality (applied to the measure $\Hcal^m\res\alpha_i$), we have that the set $F_i\subset {\alpha_i\cap \Bbf_{1/4}\setminus \Bbf_a(V)}$ of points $q\in {\alpha_i\cap \Bbf_{1/4}\setminus B_a(V)}$ obeying
\begin{equation}\label{e:Chebyshev}
    \dist^2(q,\spt(T)) \leq \frac{C}{\Hcal^m({\alpha_i\cap \Bbf_{1/4}\setminus \Bbf_a(V)})}\mathbb{E} (T, \Sbf, \Bbf_1)
\end{equation}
satisfies $\Hcal^m(F_i)\geq \left(1-\frac{1}{C}\right)\Hcal^m({\alpha_i\cap \Bbf_{1/4}\setminus \Bbf_a(V)})$ and so for a choice of $C=C(m)>0$ sufficiently large, there exist $m$ linearly independent vectors $e_1,\dots,e_m\in F_i$ with the property that every $v \in \alpha_i \cap \Bbf_1$ can be written as $v = \sum_j \lambda_j e_j$ with $|\lambda_j|\leq C$.

Combining the above distance inequalities, for each $j\in \{1,\dots,m\}$ we have
\[
\dist^2 (e_j, \pi) \leq C (\mathbb{E} (T, \Sbf, \Bbf_1) + \Abf^2 + \mathbf{E}^p (T, \Bbf_1)) \leq 
C \mathbf{E}^p (T, \Bbf_1)\, .  
\]
{For each $v \in \alpha_i \cap \Bbf_1$, using the identity $\dist (v, \pi)= |\mathbf{p}_\pi^\perp (v)|$, this in turn implies}
\[
\dist^2 (v, \pi) \leq C \mathbf{E}^p (T, \Bbf_1)\, . 
\]
We thus can use \eqref{e:Haus=max_eigen} of Corollary \ref{c:growth} to infer that
\begin{equation}\label{e:obvious-2}
\dist^2 (\alpha_i\cap \Bbf_1, \pi\cap \Bbf_1) \leq C \mathbf{E}^p (T, \Bbf_1)\, ,
\end{equation}
which in turn implies \eqref{e:obvious}. 

Observe also that we have the inequality
\[
\mathbf{E}^p (T, \Bbf_1) \leq \hat{\mathbf{E}} (T, \alpha_1,\Bbf_1) 
\leq C \hat{\mathbf{E}} (T,\Sbf, \Bbf_1) + C \boldsymbol{\mu} (\Sbf)^2  
\leq C \varepsilon^2 \mathbf{E}^p (T, \Bbf_1) + C \boldsymbol{\mu} (\Sbf)^2\, ,
\]
and so in particular we conclude
\begin{equation}\label{e:E-mu-comparable}
C^{-1} \mathbf{E}^p (T, \Bbf_1) \leq \boldsymbol{\mu} (\Sbf)^2 \leq C \mathbf{E}^p (T, \Bbf_1)\, ,
\end{equation}
if $\varepsilon$ is sufficiently small.
Moreover, since we have not used any information other than the smallness of $\Ebb (T, \Sbf, \Bbf_1) + \Abf^2$ with respect to $\Ebf^p (T, \Bbf_1)$, the same argument applies to $\Sbf^\prime$ and $T_{0, r_0}$, which allows us to conclude
\begin{equation}\label{e:E-mu-comparable-2}
C^{-1} \mathbf{E}^p (T, \Bbf_{r_0}) \leq \boldsymbol{\mu} (\Sbf^\prime)^2 \leq C \mathbf{E}^p (T, \Bbf_{r_0})\, .
\end{equation}
We next observe that from \eqref{e:obvious} {(see also \eqref{e:obvious-2})} we immediately get the {inequality on the right-hand side of} \eqref{e:no-change-E}. Indeed, we can write 
\begin{align*}
\mathbf{E}^p (T, \Bbf_{r_0}) & \leq 
\hat{\mathbf{E}} (T, \pi, \Bbf_{r_0}) \leq C r_0^{-2} \max_i \dist^2 (\alpha_i \cap \Bbf_{r_0}, \pi \cap \Bbf_{r_0}) + 
C \hat{\mathbf{E}} (T, \Sbf, \Bbf_{r_0})\\
&\leq C \mathbf{E}^p (T, \Bbf_1) + C r_0^{-m-2} \hat{\mathbf{E}} (T, \Sbf, \Bbf_1)
\leq C (1+ \varepsilon r_0^{-m-2}) \mathbf{E}^p (T, \Bbf_1)
\end{align*}
and so to conclude the right-hand inequality in \eqref{e:no-change-E} it suffices to choose $\varepsilon$ small compared to $r_0$. 

\medskip

We next argue by contradiction for the left-hand inequality in \eqref{e:no-change-E}. 
If the conclusion is false, then we could find a sequence of {area-minimizing} currents $T_k$, ambient manifolds $\Sigma_k$, cones $\Sbf_k\in \Cscr(Q,0)\setminus\Pscr(0)$, and parameters $\eps^k\downarrow 0$ for which the left-hand inequality in \eqref{e:no-change-E} fails for some fixed $r_0$, namely 
\begin{equation}\label{e:planar-decay}
    \lim_{k\to \infty} \frac{\mathbf{E}^p (T_k, \Bbf_{r_0})}{\mathbf{E}^p (T_k, \Bbf_1)} <  C_*^{-1} \, ,
\end{equation}
for some constant $C_*$ which will be specified only later and which will turn out to be independent of $r_0$.
 Without loss of generality, we can select planes $\pi_k$ {with $\mathbf{E}^p(T_k,\Bbf_1) = \hat\Ebf(T_k,\pi_k,\Bbf_1)$}, and, up to performing a rotation, we can assume they all coincide with the same fixed plane, $\pi_\infty$. We can also pass to a subsequence to ensure that the number $N_k$ of planes forming $\Sbf_k$ is a constant $N$ (obeying $N\leq Q$).

Observe that we must necessarily have 
\[
\mathbf{E}^p(T_k,\Bbf_1)\to 0\, .
\]
Indeed, we can assume $\mathbf{S}_k \to \mathbf{S}_\infty {\in \Cscr(Q,0)\setminus \Pscr(0)}$ locally in the Hausdorff topology, while $T_k$ converges weakly to some area-minimizing current $T_\infty$ and it follows easily that $\spt (T_\infty)\cap \Bbf_1 = \mathbf{S}_\infty \cap \Bbf_1$. If $\mathbf{E}^p (T_k, \Bbf_1)$ does not converge to $0$, then $T_\infty$ is not supported in a plane, so it must be an area-minimizing integral cone which is not planar. But then we see immediately that 
\[
\lim_{k\to\infty} \frac{\mathbf{E}^p (T_k, \Bbf_1)}{\mathbf{E}^p (T_k, \Bbf_{r_0})} = 1\, .
\]
In particular this would be a contradiction if we impose $C_*\geq 1$ in \eqref{e:planar-decay}. 

We conclude then that $\mathbf{S}_k$ converges indeed to $\pi_\infty$. We can therefore apply Almgren's (strong) Lipschitz approximation \cite{DLS14Lp}*{Theorem 2.4} to the sequence $T_k$ {relative to the plane $\pi_\infty$} and perform a blow-up procedure. We normalize the Lipschitz approximations by 
\[
E_k^{1/2} = \mathbf{E}^p(T_k,\Bbf_1)^{1/2}
\]
and, upon possibly applying suitable rotations (see Remark \ref{r:blow-up}), extract a Dir-minimizing function {$f_\infty:B_1({\pi_\infty})\to\Acal_Q({\pi_\infty^\perp})$} in the blow-up limit {(the Dir-minimizing property follows from \cite{DLS14Lp}*{Theorem 2.6})}. Moreover, since $\mathbb{E}(T_k,\Sbf_k,\Bbf_1)\leq (\eps^k)^2\mathbf{E}^p(T_k,\Bbf_1)$, we see that if we describe the cones $\Sbf_k$ as a union of graphs of linear maps $L_1^k,\dotsc,L_N^k:{\pi_\infty\to \pi_\infty^\perp}$ then upon extraction of a subsequence we will have that $E_k^{-1/2}L_i^k$ converges to some linear function $L_i^\infty$ for each $i=1,\dotsc,N$, because of \eqref{e:obvious}.

Observe that:
\begin{itemize}
    \item $L_i^\infty$ and $L_j^\infty$ are not necessarily distinct for every $i\neq j$;
    \item However there is {\em at least one} pair {of indices $i\neq j$} for which they are indeed distinct. 
\end{itemize}
The second fact is a simple consequence of \eqref{e:E-mu-comparable}.

Up to reordering, we can assume that there are exactly $\bar N \geq 2$ distinct linear maps $L_1^\infty, \ldots , L_{\bar{N}}^\infty$ in the collection $\{L_i^\infty\}$. Consider then the cones $\Sbf'_k\subset \Sbf_k$ given by the graph of the linear maps $L_1^k, \ldots , L_{\bar{N}}^k$. It follows from the arguments outlined so far that 
\[
\lim_{k\to \infty} \frac{\hat{\mathbf{E}} (T_k, \Sbf^\prime_k, \Bbf_1)}{\mathbf{E}^p (T_k, \Bbf_1)} \to 0\, ,
\]
\[
\liminf_{k\to \infty} \frac{\boldsymbol{\sigma} (\Sbf^\prime_k)}{\boldsymbol{\mu} (\Sbf^\prime_k)} 
> 0\, ,
\]
and
\[
\lim_{k\to\infty} \frac{\boldsymbol{\mu} (\Sbf_k)}{\boldsymbol{\mu} (\Sbf^\prime_k)} =1\, .
\]
Together with the fact that $C^{-1} \leq \frac{\mathbf{E}^p ({T_k}, \Bbf_1)}{\boldsymbol{\mu} ({\Sbf_k})^2} \leq C$, we conclude that
\[
\max_{1\leq i < j \leq \bar N} \int_{B_1} |L^\infty_i - L^\infty_j|^2 \geq {2C_*^{-1}}
\]
for {a suitable choice of $C_*=C_*(Q,m,n,\bar n)>0$.} 

We can now argue as in the proof of Proposition \ref{p:balancing-3} to conclude that 
\[
f_\infty = \sum_{i=1}^{\bar N} Q_i \llbracket L^\infty_i \rrbracket
\]
for some {\em positive} integers $Q_i\geq 1$.

Consider now for each $k$ a plane $\pi^1_k$ such that $\hat{\Ebf} (T_k, \pi^1_k, \Bbf_{r_0}) = \Ebf^p (T_k, \Bbf_{r_0})$. By Allard's $L^\infty$-$L^2$ bound, for each $v\in \pi^1_k \cap \Bbf_{r_0/2}$ there is a point $q\in \spt (T_k)\cap \Bbf_{r_0}$ such that $\mathbf{p}_{\pi^1_k} (q) = v$ and
\[
|v-p|^2\leq C r_0^2 (\Ebf^p (T_k, \Bbf_{r_0}) + r_0^2 \Abf_k^2)
\leq C r_0^{-m} \Ebf^p (T_k, \Bbf_1)\, .
\]
However, by \eqref{e:Allard-1}, $\dist^2 (p, \pi_\infty)\leq C (\Ebf^p (T_k, \Bbf_1) + \Abf_k^2) \leq C \Ebf^p (T_k, \Bbf_1)$. It follows immediately that 
\[
\dist^2 (\pi^1_k\cap \Bbf_{r_0}, \pi_\infty\cap \Bbf_{r_0}) \leq C r_0^{-m} \Ebf^p (T_k, \Bbf_1)
\]
which in turn implies 
\[
\dist^2 (\pi^1_k\cap \Bbf_1, \pi_\infty\cap \Bbf_1) \leq C r_0^{-m-2}\Ebf^p (T_k, \Bbf_1)\, .
\]
Let $A_k : \pi_\infty \to \pi_\infty^\perp$ be the linear map whose graph gives $\pi^1_k$ and note that, from the above discussion, ${\|A_k\|_{L^\infty (B_1 (0, \pi_\infty))}}\leq C r_0^{-1-m/2} E_k^{1/2}$. In particular we can extract an {$L^2$} limit $A_\infty$ of the maps $E_k^{-1/2} A_k$, up to subsequences. 
Now the estimates in \cite{DLS14Lp}*{Theorem 2.4} imply
\[
\lim_{k\to \infty} \frac{\mathbf{E}^p (T_k, \Bbf_{r_0})}{\mathbf{E}^p ({T_k}, \Bbf_1)} 
= r_0^{-m-2} \int_{B_{r_0}} \mathcal{G} (f_\infty, Q \llbracket A_\infty \rrbracket)^2 
= \int_{B_1} \mathcal{G} (f_\infty, Q \llbracket A_\infty\rrbracket)^2 \, ,
\]
where in the last equality we have used the $1$-homogeneity of both $f_\infty$ and $A_\infty$. On the other hand by the triangle inequality,
\[
\int_{B_1} \mathcal{G} (f_\infty, Q\llbracket A_\infty \rrbracket)^2 \geq \frac{1}{2}
\max_{1\leq i < j \leq \bar N} \int_{B_1} |L^\infty_i - L^\infty_j|^2\, \geq {C_*^{-1}}\, .
\]
Thus, we are in contradiction with \eqref{e:planar-decay}, which concludes the proof of \eqref{e:no-change-E}.

\medskip

{\bf Step 1: {Reduction via }enlarging, balancing, and pruning.}
{We will first show that one can without loss of generality assume also the following on $\Sbf$ and $\Sbf^\prime$:}
\begin{itemize}
\item[(B)] $\mathbf{S}$ and $\mathbf{S}^\prime$ are both $M$-balanced (with $M$ the constant of Assumption \ref{a:choose-M});
\item[(P)] For a suitably small constant $\delta = \delta (Q,m,n, \bar{n}, r_0)>0$ we have 
\begin{align}
 \mathbf{A}^2 + \mathbb{E} (T, \mathbf{S}, \Bbf_1) & \leq \delta^2 \boldsymbol{\sigma} (\mathbf{S})^2\label{e:pruned-1}\\
 r_0^2 \mathbf{A}^2 + \mathbb{E} (T, \mathbf{S}^\prime, \Bbf_{r_0}) &\leq \delta^2 \boldsymbol{\sigma} (\Sbf^\prime)^2\, .\label{e:pruned-2}
\end{align}
\end{itemize}

For the present reduction step, we will show the existence of another cone $\Sbf_1\in \mathscr{C}(Q,0)$ such that:
\begin{itemize}
    \item[(i)] (B) and \eqref{e:pruned-1} above hold for $\Sbf_1$ in place of $\Sbf$;
    \item[(ii)] $V (\Sbf) = V (\Sbf_1)$;
    \item[(iii)] $\mathbb{E} (T, \Sbf_1, \Bbf_1) \leq \bar{C} (\mathbb{E} (T, \Sbf, \Bbf_1) + \Abf^2)$, for a constant $\bar C = \bar C (Q,m,n,\bar n, \delta)$;
    \item[(iv)] $\dist^2 (\Sbf_1\cap \Bbf_1, \Sbf\cap \Bbf_1)^2 \leq \bar{C} (\mathbb{E} (T, \Sbf, \Bbf_1) + \Abf^2);$
\end{itemize}
We can then use the same argument to find a cone $\Sbf^\prime_1$ satisfying (i)--(iv) for $T_{0,r_0}$ and $\Sbf^\prime$. Given this, it is easy to check that the assumptions of the lemma hold with $T,\Sbf_1,\Sbf_1^\prime$ in place of $T,\Sbf,\Sbf^\prime$ (up to controlled constant factors), and that if one can prove the result for $T,\Sbf_1,\Sbf^\prime_1$, then the result follows for $T,\Sbf,\Sbf^\prime$. Thus, after Step 1 we will have reduced the proof to the case where we can also assume the validity of (B) and (P).

\medskip

In order to accomplish the task of this step,  we observe that, by choosing $\eps$ in {\eqref{e:smallness-abcd}} small enough, $T$, $\Sigma$, and $\Sbf$ satisfy all the requirements of Proposition \ref{p:balancing}, except possibly $\Abf^2 \leq \eps_0^2 \mathbb{E} (T, \Sbf, \Bbf_1)$, for the parameter {$\eps_0 = \eps_0(Q,m,n,\bar{n})$ in Proposition \ref{p:balancing}}. Let us assume for the moment that this does not hold, i.e.
\begin{equation}\label{e:annoying}
\mathbb{E} (T, \mathbf{S}, \Bbf_1) < \eps_0^{-2}\Abf^2\, .
\end{equation}
In this case we wish to deform $\Sbf$ to another cone $\Sbf_e\in \mathscr{C} (Q, 0)$ such that $\mathbb{E} (T, \Sbf_e, \Bbf_1) < \eps_0^{-2}\Abf^2$ and $V (\Sbf_e) = V (\Sbf)=:V$. This can be accomplished as follows. We consider a plane $\pi \in \mathscr{P} (0)$ containing $V$ and observe that, because of \eqref{e:smallness-abcd}, provided that we take $\eps< {\sqrt{N}}\eps_0$ ,
we have
\begin{equation}\label{e:planar-larger-than-A}
{N}{\hat{\mathbf{E}}} (T, \pi, \Bbf_1) > \eps_0^{-2} \Abf^2\, .
\end{equation}
Next enumerate the (distinct) planes forming $\Sbf$ as $\alpha_1, \dots , \alpha_N$, and {for each $i=1,\dotsc,N$} fix a continuous path of planes $\alpha_i (t)$ in $T_0 \Sigma$ connecting $\alpha_i (0) = \alpha_i$ and $\alpha_i (1) =\pi$, with the property that $V \subset \alpha_i (t)$ for all $t\in [0,1]$. The existence of the path can be reduced to the existence of a continuous path for the cross sections $V^\perp \cap \alpha_i (t)\subset V^\perp \cap T_0 \Sigma$ and the existence of the latter is {a consequence of} the {path} connectedness of the Grassmannian of $2$-planes in $\mathbb{R}^{2+\bar n}$).

Next we want to avoid that along the path we have $\alpha_i (t) = \alpha_j (t)$ for some $i\neq j$ and some $0<t<1$. To that end, denote by $X$ the set of $m$-dimensional subspaces of $T_0 \Sigma$ which contain $V$, and consider its $N$-fold product $X\times \cdots \times X$. Let $Z\subset X\times \cdots\times X$ be the set of elements $(\beta_1, \beta_2, \ldots , \beta_N)$ such that $\beta_i = \beta_j$ for some $i\neq j$. We then need to show that $(X\times \cdots \times X \setminus Z) \cup \{(\pi, \ldots , \pi)\}$ is path connected. Note that $Z$ can be written as the union of $Z_{ij} := \{(\beta_1, \ldots , \beta_N) : \beta_i = \beta_j\}$. On the other hand the dimension of $Z_{ij}$ equals the dimension of $Z_{12}$, while the latter can be written as $\Delta \times X \times \cdots \times X$ where $\Delta \subset X\times X$ is the diagonal $\{(\beta, \beta): \beta \in X\}$. Hence the codimension of $Z$ in $X\times \cdots \times X$ equals the codimension of the diagonal $\Delta$ in $X\times X$. Since the codimension of $\Delta$ in $X\times X$ is the dimension of $X$, which is strictly larger than $1$, we conclude that the codimension of $Z$ in $X\times \cdots \times X$ is at least $2$. This in particular implies the path connectedness of the desired set{, meaning we can assume our path obeys $\alpha_i(t)\neq\alpha_j(t)$ for all $i\neq j$ and all $t\in (0,1)$.}

Now let $t\mapsto \Sbf(t)= \alpha_1 (t) \cup \cdots \cup \alpha_{N} (t)$, $t\in [0,1]$, denote the above continuous path from $\Sbf$ to $\pi$. Observe that by definition, 
\begin{align}
&t\mapsto \hat{\Ebf} (T, \Sbf (t), \Bbf_1) \quad \mbox{is continuous for $0\leq t \leq 1$,}\\
&t\mapsto \hat{\Ebf} (\Sbf (t), T, \Bbf_1) \quad \mbox{is continuous for $0 \leq t <1$,}
\end{align}
and
\[
\lim_{t\uparrow 1} \hat{\Ebf} (\Sbf (t), T, \Bbf_1) = N \hat{\Ebf} (\pi, T, \Bbf_1)\, .
\]
These properties, combined with \eqref{e:annoying} and \eqref{e:planar-larger-than-A}, are enough to claim the existence of a $t_e\in [0, 1)$ such that for $\Sbf_e := \Sbf (t_e)$ we have 
\[
\frac{1}{2} \eps_0^{-2} \Abf^2 < \mathbb{E} (T, \Sbf_e, \Bbf_1) < \eps_0^{-2} \Abf^2\, .
\]
However, it might be that $\mathbf{S}_e$ is not an element of $\mathscr{C} (Q,0)${; this can only happen if} the intersection of two of the planes forming $\Sbf_e$ is strictly bigger than $V$ (and so is $(m-1)$-dimensional as they do not coincide). Denote by $Z'\subset X \times \cdots \times X$ the set of elements $(\beta_1, \ldots , \beta_N)$ such that ${\rm dim}\, (\beta_i\cap \beta_j) \geq m-1$ for some $i\neq j$. $Z'$ is closed and has non-empty interior {in $X\times \cdots\times X$}. In particular, if we perturb the planes forming $\Sbf_e$ slightly (within the $N$-fold product space $X\times\dots\times X$), in light of the continuity of the map $\tilde\Sbf \mapsto \Ebf(T,\tilde\Sbf,\Bbf_1)$, we can ensure that the inequality above holds and that at the same time the new cone, which we will abuse of notation and still denote by $\Sbf_e$, is an element of $\mathscr{C} (Q, 0)$.

\medskip

We have now found a suitable cone $\Sbf_e$ when the unfavourable assumption \eqref{e:annoying} holds. {Notice that for such $\Sbf_e$ we also have
\[
\mathbb{E}(T,\Sbf_e,\Bbf_1) < \eps_0^{-2}\Abf^2 \leq \eps_0^{-2}\eps^2\Ebf^p(T,\Bbf_1)
\]
by \eqref{e:smallness-abcd}, and thus if $\eps^2<\eps_0^4$, we have that the assumptions of Proposition \eqref{p:balancing} hold for $T$ and $\Sbf_e$.} If we are in the favourable situation where \eqref{e:annoying} does not hold, then we have
\[
\Abf^2 \leq \varepsilon_0^2 \mathbb{E} (T, \Sbf, \Bbf_1)\, ,
\]
and so we can simply set $\Sbf_e:=\Sbf$, and again all the assumptions of Proposition \ref{p:balancing} hold for $T$ and $\Sbf_e$.

Thus, we are in a position to apply Proposition \ref{p:balancing} to $T$ and $\Sbf_e$ to find a cone $\Sbf_1$ which satisfies (ii) and is $M$-balanced. However, it must be noticed that indeed Proposition \ref{p:balancing} is proved by showing that $\Sbf_1$ satisfies the smallness assumption \eqref{e:evensmallerexcess} in Proposition \ref{p:balancing-2}, with the choice of $\eps=\eps(m,n,\bar{n},N)$ therein. In particular, since the argument of Proposition \ref{p:balancing} allows one to ensure that $\Sbf_1$ satisfies \eqref{e:evensmallerexcess} with $\eps$ replaced by any $\delta\leq\eps$, up to further decreasing $\eps_0$ (dependent on $\delta$), we get exactly \eqref{e:pruned-1}, and moreover we are still left with the freedom of choosing $\delta$. Thus, $\Sbf_1$ satisfies (i) and (ii).

By Proposition \ref{p:balancing} we also know that 
\[
\mathbb{E} (T, \Sbf_1, \Bbf_1) \leq \bar{C} \mathbb{E} (T, \Sbf_e, \Bbf_1)\, ,
\]
{where $\bar{C} = \bar{C}(Q,m,n,\bar{n},\delta)$.} On the other hand, by construction either $\Sbf_e= \Sbf$, or 
\[
\mathbb{E} (T, \Sbf_e, \Bbf_1) \leq \varepsilon_0^{-2} \Abf^2\, ,
\]
and in particular we infer (iii) (we stress that now $\eps_0 = \eps_0(Q,m,n,\bar{n},\delta)$).

So far we have proved that (i), (ii), and (iii) hold. Property (iv) would follow from Proposition \ref{p:balancing} in the case $\Sbf_e = \Sbf$, but unfortunately this may not be the case and so we will have to provide a more subtle argument. In fact, we will use the validity of (B) and (P) for $\Sbf_1$ to prove (iv) even when $\Sbf_e\neq \Sbf$ (for $\delta$ chosen sufficiently small).

We first assume that $\delta = \delta(Q,m,n,\bar{n})>0$ is sufficiently small so that Lemma \ref{l:splitting-1} and Proposition \ref{p:Lipschitz-1} apply to $T_{0, 1/4}$ and $\Sbf_1$. We will show that this implies
\begin{equation}\label{e:not-too-far}
\dist^2 (\Sbf_1\cap \Bbf_1, \Sbf\cap \Bbf_1)^2 \leq C (\mathbb{E} (T, \Sbf_1, \Bbf_1) +{\mathbb{E}(T,\Sbf,\Bbf_1) +} \Abf^2)\, .
\end{equation}
As observed, we would need to show this when $\Ebb(T,\Sbf,\Bbf_1)<\eps_0^{-2}\Abf^2$, but the argument is in fact more general and does not use the latter information. {Combined with (iii), \eqref{e:not-too-far} proves (iv).}

We enumerate the (distinct) planes $\beta_1, \ldots , \beta_{N'}$ forming $\mathbf{S}_1$, while we also recall the enumeration $\alpha_1,\dots,\alpha_N$ for $\Sbf$. Let $W_i{\subset\Bbf_4}$ be the disjoint neighbourhoods of the planes $\beta_i$ as in Lemma \ref{l:splitting-1}, and consider their homothetic rescalings $\tilde{W}_i\coloneqq (W_i)_{0,4}$. Observe that $T\res \Bbf_{3/4} \setminus B_{1/32} (V)$ is supported in the union of the $\tilde W_i$ and let $T_i := T \res \tilde{W}_i$. By Proposition \ref{p:Lipschitz-1}(c) we know that 
\begin{equation}\label{e:a-height-bound}
\dist (p, \beta_i) \leq C (\mathbb{E} (T, \Sbf_1, \Bbf_1) + \Abf^2)^{1/2} \qquad \text{for all }p\in \spt (T_i)\, .
\end{equation}
For each fixed $i=1,\dots,N'$, consider a unit vector $e_1\in V^\perp \cap \beta_i$, let $\xi_1:= \frac{e_1}{4}$, and define the disk $B_{i}:= B_{1/32} (\xi_1, \beta_i)$. By Proposition \ref{p:Lipschitz-1}(f) we know that $((\mathbf{p}_{\beta_i})_\sharp T_i) \res (B_{1/2} (0, \beta_i) \setminus B_{1/16} (V)) = Q_i \llbracket (B_{1/2} (0, \beta_i) \setminus B_{1/16} (V)\rrbracket$ for some integer $1\leq Q_i\leq Q$, and thus
\begin{equation}\label{e:Ti-mass}
\|T_i\| (\mathbf{p}_{\beta_i}^{-1} (B_{i})) \geq c_0>0
\end{equation}
for some geometric constant $c_0 (m,n, \bar n)$. Since 
\[
\int \dist^2 (p, \Sbf)\, d\|T_i\| (p) \leq \mathbb{E} (T, \Sbf, \Bbf_1)\, ,
\]
{for a fixed constant $\tilde{C} = \frac{C_*}{c_0}$, with $C_*> 0$ to be determined, by Chebyshev's inequality and \eqref{e:Ti-mass} we have
\begin{align*}
    \|T_i\|(\{p\in \mathbf{p}_{\beta_i}^{-1}(B_i):\dist^2(p,\Sbf) > \tilde{C}\mathbb{E} (T, \Sbf, \Bbf_1)\}) &\leq C_*^{-1} \|T_i\| (\mathbf{p}_{\beta_i}^{-1} (B_{i})) \, .
\end{align*}}
Thus, the set $E_i\subset\mathbf{p}_{\beta_i}^{-1}(B_i){\cap \spt\|T_i\|}$ of points $p\in \mathbf{p}_{\beta_i}^{-1}(B_i)$ which obey
\begin{equation}\label{e:L2-Linfty-S}
\dist^2 (p, \Sbf) \leq \tilde{C} \mathbb{E} (T, \Sbf, \Bbf_1)\, ,
\end{equation}
has $\|T_i\|(E_i)\geq (1-1/C_*)\|T_i\|(\mathbf{p}_{\beta_i}^{-1}(B_i)).$

We then use Proposition \ref{p:Lipschitz-1}(iv) to estimate
\begin{align*}
Q_i\mathcal{H}^m(B_i\setminus\mathbf{p}_{\beta_i}(E_i)) & \leq C_0\|T_i\|(\mathbf{p}_{\beta_i}^{-1}(B_i\setminus \mathbf{p}_{\beta_i}(E_i))) + C_0(\Abf^2+\mathbb{E}(T,\Sbf_1,\Bbf_1))^{1+\gamma}\\
& \leq C_0\|T_i\|(\mathbf{p}_{\beta_i}^{-1}(B_i)\setminus E_i) + C_0(\Abf^2+\Ebb(T,\Sbf_1,\Bbf_1))^{1+\gamma}\\
& \leq \frac{C_0}{C_*} \|T_i\| (\mathbf{p}_{\beta_i}^{-1}(B_i)) + C_0 (\Abf^2+\Ebb(T,\Sbf_1,\Bbf_1))^{1+\gamma}\\
& \leq \frac{C_0}{C_*} Q_i \mathcal{H}^m (B_i) + C_0(\Abf^2+\Ebb(T,\Sbf_1,\Bbf_1))^{1+\gamma}
\end{align*}
where $C_0 = C_0(Q,m,n,\bar{n})$. Note however that $\Ebb (T, \Sbf_1, \Bbf_1) \leq C_1 (\Ebb (T, \Sbf, \Bbf_1) + \Abf^2)$ for another constant $C_1= C_1 (Q,m,n,\bar{n}, \delta)$. Hence, if $\Abf^2+\mathbb{E}(T,\Sbf,\Bbf_1)$ is sufficiently small and $C_*$ sufficiently large compared to $C_0$, we get
\[
\mathcal{H}^m (\mathbf{p}_{\beta_i} (E_i)) \geq \frac{1}{2} \mathcal{H}^m (B_i)\, .
\]

For each point $p\in E_i$ we let $j (p){\in \{1,\dotsc,N\}}$ be an index such that $\dist (p, \alpha_{j(p)}) = \dist (p, \Sbf)$. We then define $F_{i,j} := \mathbf{p}_{\beta_i} (\{p\in E_i : j(p)=j\})$. Clearly from the above lower bound on $\mathcal{H}^m(\mathbf{p}_{\beta_i}(E_i))$, there must be {an index $j_*\in \{1,\dots,N\}$} for which ${\Hcal^m(F_{i,j_*})\geq\frac{1}{2N}\Hcal^m(B_i)}$. {Fix now an arbitrary point $w_1$ in $F_{i,j*}$ and recall that, since it belongs to $B_i = B_{1/32} (\xi_1, \beta_i)$, $\frac{3}{8} \geq |w_1|\geq \frac{1}{8}$. Next,we let $r_*>0$ and consider the  open set
\[
\Lambda (r_*) := \{w\in \beta_i : |w\cdot w_1|^2 > {(1-r_*^2) |w_1|^2 |w|^2}\}\, .
\]
By the lower bound on the Hausdorff measure of $F_{i, j_*}$ the set $F_{i,j*}\setminus \Lambda (r_*)$ must contain at least one point $w_2$ if we choose $r_* = r_* (m,N)$ appropriately.

Clearly $w_2$ enjoys as well the bound $\frac{3}{8} \geq |w_2|\geq \frac{1}{8}$. Moreover, since $e_1$ is orthogonal to $V$, the angles formed by the $w_k$ and $V$ must both be larger than a geometric constant. Finally, the sine of the angle formed by $w_1$ and $w_2$ is at least $r_*$.  
In particular we can complete the pair $w_1, w_2$ to a basis of $\beta_i$, via an orthonormal basis $v_1, \ldots , v_{m-2}$ of $V$. {Thus,} any vector $v\in \Bbf_1 \cap \beta_i$ can be written as a linear combination
\[
v = \sum_i \lambda_i v_i + \lambda_{m-1} w_1 + \lambda_m w_2
\]
for a choice of $\lambda_i$ which satisfy $|\lambda_i|\leq C$ for some constant $C = C(m, N)$.}

For the points $w_1, w_2$, {if we let $p_1,p_2$ denote points in $E_i$ which obey $j(p)=j_*$ and $\mathbf{p}_{\beta_i}(p_k) =w_k$ for $k=1,2$, it follows from \eqref{e:L2-Linfty-S} and \eqref{e:a-height-bound} that, for $k=1,2$,}{
\begin{align*}
    |\mathbf{p}_{\alpha_{j_*}}^\perp (w_k)|^2 & = \dist^2 (w_k, \alpha_{j_*})
    \leq {2}\dist^2(w_k,p_k) + {2}\dist^2(p_k,\alpha_{j_*})\\
    & = {2}\dist^2(p_k,\beta_i) + {2}\dist^2(p_k,\Sbf) \leq C(\Ebb(T,\Sbf_1,\Bbf_1)+\Ebb(T,\Sbf,\Bbf_1) + \Abf^2)
\end{align*}
where $C = C(Q,m,n,\bar{n},\delta)$.}
  In particular, using the linearity of $\mathbf{p}_{\alpha_{j_*}}^\perp$ and the fact that $V\subset \alpha_{j_*}$, this shows that 
\[
\dist^2 (v, \alpha_{j_*}) \leq C (\mathbb{E} (T, \Sbf_1, \Bbf_1) + \mathbb{E} (T, \Sbf, \Bbf_1) + \Abf^2) \qquad \forall v\in \beta_i \cap \Bbf_1\, .
\]
Summarizing our argument so far, for every plane {$\beta_i$ in $\Sbf_1$ we can show the existence of a plane $\alpha_{j_*}$ in $\Sbf$} such that 
\begin{equation}\label{e:dist-planes}
\dist^2 (\beta_i\cap\Bbf_1, \alpha_{j_*}\cap\Bbf_1) \leq C (\mathbb{E} (T, \Sbf_1, \Bbf_1) + \mathbb{E} (T, \Sbf, \Bbf_1) + \Abf^2)\, .
\end{equation}
We now would like to prove the converse; namely, if we fix any plane $\alpha_j$ forming $\Sbf$, we look for a plane $\beta_{i_*}$ in $\Sbf_1$ satisfying \eqref{e:dist-planes} with $i_*,j$ in place of $i,j_*$. This time, we choose a unit vector $f_1\in V^\perp\cap\alpha_j$. We let $\zeta_1=f_1/4$ and set $B_j := B_{1/32} (\zeta_1, \alpha_j)$. We then know that 
\[
\int_{B_j} \dist^2 (q, \spt (T)) d\mathcal{H}^m (q) \leq \hat{\mathbf{E}} (\Sbf, T, \Bbf_1)\, .
\]
Once again applying Chebyshev's inequality, this time with $\Hcal^m$, we {get that if $F_j\subset B_j$ is the set of points $q\in B_j$ which obey} 
\begin{equation}\label{e:L2-Linfty-reverse}
\dist^2 (q, \spt (T)) \leq (2/\Hcal^m(B_j)) \mathbb{E} (T, \Sbf, \Bbf_1)\, , 
\end{equation}
{then we have $\Hcal^m(F_j)\geq \frac{1}{2}\Hcal^m(B_j)$.}
Now for each $q\in F_j$ let $p (q)\in\spt (T)$ be such that $\dist (q, \spt (T)) = |q-p(q)|$. {Thus, \eqref{e:L2-Linfty-reverse} tells us that if $\eps$ is chosen} small enough, we can guarantee that $|q-p(q)|{<1/16}$  and so in particular $p(q)\in \Bbf_{3/4}\setminus B_{1/32} (V)$. It thus follows that {$p(q)\in \tilde{W}_i$ for some $i\in\{1,\dots,N'\}$}. For each such $i$, let $F_{j,i}$ be the set of points $q\in F_j$ for which at least one such $p(q)$ belongs to $\tilde{W}_i$. Arguing as before, it follows that for some $i_*$ we have ${\Hcal^m(F_{j,i_*})\geq \frac{1}{2N'}\Hcal^m(B_j)}$. We can now argue as above that $F_{j,i_*}$ contains two appropriate points $w_1, w_2$, complete the pair with vectors in $V$ to form an appropriate base of $\alpha_j$ and use this base to prove
\[
\dist^2 (\alpha_j \cap \Bbf_1, \beta_{i_*}\cap \Bbf_1) \leq C (\mathbb{E} (T, \Sbf_1, \Bbf_1) + \mathbb{E} (T, \Sbf, \Bbf_1) + \Abf^2)\, .
\]
{Combined with \eqref{e:dist-planes},} this therefore completes the proof of {\eqref{e:not-too-far} and hence} the first step.

\medskip

{\bf Step 2. Concluding from the reduction.} Now that we have reduced to the setting where we may assume the validity of (B) and (P), we conclude the proof of the proposition by showing that the claimed conclusions hold under these additional assumptions, if we choose $\delta$ small enough. First of all we observe that, by (iv) in Step 1, \eqref{e:no-change-E}, \eqref{e:E-mu-comparable}, and \eqref{e:E-mu-comparable-2} we have 
\begin{align}\label{e:E-mu-comparable-3}
C^{-1} \mathbf{E}^p (T, \Bbf_1) \leq \min \{\boldsymbol{\mu} (\Sbf)^2, \boldsymbol{\mu} (\Sbf')^2\}
\leq \max \{\boldsymbol{\mu} (\Sbf)^2, \boldsymbol{\mu} (\Sbf')^2\} \leq C  \mathbf{E}^p (T, \Bbf_1)\, .
\end{align}
So we can appeal to Lemma \ref{l:spine} to conclude \eqref{e:align-10} from \eqref{e:near-10}. We will thus focus on proving \eqref{e:near-10}.

The latter task is similar to the one accomplished at the end of Step 1; while it is more complicated by the fact that $V (\mathbf{S})$ and $V (\mathbf{S}')$ might be different {(unlike $V(\Sbf)$ and $V(\Sbf_1)$ in Step 1)}, we can now take advantage of (B) and (P) for both {(whereas previously, we only knew that $\Sbf_1$ was balanced)}. Indeed, by choosing $\delta$ small enough we can ensure, using Lemma \ref{l:splitting-1} and Proposition \ref{p:Lipschitz-1} for $\Sbf', \ T_{0, r_0/4}$, and $\Sbf, \ T_{0,1/4}$, respectively, that the following situation holds:
\begin{itemize}
\item[(i)] If we enumerate the planes $\beta_1, \ldots , \beta_{N'}$ forming $\Sbf'$ and set $V' := V (\Sbf')$, then there are  pairwise disjoint neighborhoods $W'_j$ of $\beta_j \cap \Bbf_{r_0/2} \setminus B_{r_0/16} (V')$, a constant $C=C(Q,m,n,\bar n)>0$ and positive integers $Q'_j$ such that 
\begin{align}
& \spt (T) \cap \Bbf_{r_0/2} \setminus B_{r_0/16} (V') \subset \bigcup_j W'_j\, ;\label{e:i-crude-1}\\
&\dist^2 (p, \beta_j) \leq C (\mathbb{E} (T, \Sbf', \Bbf_{r_0}) + \Abf^2 r_0^2) r_0^2 \qquad \forall p\in \spt (T) \cap W'_j\, ;\label{e:i-crude-2}\\
&(\mathbf{p}_{\beta_j})_\sharp (T\res  W'_j \cap \mathbf{p}_{\beta_j}^{-1} (B_{r_0/2} (0, \beta_j) \setminus B_{r_0/16} (V')))\\
& \hspace{15em} = Q'_j \llbracket B_{r_0/2} (0, \beta_j) \setminus B_{r_0/16} (V')\rrbracket \, .\nonumber
\end{align}
\item[(ii)] If we enumerate the planes $\alpha_1, \ldots , \alpha_N$ forming $\Sbf$ and set $V = V (\Sbf)$, then there are pairwise disjoint neighborhoods $W_i$ of $\alpha_i \cap \Bbf_{1/2} \setminus B_{r_0/16} (V)$, constants $C=C(Q,m,n,\bar n)>0$, $\bar C=\bar C(Q,m,n,\bar n,r_0)>0$ and positive integers $Q_i$ such that
\begin{align}
&\spt (T) \cap \Bbf_{1/2} \setminus B_{r_0/16} (V) \subset \bigcup_i W_i\\
&\dist^2 (p, \alpha_i) \leq C (\mathbb{E} (T, \Sbf, \Bbf_1) + \Abf^2) \notag \\
&\qquad\qquad\hspace{1em}\leq \bar C (\mathbb{E} (T, \Sbf, \Bbf_1) + \Abf^2 r_0^2) r_0^2 \qquad \forall p\in \spt (T) \cap W_i \cap \Bbf_{r_0}\label{e:ii-crude-2}\\
&(\mathbf{p}_{\alpha_i})_\sharp (T\res W_i \cap \mathbf{p}_{\alpha_i}^{-1} (B_{r_0} (0, \alpha_i) \setminus B_{r_0/16} (V)))\label{e:ii-crude-3}\\
& \hspace{15em} = Q_i \llbracket B_{r_0} (0, \alpha_i) \setminus B_{r_0/16} (V)\rrbracket \, .\nonumber
\end{align} 
\end{itemize}
We will now proceed to show that all these estimates imply, for each plane $\alpha_i$ in $\Sbf$, the existence of a plane $\beta_j$ in $\Sbf'$ such that 
\begin{equation}\label{e:again-i-j}
\dist^2 (\alpha_i \cap \Bbf_1, \beta_j \cap \Bbf_1) \leq \bar C (\mathbb{E} (T, \Sbf, \Bbf_1) + \mathbb{E} (T, \Sbf', \Bbf_1) + \Abf^2)\, . 
\end{equation}
Since the argument can be symmetrized {to switch the roles of $\alpha_i$ and $\beta_j$}, this would conclude the proof of \eqref{e:near-10}.
Without loss of generality we assume that $i=1$ and we start by fixing a vector $e\in \alpha_1$ with $|e|=\frac{r_0}{4}$ such that $\Bbf_{2 c_0 r_0} (e) \subset \Bbf_{r_0/2}\setminus (B_{r_0/16} (V) \cup B_{r_0/16} (V'))$, where $c_0 = c_0 (m)$ is a positive dimensional constant. For each point $q\in B_{c_0 r_0} (e, \alpha_1)${, by \eqref{e:ii-crude-2} and \eqref{e:ii-crude-3} above} we {can find} a point $p=p(q)\in \spt (T)\cap W_{1}$ such that 
\[
{|p-q|^2} \leq \bar C (\mathbb{E} (T, \Sbf, \Bbf_1) + \Abf^2 r_0^2) r_0^2
\]  
and then {by \eqref{e:i-crude-1}} we find $W'_j$ (for some $j=j(q)\in \{1,\dots,N'\}$) such that $p\in W'_j$. In particular, setting $q_j := \mathbf{p}_{\beta_j} (p)$, we conclude {from \eqref{e:i-crude-2}} that 
\[
|p-q_j|^2 \leq C (\mathbb{E} (T, \Sbf, \Bbf_1) + \Abf^2 r_0^2 + \mathbb{E} (T, \Sbf^\prime, \Bbf_{r_0})) r_0^2\, .
\]
So, for each point $q\in B_{c_0 r_0} (e,\alpha_1)$ we conclude that there is {a plane $\beta_{j(q)}$ in $\Sbf'$} such that
\begin{equation}\label{e:on-the-whole-disk}
\dist (q, {\beta_{j(q)}})^2 \leq C (\mathbb{E} (T, \Sbf, \Bbf_1) + \Abf^2 r_0^2 + \mathbb{E} (T, \Sbf^\prime, \Bbf_{r_0})) r_0^2\, ,
\end{equation}
where $C=C(Q,m,n,\bar n, r_0)>0$. Select now $m$ linearly independent vectors $v_1, \ldots , v_m \in B_{c_0 r_0} (e,\alpha_1)$ with the property that, if we set $e_i:= \frac{v_i}{r_0}$, then every vector $v\in \Bbf_1\cap \alpha_1$ is a linear combination
\[
v=\sum_i \lambda_i e_i
\] 
with $|\lambda_i|\leq C$, for some constant $C$ which depends only on $c_0$ and therefore only on $m$.

{In light of the above argument, for each vector $v_i$ we may find a plane $\beta_{j(i)}$ in $\Sbf'$} such that 
\begin{equation}\label{e:another-base}
\dist^2 (v_i, \beta_{j(i)}) \leq C (\mathbb{E} (T, \Sbf, \Bbf_1) + \Abf^2 r_0^2 + \mathbb{E} (T, \Sbf^\prime, \Bbf_{r_0})) r_0^2\, .
\end{equation}
{Now, we would like to achieve \eqref{e:another-base} but with \emph{the same plane}, say $\beta_{j(1)}$, for each $v_i$, up to increasing the constant by a fixed amount.} Suppose that $j(i)\neq j(1)$ for some $i>1$. Consider the line segment $[v_1, v_i]$ and parameterize it with a constant speed curve $\gamma :[0,1]\to [v_1, v_i]$ {with $\gamma(0)=v_1$ and $\gamma(1)=v_i$}. 
By continuity {of $t\mapsto \dist (\gamma (t), \beta_{j(1)})$}, for some {$\sigma\in (0,1]$} we have 
\[
\dist^2(\gamma (t), \beta_{j(1)}) \leq 2C (\mathbb{E} (T, \Sbf, \Bbf_1) + \Abf^2 r_0^2 + \mathbb{E} (T, \Sbf^\prime, \Bbf_{r_0})) r_0^2\qquad \forall t\in [0, \sigma]\, .
\]
Now let $\tau$ be the maximal number in $[0,1]$ such that 
\begin{equation}\label{e:maximal}
\dist^2 (\gamma (\tau), \beta_{j(1)}) \leq 2C (\mathbb{E} (T, \Sbf, \Bbf_1) + \Abf^2 r_0^2 + \mathbb{E} (T, \Sbf^\prime, \Bbf_{r_0})) r_0^2\, .
\end{equation}
If $\tau =1$ then we {indeed arrive at \eqref{e:another-base} with $\beta_{j(i)}$ replaced by $\beta_{j(1)}$ and $C$ replaced by $2C$}. Otherwise, if $\tau<1$, then we must have equality in \eqref{e:maximal}. Since $\gamma (\tau) \in B_{c_0r_0} (e, \alpha_1)$, \eqref{e:on-the-whole-disk} implies that there must be another index $j'={j'(\tau)}\neq j(1)$ such that 
\begin{equation}\label{e:close-to-j'}
\dist^2 (\gamma (\tau), \beta_{j'}) \leq C (\mathbb{E} (T, \Sbf, \Bbf_1) + \Abf^2 r_0^2 + \mathbb{E} (T, \Sbf^\prime, \Bbf_{r_0})) r_0^2\, .
\end{equation}
If {$p_{j(1)} = \mathbf{p}_{\beta_{j(1)}}(\gamma(\tau))$} and {$p_{j'} = \mathbf{p}_{\beta_{j'}}(\gamma(\tau))$} are the respective nearest point projections of $\gamma (\tau)$, and if we ensure that {$\eps$ is taken to be small enough}, we can guarantee that both points belong to $\Bbf_{2c_0 r_0} (e)$. Recall that our choice of the vector $e$ guarantees that this ball does not intersect $B_{r_0/16} (V')$. In particular, { if we set $q_{j(1)} = p_{j(1)}/r_0$ and $q_{j^\prime} = p_{j^\prime}/r_0$} we find $q_{j(1)}\in \beta_{j(1)}\cap \Bbf_1 \setminus B_{1/16} ({V'})$ and $q_{j'}\in \beta_{j'}\cap \Bbf_1 \setminus B_{1/16} ({V'})$ are such that 
\[
|q_{j(1)} - q_{j'}|^2 \leq {6} C (\mathbb{E} (T, \Sbf, \Bbf_1) + \Abf^2 r_0^2 + \mathbb{E} (T, \Sbf^\prime, \Bbf_{r_0})) \, .
\]
Since the pair of planes have $V'$ as common spine and are $M$-balanced, we can {combine the above with} Corollary \ref{c:growth} {(namely, first use \eqref{e:Haus=max_eigen}, then $M$-balanced, then the first inequality in \eqref{e:Morgan-max-min}, then the above)} to conclude that
\begin{equation}\label{e:j'-close-to-1}
\dist^2 (\beta_{j(1)}\cap \Bbf_1, \beta_{j'}\cap \Bbf_1) \leq C' (\mathbb{E} (T, \Sbf, \Bbf_1) + \Abf^2 r_0^2 + \mathbb{E} (T, \Sbf^\prime, \Bbf_{r_0}))
\end{equation}
for some constant $C'$ which depends on $M$ and the previous constant $C$. 

We now look at the point $\tau'$ which is the maximum in the set of $t\in [0,1]$ for which
\begin{equation}\label{e:next-step}
\dist^2 (\gamma (t)), \beta_{j'})^2 \leq 2C (\mathbb{E} (T, \Sbf, \Bbf_1) + \Abf^2 r_0^2 + \mathbb{E} (T, \Sbf^\prime, \Bbf_{r_0})) r_0^2\, ,
\end{equation}
{where $C$ is again the constant in \eqref{e:on-the-whole-disk}.} Now $\tau'$ must be strictly larger than $\tau$, by \eqref{e:close-to-j'}. Moreover, if $\tau'<1$ then 
\begin{equation}\label{e:maximality-3}
\dist^2 (\gamma (\tau'), \beta_{j'}) = 2C (\mathbb{E} (T, \Sbf, \Bbf_1) + \Abf^2 r_0^2 + \mathbb{E} (T, \Sbf^\prime, \Bbf_{r_0})) r_0^2 \, .
\end{equation}
In this case we can look at the closest plane $\beta_{j''}$ to $\gamma ({\tau'})$, which satisfies
\begin{equation}\label{e:close-to-j''}
\dist^2 (\gamma (\tau'), \beta_{j''}) \leq C (\mathbb{E} (T, \Sbf, \Bbf_1) + \Abf^2 r_0^2 + \mathbb{E} (T, \Sbf^\prime, \Bbf_{r_0})) r_0^2\, ,
\end{equation}
where $C$ is the constant in \eqref{e:on-the-whole-disk}.
Obviously by construction $j''\not\in \{ j(1), j'\}$. But then we can repeat the argument above with $j'$ replacing $j(1)$ and $j''$ replacing $j'$ and estimate 
\begin{equation}\label{e:j''-close-to-j'}
\dist^2 (\beta_{j'}\cap \Bbf_1, \beta_{j''}\cap \Bbf_1) \leq C' (\mathbb{E} (T, \Sbf, \Bbf_1) + \Abf^2 r_0^2 + \mathbb{E} (T, \Sbf^\prime, \Bbf_{r_0}))
\end{equation}
with the same constant $C'$ of \eqref{e:close-to-j'}. We can iterate this procedure until we reach the right endpoint $t=1$. Any time that we do not stop, we find a {\em new} plane in the collection forming $\Sbf^\prime$. Since there is a finite number $N'$ of such planes, we conclude that the procedure stops after at most $N'-1$ steps. The last plane we find is a plane $\beta_{\kappa_i}$ which satisfies \eqref{e:another-base} with $2C$ in place of $C$ {and $\beta_{\kappa_i}$ in place of $\beta_{j(i)}$}. On the other hand we have a chain of at most $N'-1$ inequalities of the form \eqref{e:j''-close-to-j'} between each pair of consecutive planes found during the procedure. In particular { for every $i$ we have}
\[
\dist^2 (\beta_{j (1)}\cap \Bbf_1, \beta_{{\kappa_i}}\cap \Bbf_1) \leq (N'-1)^2 C' (\mathbb{E} (T, \Sbf, \Bbf_1) + \Abf^2 r_0^2 + \mathbb{E} (T, \Sbf^\prime, \Bbf_{r_0}))\, ,
\]
{which combined with \eqref{e:another-base} gives that \eqref{e:another-base} holds with $\beta_{j(1)}$ in place of $\beta_{j(i)}$ with a larger constant $C$, for every $i$, which is what we wanted.}
Summarizing the discussion above, we have found a larger constant $C$ such that, if we set $j = j(1)$, then 
\[
\dist^2 (e_i, \beta_j) = |\mathbf{p}_{\beta_j}^\perp (e_i)|^2 \leq C  (\mathbb{E} (T, \Sbf, \Bbf_1) + \Abf^2 r_0^2 + \mathbb{E} (T, \Sbf^\prime, \Bbf_{r_0}))\, 
\] 
for every $i$.
{Writing an arbitrary vector in $\Bbf_1\cap\alpha_1$ in terms of the basis $e_i$ as described above,} we then conclude that
\[
\dist^2 (\alpha_1\cap \Bbf_1, \beta_j\cap \Bbf_1)^2 \leq C  (\mathbb{E} (T, \Sbf, \Bbf_1) + \Abf^2 r_0^2 + \mathbb{E} (T, \Sbf^\prime, \Bbf_{r_0}))\, .
\]
This completes the proof of \eqref{e:again-i-j} and thus completes the proof of the lemma.
\end{proof}

Thus, to prove Theorem \ref{c:decay}, now we just need to prove Theorem \ref{t:weak-decay}.

\section{Estimates at the spine}\label{p:estimates}

In this section we address the pivotal estimates needed at the spine of the cone for the proof of Theorem \ref{c:decay}, which are a suitable adaptation of the groundbreaking work of Simon in \cite{Simon_cylindrical}. They will be crucial for demonstrating that in the end, after a blow-up procedure {under the assumption that $Q$-points accumulate across the spine}, the graphical approximations constructed in Proposition \ref{p:first-blow-up} will remain {controlled} as one approaches the spine, and will converge to a Dir-minimizer that has an $(m-2)$-dimensional subspace of $Q$-points. We start by detailing the assumptions which will be used through this section. Note that, in contrast to \cite{DLHMSS}, we denote the estimates in Theorem \ref{t:HS}, Corollary \ref{c:HS-patch} and Proposition \ref{p:HS-3} collectively as \emph{Simon's estimates}. Indeed, although the first occurrence of an estimate analogous to \eqref{e:HS} below first appeared in the work of Hardt and Simon in \cite{HS}, the framework therein is significantly simpler. The first appearance of the estimates \eqref{e:HS}, \eqref{e:HS-additional}, \eqref{e:HS-patch} and \eqref{e:HS-3}, albeit in a multiplicity one setting, is in \cite{Simon_cylindrical}; the corresponding estimates in a setting of higher multiplicity, as is the case in the present work, first appeared in \cite{W14_annals}. We note that \eqref{e:HS-4} is a more refined version of an estimate in \cite{Simon_cylindrical}, but in this form it appeared first in \cite{W14_annals}. One important difference between our work and \cites{Simon_cylindrical,W14_annals} is that, even when we are at a fixed distance from the spine of the cone $\Sbf$ and the excess of the current from $\Sbf$ is very small, the current is not necessarily a graph, nor a multigraph (it is merely approximated suitably by the latter), and therefore we need to deal with appropriate additional error estimates coming from regions of non-graphicality. 

Throughout this section, we will work under an analogous assumption to Assumption \ref{a:refined}, but with possibly smaller parameters:
\begin{assumption}[Assumptions for Simon's estimates]\label{a:HS}
Suppose $T$ and $\Sigma$ are as in Assumption \ref{a:main} and $\|T\| (\Bbf_4) \leq 4^m (Q+\frac{1}{2}) \omega_m$. Suppose $\mathbf{S}=\alpha_1\cup \cdots \cup \alpha_N$ is a cone in $\mathscr{C} (Q)\setminus \mathscr{P}$ which is $M$-balanced, where $M>0$ is a given fixed constant, and $V$ is the spine of $\mathbf{S}$. For a sufficiently small constant $\eps = \eps(Q,m,n,\bar{n},M)$ smaller than the {$\eps$-}threshold  in Assumption \ref{a:refined}, whose choice will be fixed by the statements of Theorem \ref{t:HS}, Corollary \ref{c:HS-patch}, and Proposition \ref{p:HS-3} below, suppose that
\begin{equation}\label{e:smallness-alg-2}
\mathbb{E} (T, \mathbf{S}, \Bbf_4) + \mathbf{A}^2 \leq \varepsilon^2 \boldsymbol{\sigma} (\Sbf)^2\, .
\end{equation}
\end{assumption}

We recall once again the notation
\begin{align*}
\boldsymbol{\sigma} (\mathbf{S}) := \min_{i<j} \dist (\alpha_i \cap \Bbf_1, \alpha_j \cap \Bbf_1)\, , \qquad \text{and} \qquad
\boldsymbol{\mu} (\mathbf{S}):= \max_{i<j} \dist (\alpha_i \cap \Bbf_1, \alpha_j\cap \Bbf_1)\, . 
\end{align*}
{Before stating the inequalities that we need, let us introduce the following short-hand notation for points $q\in {\spt}\, (T)$ at which the $m$-rectifiable set ${\spt}\, (T)$ has an approximate tangent {plane} $\pi (q)$ oriented by the simple $m$-vector $\vec{T} (q)$:}
\begin{itemize}
\item $\mathbf{p}_{\vec{T}}$ and $\mathbf{p}_{\vec{T}}^\perp$ will denote the orthogonal projections onto $\pi (q)$ and its orthogonal complement $(\pi (q))^\perp$ {respectively};
\item $q_\parallel$ and $q^\perp$ will denote the vectors $\mathbf{p}_{\vec{T}} (q)$ and $\mathbf{p}_{\vec{T}}^\perp (q)$ {respectively}.
\end{itemize}

We now split the key estimates into three separate statements and we will proceed with their proofs afterwards.

\begin{theorem}[Simon's error and gradient estimates]\label{t:HS}
Assume $T$, $\Sigma$, and $\mathbf{S}$ are as in Assumption \ref{a:HS}, suppose in addition that $\Theta (T, 0) \geq Q$ and set $r= \frac{1}{3\sqrt{m-2}}$. Then there is a constant $C=C(Q,m,n,\bar{n},M)>0$ and a choice of $\varepsilon = \varepsilon(Q,m,n,\bar{n},M)>0$ in Assumption \ref{a:HS} sufficiently small such that
\begin{align}
\int_{\Bbf_r} \frac{|q^\perp|^2}{|q|^{m+2}}\, d\|T\| (q)
&\leq C (\mathbf{A}^2 + {\hat{\Ebf} (T, \mathbf{S}, \Bbf_4))}\, \label{e:HS}\\
\int_{\Bbf_r} |\mathbf{p}_V\circ \mathbf{p}_{\vec{T}}^\perp|^2\, d\|T\| &\leq C (\mathbf{A}^2 + {\hat{\Ebf} (T, \mathbf{S}, \Bbf_4))}\, .\label{e:HS-additional}
\end{align}
\end{theorem}
We refer to the first estimate \eqref{e:HS} as ``Simon's error estimate", since it is estimating the error in the monotonicity formula. Meanwhile, we refer to the second estimate \eqref{e:HS-additional} as ``Simon's gradient estimate", since if $T$ were the graph of a function, this {would be an $L^2$ bound on the derivative of the function in the directions parallel to $V$.}

An important corollary which will not require much additional work is the following.

\begin{corollary}[Simon's non-concentration estimate]\label{c:HS-patch}
Assume $T$, $\Sigma$, $\mathbf{S}$ and $r$ are as in Theorem \ref{t:HS}. Then, there is a choice of $\varepsilon = \varepsilon(Q,m,n,\bar{n},M)$ in Assumption \ref{a:HS}, possibly smaller than that in Theorem \ref{t:HS}, such that for every $\kappa\in (0,m+2)$, 
\begin{equation}\label{e:HS-patch}
\int_{\Bbf_r} \frac{\dist^2 (q, \mathbf{S})}{|q|^{m+2-\kappa}}\, d\|T\| (q)
\leq C_\kappa (\mathbf{A}^2 + {\hat{\Ebf} (T, \mathbf{S}, \Bbf_4))}\, ,
\end{equation}
where here $C_\kappa = C_\kappa(Q,m,n,\bar{n},M,\kappa)$. 
\end{corollary}
Finally, this part will be concluded by deriving the following consequence of Corollary \ref{c:HS-patch}, which will require a subtle geometric consideration.
\begin{proposition}[Simon's shift inequality]\label{p:HS-3}
Assume $T$, $\Sigma$, and $\mathbf{S}$ are as in Assumption \ref{a:HS} and in addition $\{\Theta (T, \cdot)\geq Q\}\cap \Bbf_\varepsilon (0) \neq \emptyset$. Then there is a radius $r=r(Q,m,n,\bar{n})$ and a choice of $\varepsilon = \varepsilon(Q,m,n,\bar{n},M)$ in Assumption \ref{a:HS}, possibly smaller than those in Theorem \ref{t:HS} and Corollary \ref{c:HS-patch} such that for each $\kappa\in (0,m+2)$, there are constants $\bar{C}_\kappa=\bar{C}_\kappa(Q,m,n,\bar{n},M,\kappa)>0$ and $C=C(Q,m,n,\bar{n},M)$ such that the following holds. If $q_0\in \Bbf_r (0)$ and $\Theta (T, q_0)\geq Q$, then
\begin{align}
\int_{\Bbf_{4r} (q_0)} \frac{\dist^2 (q, q_0 + \mathbf{S})}{|q-q_0|^{m+2-\kappa}}\, d\|T\| (q) &\leq
\bar{C}_\kappa (\mathbf{A}^2 + {\hat{\Ebf} (T, \mathbf{S}, \Bbf_4))}\, .\label{e:HS-3} \\
|\mathbf{p}_{\alpha_1}^\perp (q_0)|^2 + \boldsymbol{\mu} (\mathbf{S})^2 |\mathbf{p}_{V^\perp\cap \alpha_1} (q_0)|^2 &\leq C (\mathbf{A}^2 + {\hat{\Ebf} (T, \mathbf{S}, \Bbf_4))}\label{e:HS-4}\, .
\end{align}
\end{proposition}

\begin{remark}\label{r:displacement-of-Q-points}
Observe that $\mathbf{p}_{V}^\perp = \mathbf{p}_{\alpha_1}^\perp + \mathbf{p}_{V^\perp \cap \alpha_1}$. In particular, when we know that $\mathbf{S}$ is at a fixed positive distance from any plane {(namely, $\boldsymbol{\mu}(\Sbf)$ is bounded from below away from $0$)}, then \eqref{e:HS-4} gives a control on how far $q_0$ can be from the spine $V$. This particular case corresponds to what Simon proves in \cite{Simon_cylindrical}, while, as already mentioned, \eqref{e:HS-4} is a refinement which appears first in the work \cite{W14_annals} of Wickramasekera.
\end{remark}

\subsection{Proof of the Simon's error and gradient estimates (Theorem \ref{t:HS})}

\subsubsection{Monotonicity formula}\label{ss:monotonicity}
We begin by recalling the monotonicity formula for mass ratios, along with some consequences. For $T$ as in Assumption \ref{a:main} and $\rho \in (0,4]$, the monotonicity formula reads
\[
\int_{\Bbf_\rho} \frac{|q^\perp|^2}{|q|^{m+2}}\, d\|T\| (q) =
\frac{\|T\| (\Bbf_\rho)}{\rho^m} - \omega_m \Theta (T,0) 
-\frac{1}{m} \int_{\Bbf_\rho} (q^\perp\cdot \vec{H}_T (q)) (|q|^{-m} - \rho^{-m})\, d\|T\| (q)\, .
\]
The {mean curvature} vector $\vec{H}_T (q)$ is defined as 
\[
\vec{H}_T (q) \coloneqq \sum_{i=1}^m A_\Sigma (e_i, e_i)\, ,
\]
where $A_\Sigma$ is the second fundamental form of $\Sigma$ and $e_1, \ldots, e_m$ is an orthonormal base of the approximate tangent $\pi (q)$ to $\spt (T)$ at $q$. 

We next assume that $\Theta (T,0)\geq Q$, fix the multiplicities $Q_1, \ldots , Q_N$ given by Proposition \ref{p:coherent} and observe that {for $\rho$ as above, we have}
\begin{equation}\label{e:split-in-planes}
\sum_i Q_i \mathcal{H}^m (\alpha_i \cap \Bbf_\rho) = Q \omega_m \rho^m \leq \omega_m \Theta (T,0) \rho^m\, .
\end{equation}
Fix $r = \frac{1}{3\sqrt{m-2}}$ as in the statement and note that any constants depending on $r$ become in turn dimensional constants. We moreover fix a smooth, monotone non-increasing function $\chi: [0, \infty)\to \mathbb R$ with $\chi \equiv 1$ on $[0,r]$ and $\chi \equiv 0$ on $[{2r}, \infty)$ and we introduce the function
\[
\Gamma (t) := -\int_t^\infty \frac{d}{ds} (\chi (s)^2) s^m\, ds\, .
\]
Recall the elementary equalities
\begin{align}
\int_{\Bbf_R} \chi (|q|)^2\, d\mu (q) &= \int_0^R \chi (t)^2 \frac{d}{dt} (\mu (\Bbf_t))\, dt\label{e:Fubini-1}\\
\int_{\Bbf_R} \Gamma (|q|)\, d\mu (q) &= \int_0^R \chi (t)^2 \frac{d}{dt} (t^m\mu (\Bbf_t))\, dt\, ,\label{e:Fubini-2}
\end{align}
which are valid for any Radon measure $\mu$ with $\mu (\{0\}) =0$ via Fubini's Theorem and the definition of the distributional derivative of the BV function $t\mapsto \mu(\Bbf_t)$. We then multiply both sides of the monotonicity formula by $\rho^m$, differentiate the resulting identity in $\rho$, multiply by $\chi (\rho)^2$, integrate {over $\rho \in[0,2r]$} and use \eqref{e:Fubini-1}, \eqref{e:Fubini-2}, and \eqref{e:split-in-planes} {(treating, for example, $\frac{|q^\perp|^2}{|q|^{m+2}} d\|T\|(q)$ as a Radon measure $d\mu$ as above)}. Since $\Gamma \geq C^{-1} \mathbf{1}_{[0, r]}$ for some constant $C(m,{r})>0$, we get
\begin{align}
\int_{\Bbf_r} \frac{|q^\perp|^2}{|q|^{m+2}}\, d\|T\| (q)
&\leq C\left[\int \chi^2 (|q|)\, d\|T\| (q) - \sum_i Q_i \int_{\alpha_i} \chi^2 (|q|)\, d\mathcal{H}^m (q)\right]\nonumber\\
&\qquad + C \underbrace{\int \frac{|\Gamma (|q|)q^\perp\cdot \vec{H}_T (q)|}{|q|^m}\,  d\|T\| (q)}_{=:\,\text{(A)}}\, .\label{e:monot-10}
\end{align}
We first demonstrate that the error term (A) coming from the mean curvature of $T$ may be controlled by $\Abf^2$ and to that end we observe that 
\begin{align}
|\vec{H}_T (q)| &\leq m \Abf\label{e:A-1}\, ,\\
|T_q \Sigma - T_0 \Sigma| & \leq C \Abf |q|\label{e:A-3}\\
|\mathbf{p}_{T_q \Sigma}^\perp (q)| &\leq C \Abf |q| \qquad \forall q\in \Sigma \label{e:A-2}\, .
\end{align}
Indeed the first inequality is obvious by the very definition of $\vec{H}_T$, while the second and the third are a simple exercise in differential geometry and we leave them to the reader.

Observe that since $T$ is area-minimizing in $\Sigma$, $\vec{H}_T (q)$ is orthogonal to $T_q \Sigma$, while $\pi (q)\subset T_q \Sigma$, so
\begin{equation}\label{e:perp}
q^\perp\cdot \vec{H}_T (q) = q \cdot \vec{H}_T (q) = \mathbf{p}_{T_q \Sigma}^\perp (q) \cdot \vec{H}_T (q)\, .
\end{equation}
In particular, since $\spt (T) \subset \Sigma$, \eqref{e:A-1} and \eqref{e:A-2} imply
\[
\text{(A)}\leq C \mathbf{A}^2 \int_{\Bbf_r} |q|^{1-m}\, d\|T\| (q) \leq C \mathbf{A}^2\, 
\]
where again $C=C(m,r)>0$. We can therefore write
\begin{align}
\int_{\Bbf_r} \frac{|q^\perp|^2}{|q|^{m+2}}\, d\|T\| (q)
&\leq C\left[\int \chi^2 (|q|)\, d\|T\| (q) - \sum_i Q_i \int_{\alpha_i} \chi^2 (|q|)\, d\mathcal{H}^m (q)\right] + C \mathbf{A}^2 .\label{e:monot-11}
\end{align}
{We now wish to estimate the first term on the right hand side of \eqref{e:monot-11} in terms of the excess $\Ebb(T,\Sbf,\Bbf_4)$; to achieve this, we will test the first variation formula for both $T$ and the current whose support is $\Sbf$ with the vector field}
\[
X (q) = \chi (|q|)^2 \mathbf{p}_{V}^\perp (q)\, .
\]
The first variation formula for $T$ reads
\begin{equation}\label{e:first-variation}
\int {\rm div}_{\vec{T}}\, X (q)\, d\|T\| (q) = - \underbrace{\int X^\perp (q)\cdot \vec{H}_T (q)\, d\|T\| (q)}_{=:\, \text{(B)}}\, , 
\end{equation}
where we use the shorthand notation $X^\perp (q) = \mathbf{p}_{\vec{T}}^\perp (X (q))$ and
\begin{equation}
{\rm div}_{\vec{T}}\, X (q) = \sum_i \partial_{e_i} X (q) \cdot e_i
\end{equation}
for any orthonormal base $e_1, \ldots, e_m$ of the approximate tangent $\pi (q)$ of $\spt (T)$.

Writing $x=\mathbf{p}_{V}^\perp (q)$, we can now use \eqref{e:A-1} to write
\[
|X^\perp (q)\cdot \vec{H}_T (q)| = \chi (|q|)^2 |\mathbf{p}_{T_q \Sigma}^\perp (x) \cdot \vec{H}_T (q)| \leq m \mathbf{A} |\mathbf{p}_{T_q \Sigma}^\perp (x)|\, .
\]
Observe next that $x = \mathbf{p}_V^\perp (q) = q - \mathbf{p}_V (q)$ and $V\subset T_0 \Sigma$. Therefore for $q\in \spt (T)$ we can use \eqref{e:A-3} and \eqref{e:A-2} to estimate
\begin{align*}
|\mathbf{p}_{T_q \Sigma}^\perp (x)| &\leq |\mathbf{p}_{T_q \Sigma}^\perp (q)| + |\mathbf{p}_{T_q\Sigma}^\perp (\mathbf{p}_V (q))|
\leq C \Abf |q| + |\mathbf{p}_{T_q \Sigma}^\perp - \mathbf{p}_{T_0 \Sigma}^\perp| |\mathbf{p}_V (q)|
\leq C \Abf |q|\, .
\end{align*}
In particular 
\[
|\text{(B)}| \leq C \mathbf{A}^2 \int \chi (|q|)^2 |q|\, d\|T\| (q) \leq C \mathbf{A}^2\, ,
\]
which in turn leads to 
\begin{equation}\label{e:first-variation-2}
\int {\rm div}_{\vec{T}} X (q)\, d\|T\| (q) = O (\mathbf{A}^2)\, .
\end{equation}
For $x=\mathbf{p}^\perp_V(q)$, $q\in \spt(T)$, let $\mathbf{p}_{\vec{T}}(x) \equiv\mathbf{p}_{\pi(q)}(x)$ denote the orthogonal projection of $x$ onto the approximate tangent plane $\pi(q)$ at $q$. We next compute $\diverg_{\vec{T}} X$:
\[
{\rm div}_{{\vec{T}}}\, X (q) = 2 \chi (|q|)\, \mathbf{p}_{\vec{T}} (x) \cdot \nabla \chi (|q|) + \chi (|q)^2 {\rm div}_{\vec{T}} x\, .
\]
Let us now compute ${\rm div}_{\vec{T}} x$. Complete $e_1, \ldots , e_m$ to an orthonormal basis $e_1, \ldots , e_m, \nu_1\, \ldots , \nu_n$ of $\mathbb R^{m+n}$, and compute
\begin{align*}
{\rm div}_{\vec{T}} x &= \sum_i e_i \cdot \partial_{e_i} x = \sum_i e_i \cdot \mathbf{p}_{V}^\perp (e_i)\\
&= m - \sum_i e_i \cdot \mathbf{p}_V (e_i) = m - {\rm tr}\, (\mathbf{p}_{V}) + \sum_j \nu_j \cdot \mathbf{p}_{V} (\nu_j)
= 2 + \sum_j \nu_j \cdot \mathbf{p}_{V} (\nu_j)\, ,
\end{align*}
{where the last equality follows from the fact that $V$ is $(m-2)$-dimensional. Let us now rewrite the sum on the right-hand side in a more convenient form.} Observe that 
\[
{\rm tr}\, (\mathbf{p}_V \circ \mathbf{p}_{\vec{T}}^\perp) = {\rm tr}\, (\mathbf{p}_{\vec{T}}^\perp \circ \mathbf{p}_V \circ \mathbf{p}_{\vec{T}}^\perp)\, ,
\]
and also, since $\mathbf{p}_V^T = \mathbf{p}_V$, $\mathbf{p}_{\vec{T}}^T= \mathbf{p}_{\vec{T}}$, and $\mathbf{p}_V\circ \mathbf{p}_V=\mathbf{p}_V$, we have
\begin{align*}
|\mathbf{p}_V \circ \mathbf{p}_{\vec{T}}^\perp|^2 &= {\rm tr}\, ((\mathbf{p}_V \circ \mathbf{p}_{\vec{T}}^\perp)^T \circ
(\mathbf{p}_V \circ \mathbf{p}_{\vec{T}}^\perp)) = {\rm tr}\, ((\mathbf{p}_{\vec{T}}^\perp)^T \circ \mathbf{p}_V^T \circ
\mathbf{p}_V \circ \mathbf{p}_{\vec{T}}^\perp)\\
&= {\rm tr}\, (\mathbf{p}_{\vec{T}}^\perp \circ \mathbf{p}_V \circ
\mathbf{p}_V \circ \mathbf{p}_{\vec{T}}^\perp) = {\rm tr}\, (\mathbf{p}_{\vec{T}}^\perp \circ \mathbf{p}_V \circ \mathbf{p}_{\vec{T}}^\perp)\, ,
\end{align*}
Thus, we deduce that
\begin{align*}
\sum_j \nu_j \cdot \mathbf{p}_{V} (\nu_j) &= {\rm tr}\, (\mathbf{p}_V \circ \mathbf{p}_{\vec{T}}^\perp) =
|\mathbf{p}_V \circ \mathbf{p}_{\vec{T}}^\perp|^2\, ,
\end{align*}
{so in summary,
\begin{equation}\label{e:test-vf}
    {\rm div}_{\vec{T}}\, X (q) = 2 \chi (|q|)\, \mathbf{p}_{\vec{T}} (x) \cdot \nabla \chi (|q|) + \chi (|q|)^2 (2 + |\mathbf{p}_V\circ\mathbf{p}_{\vec{T}}^\perp|^2)\, .
\end{equation}}
Plugging this into \eqref{e:first-variation-2} we then conclude that
\begin{equation}\label{e:first-variation-3}
\int \chi{(|q|)}^2 (2 + |\mathbf{p}_V \circ \mathbf{p}_{\vec{T}}^\perp|^2)\, d\|T\|(q)
\leq C \mathbf{A}^2 - \int 2 \chi (|q|)\, \mathbf{p}_{\vec{T}} (x) \cdot \nabla \chi (|q|)\, d\|T\| (q)\, .
\end{equation}
{Since $\id = \mathbf{p}_V + \mathbf{p}_V^\perp = \mathbf{p}_{\vec{T}} + \mathbf{p}_{\vec{T}}^\perp$ and $x=\mathbf{p}_V^\perp(q)$, we have
\begin{align}
\mathbf{p}_{\vec{T}} (x) \cdot \nabla \chi (|q|) &=
\mathbf{p}_{\vec{T}} (x) \cdot \mathbf{p}_V (\nabla \chi (|q|))
+ \mathbf{p}_{\vec{T}} (x) \cdot \mathbf{p}_{V^\perp} (\nabla \chi (|q|)) \notag \\
&= - \mathbf{p}_{\vec{T}}^\perp (x) \cdot \mathbf{p}_V (\nabla \chi (|q|))
+ \mathbf{p}_{\vec{T}} (x) \cdot \mathbf{p}_{V^\perp} (\nabla \chi (|q|))\, \label{e:first-var}
\end{align}}
On the other hand, once again using that $\mathbf{p}_V^T = \mathbf{p}_V$, we have
\begin{align*}
|\mathbf{p}_{\vec{T}}^\perp (x) \cdot \mathbf{p}_{V} (\nabla \chi (|q|))|
&{=} |(\mathbf{p}_V\circ \mathbf{p}_{\vec{T}}^\perp) (\mathbf{p}_{\vec{T}}^\perp (x)) \cdot \nabla \chi (|q|)|
\leq C |\mathbf{p}_V\circ \mathbf{p}_{\vec{T}}^\perp||\mathbf{p}_{\vec{T}}^\perp (x)|\, .
\end{align*}
Introducing the short-hand notation $x^\perp$ for $\mathbf{p}_{\vec{T}}^\perp (x)$, as well as $\nabla_{V} \chi (|q|)$ and $\nabla_{V^\perp} \chi (|q|)$ for $\mathbf{p}_{V} (\nabla \chi (|q|))$ and $\mathbf{p}_{V^\perp} (\nabla \chi (|q|))$ respectively, we arrive at
\begin{align}
- 2 \chi (|q|)\, \mathbf{p}_{\vec{T}} (x) \cdot \nabla \chi (|q|)
&\leq - 2 \chi (|q|)\, \mathbf{p}_{\vec{T}} (x)\cdot \nabla_{V^\perp} \chi (|q|)
+ C|\chi (|q|)||\mathbf{p}_V\circ \mathbf{p}_{\vec{T}}^\perp||x^\perp|\nonumber\\
&\leq - 2 \chi (|q|)\, \mathbf{p}_{\vec{T}} (x)\cdot \nabla_{V^\perp} \chi (|q|) + \frac{1}{2} \chi (|q|)^2 |\mathbf{p}_V\circ \mathbf{p}_{\vec{T}}^\perp|^2 + C |x^\perp|^2{\mathbf{1}_{\Bbf_{2r}}}\, .
\end{align}
Inserting the latter inequality in \eqref{e:first-variation-3} we arrive at 
\begin{align}
\int \chi^2(|q|) \left({1+{\textstyle{\frac{1}{4}}}}|\mathbf{p}_V\circ \mathbf{p}_{\vec{T}}^\perp|^2\right)\, d\|T\|(q) &\leq C \mathbf{A}^2 + C \int_{\Bbf_{2r}} |x^\perp|^2 d\|T\| \notag
\\
&\qquad- \int \chi (|q|)\, \mathbf{p}_{\vec{T}} (x)\cdot \nabla_{V^\perp} \chi (|q|)\, d\|T\| (q)\, . \label{e:first-variation-4}
\end{align}
Consider next the current $S := \sum_i Q_i \llbracket \alpha_i \rrbracket$, where we recall that the $Q_i$ are the multiplicities on the outer region. Observe that, since $X (q)=\chi (|q|)^2 x$, if $\{\Phi_t\}_t$ denotes the one-parameter family of diffeomorphisms generated by $X$, then $(\Phi_t)_\sharp S =S$ for each $t$, since $\Sbf$ is invariant under rescalings in any direction in $V^\perp$. In particular, the first variation formula tells us that we must have
\[
\int {\rm div}_{\vec{S}} X\, d\|S\| = 0\, .
\]
Now we may repeat the computation above leading to \eqref{e:test-vf}, but with $\vec{S}$ in place of $\vec{T}$; notice however that \eqref{e:test-vf} is in fact simpler because $\mathbf{p}_V\circ \mathbf{p}^\perp_{\vec{S}} = 0$. Also notice that \eqref{e:first-var} is simpler, because for $q\in \Sbf = \spt(S)$ we have that $x$ is tangent to $S$ and therefore $\mathbf{p}^\perp_{
\vec{S}} = 0$ whilst $\mathbf{p}_{\vec{S}}(x) = x$, meaning the first term on the right hand side of \eqref{e:first-var} vanishes and the second is simply $x\cdot\mathbf{p}_{V^\perp}(\nabla \chi(|q|))$. Thus, in place of \eqref{e:first-variation-4} we get
\[
\int \chi^2(|q|)\, d\|S\|(q) = -  \int \chi (|q|) x \cdot \nabla_{V^\perp} \chi (|q|)\, d\|S\| (q)\, .
\]
Subtracting this from \eqref{e:first-variation-4} and rearranging, we arrive at
\begin{align}
\int \chi^2(|q|)\, d\|T\|(q) &- \int \chi^2(|q|)\, d\|S\|(q) \nonumber \\
&\leq\int \chi^2(|q|)\, d\|T\|(q) - \int \chi^2(|q|)\, d\|S\|(q) + \frac{1}{4}\int \chi^2(|q|)|\mathbf{p}_V\circ\mathbf{p}_{\vec{T}}^\perp|^2\, d\|T\|(q)\nonumber \\
&\leq C \mathbf{A}^2 + C \int_{\Bbf_{2r}} |x^\perp|^2 d\|T\| + \int \chi (|q|) x \cdot \nabla_{V^\perp} \chi (|q|)\, d\|S\|(q) \notag \\
&\qquad - \int \chi (|q|)\, \mathbf{p}_{\vec{T}} (x)\cdot \nabla_{V^\perp} \chi (|q|)\, d\|T\| (q)
\, .\label{e:first-variation-5}
\end{align}
On the other hand, observe that for any function $f$ on $\R^{m+n}$,
\begin{align*}
\int f(q)\, d\|S\|(q) &= \sum_i Q_i \int_{\alpha_i} f(q)\, d\mathcal{H}^m(q)\, .
\end{align*}
Therefore, combining \eqref{e:first-variation-5} with our monotonicity formula estimate \eqref{e:monot-11}, we arrive at 
\begin{align}
\int_{\Bbf_r} \frac{|q^\perp|^2}{|q|^{m+2}}\, d\|T\| (q) &\leq C \mathbf{A}^2 + C \int_{\Bbf_{2r}} |x^\perp|^2 d\|T\| + \sum_i Q_i \int_{\alpha_i} \chi(|q|)\, x \cdot \nabla_{V^\perp} \chi(|q|)\, d\mathcal{H}^m(q) \nonumber \\
&\qquad- \int \chi(|q|)\, \mathbf{p}_{\vec{T}} (x)\cdot \nabla_{V^\perp} \chi(|q|)\, d\|T\|(q)\, .\label{e:first-variation-6}
\end{align}
Observe however that \eqref{e:monot-11} gives 
\[
\int \chi^2(|q|) d\|T\|(q) - \int \chi^2(|q|) d\|S\|(q) \geq - C \mathbf{A}^2\, .
\]
Thus, from \eqref{e:first-variation-5} we may further infer that
\begin{align}
\int_{\Bbf_r} |\mathbf{p}_V\circ \mathbf{p}_{\vec{T}}^\perp|^2\, d\|T\|&\leq C \mathbf{A}^2 + C \int_{\Bbf_{2r}} |x^\perp|^2 d\|T\| + \sum_i Q_i \int_{\alpha_i} \chi(|q|)\, x \cdot \nabla_{V^\perp} \chi(|q|)\, d\mathcal{H}^m(q) \nonumber \\
&\qquad- \int \chi(|q|)\, \mathbf{p}_{\vec{T}} (x)\cdot \nabla_{V^\perp} \chi(|q|)\, d\|T\|(q)\, .\label{e:first-variation-6-bis}
\end{align}

\subsubsection{Key estimates}
The rest of the proof is dedicated to estimating the terms in the right hand side of \eqref{e:first-variation-6} and, equivalently, \eqref{e:first-variation-6-bis}. To that end we will use the refined graphical approximation of Proposition \ref{p:refined} and observe that $r$ has been chosen so that $\mathbf{B}_{2r}\subset R\cup V$. Observe moreover that $V$ is a negligible set in all the integrals appearing in the right and side of \eqref{e:first-variation-6} and we will therefore ignore it. Recalling the inner, central, and outer regions (cf. Definition \ref{d:regions}), we will split our task into four estimates {, depending on whether we are integrating over the inner region (near the spine), or central or outer regions (the latter two will be coupled together)}:
\begin{align}
&\hspace{11.5em}\underbrace{\int_{R^{in}} |x^\perp|^2 d\|T\|}_{=:\,\text{(C)}} \leq C (\mathbf{A}^2 + {\hat{\mathbf{E}} (T, \mathbf{S}, \Bbf_4)})\label{e:inner-1}\\
&\underbrace{\sum_j Q_j \int_{\alpha_j\cap R^{in}} \chi(|q|)\, |x \cdot \nabla_{V^\perp} \chi(|q|)|\, d\mathcal{H}^m(q) + \int_{R^{in}} \chi(|q|)\, |\mathbf{p}_{\vec{T}} (x)\cdot \nabla_{V^\perp} \chi(|q|)|\, d\|T\|(q)}_{=:\,\text{(D)}} \nonumber \\
&\hspace{24em}\leq C (\mathbf{A}^2 + {\hat{\mathbf{E}} (T, \mathbf{S}, \Bbf_4)})\label{e:inner-2}\\
&\hspace{10.5em}\underbrace{\int_{R^o\cup R^c} |x^\perp|^2 d\|T\|}_{=:\,\text{(E)}} \leq C (\mathbf{A}^2 + {\hat{\mathbf{E}} (T, \mathbf{S}, \Bbf_4)})\label{e:central-outer-1}\\
&\underbrace{\left|\sum_j Q_j \int_{\alpha_j\cap (R^o\cup R^c)} \chi(|q|)\, x \cdot \nabla_{V^\perp} \chi(|q|)\, d\mathcal{H}^m(q) - \int_{R^o\cup R^c} \chi(|q|)\, \mathbf{p}_{\vec{T}} (x)\cdot \nabla_{V^\perp} \chi(|q|)\, d\|T\|(q)\right|}_{=:\, \text{(F)}} \nonumber \\
&\hspace{24em}\leq C (\mathbf{A}^2 + {\hat{\mathbf{E}} (T, \mathbf{S}, \Bbf_4)})\, .\label{e:central-outer-2}
\end{align}
Once we have established these four estimates, the result follows from \eqref{e:first-variation-6} and \eqref{e:first-variation-6-bis}. We stress that, whilst in \eqref{e:inner-2} we do not care about subtracting the two terms as in \eqref{e:first-variation-6}, \eqref{e:first-variation-6-bis} (indeed, our estimates on the inner region will suffice there), it is important in the outer and central regions that we are subtracting the two terms, as in \eqref{e:central-outer-2}; the argument will exploit in a crucial way a cancellation effect due to the fact that this is a difference between two nearly equal quantities (so, we do not wish to crudely estimate (F) by the sum of the two terms therein at any point, unlike in (D)).

\subsubsection{Estimates in the inner region} 
This section is dedicated to prove the first two inequalities, \eqref{e:inner-1} and \eqref{e:inner-2}. First of all observe that 
\begin{align}
|x^\perp|^2 &\leq |x|^2 = \dist (q, V)^2\label{e:from-spine-1}
\end{align}
Note that,
\begin{equation}\label{e:Simon-error-integrand}
\chi (|q|) |x\cdot \nabla_{V^\perp} \chi (|q|)| = \chi (|q|) |x|^2 \frac{|\chi' (|q|)|}{|q|}\, .
\end{equation}
On the other hand $\chi' (|q|)=0$ if $|q|\leq r$ {and $|\chi'(|q|)| \leq \frac{C}{r}$ otherwise; in particular we conclude that for $C=C(m)>0$ we have}
\begin{equation}\label{e:from-spine-2}
\chi (|q|) |x\cdot \nabla_{V^\perp} \chi (|q|)|\leq C \dist (q,V)^2\, .
\end{equation}
On the other hand 
\[
\chi(|q|)|\mathbf{p}_{\vec{T}}(x)\cdot\nabla_{V^\perp}\chi(|q|)|\leq \chi(|q|)|\mathbf{p}_{\vec{T}}(x)\cdot x|\frac{|\chi^\prime(|q|)|}{|q|} = \chi(|q|)|\mathbf{p}_{\vec{T}}(x)|^2\frac{|\chi^\prime(|q|)|}{|q|}
\]
and since $|\mathbf{p}_{\vec{T}}(x)|\leq |x|$, we can combine it with the second inequality of \eqref{e:from-spine-1} to estimate
\begin{equation}\label{e:from-spine-3}
\chi (|q|) |\mathbf{p}_{\vec{T}} (x)\cdot \nabla_{V^\perp} \chi (|q|)|\leq C \dist (q,V)^2\, .
\end{equation}
Using the monotonicity formula {for $y\in \Large[-\frac{1}{\sqrt{m-2}}, \frac{1}{\sqrt{m-2}}\Large]^{m-2} \subset V$ and $\rho \leq 2$ (which in particular tells us that $\|T\|(\Bbf_\rho(y)) \leq C\rho^m$),} we easily conclude that
\begin{align}
\int_{\Bbf_\rho(y)} (|x^\perp|^2 + \chi (|q|) |\mathbf{p}_{\vec{T}} (x)\cdot \nabla_{V^\perp} \chi(|q|)|)\, d\|T\|(q)
& \leq C \rho^{m+2}\, ,\label{e:Vitali-1}\\
\sum_j Q_j \int_{\alpha_j\cap \Bbf_\rho(y)} \chi (|q|) |x\cdot \nabla_{V^\perp} \chi(|q|)|\, d\mathcal{H}^m(q) &\leq C \rho^{m+2}\, .\label{e:Vitali-2}
\end{align}
Consider now the cubes $L\in \mathcal{G}^{in}$ as in Definition \ref{d:regions} and enumerate them as $\{L_k\}_k$. Let $y_k:= y_{L_k}{\in V}$ be the corresponding centers and let $\rho_k = 2^{2-\ell (L_k)}$ be the radii of the balls $\Bbf (L_k)$ as defined in Section \ref{ss:whitney}. Clearly $\{\Bbf_{\rho_k} (y_k)\}$ is a covering of $R^{in}$. 
We now appeal to \eqref{e:0-control-inner} in Lemma \ref{l:controls_Whitney} and use \eqref{e:Vitali-1}, \eqref{e:Vitali-2} with $y=y_k$ and $\rho = \rho_k$ to estimate 
\begin{align}
 \text{(C)}+\text{(D)} &\leq C \sum_k \rho_k^{m+2} \mathbf{E} (L_k, 0) = C \sum_k \int_{\Bbf^h (L_k)} \dist^2 (q, \mathbf{S})\, d\|T\| (q)\, .
\end{align}
Recall that, by Lemma \ref{l:whitney}(iv), the collection of sets $\Bbf^h (L_k)$ has a control on their overlaps (each point in $R$ belongs to at most $C(m)$ such sets). Therefore
\[
 \text{(C)}+\text{(D)} \leq C \int_{\Bbf_4} \dist^2 (q, \Sbf)\, d\|T\| (q)\, , 
\]
which gives 
\eqref{e:inner-1} and \eqref{e:inner-2}.

\subsubsection{Estimates in the central and outer regions} In this section we prove \eqref{e:central-outer-1} and \eqref{e:central-outer-2} and hence conclude the proof of Theorem \ref{t:HS}. We first observe that, because of Lemma \ref{l:controls_Whitney}(iii), (iv) and Lemma \ref{l:whitney}(iv), (v) we have
\begin{equation}\label{e:summing}
\sum_{L\in \mathcal{G}^c\cup \mathcal{G}^o} 2^{-(m+2) \ell (L)} (\mathbf{E} (L)+2^{-2\ell(L)}\mathbf{A}^2) 
\leq C ({\hat{\mathbf{E}}} (T, \mathbf{S}, \Bbf_4) + \mathbf{A}^2)\, .
\end{equation}
For each $L\in \mathcal{G}^c\cup \mathcal{G}^o$ we consider the approximations $u_L$ given by Proposition \ref{p:refined}. These approximations consist of $Q_{L,i}$-valued functions defined over the regions $\lambda L_i$ for some subcollection of the planes $\{\alpha_1, \ldots , \alpha_N\}$ for each $L$. In order to simplify our notation, we do not keep track of these collections for different $L$. 

The estimate \eqref{e:central-outer-1} is less laborious than \eqref{e:central-outer-2}, thanks to the fact that we do not need to exploit any cancellation effect. We can in particular write 
\begin{align}
\text{(E)} \leq \sum_{L\in \mathcal{G}^c\cup\mathcal{G}^o} \int_{R(L)} |x^\perp|^2\, d\|T\| \label{e:central-outer-10}\, .
\end{align}
Now, for each region $R(L)$ with $L
\in \Gcal^c\cup\Gcal^o$, we can use that $|x^\perp|^2 \leq C 2^{-2\ell (L)}$ on $\spt (T)\cap R(L)$, together with Proposition \ref{p:refined}(iii) and \eqref{e:error-refined-3}, to estimate
\[
\int_{R(L)} |x^\perp|^2 d\|T\| \leq \sum_i \int_{\Omega(L)} |x^\perp|^2\, d\|\mathbf{G}_{u_{L,i}}\|
+ C 2^{-(m+2)\ell (L)} (\mathbf{E} (L) + 2^{-2\ell (L)} \mathbf{A}^2)\, .
\]
Using \eqref{e:summing} we then conclude that
\begin{equation}\label{e:central-outer-11}
\text{(E)} \leq \sum_{L\in \mathcal{G}^c\cup\mathcal{G}^o} \sum_i \int_{\Omega(L)} |x^\perp|^2\, d\|\mathbf{G}_{u_{L,i}}\|
+ C (\hat{\mathbf{E}} (T, \mathbf{S}, \Bbf_4)+\Abf^2)\, .
\end{equation}
Consider now coordinates $z$ on $\alpha_i$ and recall that, for a point $q= (z, (u_{L,i})_j (z))$, for some $j=1,\dots,Q_{L,i}$, in the graph of the $Q_{L,i}$-valued function $u_{L,i} = \sum_j \llbracket (u_{L,i})_j\rrbracket$, the vector $x^\perp$ is the projection of $x = \mathbf{p}_{V}^\perp(q)$ onto the orthogonal complement of the tangent plane $\vec{\mathbf{G}}_{u_{L,i}}$ to the graph of $(u_{L,i})_j$ at $q$. In particular, since at such a point $q$ we have $\big|\mathbf{p}^\perp_{{\vec{\mathbf{G}}_{u_{L,i}}}} - \mathbf{p}^\perp_{\alpha_i}\big| \leq C|\nabla(u_{L,i})_j(z)|$, and moreover because $\mathbf{p}^\perp_{\alpha_i}\circ\mathbf{p}^\perp_V = \mathbf{p}^\perp_{\alpha_i}$, this yields (noting that $|x| \leq |q| \leq C2^{-\ell(L)}$)
\begin{align}
    |x^\perp| = \big|\mathbf{p}_{{\vec{\mathbf{G}}_{u_{L,i}}}}^\perp(x)\big| & \leq \big|\mathbf{p}_{{\vec{\mathbf{G}}_{u_{L,i}}}}^\perp-\mathbf{p}_{\alpha_i}^\perp\big|\cdot|x| + |\mathbf{p}_{\alpha_i}^\perp(\mathbf{p}_{V}^\perp(q))| \nonumber \\
    &\leq C|\nabla(u_{L,i})_j(z)|\cdot 2^{-\ell(L)} + |(u_{L,i})_j(z)| \label{e:orth-proj}
\end{align}
Now if we square this expression and integrate it, we get from the two bounds in \eqref{e:error-refined-1} (to control the integrand) as well as the Lipschitz bound \eqref{e:error-refined-2} (to control the Jacobian) that
\begin{align*}
\sum_i\int_{\Omega(L)}|x^\perp|^2\, d\|\mathbf{G}_{u_{L,i}}\| & \leq C\sum_{i,j}\int_{\Omega_i(L)}|\nabla(u_{L,i})_j(z)|^2 2^{-2\ell(L)} + |(u_{L,i})_j(z)|^2 d\|\mathbf{G}_{u_{L,i}}\| \\
& \leq C2^{-(m+2)\ell(L)}(\Ebf(L)+2^{-2\ell(L)}\Abf^2).
\end{align*}
In particular, plugging the latter estimate in \eqref{e:central-outer-11} and using again \eqref{e:summing} we reach \eqref{e:central-outer-1}.

We now come to the proof of \eqref{e:central-outer-2}, which is more laborious. First of all, because $\|\mathbf{G}_{u_{L,i}}\| (\partial R (L)) =0$ for each $L \in \Gcal^c\cup\Gcal^o$, we can use Proposition \ref{p:refined}(iii), the estimate \eqref{e:error-refined-3} and the estimate \eqref{e:from-spine-3} to deduce that
\[
\left|\int_{\partial R (L)} \chi(|q|)\, \mathbf{p}_{\vec{T}} (x)\cdot \nabla_{V^\perp} \chi(|q|)\, d\|T\|(q)\right| \leq C 2^{-(m+2)\ell (L)} (\mathbf{E} (L) + 2^{-2\ell (L)} \mathbf{A}^2)\, .
\]
Thus, letting $(R (L))^\circ$ denote the interior of $R(L)$, we can once again use \eqref{e:summing} to estimate
\begin{align}
\text{(F)} &\leq \sum_{L\in \mathcal{G}^c\cup \mathcal{G}^o} \left| \sum_j Q_j \int_{\alpha_j\cap (R(L))^\circ} \chi(|q|)\, x \cdot \nabla_{V^\perp} \chi(|q|)\, d\mathcal{H}^m\right.\nonumber\\
&\hspace{7em}\left.- \int_{(R (L))^\circ} \chi(|q|)\, \mathbf{p}_{\vec{T}} (x)\cdot \nabla_{V^\perp} \chi(|q|)\, d\|T\|\right| + C ({\hat{\mathbf{E}}(T,\Sbf,\Bbf_4)} + \mathbf{A}^2)\, .\label{e:central-outer-20}
\end{align}
We remove the boundary as above to ensure that our regions are disjoint. Now we want to straighten out the curved regions (using Lemma \ref{l:curved}) to cylinders over disjoint cubes in our planes where we can pass to graphs. So, next, again taking into account \eqref{e:from-spine-3}, and then (recalling the notation $L_j = \alpha_j \cap R (L)$) using Lemma \ref{l:curved} with $U=(R(L))^\circ$ and $\tilde{U}=\bigcup_j\mathbf{p}_{\alpha_j}^{-1} (L_j)$, since $\Hcal^{m-1}(\partial L_j) \leq C2^{-(m-1)\ell(L)}$, we can further estimate 
\begin{align*}
 &\left|\int_{(R (L))^\circ} \chi(|q|)\, \mathbf{p}_{\vec{T}} (x)\cdot \nabla_{V^\perp} \chi(|q|)\, d\|T\|
- \sum_j \int_{\mathbf{p}_{\alpha_j}^{-1} ({L_j})} \chi(|q|)\, \mathbf{p}_{\vec{T}} (x)\cdot \nabla_{V^\perp} \chi(|q|)\, d\|T_{L,{j}}\| \right|\\
&\qquad\leq C 2^{-(m+2)\ell (L)} (\mathbf{E} (L) + 2^{-2\ell (L)} \mathbf{A}^2)\, .
\end{align*}
We can now yet again use \eqref{e:from-spine-3}, as well as the estimate, \eqref{e:error-refined-3} of Proposition \ref{p:refined} to further estimate
\begin{align*}
& \left|\sum_j\int_{{\mathbf{p}_{\alpha_j}^{-1} (L_j)}} \chi(|q|)\, \mathbf{p}_{\vec{T}} (x)\cdot \nabla_{V^\perp} \chi(|q|)\, d\|T_{L,j}\|\right.\\
& \hspace{5em} \left. - \sum_j \int_{{\mathbf{p}_{\alpha_j}^{-1} (L_j)}} \chi(|q|)\, \mathbf{p}_{\vec{\mathbf{G}}_{u_{L,j}}} (x)\cdot \nabla_{V^\perp} \chi(|q|)\, d\|\mathbf{G}_{u_{L,j}}\| \right|\\
&\hspace{20em} \leq C 2^{-(m+2)\ell (L)} (\mathbf{E} (L) + 2^{-2\ell (L)} \mathbf{A}^2)
\end{align*}
Combining the above estimates with \eqref{e:central-outer-20} and again making use of \eqref{e:summing}, we arrive at
\begin{align}
\text{(F)} & \leq \sum_{L\in \mathcal{G}^c\cup \mathcal{G}^o} \left| \sum_j \left[Q_j \int_{\alpha_j\cap (R(L))^\circ} \chi(|q|)\, x \cdot \nabla_{V^\perp} \chi(|q|)\, d\mathcal{H}^m\right.\right.\nonumber\\
& \hspace{10em}\left.\left.
- \int_{{\mathbf{p}_{\alpha_j}^{-1} (L_j)}} \chi(|q|)\, \mathbf{p}_{\vec{\mathbf{G}}_{u_{L,j}}} (x)\cdot \nabla_{V^\perp} \chi(|q|)\, d\|\mathbf{G}_{u_{L,j}}\|\right]\right|\nonumber\\
&\hspace{24em} + C ({\hat{\mathbf{E}}} (T, \mathbf{S}, \Bbf_4) + \mathbf{A}^2)\label{e:central-outer-21}\, .
\end{align}
Now note that the multiplicities $Q_j$ are the ones from the outer region, and so they do not necessarily match the multiplicities $Q_{L,j}$ of the multi-valued functions $u_{L,j}$ when $L\in\Gcal^c$. However recall the computation (see \eqref{e:Simon-error-integrand}),
\begin{equation}\label{e:cutoff-gradient}
\chi(|q|) x\cdot \nabla_{V^\perp} \chi (|q|) = \frac{\chi (|q|) \chi' (|q|)}{|q|} |x|^2\, ,
\end{equation}
which since $x\in V^\perp$ shows that the integrand {$\chi(|q|) x\cdot \nabla_{V^\perp} \chi (|q|)$} is invariant under rotations which keep the spine $V$ fixed. Since for every $j,k \in \{1,\dots,N\}$ there is a rotation which maps $\alpha_j$ onto $\alpha_k$ and fixes $V$, the integral
\[
\int_{\alpha_j\cap (R(L))^\circ} \chi(|q|)\, x \cdot \nabla_{V^\perp} \chi(|q|)\, d\mathcal{H}^m
\]
is independent of the plane $\alpha_j$. In particular, given that $\sum_j Q_{L,j} = \sum_j Q_j = Q$, we can in fact write {from \eqref{e:central-outer-21}} (using $L_j = \alpha_j \cap R (L)$
\begin{align}
\text{(F)} & \leq \sum_{L\in \mathcal{G}^c\cup \mathcal{G}^o} \sum_j \left| Q_{L,j} \int_{L_j} \chi(|q|)\, x \cdot \nabla_{V^\perp} \chi(|q|)\, d\mathcal{H}^m\right.\nonumber\\
& \hspace{11em}\left. - \int_{{\mathbf{p}_{\alpha_j}^{-1} (L_j)}} \chi(|q|)\, \mathbf{p}_{\vec{\mathbf{G}}_{u_{L,j}}} (x)\cdot \nabla_{V^\perp} \chi(|q|)\, d\|\mathbf{G}_{u_{L,j}}\|\right|\nonumber\\
&\hspace{24em} + C ({\hat{\mathbf{E}}} (T, \mathbf{S}, \Bbf_4) + \mathbf{A}^2)\label{e:central-outer-22}
\end{align}
Next, following the same computation as for \eqref{e:from-spine-3} we get 
\[
|\chi(|q|)\, \mathbf{p}_{\vec{\mathbf{G}}_{u_{L,j}}} (x)\cdot \nabla_{V^\perp} \chi(|q|)|
\leq C \dist (q, V)^2 \leq C\cdot 2^{-2\ell(L)}\, ,
\]
and hence, through the usual Taylor expansion of the area functional for a multi-valued graph, for $k=1,\dots, Q_{L,j}$ letting $q_k\coloneqq (z,(u_{L,j})_k(z))\equiv z + (u_{L,j})_{k}(z) \in \alpha_j \times \alpha_j^\perp$, we get
\begin{align*}
&\left|\int_{{\mathbf{p}_{\alpha_j}^{-1} (L_j)}} \chi(|q|)\, \mathbf{p}_{\vec{\mathbf{G}}_{u_{L,j}}} (x)\cdot \nabla_{V^\perp} \chi(|q|)\, d\|\mathbf{G}_{u_{L,j}}\|\right.\\
& \left.\hspace{10em}- \int_{L_j} {\sum_k\chi(|q_k|)\, \mathbf{p}_{\vec{\mathbf{G}}_{{(u_{L,j})_k}}} (\mathbf{p}_V^\perp(q_k))\cdot \nabla_{V^\perp} \chi(|q_k|)}\, d\mathcal{H}^m(z)\right| \\
&\hspace{20em}\leq C
2^{-2\ell (L)} \int_{L_j} |Du_{L,j}|^2\, d\mathcal{H}^m\, .
\end{align*}
We can in particular use \eqref{e:central-outer-22}, \eqref{e:error-refined-1}, and \eqref{e:summing} to obtain
\begin{align}
\text{(F)} & \leq \sum_{L\in \mathcal{G}^c\cup \mathcal{G}^o} \sum_j \left| \int_{L_j} \left.[Q_{L,j} \chi(|z|)\, \mathbf{p}_V^\perp(z) \cdot \nabla_{V^\perp} \chi(|z|)\right.\right.\nonumber\\
& \hspace{10em} \left. - {\sum_k\chi(|q_k|)\, \mathbf{p}_{\vec{\mathbf{G}}_{{(u_{L,j})_k}}} (\mathbf{p}_V^\perp(q_k))\cdot \nabla_{V^\perp} \chi(|q_k|)}]\, d\mathcal{H}^m(z)\right|\nonumber\\
&\hspace{20em} + C ({\hat{\mathbf{E}}}(T, \mathbf{S}, \Bbf_4) + \mathbf{A}^2)\label{e:central-outer-23}
\end{align}
Now use the coordinates $z=(\zeta,\xi)\in V\times V^\perp$ for points $z\in \alpha_j$. Recalling \eqref{e:cutoff-gradient}, the first integrand above is then given by
\[
h (z) \coloneqq Q_{L,j}\frac{\chi (|z|) \chi'(|z|)}{|z|} |\xi|^2\, .
\]
Meanwhile, we write the second integrand as
\[
g (z) := \sum_{k=1}^{Q_{L,j}} \frac{\chi (|q_k|)\chi'(|q_k|)}{|q_k|} \mathbf{p}_{\vec{\mathbf{G}}_{(u_{L,j}})_k} (\mathbf{p}_{V}^\perp (q_k)) \cdot \mathbf{p}_{V}^\perp (q_k)\, .
\]
Considering that $\chi' (|q|)=0$ when $|q|\leq r$, for each $k=1,\dots,Q_{L,j}$ we can estimate (by Taylor expansion)
\[
\left|\frac{\chi (|q_k|)\chi'(|q_k|)}{|q_k|} - \frac{\chi (|z|) \chi'(|z|)}{|z|}\right| \leq C 2^{2\ell (L)} |(u_{L,j})_k (z)|^2\, .
\]
On the other hand we have
\[
\left|\mathbf{p}_{\vec{\mathbf{G}}_{(u_{L,j}})_k} (\mathbf{p}_{V}^\perp (q_k)) \cdot \mathbf{p}_{V}^\perp (q_k)\right|
\leq C 2^{-2\ell (L)}\, .
\]
In particular, if we define
\[
\bar{g} (z) := \sum_{k=1}^{Q_{L,j}} \frac{\chi (|z|) \chi'(|z|)}{|z|} \mathbf{p}_{\vec{\mathbf{G}}_{(u_{L,j})_k}} (\mathbf{p}_{V}^\perp (q_k)) \cdot \mathbf{p}_{V}^\perp (q_k)\, ,
\]
we then get
\begin{equation}\label{e:g-gbar}
|g (z) - \bar{g} (z)|\leq C |u_{L,j} (z)|^2\, .
\end{equation}
Next, recalling the definition of $q_k$, notice that 
\[
\mathbf{p}_{V}^\perp (q_k) = \xi + (u_{L,j})_k (z)\, ,
\]
On the other hand, 
\begin{align*}
|\mathbf{p}_{\vec{\mathbf{G}}_{(u_{L,j})_k}} &  (\mathbf{p}_{V}^\perp (q_k)) \cdot \mathbf{p}_{V}^\perp (q_k) - |\xi|^2|\\
& = \left||\mathbf{p}_{\vec{\mathbf{G}}_{(u_{L,j})_k}}(\xi+(u_{L,j})_k)|^2 - |\xi|^2\right|\\
& = \left||\mathbf{p}_{\vec{\mathbf{G}}_{(u_{L,j})_k}}(\xi)|^2 - |\xi|^2\right| + |\mathbf{p}_{\vec{\mathbf{G}}_{(u_{L,j})_k}}((u_{L,j})_k)|^2\\
&\qquad\qquad + 2\left|\mathbf{p}_{\vec{\mathbf{G}}_{(u_{L,j})_k}}(\xi)\cdot\mathbf{p}_{\vec{\mathbf{G}}_{(u_{L,j})_k}}((u_{L,j})_k)\right|\\
& \leq |\mathbf{p}^\perp_{\vec{\mathbf{G}}_{(u_{L,j})_k}}(\xi)|^2 + |(u_{L,j})_k|^2 + 2|\mathbf{p}_{\vec{\mathbf{G}}_{(u_{L,j})_k}}(\xi)|\left|(\mathbf{p}_{\vec{\mathbf{G}}_{(u_{L,j})_k}}-\mathbf{p}_{\alpha_j})((u_{L,j})_k)\right|\\
& \leq |\mathbf{p}^\perp_{\vec{\mathbf{G}}_{(u_{L,j})_k}}-\mathbf{p}^\perp_{\alpha_j}|^2|\xi|^2 + |(u_{L,j})_k|^2 + 2|\xi|\cdot|D(u_{L,j})_k||(u_{L,j})_k|\\
& \leq |D(u_{L,j})_k|^22^{-2\ell(L)} + |(u_{L,j})_k|^2 + 2^{1-\ell(L)}|D(u_{L,j})_k||(u_{L,j})_k|\\
& \leq C2^{-2\ell(L)}|D(u_{L,j})_k|^2 + C|(u_{L,j})_k|^2
\end{align*}
where we have used that $\mathbf{p}_{\alpha_j}((u_{L,j})_k) = 0$ and $\mathbf{p}^\perp_{\alpha_j}(\xi) = 0$. In particular, we arrive at
\[
|\bar{g}(z)-h(z)| \leq C|u_{L,j}(z)|^2 + 2^{-2\ell(L)}|Du_{L,j}(z)|^2\, .
\]
Combining this last inequality and \eqref{e:g-gbar} with the bound \eqref{e:error-refined-1}, for each $L\in\Gcal^c\cup\Gcal^o$ and each index $j$ enumerating the planes for the corresponding cone associated to $L$, we thus achieve
\begin{align*}
& \left| \int_{L_j} Q_{L,j} \chi(|z|)\, \mathbf{p}_{V^\perp}(z) \cdot \nabla_{V^\perp} \chi(|z|)\, d\mathcal{H}^m(z)
\right.\\
& \hspace{10em}\left.- \sum_k \int_{L_j} \chi(|q_k|)\, \mathbf{p}_{\vec{\mathbf{G}}_{{(u_{L,j})_k}}} (\mathbf{p}_{V}^\perp(q_k))\cdot \nabla_{V^\perp} \chi(|q_k|)\, d\mathcal{H}^m(z)\right| \\
& \hspace{20em} \leq C 2^{-(m+2) \ell (L)} (\mathbf{E} (L)+ 2^{-2\ell (L)} \mathbf{A}^2)\,
\end{align*}
Together with \eqref{e:central-outer-23} and \eqref{e:summing}, this completes the proof of \eqref{e:central-outer-2} and thus the proof of the theorem.
\qed

\subsection{Proof of Simon's non-concentration estimate (Corollary \ref{c:HS-patch})}

We observe that Corollary \ref{c:HS-patch} is a direct consequence of Theorem \ref{t:HS} and of the following lemma (after an appropriate rescaling to adjust the radius).

\begin{lemma}\label{l:variation-test-2}
Let $T$, $\Sigma$ and $\Sbf$ be as in Assumption \ref{a:HS} with $\mathbf{B}_1 \subset \Omega$ and $\Theta (T, 0)\geq Q$. Then we may choose $\eps$ sufficiently small in Assumption \ref{a:HS} such that for each $\kappa > 0$ we have
\begin{equation}\label{e:HS-20}
\int_{\Bbf_1} \frac{\dist ^2(q, \mathbf{S})}{|q|^{m+2-\kappa}}\, d\|T\| (q)
\leq C_\kappa \int_{\Bbf_1} \frac{|q^\perp|^2}{|q|^{m+2}}\, d\|T\| (q)+ C_\kappa (\hat{\mathbf{E}} (T, \mathbf{S}, \Bbf_4) + \mathbf{A}^2)\, ,
\end{equation}
where $C_\kappa=C_\kappa(Q,m,n,\bar n,M,\kappa)>0$.
\end{lemma}

\begin{proof}[Proof of Lemma \ref{l:variation-test-2}]
Fix $\kappa \in (0,m+2)$. We test the first variation identity \eqref{e:first-variation} for $T$ with the vector field
\[
X (q) := \dist^2 (q,\mathbf{S}) \underbrace{(\max \{r, |q|\}^{-m-2+\kappa} - 1)_+}_{=:\, f(q)} q\, ,
\]
where $r> 0$ and for a function $g$ we define $g_+(q) \coloneqq \max\{g(q),0\}$. Note that $f$ is supported in $\Bbf_1\setminus\Bbf_r$. 
In order to estimate the integrand in (B) of \eqref{e:first-variation}, we first observe that 
\[
f(q)\dist^2 (q, \mathbf{S}) \leq |q|^{{\kappa}-m}\, .
\]
Recalling \eqref{e:perp}, \eqref{e:A-1}, and \eqref{e:A-2}, we have
\[
|\vec{H}_T \cdot X^\perp (q)| \leq C \mathbf{A}^2 |q|^{1+\kappa-m}\, .
\]
Hence we can estimate
\[
|\text{(B)}| \leq C \mathbf{A}^2 \int_{\Bbf_1} |q|^{1+\kappa-m}\, d\|T\| (q) \leq C \mathbf{A}^2 \, .
\]
In particular we conclude that
\begin{equation}\label{e:first-variation-21}
\int {\rm div}_{\vec{T}} X\, \, d\|T\| = O (\mathbf{A}^2)\, .
\end{equation}
We next compute
\begin{equation}\label{e:compute-divergence-1}
{\rm div}_{\vec{T}} X (q) = m \dist^2 (q, \mathbf{S}) f (q) + f (q) \nabla \dist^2 (q, \mathbf{S}) \cdot q_\parallel + \dist^2 (q, \mathbf{S}) \nabla f (q) \cdot q_\parallel\, \, .
\end{equation}
Moreover, we can explicitly compute
\begin{align}
\nabla f (q) \cdot q_\parallel & = -(m+2-\kappa) |q|^{-m-4+\kappa} |q_\parallel|^2\mathbf{1}_{\Bbf_1\setminus \Bbf_r} (q)\nonumber\\
&= -(m+2-\kappa) |q|^{-m-2+\kappa} \mathbf{1}_{\Bbf_1\setminus \Bbf_r} (q) + (m +2-\kappa) |q|^{-m -4+\kappa} |q^\perp|^2 \mathbf{1}_{\Bbf_1\setminus \Bbf_r} (q)\, .\label{e:compute-divergence-2}
\end{align}
On the other hand, using the $2$-homogeneity of $\dist^2 (q, \mathbf{S})$ we can likewise compute
\begin{align*}
\nabla \dist^2 (q,\mathbf{S}) \cdot q_\parallel &= \nabla \dist^2 (q,\Sbf) \cdot q - \nabla \dist^2 (q, \mathbf{S}) \cdot q^\perp\\
&= 2 \dist^2 (q, \Sbf) - \nabla \dist^2 (q, \mathbf{S}) \cdot q^\perp\, .
\end{align*}
Inserting this and \eqref{e:compute-divergence-2} into \eqref{e:compute-divergence-1} we reach
\begin{align}
{\rm div}_{\vec{T}} X (q) &= \kappa \dist (q,\mathbf{S})^2 |q|^{-m-2+\kappa} \mathbf{1}_{\Bbf_1\setminus \Bbf_r} (q)\nonumber\\
&\quad + (m+2) r^{-m-2+\kappa} \dist^2 (q,\mathbf{S}) \mathbf{1}_{\Bbf_r} (q) - (m+2) \dist^2 (q,\mathbf{S}) \mathbf{1}_{\Bbf_1} (q)\nonumber\\
&\quad + (m+2-\kappa) \dist^2 (q,\mathbf{S}) |q^\perp|^2 |q|^{-m-4+\kappa} \mathbf{1}_{\Bbf_1\setminus \Bbf_r} (q) - \nabla \dist^2 (q,\mathbf{S}) \cdot q^\perp f(q)\, .
\end{align}
In turn, inserting this in \eqref{e:first-variation-21} we achieve 
\begin{align}
\kappa\int_{\Bbf_1\setminus \Bbf_r} \frac{\dist^2 (q,\mathbf{S})}{|q|^{m+2-\kappa}}\, d\|T\| (q) &\leq C \mathbf{A}^2 + (m+2) \hat{\mathbf{E}} (T, \mathbf{S}, \Bbf_1)\nonumber\\
&\quad + \underbrace{\int \nabla \dist (\mathbf{S},q)^2 \cdot q^\perp f(q)\, d\|T\| (q)}_{=:\, \text{(C)}}\label{e:first-variation-22}\, .
\end{align}
On the other hand, using that $\dist (q, \mathbf{S})$ is $1$-Lipschitz, we can estimate 
\begin{align}
|\text{(C)}| &\leq 2 \int \dist (q,\mathbf{S}) |q^\perp| f (q)\, d\|T\| (q)\nonumber\\
&\leq \kappa \int \dist^2 (q,\mathbf{S}) f(q)\, d\|T\| (q) + \frac{1}{\kappa} \int |q^\perp|^2 f(q)\, d\|T\| (q)\nonumber\\
&\leq \kappa \int_{\Bbf_1\setminus \Bbf_r} \frac{\dist^2 (q, \mathbf{S})}{|q|^{m+2-\kappa}}\, d\|T\| (q)\nonumber\\
&\qquad + \kappa r^{-m-2+\kappa} \int_{\Bbf_r} \dist^2 (q,\mathbf{S})\, d\|T\| (q)
 + \frac{1}{\kappa} \int_{\Bbf_1} \frac{|q^\perp|^2}{|q|^{m+2}}\, d\|T\| (q)\, .\label{e:first-variation-23}
\end{align}
Inserting \eqref{e:first-variation-23} into \eqref{e:first-variation-22} we then reach 
\begin{align}
\int_{\Bbf_1\setminus \Bbf_r} \frac{\dist^2 (q,\mathbf{S})}{|q|^{m+2-\kappa}}\, d\|T\| (q) &\leq C_\kappa \int_{\Bbf_1} \frac{|q^\perp|^2}{|q|^{m+2}}\, d\|T\| (q) + C_\kappa (\mathbf{A}^2 + \hat{\mathbf{E}} (T, \mathbf{S}, \Bbf_1)) + C\frac{\|T\| (\Bbf_r)}{r^{m-\kappa}}\, .
\end{align}
Letting $r\downarrow 0$ and using $\|T\| (\Bbf_r)\leq C r^m$ we reach \eqref{e:HS-20}. 
\end{proof}

\subsection{Proof of Simon's shift inequality (Proposition \ref{p:HS-3})}

In Lemma \ref{l:HS-4} below we will show the following inequality for each $\kappa\in (0,m+2)$, under the assumption that $\varepsilon$ is chosen sufficiently small and that $\rho=\rho(m)$ is a dimensional constant:
\begin{equation}\label{e:shifted-HS-3}
\int_{\Bbf_{\rho} (q_0)} \frac{\dist (q, q_0+\mathbf{S})^2}{|q-q_0|^{m+2-\kappa}}\, d\|T\| (q)
\leq C^\star_\kappa (\mathbf{A}^2 + \hat{\mathbf{E}} (T, \mathbf{S}, \Bbf_4) +  |\mathbf{p}_{\alpha_1}^\perp (q_0)|^2 +  \boldsymbol{\mu} (\mathbf{S})^2 |\mathbf{p}_{V^\perp\cap\alpha_1} (q_0)|^2)\, ,
\end{equation}
for $C^\star_\kappa=C^\star_\kappa(Q,m,n,\bar{n})$. Assuming the validity of this, our aim is therefore to show that we also have 
\begin{equation}\label{e:control-on-Q-point}
|\mathbf{p}_{\alpha_1}^\perp (q_0)|^2 + \boldsymbol{\mu} (\mathbf{S})^2 |\mathbf{p}_{V^\perp\cap\alpha_1} (q_0)|^2 \leq C (\mathbf{A}^2 + \hat{\mathbf{E}} (T, \mathbf{S}, \Bbf_4))\, ,
\end{equation}
for $C=C(Q,m,n,\bar{n})$. Observe that \eqref{e:shifted-HS-3} and \eqref{e:control-on-Q-point} together yield \eqref{e:HS-3}.

Fix now a scale $\bar r\leq \frac{\rho}{2}$ (for $\rho$ fixed as in \eqref{e:shifted-HS-3}), whose choice will be specified later. Observe that, by assuming $\varepsilon$ is small enough depending on $\bar r$, Proposition \ref{p:refined}(v) guarantees that $\dist (q_0, V) \leq \frac{\bar r}{2}$. We can now apply a scaled version of Lemma \ref{l:shifting} to find an index $j$ and a subset $\Omega \subset \alpha_j \cap \Bbf_{\bar r} (\mathbf{p}_{V} (q_0))\setminus B_{\bar r/2} (V)$ with the property that 
\begin{equation}\label{e:shifting-LB-10}
|\mathbf{p}_{\alpha_1}^\perp (q_0)|^2 + \boldsymbol{\mu} (\mathbf{S})^2 |\mathbf{p}_{V^\perp\cap\alpha_1} (q_0)|^2 \leq \bar{C} \dist (z, q_0+\mathbf{S})^2 \qquad \forall z\in \Omega
\end{equation}
and $\mathcal{H}^m (\Omega) \geq \bar{C}^{-1} \bar r^m$ for some geometric constant $\bar{C}>0$ (from the proof of Lemma \ref{l:shifting}, 
we know $\bar{C} = \bar{C}(Q,m,n,\bar{n})$ due to the specific choice of $U=\Bbf_{\bar r}\setminus B_{\bar r/2} (V)$ here in Lemma \ref{l:shifting}; 
see also Remark \ref{r:shifting-scaling-invariant}). If we choose $\varepsilon$ sufficiently small, depending on $\bar r$, by Proposition \ref{p:refined} we can find a further subset $\Omega'\subset\Omega$ with measure larger than $(2\bar C)^{-1} \bar r^m$ and such that over each $z\in \Omega'$ we can find a point $p$ in the outer graphical approximation lying in $z+ \alpha_j^\perp$ and in the support of the current $T$ with the property that 
\[
 \dist (z, q_0+\mathbf{S}) \leq \dist (p, q_0+\mathbf{S}) + C \bar r^{-m/2} (\hat{\mathbf{E}} (T, \mathbf{S}, \Bbf_4) + \bar r^2 \mathbf{A}^2)^{1/2}\, . 
\]
In particular we achieve 
\begin{equation}\label{e:inequality-Omega-second}
|\mathbf{p}_{\alpha_1}^\perp (q_0)|^2 + \boldsymbol{\mu} (\mathbf{S})^2 |\mathbf{p}_{V^\perp\cap\alpha_1} (q_0)|^2 \leq \bar C \dist^2 (p, q_0+\mathbf{S}) + C \bar r^{-m} (\hat{\mathbf{E}} (T, \mathbf{S}, \Bbf_4) + \bar r^2\mathbf{A}^2)
\end{equation}
for all $p$ in the subset $\Omega''$ of points $p\in \spt (T)$ which coincide with the outer graphical approximation of Proposition \ref{p:refined} restricted to the subset $\Omega'$. For this set we clearly have $\|T\| (\Omega'') \geq \mathcal{H}^m (\Omega') \geq (2\bar{C})^{-1} \bar r^m$. Observe moreover that, since $\Omega' \subset \Bbf_{\bar r} (\mathbf{p}_V (q_0))\subset \Bbf_{3\bar r/2} (q_0)$, if $\varepsilon$ is sufficiently small (depending on $\bar r$), then $\Omega'' \subset \Bbf_{2\bar r} (q_0)$. We now integrate \eqref{e:inequality-Omega-second} over $\Omega''$ with respect to $d\|T\|$ to find 
\begin{align*}
|\mathbf{p}_{\alpha_1}^\perp (q_0)|^2 + \boldsymbol{\mu} (\mathbf{S})^2 |\mathbf{p}_{V^\perp\cap\alpha_1} (q_0)|^2 & \leq C \bar r^{-m} \int_{\Bbf_{2 \bar r} (q_0)} \dist^2 (p, q_0+\mathbf{S})\, d\|T\| (p)\\
&\qquad \qquad + C \bar r^{-m} (\hat{\mathbf{E}} (T, \mathbf{S}, \Bbf_4) + \bar r^2 \mathbf{A}^2)\, .
\end{align*}
We can in particular write
\begin{align}
|\mathbf{p}_{\alpha_1}^\perp (q_0)|^2 + \boldsymbol{\mu} (\mathbf{S})^2 |\mathbf{p}_{V^\perp\cap\alpha_1} (q_0)|^2 &\leq C \bar r^{\frac{7}{4}} \int_{\Bbf_{2 \bar r} (q_0)} \frac{\dist^2 (q, q_0+\mathbf{S})}{|q-q_0|^{m+7/4}}\, d\|T\|(q) \nonumber\\
&\qquad \qquad + C \bar r^{-m} (\hat{\mathbf{E}} (T, \mathbf{S}, \Bbf_4) + \bar r^2 \mathbf{A}^2)\, .\label{e:will-be-absorbed}
\end{align}
Given that the constant $C$ is independent of the radius $\bar r$, for any fixed $\delta>0$ if we choose $\bar r=\bar r(Q,m,n,\bar{n},\delta)$ sufficiently small we achieve 
\begin{align*}
|\mathbf{p}_{\alpha_1}^\perp (q_0)|^2 + \boldsymbol{\mu} (\mathbf{S})^2 |\mathbf{p}_{V^\perp\cap\alpha_1} (q_0)|^2 &\leq \delta \int_{\Bbf_{2 \bar r} (q_0)} \frac{\dist (q, q_0+\mathbf{S})^2}{|q-q_0|^{m+7/4}}\, d\|T\|(q) \\
&\qquad+ C \bar r^{-m} (\hat{\mathbf{E}} (T, \mathbf{S}, \Bbf_4) + \bar r^2 \mathbf{A}^2)\, .
\end{align*}
We can now insert the latter in \eqref{e:shifted-HS-3} (with $\kappa=\frac{1}{4}$) and, upon fixing $\delta =\delta(Q,m,n,\bar{n})$ small enough (recalling that $2 \bar r\leq\rho$), conclude that
\begin{align}
\int_{\Bbf_{2 \bar r} (q_0)} \frac{\dist (q, q_0+\mathbf{S})^2}{|q-q_0|^{m+7/4}}\, d\|T\| (q)
&\leq\int_{\Bbf_{\rho} (q_0)} \frac{\dist (q, q_0+\mathbf{S})^2}{|q-q_0|^{m+7/4}}\, d\|T\| (q)\nonumber\\
&\leq C \bar r^{-m} (\bar r^2 \mathbf{A}^2 + C \hat{\mathbf{E}} (T, \mathbf{S}, \Bbf_4))\nonumber\\
&\hspace{6em} + \frac{1}{2}
\int_{\Bbf_{2 \bar r} (q_0)} \frac{\dist (q, q_0+\mathbf{S})^2}{|q-q_0|^{m+7/4}}\, d\|T\| (q)\, .
\label{e:shifted-HS-5}
\end{align}
In particular we conclude
\[
\int_{\Bbf_{2 \bar r} (q_0)} \frac{\dist (q, q_0+\mathbf{S})^2}{|q-q_0|^{m+7/4}}\, d\|T\| (q)
\leq  C \bar r^{-m} (\bar r^2 \mathbf{A}^2 + \hat{\mathbf{E}} (T, \mathbf{S}, \Bbf_4)) \, ,
\]
for this fixed choice of $\bar r$ and inserting the latter into \eqref{e:will-be-absorbed} we achieve \eqref{e:control-on-Q-point}.

We are left with the task of showing that \eqref{e:shifted-HS-3} holds. This is accomplished in the following: 

\begin{lemma}\label{l:HS-4}
If $T$, $\Sigma$, and $\mathbf{S}$ are as in Assumption \ref{a:HS} and $\rho= \frac{1}{12 (m-2)}$, then \eqref{e:shifted-HS-3} holds for every point $q_0\in \Bbf_\rho$ with the property that $\Theta (T, q_0) \geq Q$.
\end{lemma}
\begin{proof}
First of all, observe the following elementary fact for all $q,q_0,\Sbf$: 
\begin{equation}\label{e:elementary-shift}
\dist (q, q_0+\mathbf{S})\leq |\mathbf{p}_{\alpha_1}^\perp (q_0)| + C \boldsymbol{\mu} (\mathbf{S}) |\mathbf{p}_{V^\perp\cap\alpha_1} (q_0)| + \dist (q, \mathbf{S})\, .
\end{equation}
Fix $\bar{\rho} = \frac{1}{4\sqrt{m-2}}$.
Next note that, assuming $\varepsilon$ is sufficiently small, we can assume that $\Bbf_{4\bar \rho}\setminus B_{a\bar\rho/8} (V)$ is in the outer region $R^o$ and moreover that $\dist(q_0,V)<a\bar\rho/8$, by Proposition \ref{p:refined}(v). Thus, using \eqref{e:elementary-shift} with $q_0$ as in the statement of the lemma, we gain the inequality
\begin{equation}\label{e:D^2}
\Ebb (T_{q_0,\bar\rho}, \mathbf{S}, \Bbf_4) + \mathbf{A}^2 \leq 
\underbrace{C_0 (\hat{\Ebf} (T, \mathbf{S}, \Bbf_4) + \mathbf{A}^2 + |\mathbf{p}_{\alpha_1}^\perp (q_0)|^2 + \boldsymbol{\mu} (\mathbf{S})^2 |\mathbf{p}_{V^\perp\cap\alpha_1} (q_0)|^2)}_{=:\, D^2}\, ,
\end{equation}
for some constant $C_0$ which is now independent of $\varepsilon$.

Note that, if $D^2\leq \bar{\varepsilon} \boldsymbol{\sigma} (\mathbf{S})^2$ for $\bar\varepsilon = \bar\varepsilon(Q,m,n,\bar{n},M)$ which is the threshold needed to apply Corollary \ref{c:HS-patch}, then we could apply Corollary \ref{c:HS-patch} with $T_{q_0,\bar{\rho}}$ in place of $T$ and the desired inequality \eqref{e:shifted-HS-3} would then follow. 

To handle the general case we fix a suitable $\delta>0$, which will be chosen depending on $\bar\varepsilon$. We wish to apply the Pruning Lemma \ref{l:pruning} with this choice of $\delta$ and $D$. Let $\eps_*= \eps_* (\delta,N)$ be the threshold needed for the applicability of Lemma \ref{l:pruning} and observe that we need to prove that
\begin{equation}\label{e:needed-for-pruning}
D^2 \leq \eps_*^2 \boldsymbol{\mu} (\Sbf)^2\, .
\end{equation}
Note that we can assume
\begin{equation}\label{e:piece-1}
C_0 (\hat{\Ebf} (T, \mathbf{S}, \Bbf_4) + \mathbf{A}^2) \leq \frac{\eps_*^2}{2} \boldsymbol{\mu} (\Sbf)^2
\end{equation}
by choosing $\eps$ sufficiently small, depending on $\eps_*$ in addition to existing dependencies.

Next observe that for every fixed $\lambda>0$, by choosing $\varepsilon$ sufficiently small depending on $\lambda$, the set $\Bbf_\rho \setminus B_\lambda (V)$ will be contained in the outer region $R^o$. In particular, applying Proposition \ref{p:refined}(v) as before, it follows that $\dist (q_0, V) \leq \lambda$
and thus for $\lambda\leq \eps_*/(2\sqrt{C_0})$ we gain
\begin{equation}\label{e:piece-2}
C_0\boldsymbol{\mu} (\mathbf{S})^2 |\mathbf{p}_{V^\perp\cap\alpha_1} (q_0)|^2 \leq \frac{\varepsilon_*^2}{4}
\boldsymbol{\mu} (\Sbf)^2\, .
\end{equation}
Consider moreover the ball $\Bbf := \Bbf_{4\lambda} (\mathbf{p}_V (q_0))$ and estimate
\[
\hat{\Ebf} (T, \alpha_1, \Bbf) \leq C_1 \boldsymbol{\mu} (\Sbf)^2 + C_1 \hat \Ebf (T, \Sbf, \Bbf)
\leq C_1 \boldsymbol{\mu} (\Sbf)^2 + C_1 \lambda^{-m-2} \hat{\Ebf} (T, \Sbf, \Bbf_4)\, ,
\]
where $C_1=C_1(m,n,\bar{n},Q)$.
In particular applying Allard's $L^2$--$L^\infty$ bound for $\alpha_1$ we conclude that
\[
|\mathbf{p}_{\alpha_1}^\perp (q_0)|^2 \leq C_1 \lambda^2 \boldsymbol{\mu} (\Sbf)^2
+ C_1 \lambda^{-m} \hat{\Ebf} (T, \Sbf, \Bbf_4) + C_1 \lambda^2 \Abf^2\, .
\]
We can now select $\lambda$ so that $C_0C_1 \lambda^2 \leq \frac{\varepsilon_*^2}{16}$, yielding
\[
    C_0|\mathbf{p}_{\alpha_1}^\perp (q_0)|^2 \leq \frac{\varepsilon_*^2}{16} \boldsymbol{\mu}(\Sbf)^2 + C\varepsilon_*^{-m-2}\hat{\Ebf} (T, \Sbf, \Bbf_4) + C\varepsilon_*^2 \Abf^2\, .
\]
We are now in the position to choose $\varepsilon$ sufficiently small (depending on $\eps_*$) so that
\begin{equation}\label{e:piece-3}
C_0|\mathbf{p}_{\alpha_1}^\perp (q_0)|^2\leq \frac{\varepsilon_*^2}{8}  \boldsymbol{\mu} (\Sbf)^2\, .
\end{equation}
Summing \eqref{e:piece-1}, \eqref{e:piece-2}, and \eqref{e:piece-3} we obtain \eqref{e:needed-for-pruning}. 

So we can now indeed apply Lemma \ref{l:pruning} to find a cone $\mathbf{S}'\subset \mathbf{S}$, indexed as $\mathbf{S}' = \bigcup_{i\in I} \alpha_i$, with the properties that 
\[
D^2 + \max_{j} \min_{i\in I} \dist^2 (\alpha_i\cap \Bbf_1, \alpha_j\cap \Bbf_1) \leq 
\delta^2 \boldsymbol{\sigma} (\mathbf{S}')^2\, ,
\]
and 
\begin{equation}\label{e:pruning-again}
\max_{j} \min_{i\in I} \dist^2 (\alpha_i\cap \Bbf_1, \alpha_j\cap \Bbf_1) \leq \Gamma^2D^2\, 
\end{equation}
for $\Gamma = \Gamma(\delta)$ given by Lemma \ref{l:pruning}.
Now, since by \eqref{e:D^2} and the triangle inequality
\begin{equation}\label{e:excess-S'}
\hat{\mathbf{E}} (T_{q_0,\bar{\rho}}, \mathbf{S}', \Bbf_4) \leq C D^2 + C \max_{j} \min_{i\in I} \dist^2 (\alpha_i\cap \Bbf_1, \alpha_j\cap \Bbf_1) \leq 
C \delta^2 \boldsymbol{\sigma} (\mathbf{S}')^2\, ,
\end{equation}
where $C$ is a constant independent of $\delta$, and since $\Sbf^\prime\subset\Sbf$ we have by \eqref{e:D^2} again that $\hat{\Ebf}(\Sbf^\prime,T_{q_0,\bar\rho},\Bbf_4)+\Abf^2\leq C_0 D^2 \leq C\delta^2\boldsymbol{\sigma}(\Sbf^\prime)^2$, we can choose $\delta = \delta(Q,m,n,\bar{n},M)$ small to achieve the applicability of Corollary \ref{c:HS-patch} with $\mathbf{S}'$ replacing $\mathbf{S}$ and $T_{q_0, \bar{\rho}}$ replacing $T$. Given that $\rho = \frac{\bar{\rho}}{3 \sqrt{m-2}}$ the estimate \eqref{e:HS-patch}, together with \eqref{e:excess-S'} and \eqref{e:pruning-again}, give 
\begin{equation}\label{e:shifted-HS-10}
\int_{\Bbf_{\rho} (q_0)} \frac{\dist^2 (q, q_0+\mathbf{S'})}{|q-q_0|^{m+2-\kappa}}\, d\|T\| (q)
\leq C (\mathbf{A}^2 + \Gamma^2 D^2) \, .
\end{equation}
Since however we have chosen $\delta$ and thus we have $\Gamma$ fixed, we can treat the latter as a constant depending only on $Q$, $m$, $n$, $\bar{n}$, and $M$.
Since $\mathbf{S}\supset \mathbf{S}'$, we trivially have $\dist (q, q_0+\mathbf{S})
\leq \dist (q, q_0+\mathbf{S'})$, and hence, given our definition of $D^2$, \eqref{e:shifted-HS-10} implies \eqref{e:shifted-HS-3}.
\end{proof}

\section{Linearization}\label{p:linear}

The aim of this section is to collect some results on Dir-minimizing functions which will be pivotal to close the proof of Theorem \ref{c:decay}.  We follow the notation of \cite{DLS_MAMS}. 

\begin{definition}\label{d:1-homogeneous}
Let $Q, m,$ and $n$ be positive integers. We denote by:
\begin{itemize}
    \item $\mathscr{H}_1$ the space of $1$-homogeneous locally Dir-minimizing functions $u: \mathbb R^m \to \mathcal{A}_Q (\mathbb R^n)$;
    \item $\mathscr{L}_1$ the subspace of $u\in \mathscr{H}_1$ which are invariant by translations along at least $(m-2)$-independent directions;
    \item $\mathscr{H}_1^0$ and $\mathscr{L}_1^0$ the subsets of $\mathscr{H}_1$ and $\mathscr{L}_1$ consisting of maps $u$ such that $\boldsymbol{\eta} \circ u \equiv 0$.
\end{itemize}
\end{definition}

Observe that all these spaces are locally compact in the $L^2_{\text{loc}}$ topology. In particular, all the minima appearing in the inequalities below (e.g. see the right hand side of \eqref{e:decay-1}) are attained. We begin with a suitable decay lemma.

\begin{theorem}\label{t:decay}
For every $Q, m, n$, and $\varepsilon >0$, there is $\rho = \rho(Q,m,n,\eps)\in (0, \frac{1}{2})$ with the following property. Assume that $u:\mathbb R^m \supset B_1 \to \mathcal{A}_Q (\mathbb R^n)$ is Dir-minimizing, that $u (0) = Q \llbracket 0 \rrbracket$, and $I_{0,u} (0) \coloneqq \lim_{r\to 0} I_{0,u}(r) \geq 1$. Then
\begin{equation}\label{e:decay-1}
\min_{v\in \mathscr{H}_1} \int_{B_r} \mathcal{G} (u,v)^2 \leq \varepsilon r^{m+2} \int_{B_1} |u|^2\qquad \forall r\leq \rho\, .
\end{equation}
If additionally for some constant $\vartheta>0$ and for every $r\leq \frac{1}{2}$ there are $m-2$ points $z_1, \ldots, z_{m-2}\in B_r$ such that $u (z_i) = Q \llbracket 0 \rrbracket$, $I_{z_i, u} (0)\geq 1$ and $\det ( (z_i\cdot z_j)_{i,j}) \geq \vartheta r^{m-2}$, then
\begin{equation}\label{e:decay-2}
\min_{v\in \mathscr{L}_1} \int_{B_r} \mathcal{G} (u,v)^2 \leq \varepsilon r^{m+2}\int_{B_1}|u|^2 \qquad \forall r\leq \rho\, 
\end{equation}
(where $\rho$ will depend also on $\vartheta$). 
Finally, when $\boldsymbol{\eta} \circ u \equiv 0$ we have the equalities
\begin{align}
&\min_{v\in \mathscr{H}_1} \int_{B_r} \mathcal{G} (u,v)^2 = \min_{v\in \mathscr{H}^0_1} \int_{B_r} \mathcal{G} (u,v)^2\label{e:average-free-1}\, ,\\
&\min_{v\in \mathscr{L}_1} \int_{B_r} \mathcal{G} (u,v)^2 = \min_{v\in \mathscr{L}^0_1} \int_{B_r} \mathcal{G} (u,v)^2\label{e:average-free-2}\, .
\end{align}
\end{theorem}

Let us first remark that the identities \eqref{e:average-free-1} and \eqref{e:average-free-2} are simple consequences of the following elementary algebraic fact. Consider any pair $T= \sum_i \llbracket T_i\rrbracket$, $S= \sum_i \llbracket S_i \rrbracket\in \mathcal{A}_Q (\mathbb R^n)$ and define
\begin{align*}
    T' &= \sum_i \llbracket T_i - \boldsymbol{\eta} (T) \rrbracket\, ,\qquad \text{and}\qquad S' = \sum_i \llbracket S_i - \boldsymbol{\eta} (S) \rrbracket\, ,
\end{align*}
then
\[
\mathcal{G} (T,S)^2 = Q |\boldsymbol{\eta} (T)-\boldsymbol{\eta} (S)|^2 + \mathcal{G} (T',S')^2\, .
\]
Given this then, for example, in \eqref{e:average-free-1}, we would have for $v\in \mathscr{H}_1$,
\[
\mathcal{G}(u,v)^2 = Q|\boldsymbol{\eta}\circ v|^2 + \mathcal{G}(u,v^\prime)^2
\]
where $v^\prime = v-\boldsymbol{\eta}\circ v$, meaning $\min_{v\in\mathscr{H}_1}\int_{B_r}\mathcal{G}(u,v)^2\geq\min_{v\in\mathscr{H}^0_1}\int_{B_r}\mathcal{G}(u,v)^2$; the other inequality is of course trivial as we are minimizing over a larger set. Thus, to prove Theorem \ref{t:decay} we just need to prove \eqref{e:decay-1} and \eqref{e:decay-2}. We will first show, using a compactness argument, that if the frequency of a Dir-minimizer is pinched between $1$ and $1+\eta$ then the function is very close to be $1$-homogeneous. We begin with the following intermediate lemma.

\begin{lemma}\label{l:compactness}
For every $Q,m,n,\bar \varepsilon >0$, there is a constant $\eta = \eta(Q,m,n,\bar{\eps})>0$ with the following property. Fix $r>0$. Then if $u:B_r\to \mathcal{A}_Q (\mathbb R^n)$ is a Dir-minimizing function such that $u (0) = Q\llbracket 0\rrbracket$, $I_{0,u} (0) \geq 1$ and $I_{0,u} (r) \leq 1+\eta$, then
\begin{equation}\label{e:compactness-estimate}
\min_{v\in \mathscr{H}_1} \int_{B_r} \mathcal{G} (u,v)^2 \leq \bar \varepsilon r \int_{\partial B_r} |u|^2\, .
\end{equation}
If additionally for some constant $\vartheta>0$ there are $m-2$ points $z_1, \ldots, z_{m-2}\in B_{r/2}$ such that $u (z_i) = Q \llbracket 0 \rrbracket$, $I_{z_i, u} (0)\geq 1$ and $\det ((z_i\cdot z_j)_{i,j}) \geq \vartheta r^{m-2}$, then, under the assumption that the parameter $\eta$ is small enough, depending also on the value of $\vartheta$, we have
\begin{equation}\label{e:compactness-estimate-2}
\min_{v\in \mathscr{L}_1} \int_{B_r} \mathcal{G} (u,v)^2 \leq \bar \varepsilon r \int_{\partial B_r} |u|^2\, .
\end{equation}
\end{lemma}
\begin{proof}
By scaling the domain and $u$, we can assume without loss of generality that $r=1$ and $H_{u} (1) \coloneqq \int_{\partial B_1} |u|^2 =1$. We then argue by contradiction and assume the statements to be false for some fixed $\bar \varepsilon$ (and $\vartheta$) no matter how small $\eta$ is. In particular we can set $\eta = \frac{1}{k}$ and select corresponding maps $u_k$ satisfying $u_k (0) = Q \llbracket 0 \rrbracket$, $1\leq I_{0,u_k} (0) \leq I_{0,u_k} (1) \leq 1 + \frac{1}{k}$, $H_{u_k} (1) =1$ and 
\begin{equation}\label{e:compactness-contradiction}
\min_{v\in \mathscr{H}_1} \int_{B_1} \mathcal{G} (u_k,v)^2 \geq \bar\varepsilon \, .
\end{equation}
As for the second statement of the lemma, we know additionally
\begin{itemize}
    \item[$(\star)$] the existence of points $z^k_1, \ldots, z^k_{m-2}\in B_{1/2}$ such that $u_k (z^k_i) = Q \llbracket 0 \rrbracket$, $I_{z^k_i, u_k} (0)\geq 1$ and 
$\det ((z^k_i\cdot z^k_j)_{i,j}) \geq \vartheta$,
\end{itemize}
while 
\begin{equation}\label{e:compactness-contradiction-2}
\min_{v\in \mathscr{L}_1} \int_{B_1} \mathcal{G} (u_k,v)^2 \geq \bar\varepsilon \, .
\end{equation}
As $\int_{B_1} |Du_k|^2 = I_{0,u_k} (1) \cdot H_{u_k} (1) \leq 1 +\frac{1}{k}$, we can appeal to the compact embedding of the $W^{1,2} (B_1;\Acal_Q(\R^n))$ in $L^2 (B_1;\Acal_Q(\R^n))$ (cf. \cite{DLS_MAMS}*{Proposition~2.11}) to extract (in both cases) a subsequence, not relabeled, which converges strongly in $L^2$ to some $u$. In fact (up to extraction of another subsequence), we can also assume that $u_k|_{\partial B_1} \to u |_{\partial B_1}$ by the trace property of $W^{1,2}$ (cf. \cite{DLS_MAMS}*{Proposition~2.10}), while $u$ is Dir-minimizing (cf. \cite{DLS_MAMS}*{Proposition~3.20}) and $u_k \to u$ strongly in $W^{1,2}_{\text{loc}}(B_1)$. It also follows, from H\"older estimates (cf. \cite{DLS_MAMS}*{Theorem~3.9} that $u (0) = \llbracket 0 \rrbracket$. Also, from upper semi-continuity of the frequency, $I_{0,u} (0) \geq \limsup_{k\to \infty} I_{0,u_k} (0) \geq 1$, while $I_{0,u} (1) \leq \liminf_{k\to \infty} I_{0,u_k} (1) =1$. Hence, by the monotonicity of the frequency function (cf. \cite{DLS_MAMS}*{Theorem 3.15}) $I_{0,u} (r) \equiv 1$, which in turn implies that $u$ is $1$-homogeneous (cf. \cite{DLS_MAMS}*{Corollary 3.16}). In particular $u\in \mathscr{H}_1$ and so the $L^2$ convergence of $u_k$ to $u$ contradicts \eqref{e:compactness-contradiction}.

Under the additional assumption ($\star$) we can assume, up to extraction of a further subsequence, that $z^k_i \to z_i \in B_{1/2}$ and obviously $\det ((z_i\cdot z_j)_{i,j}) \rangle \geq \vartheta$. In particular the vectors $v_i$ span an $(m-2)$-dimensional subspace $V$. Again by the H\"older continuity of $u_k$ and upper semi-continuity of the frequency function we conclude $u (z_i) = Q\llbracket 0 \rrbracket$ and $I_{z_i,u} (0)\geq 1$. But because of the $1$-homogeneity of $u$ we have the same properties for $u (\sigma z_i)$ for every $\sigma\in (0,1]$, in particular, again by upper semi-continuity of the frequency, $I_{z_i,u} (0)\leq 1$. Arguing as in \cite{DLS_MAMS}*{Proof of Lemma 3.4}, it follows that $u (x+\lambda z_i) = u (x)$ for every $x$, every $i$, and every $\lambda \in \mathbb R$. In particular $u\in \mathscr{L}_1$. By the strong $L^2$-convergence of $u_k$ to $u$ this contradicts \eqref{e:compactness-contradiction-2}.
\end{proof}

\begin{proof}[Proof of Theorem \ref{t:decay}]
We fix $\varepsilon$ and $\vartheta$ as in the statement of the Theorem. Fix $\bar\eps>0$ (to be determined, possibly smaller than $\eps$), and let $\eta = \eta(Q,m,n,\bar\eps)$ be the parameter given by Lemma \ref{l:compactness} with this choice of $\bar\eps$. We then define
\[
\varrho:= \inf \{0\leq r\leq {\textstyle{\frac{1}{2}}}: I_{0,u} (r) \geq 1+\eta\}\, .
\]
In particular, we can apply Lemma \ref{l:compactness} to infer, respectively in each case (\eqref{e:decay-1} or \eqref{e:decay-2}) that
\begin{align}
&\min_{v\in \mathscr{H}_1} \int_{B_r} \mathcal{G} (u,v)^2 \leq \bar \varepsilon r \int_{\partial B_r} |u|^2\, \qquad \forall r\leq\varrho \label{e:small-radii-1}\\
&\min_{v\in \mathscr{L}_1} \int_{B_r} \mathcal{G} (u,v)^2 \leq \bar \varepsilon r \int_{\partial B_r} |u|^2\, \qquad \forall r\leq \varrho \label{e:small-radii-2}\, .
\end{align}
 We know by the decay of the $L^2$ height in terms of the frequency (cf. \cite{DLS_MAMS}*{Corollary 3.18}) that
\begin{align}
\int_{\partial B_r} |u|^2 \leq \left(\frac{r}{s}\right)^{m-1+2 I_{0,u}(r)}\int_{\partial B_s}|u|^2 \qquad \forall r\leq s\leq 1\, .\label{e:height-decay}
\end{align}
Moreover, we have
\begin{equation}\label{e:min-sphere-int}
\min_{t\in [1/2,1]} \int_{\partial B_t} |u|^2 \leq C \int_{B_1} |u|^2\, ,
\end{equation}
for some constant $C = C(Q,m,n)$.  Thus, we infer from \eqref{e:small-radii-1} (respectively \eqref{e:small-radii-2}), combined with \eqref{e:height-decay} with $s$ chosen to be $t\in [1/2,1]$ realizing the minimum and the fact that $I_{0,u}(r) \geq 1$ for all $r> 0$, that
\begin{align}
&\min_{v\in \mathscr{H}_1} \int_{B_r} \mathcal{G} (u,v)^2 \leq C \bar \varepsilon r^{m+2} \int_{B_1} |u|^2\, \qquad \forall r\leq \varrho \label{e:small-radii-3}\\
&\min_{v\in \mathscr{L}_1} \int_{B_r} \mathcal{G} (u,v)^2 \leq C \bar \varepsilon r^{m+2} \int_{B_1} |u|^2\, \qquad \forall r\leq \varrho \label{e:small-radii-4}\, .
\end{align}
This proves the result for $r\leq\varrho$; however, $\varrho$ is not a geometric constant (it depends on $u$), so we are not done. Now, by monotonicity of the frequency and the definition of $\varrho$ we know that $I_{0,u} (r) \geq 1+\eta$ for every $r> \varrho$. So, again by \eqref{e:height-decay} and \eqref{e:min-sphere-int} (choosing appropriate $t\in [1/2,1]$ again), for every $r \in (\varrho, 1]$ we have 
\[
\int_{\partial B_s} |u|^2 \leq \left(\frac{s}{r}\right)^{m+1}\int_{\partial B_r}|u|^2 \leq s^{m+1}\left(\frac{r}{t}\right)^{2\eta}\int_{\partial B_t}|u|^2 \leq C s^{m+1} r^{2\eta} \int_{B_1} |u|^2 \qquad \forall s\leq r\, .
\]
Integrating the latter over $s\in [0,r]$ we get 
\[
\int_{B_r} |u|^2 \leq C r^{m+2+2\eta} \int_{B_1} |u|^2 \qquad \forall \varrho < r \leq 1\, .
\]
Since, however, the function $v\equiv Q\llbracket 0\rrbracket$ belongs to $\mathscr{L}_1\subset \mathscr{H}_1$, we can combine the latter inequality with \eqref{e:small-radii-3} (resp. \eqref{e:small-radii-4}) to get 
\begin{equation}
\min_{v\in \mathscr{H}_1} \int_{B_r} \mathcal{G} (u,v)^2
\leq C \max \{\bar \varepsilon, r^{2\eta}\} r^{m+2} \int_{B_1} |u|^2 \qquad \forall r\leq \frac{1}{2}\, ,
\end{equation}
and respectively
\begin{equation}
\min_{v\in \mathscr{L}_1} \int_{B_r} \mathcal{G} (u,v)^2
\leq C \max \{\bar \varepsilon, r^{2\eta}\} r^{m+2} \int_{B_1} |u|^2\qquad \forall r \leq \frac{1}{2}, .
\end{equation}
We now choose first $\bar\varepsilon$ so that $C \bar \varepsilon \leq \varepsilon$, which in turn fixes the value of $\eta>0$ determined by Lemma \ref{l:compactness}. Hence we can choose $\rho = \rho(Q,m,n,\eps)>0$ such that $C \rho^{2\eta} \leq \varepsilon$. Then the desired estimates \eqref{e:decay-1} and \eqref{e:decay-2} for $r\leq\rho$ follow.
\end{proof}

The second result we will need is a suitable ``removability result'' which we stated previously in Proposition \ref{p:remove-spine}. We recall the statement for the convenience of the reader.
\begin{proposition}\label{p:remove-spine-2}
Assume $\Omega\subset \mathbb R^m$ is a Lipschitz domain, $V\subset \mathbb R^m$ is an $(m-2)$-dimensional plane, and $v\in W^{1,2}(\Omega;\Acal_Q(\R^n))$ is a map with the property that the restriction of $v$ to $\Omega_\varepsilon := \Omega\setminus B_\varepsilon (V)$ is Dir-minimizing for every $\varepsilon > 0$. Then, $v$ is Dir-minimizing in $\Omega$.
\end{proposition}

We remark an important subtlety: the set which is removed from $\Omega$ to get to $\Omega_\varepsilon$ is {\em not} compactly contained in $\Omega$ when $m\geq 3$. 

We will in fact prove the following lemma, where we use the notion of $p$-capacity, ${\text{Cap}}^m_{p}(A)$, of a set $A\subset \R^m$. We refer the reader to e.g. \cite{EvansGariepy} for the definition and preliminaries.

\begin{lemma}\label{l:Dir_min_cap_zero}
Let $\Omega \subset \R^m$ be any bounded Lipschitz domain and $A\subset \mathbb R^m$ any set with $\textnormal{Cap}^m_2 (A) =0$. Then, for any two maps $v, u\in W^{1,2} (\Omega, \mathcal{A}_Q (\mathbb R^n))$ such that $v=u$ on $\partial \Omega$ and for any $\delta>0$, we can find a third map $w\in W^{1,2}(\Omega,\mathcal{A}_Q(\R^n))$ which coincides with $v$ on $\partial \Omega \cup (\Omega\cap U)$ for some neighborhood $U$ of $A$ and such that 
\[
\int_\Omega |Dw|^2 \leq \int_\Omega |Du|^2 + \delta\, .
\]
\end{lemma}

Let us first prove Proposition \ref{p:remove-spine-2} from Lemma \ref{l:Dir_min_cap_zero}.

\begin{proof}[Proof of Proposition \ref{p:remove-spine-2}]
    First note that since $V$ is an $(m-2)$-dimensional plane, it obeys $\text{Cap}^m_2(V) = 0$ (note here $m>2$, and so this follows because sets with finite $\mathcal{H}^{m-2}$-measure have vanishing 2-capacity, see \cite{EvansGariepy}; the result still holds for $m=2$ by using the classical logarithmic cut-off trick).
    
    Suppose for contradiction that $v$ is not Dir-minimizing in $\Omega$. Then we can find a competitor $u\in W^{1,2}(\Omega;\Acal_Q(\R^n))$ with $v=u$ on $\partial \Omega$ and
    $$\int_{\Omega}|Du|^2 < \int_{\Omega}|Dv|^2 - \delta$$
    for some $\delta>0$. Now apply Lemma \ref{l:Dir_min_cap_zero} with this choice of $v,u$ and $\delta$ to find a map $w\in W^{1,2}(\Omega;\Acal_Q(\R^n))$ which coincides with $v$ on $\partial\Omega\cup (\Omega\cap B_{\eps_*}(V))$ for some $\eps_*>0$. In particular, we have
    $$\int_{\Omega}|Dw|^2 \leq \int_{\Omega}|Du|^2 + \delta <\int_{\Omega}|Dv|^2.$$
    However, this contradicts $v$ being Dir-minimizing on $\Omega_{\eps_*}$ (since $v=w$ on $B_{\eps_*}(V)$). 
\end{proof}

All that remains is to prove Lemma \ref{l:Dir_min_cap_zero}.

\begin{proof}[Proof of Lemma \ref{l:Dir_min_cap_zero}]
We fix $\Omega$ a Lipschitz domain and we let $v, u:\Omega \to \Acal_Q(\R^n)$ be any pair of $W^{1,2}(\Omega;\Acal_Q(\R^n))$ functions with $v|_{\partial\Omega}= u|_{\partial\Omega}$, where the boundary data is defined in the Sobolev trace sense (see \cite{DLS_MAMS}*{Section 2.2.2}). Consider now $\tilde{u} \coloneqq \boldsymbol{\xi}_{BW}\circ u$ and $\tilde{v} \coloneqq \boldsymbol{\xi}_{BW}\circ v$, where $\boldsymbol{\xi}_{BW}: \Acal_Q(\R^n)\to \R^N$ is the special bi-Lipschitz embedding of White, which is constructed via a modification of the original embedding of Almgren (see \cite{SXChang} or \cite{DLS_MAMS}*{Corollary 2.2}). 

Let $z:= \tilde{u}-\tilde{v}$ and extend it to be identically equal to $0$ on $\mathbb R^m\setminus \Omega$. We claim that there is a sequence $\{\zeta_k\} \subset W^{1,2} (\mathbb R^m, \R^N)$ with the following properties:
\begin{itemize}
\item $\zeta_k$ vanishes identically on $(\mathbb R^m \setminus \Omega) \cup U_k (A)$, where $U_k (A)$ is some open neighborhood of $A$ (which depends on $k$);
\item $\|\zeta_k-z\|_{W^{1,2}}\to 0$ as $k\uparrow \infty$.
\end{itemize}
Assuming the existence of such a sequence $\{\zeta_k\}$ for now, we then let $w_k := \boldsymbol{\xi}_{BW}^{-1} \circ (\boldsymbol{\rho} \circ (\tilde{v} + \zeta_k))$, where $\boldsymbol{\rho}$ is Almgren's Lipschitz retraction of $\R^N$ onto $\boldsymbol{\xi}_{BW} (\mathcal{A}_Q (\R^N))$ (cf. \cite{DLS_MAMS}*{Chapter~2}). Observe that $\tilde{v}+\zeta_k =\tilde{v}$ on $\partial \Omega \cup (\Omega\cap U_k (A))$ and so $w_k = v$ on $\partial \Omega \cup (\Omega\cap U_k (A))$. Moreover 
\[
\int_\Omega |Dw_k|^2 = \int_\Omega |D (\boldsymbol{\rho} \circ (\tilde{v}+\zeta_k))|^2
\]
and 
\[
\int_\Omega |Du|^2 = \int_\Omega |D \tilde{u}|^2\, .
\]
On the other hand $\tilde{v}+\zeta_k$ converges to $\tilde{v} + z = \tilde{v}+(\tilde{u}-\tilde{v}) = \tilde{u}$ strongly in $W^{1,2} (\R^m;\R^N)$ and, 
because $\boldsymbol{\rho}$ is Lipschitz and equal to the identity on 
$\boldsymbol{\xi}_{BW} (\mathcal{A}_Q (\R^n))$, $\boldsymbol{\rho} \circ (\tilde{v} + \zeta_k)$ converges strongly in $W^{1,2}$ to $\boldsymbol{\rho} \circ \tilde{u} = \tilde{u}$. In particular we conclude from this and the above two identities that the Dirichlet energy of $w_k$ converges to the Dirichlet energy of $u$, and thus to conclude the result of the lemma we just need to take $w=w_k$ for $k$ sufficiently large (depending on $\delta$).

So now we just need to show the existence of $\zeta_k$ satisfying the desired properties. 
By the definition of ${\textrm{Cap}}^m_2$, given $\eps> 0$, we may choose $\rho_\eps \in C_c^\infty(B_\eps(A))$ with $\rho_\eps \geq \mathbf{1}_{B_{\eps/2}(A)}$ and such that 
\[
\int |D\rho_\eps|^2 \leq {\textrm{Cap}}^m_2(A) + \eps = \eps.
\]
In particular we can assume, by truncating $\rho_\eps$, that it takes values between $0$ and $1$; clearly it also follows that $\rho_\eps \to 0$ pointwise almost everywhere as $\eps \downarrow 0$. 

Next we choose $M$ sufficiently large and consider the truncation function $z_M = (z_M^1,\dots,z_M^N)$ given by
\[
    z_M^i\coloneqq \min\{z^i,M\}\mathbf{1}_{\{z^i > 0\}} + \max\{z^i,-M\}\mathbf{1}_{\{z^i \leq 0\}}\, ,
\]
where we have written $z = (z^1,\dotsc,z^N)$. It is simple to see that $z_M \to z$ in $W^{1,2}$ as $M\to\infty$, and so for each $k\in \{1,2,\dotsc\}$ we can select $M=M(k)$ so that $\|z_M-z\|_{W^{1,2}} \leq \frac{1}{2k}$. Observe also that $z_M$ always vanishes outside of $\Omega$. Next we take $\zeta_k := z_M (1-\rho_{\eps_k})$ for a sufficiently small $\eps_k$ to be chosen. Note indeed that for fixed $M$, $\|z_M\rho_\varepsilon\|_{L^2} \to 0$ as $\varepsilon\downarrow 0$ by dominated convergence. On the other hand
\[
D (z_M (1-\rho_\varepsilon)) -Dz_M
= -\rho_\eps Dz_M - z_M D\rho_\varepsilon
\]
and $\|Dz_M\rho_\varepsilon\|_{L^2} \to 0$ as $\eps\downarrow 0$, again by dominated convergence, while we have $\|z_M D \rho_\varepsilon\|_{L^2}\leq \sqrt{N} M \|D\rho_\varepsilon\|_{L^2}$ converges to $0$ as well as long as we keep $M$ fixed while $\varepsilon \downarrow 0$. Thus, $z_M(1-\rho_\eps)\to z_M$ in $W^{1,2}$, and so we can choose $\eps = \eps_k$ sufficiently small so that $\|z_M(1-\rho_{\eps_k})-z_M\|_{W^{1,2}}\leq \frac{1}{2k}$. This is the choice of $\eps_k$ used to define $\zeta_k$. Note then that $\|\zeta_k-z\|_{W^{1,2}} \leq 1/k\to 0$, and also that $\zeta_k$ vanishes identically on $(\R^m\setminus\Omega)\cup B_{\eps_k/2}(A)$; this completes the proof.
\end{proof}

\section{Final blow-up argument}\label{p:blowup}

In this part we complete the proof of Theorem \ref{t:weak-decay}, which we recall in turn implies the validity of Theorem \ref{c:decay}, as demonstrated in Section \ref{s:reduction}.

\subsection{Two regimes}

The proof of Theorem \ref{t:weak-decay} will be split into two cases, which will both be proved via a blow-up argument. We start by giving the detailed statements.

\begin{proposition}[Collapsed decay]\label{p:decay-collapsed}
For every $Q, m, n, \bar n$, and $\varsigma_1>0$ there are positive constants $\varepsilon_c = \eps_c(Q,m,n,\bar{n},\varsigma_1)\leq 1/2$ and $r_c = r_c(Q,m,n,\bar{n},\varsigma_1) \leq 1/2$ with the following property. Assume that 
\begin{itemize}
\item[(i)] $T$ and $\Sigma$ are as in Assumption \ref{a:main}, and $\|T\| (\Bbf_1) \leq \omega_m (Q+\frac{1}{2})$;
\item[(ii)] There is a cone $\mathbf{S}\in \mathscr{C} (Q, 0)$ which is $M$-balanced (with $M$ as in Assumption \ref{a:choose-M}), such that \eqref{e:smallness-weak} and \eqref{e:no-gaps-weak} hold with $\varepsilon_c$ in place of $\varepsilon_1$, and in addition
$\boldsymbol{\mu} (\mathbf{S}) \leq \varepsilon_c$;
\item [(iii)] $\Abf^2\leq \eps_c^2\mathbb{E}(T,\tilde{\Sbf},\Bbf_1)$ for every $\tilde{\Sbf}\in \mathscr{C}(Q,0)$.
\end{itemize}
Then, there is another cone $\mathbf{S}'\in \mathscr{C} (Q,0) \setminus \mathscr{P} (0)$ such that 
\begin{equation}\label{e:decay-collapsed}
\mathbb{E} (T, \mathbf{S}', \Bbf_{r_c})
\leq \varsigma_1 \mathbb{E} (T, \mathbf{S}, \Bbf_1)\, .
\end{equation}
\end{proposition}

\begin{proposition}[Non-collapsed decay]\label{p:decay-noncollapsed}
For every $Q, m, n, \bar n$, $\varepsilon^\star_c >0$, and $\varsigma_1>0$, there are positive constants $\varepsilon_{nc} = \eps_{nc}(Q,m,n,\bar{n},\varepsilon^\star_c,\varsigma_1)\leq 1/2$ and $r_{nc} = r_{nc}(Q,m,n,\bar{n},\varepsilon^\star_c,\varsigma_1) \leq \frac{1}{2}$ with the following property. Assume that 
\begin{itemize}
\item[(i)] $T$ and $\Sigma$ are as in Assumption \ref{a:main} and $\|T\| (\Bbf_1) \leq \omega_m (Q+\frac{1}{2})$;
\item[(ii)] There is $\mathbf{S}\in \mathscr{C} (Q, 0)$ which is $M$-balanced (with $M$ as in Assumption \ref{a:choose-M}), such that 
\eqref{e:smallness-weak} and \eqref{e:no-gaps-weak} hold with $\varepsilon_{nc}$ in place of $\varepsilon_1$ and in addition
$\boldsymbol{\mu} (\mathbf{S}) \geq \varepsilon^\star_c$;
\item [(iii)] $\Abf^2\leq \eps^2_{nc}\mathbb{E}(T,\tilde{\Sbf},\Bbf_1)$ for every $\tilde{\Sbf}\in \mathscr{C}(Q,0)$.
\end{itemize}
Then, there is another cone $\mathbf{S}'\in \mathscr{C} (Q,0) \setminus \mathscr{P} (0)$ such that 
\begin{equation}\label{e:decay-noncollapsed}
\mathbb{E} (T, \mathbf{S}', \Bbf_{r_{nc}})
\leq \varsigma_1 \mathbb{E} (T, \mathbf{S}, \Bbf_1)\, .
\end{equation}
\end{proposition}

\begin{remark}\label{r:A-excess}
    In fact, for the proof of Proposition \ref{p:decay-collapsed} and Proposition \ref{p:decay-noncollapsed}, instead of condition (iii) all we will need to assume is that $\Abf^2\leq \eps_c^2\mathbb{E}(T,\Sbf,\Bbf_1)$ or $\Abf^2\leq\eps^2_{nc}\mathbb{E}(T,\Sbf,\Bbf_1)$, respectively, for the cone $\Sbf$ as in (ii) in each proposition, respectively.
\end{remark}

Theorem \ref{t:weak-decay} obviously follows from the two propositions above. In fact, being given $Q,m,n,\bar n$, and $\varsigma_1$ as in Theorem \ref{t:weak-decay}, we first apply Proposition \ref{p:decay-collapsed} and get $\varepsilon_c$ and $r_c$. We then apply Proposition \ref{p:decay-noncollapsed} with the same choice of $Q,m,n,\bar n$, and $\varsigma_1$, and with $\varepsilon_c^\star = \varepsilon_c$, to get $\varepsilon_{nc}$ and $r_{nc}$. It is then clear that Theorem \ref{t:weak-decay} holds if we set $\varepsilon_1 := \min \{\varepsilon_c, \varepsilon_{nc}\}$, $r^1_2 := r_c$ and $r^2_2 := r_{nc}$.

Both propositions will be proved by a blow-up procedure, which will assume, by seeking a contradiction, that the statements fail. In the proof of Proposition \ref{p:decay-noncollapsed} this means that $\varepsilon^\star_c$ is fixed while $\varepsilon_{nc}$ will be taken to be arbitrarily small and so in particular $\boldsymbol{\mu} (\mathbf{S})$ will stay away from $0$ while the ratio
\begin{equation}\label{e:key-ratio}
\frac{\mathbb{E} (T, \mathbf{S}, \Bbf_1)}{\boldsymbol{\sigma} (\mathbf{S})^2}
\end{equation}
is arbitrarily small.
We will call the latter a {\em Simon blow-up} (since in the work \cite{Simon_cylindrical} the cones are uniformly non-collapsed, analogously to this case), and in this case we will take limits of suitable rescalings of the coherent outer approximations of Proposition \ref{p:coherent}. 

In the proof of Proposition \ref{p:decay-collapsed} we will instead assume that $\varepsilon_c$ is arbitrarily small. This will mean that {\em both} the ratio in \eqref{e:key-ratio} and $\boldsymbol{\mu} (\mathbf{S})$ are arbitrarily small. In this case we will approximate the current in the outer region by reparameterizing the graphs of the coherent outer approximations over a single plane, which we can fix to be one of the planes forming $\mathbf{S}$, say $\alpha_1$. We will call these the {\em transversal coherent approximations}, as opposed to the ones of the Simon blow-up, which will be called the {\em orthogonal coherent approximations}. By construction the transversal approximations are superpositions of multi-valued maps, each of them close to the linear map describing $\alpha_j$ as a graph over $\alpha_1$. We will then subtract the latter linear map from the sheet of the corresponding multi-valued approximation and study the limits of appropriate rescalings. We will call this procedure a {\em Wickramasekera blow-up} (since the work \cite{W14_annals} is the first appearance of such a blow-up where the cones are collapsing to a high-multiplicity flat plane).

\subsection{Transversal coherent approximation} 
In this section we work under the assumption of Proposition \ref{p:decay-collapsed} and recast the coherent outer approximation of Proposition \ref{p:coherent} as an approximation through a Lipschitz multi-valued function over a single plane (the latter function representing the cone $\Sbf$). The choice of the plane is not really important as long as its Hausdorff distance from $\mathbf{S}$ in $\Bbf_1$ is comparable to $\boldsymbol{\mu} (\mathbf{S})$; we may without loss of generality take it to be the first indexed plane $\alpha_1$ forming $\mathbf{S}$. Given a multi-valued function $g = \sum_i \llbracket g_i \rrbracket$ and a single-valued function $f$ with the same domain, we also use the notational shorthand $g\ominus f$ for 
\[
g\ominus f (x) \coloneqq \sum_i \llbracket g_i (x) - f (x) \rrbracket\, ,
\]
and similarly we define $g\oplus f$.
We moreover recall the set $\mathscr{N} (L)$ of ``neighbouring cubes'' to $L$ and the quantity $\bar{\Ebf} (L)$ introduced in Definition \ref{d:neighbors} and introduce 
\[
\widetilde{\Ebf} (L) := \max \{\bar{\Ebf} (L') : L'\in \mathscr{N} (L)\}
= \max \{ \Ebf (L'') : L''\in \mathscr{N} (L'), \, L'\in \mathscr{N} (L)\}
\]
i.e. we maximize the excess of neighbours of neighbours.

\begin{proposition}[Transversal coherent approximation]\label{p:transversal}
Let $T$, $\Sigma$, and $\mathbf{S}$ be as in Proposition \ref{p:refined}, let $\mathcal{\ell}\in \N$ be the maximal number such that $\mathcal{G}_{\ell+2} \subset \mathcal{G}^o$, and consider the regions $\widetilde{R}^o_i\coloneqq \alpha_i \cap \bigcup_{L\in \mathcal{G}_{\leq\ell}} R (L)$. Let $u_i:R_i^o\to \Acal_{Q_i}(\alpha_i^\perp)$ be the $Q_i$-valued maps in Proposition \ref{p:coherent}. Under the additional assumption that $\boldsymbol{\mu} (\mathbf{S})\leq c_0$ for a geometric constant $c_0 = c_0(m,n,\bar{n})>0$, the following holds.
\begin{itemize}
    \item[(a)] Each plane $\alpha_i$ is the graph over $\alpha_1$ of a linear map $A_i : \alpha_1\to \alpha_1^\perp$ with $|A_i|\leq C \boldsymbol{\mu} (\mathbf{S})$ and $\ker (A_i) = V$.
    \item[(b)] For each $i$ there is a map $v_i : \widetilde{R}^o_1 \to \mathcal{A}_{Q_i} (\alpha_1^\perp)$ with the property that 
    \[
    \mathbf{G}_{v_i} = \mathbf{G}_{u_i} \res \mathbf{p}_{\alpha_1}^{-1} (\widetilde{R}_1^o)\, .
    \]
    \item[(c)] If we let $v:\widetilde{R}^o_1\to \mathcal{A}_Q(\alpha_1^\perp)$ be the $Q$-valued function $v:= \sum_i \llbracket v_i \rrbracket$ (note that $Q=\sum_{i=1}^N Q_i$), then 
    \begin{align}
    \|v\|_{L^\infty} + \|Dv\|_{L^2} &\leq C \boldsymbol{\mu} (\Sbf)
    \end{align}
    \item[(d)] If we let 
    \[
    K := \mathbf{p}_{\alpha_1} ((\spt (T)\cap \Bbf_{1/2}\cap \mathbf{p}_{\alpha_1}^{-1} (\widetilde{R}_1^o))\setminus {\rm gr}\, (v))
    \]
    then, for all $L\in \mathcal{G}_{\leq \ell}$,
    \begin{equation}\label{e:mass-error-transersal}
    |L_1\setminus K| + \|T\| (\mathbf{p}_{\alpha_1}^{-1} (L_1\setminus K)) \leq C 2^{-m\ell (L)} (\widetilde{\mathbf{E}} (L)+2^{-2\ell (L)} \mathbf{A}^2)^{1+\gamma}\, ,
    \end{equation}
    where $L_1$ is as in Proposition \ref{p:coherent}.
    \item[(e)] If we let $w_i := v_i \ominus A_i: \tilde{R}_1^o \to \Acal_{Q_i}(\alpha_1^\perp)$, then
    \begin{align}
    2^{2\ell (L)} \|w_i\|_{L^\infty (L_1)}^2 + 2^{m\ell (L)} \|Dw_i\|_{L^2 (L_1)}^2 &\leq C (\widetilde{\mathbf{E}} (L) + 2^{-2\ell(L)}\mathbf{A}^2)\label{e:local-transverse-1}\\
    \|Dw_i\|_{L^\infty(L_1)} & \leq C (\widetilde{\mathbf{E}} (L) + 2^{-2\ell(L)}\mathbf{A}^2)^\gamma  \label{e:Lipschitz-transverse}
    \end{align}
\end{itemize}
Here, $C = C(Q,m,n,\bar{n})$.
\end{proposition}
\begin{proof}
Claim (a) is obvious. For the rest, all we are essentially doing is changing the coordinates of $u_i$, to give a parameterization over $\alpha_1$ rather than $\alpha_i$. Claims (b) and and the $L^\infty$-estimate in (c) follow immediately from \cite{DLS_multiple_valued}*{Proposition 5.2}. The $L^2$-gradient estimate in (c) follows immediately from the $L^2$ bound in \eqref{e:local-transverse-1} and (a). Conclusion (d) follows immediately from the corresponding estimate in Proposition \ref{p:coherent}. As for \eqref{e:Lipschitz-transverse}, observe that, for a.e. $x\in \tilde{R}_1^o$, $|Dw_i| (x)$ equals, up to a geometric constant, the following quantity:
\[
\sup \{|\vec{\mathbf{G}}_{v_i} (x,y) - \vec{\alpha}_i| : (x,y)\in {\rm gr}\, (v_i)\}
\]
while an analogous formula holds for $|Du_i| (x)$, with $x\in R^o_i$: since the graph of $v_i$ equals that of $u_i$ in a different coordinate system, the estimate \eqref{e:Lipschitz-transverse} follows from the analogous one for $u_i$ in Proposition \ref{p:coherent}. Likewise, $\|Dw_i\|_{L^2 (L_1)}^2$ is equivalent, up to constants, to
\[
\int_{\mathbf{p}_{\alpha_1}^{-1} (L_1)} |\vec{\mathbf{G}}_{v_i} (p) - \vec{\alpha}_i|^2 \, d\|\mathbf{G}_{v_i}\| (p)\, .
\]
Thus from the corresponding inequality for $u_i$, we obtain the inequality 
\[
\|Dw_i\|_{L^2 (L_1)}^2 \leq C 2^{-m\ell(L)}(\widetilde{\mathbf{E}} (L) + 2^{-2\ell(L)}\mathbf{A}^2)\, .
\]
As for $\|w_i\|_{L^\infty (L_1)}$ it is easy to see that it is controlled by $\|u_i\|_{L^\infty (\lambda L_i)}$ for a choice of $\lambda$ slightly larger than $1$ (cf. the arguments in \cite{DLS_multiple_valued}*{Section 5}). This establishes \eqref{e:local-transverse-1} and so completes the proof.
\end{proof}

\subsection{Non-concentration estimates}

In this section we draw some important conclusions from Simon's estimates in Section \ref{p:estimates}. We will work under the following assumptions throughout this section.

\begin{assumption}\label{a:small-excess-and-gaps}
$T$ and $\Sigma$ are as in Assumption \ref{a:main}. $\mathbf{S}\in \mathscr{C} (Q, 0)\setminus \Pscr(0)$ is such that 
\begin{equation}\label{e:again-small-excess}
\mathbf{A}^2 + \mathbb{E} (T, \mathbf{S}, \Bbf_1) \leq \varepsilon^2 \boldsymbol{\sigma} (\mathbf{S})^2
\end{equation}
and
\begin{equation}\label{e:again-no-gaps}
\Bbf_\varepsilon (y)\cap \{\Theta (T, \cdot) \geq Q\}\neq \emptyset \quad \mbox{for all\;} y\in V (\mathbf{S})\cap \Bbf_1\, ,
\end{equation}
for some $\eps>0$, to be determined.
\end{assumption}

Next, consider recall the family of cubes $\mathcal{G}$, where $L\in\mathcal{G}$ has center point $y_L\in V$, as defined in Section \ref{ss:whitney}.

\begin{definition}\label{d:nails}
Let $T$, $\Sigma$, and $\mathbf{S}$ be as in Assumption \ref{a:small-excess-and-gaps}. For every $L\in \mathcal{G}$ we let $\beta (L)$ be a point $p$ in $\{\Theta (T, \cdot)\geq Q\}$ which minimizes the distance to $y_L$. This point will be called the {\em nail} of $L$.
\end{definition}

Of course, there might be more than one candidate for the nail of $L$, but under Assumption \ref{a:small-excess-and-gaps}, one nail always exists as the set $\{\Theta(T,\cdot)\geq Q\}$ is closed by upper semi-continuity of density. In the rest of the paper we will just assume that some arbitrary choice for each given nail has been made. The main conclusion of this section is then the following.

\begin{proposition}\label{p:no-concentration}
There is a constant $C= C(m,n,\bar n,Q)$ such that, for every $\varrho>0$ there exists $\varepsilon = \varepsilon (Q,m,n,\bar{n},\varrho)>0$ with the following property. Assume $T$, $\Sigma$, and $\mathbf{S}$ are as in Assumption \ref{a:small-excess-and-gaps} with this choice of $\eps$, let $r_*$ denote the radius of Proposition \ref{p:HS-3} and let $r=\frac{r_*}{4}$.  Then
\begin{equation}\label{e:no-concentration}
\int_{\Bbf_r} \frac{\dist^2 (q, \mathbf{S})}{\max \{\varrho, \dist (q, V)\}^{3/2}}\, d\|T\| (q) \leq C (\hat{\mathbf{E}} (T, \mathbf{S}, \Bbf_1) +\mathbf{A}^2)\, .
\end{equation}
 Moreover, if $u_1, \ldots, u_N$ are the coherent approximations of Proposition \ref{p:coherent}, upon suitably modifying them (without relabelling) we can assume that they satisfy the following stronger estimates for each $j\in\{1,\dots,N\}$:
\begin{align}
\int_{(\Bbf_r \cap \alpha_j)\setminus B_\varrho (V)} \frac{|Du_j (z)|^2}{\dist (z, V)^{3/2}}\, dz &\leq
C (\hat{\mathbf{E}} (T, \mathbf{S}, \Bbf_1)+\Abf^2) \label{e:no-concentration-2}\\
\sum_{i:\, 2^{-i-1} \geq \varrho} \sum_{L\in \mathcal{G}_i} \int_{L_j} \frac{|u_j (z)\ominus (\mathbf{p}_{\alpha_j}^\perp (\beta (L)))|^2}{\dist (z, V)^{7/2}}\, dz &\leq C (\hat{\mathbf{E}} (T, \mathbf{S}, \Bbf_1)+\Abf^2)\, .\label{e:no-concentration-3}
\end{align}
Finally, under the additional assumption that $\boldsymbol{\mu} (\mathbf{S})\leq c_0$ for the constant $c_0 = c_0(m,n,\bar{n})$ of Proposition \ref{p:transversal}, the maps $w_1, \ldots, w_N$ and $A_1, \ldots, A_N$ defined therein satisfy the corresponding estimates
\begin{align}
\int_{(\Bbf_r \cap \alpha_1)\setminus B_\varrho (V)} \frac{|Dw_j (z)|^2}{\dist (z, V)^{3/2}}\, dz &\leq
C (\mathbf{A}^2 + \hat{\mathbf{E}} (T, \mathbf{S}, \Bbf_1)) \label{e:no-concentration-4}
\end{align}
\begin{align}
\sum_{2^{-i-1} \geq \varrho} \sum_{L\in \mathcal{G}_i} \int_{L_1} \frac{|w_j (z)\ominus (\mathbf{p}_{\alpha_1}^\perp (\beta (L)) + A_j (\mathbf{p}_{V^\perp\cap \alpha_1} (\beta (L))))|^2}{\dist (z, V)^{7/2}}\, dz &\leq C (\mathbf{A}^2 + \hat{\mathbf{E}} (T, \mathbf{S}, \Bbf_1))\, .\label{e:no-concentration-5}
\end{align}
\end{proposition}

\begin{proof}
    We begin with the first estimate \eqref{e:no-concentration}; we may assume that $\varrho<1$ is smaller than half the radius in Proposition \ref{p:HS-3}, as if not the inequality follows trivially. We begin by estimating the part of the integral on the left-hand side that is over $\Bbf_r\cap B_\varrho(V)$. Cover $V\cap \Bbf_r$ with $C \varrho^{-(m-2)}$ balls $\Bbf_\varrho (y_i)$ of radius $\varrho$ (of course, we can assume $\Bbf_{\varrho}(y_i)\cap V\neq\emptyset$ for each $i$) and observe that $\{\Bbf_{2\varrho} (y_i)\}_i$ covers $B_\varrho (V) \cap \Bbf_r$. If \eqref{e:again-no-gaps} holds with $\eps<\varrho$, then for each $i$ we can find a point $p_i\in \Bbf_{2\varrho}(y_i)$ such that $\Theta (T, p_i) \geq Q$. In particular we can estimate, using Proposition \ref{p:HS-3} with $\kappa = 1/4$, centered at each $p_i$, 
\begin{align}
\varrho^{-3/2} & \int_{\Bbf_r\cap B_\varrho (V)} \dist^2 (q, \mathbf{S})\, d\|T\| (q) \leq \varrho^{-3/2} \sum_i 
\int_{\Bbf_r\cap B_{2\varrho} (y_i)} \dist^2 (q, \mathbf{S})\, d\|T\| (q)\nonumber\\
&\leq C \varrho^{-3/2} \sum_i \int_{\Bbf_r\cap \Bbf_{2\varrho} (y_i)} \dist^2 (q, p_i + \mathbf{S})\, d\|T\| (q)\nonumber\\
& \hspace{10em} + C  \varrho^{-3/2} \sum_i \varrho^m \big(|\mathbf{p}_{\alpha_1}^\perp (p_i)|^2 + \boldsymbol{\mu} (\mathbf{S})^2 |\mathbf{p}_{V^\perp \cap \alpha_1} (p_i)|^2\big)\nonumber\\
&\leq C \varrho^{m+7/4-3/2} \sum_i \int_{\Bbf_r\cap \Bbf_{4\varrho} (p_i)} \frac{\dist^2 (q, p_i+\mathbf{S})}{|q-p_i|^{m+7/4}}\, d\|T\| (q)\nonumber\\
&\hspace{18em} + C \varrho^{-(m-2)} \varrho^{m-3/2} (\hat{\mathbf{E}} (T, \mathbf{S}, \Bbf_1) + \mathbf{A}^2)\nonumber\\
&\leq C \varrho^{-(m-2)} \varrho^{m+1/4} (\hat{\mathbf{E}} (T, \mathbf{S}, \Bbf_1) + \mathbf{A}^2) + C \varrho^{1/2} (\hat{\mathbf{E}} (T, \mathbf{S}, \Bbf_1) + \mathbf{A}^2)\nonumber\\
&\leq C (\hat{\mathbf{E}} (T, \mathbf{S}, \Bbf_1) + \mathbf{A}^2)\, .\label{e:no-concentration-close-to-V}
\end{align}
We next consider the region $\Bbf_r \setminus B_\varrho (V)$ and note that we can we cover it with the regions $R(L)$ for $L\in \mathcal{G}$ such that $2^{-\ell (L) -1} \geq \varrho$. However, we only include in the cover the cubes $L$ for which $R(L)$ have a nonempty intersection with $\Bbf_r$: denote the latter family of cubes by $\mathscr{F}$. For each such cube we denote by $\beta (L)$ its nail as in Definition \ref{d:nails}. Note in particular that, if $\varepsilon < \frac{\varrho}{4}$ in \eqref{e:again-no-gaps}, then
\[
C^{-2} 2^{-\ell (L)} \leq C^{-1} |q-\beta(L)| \leq \dist (q, V) \leq C |q-\beta(L)| \leq C^2 2^{-\ell (L)} \qquad \forall q\in R(L)\, ,
\]
where $C$ is a geometric constant. Again combining with Proposition \ref{p:HS-3} with $\kappa=\frac{1}{4}$, this allows us to estimate 
\begin{align}
\int_{\Bbf_r\setminus B_\varrho (V)} & \frac{\dist^2 (q, \mathbf{S})}{\dist (q, V)^{3/2}}\, d\|T\| (q)\nonumber\\
&\leq  C \sum_{i\leq -1-\log_2(\varrho)} \sum_{L\in \mathcal{G}_i\cap \mathscr{F}} 2^{3i/2} \int_{R(L)} 
\dist^2 (q, \mathbf{S})\, d\|T\| (q)\nonumber\\
&\leq C \sum_{i\leq -1-\log_2(\varrho)} \sum_{L\in \mathcal{G}_i\cap \mathscr{F}} 2^{3i/2} \int_{R(L)} \dist^2 (q, \beta (L) + \mathbf{S})\, d\|T\| (q)\nonumber\\
&\quad\quad + C \sum_{i\leq -1-\log_2(\varrho)} \sum_{L\in \mathcal{G}_i\cap \mathscr{F}} 2^{3i/2} 2^{-m i} \big(|\mathbf{p}_{\alpha_1}^\perp (\beta (L))|^2 + \boldsymbol{\mu} (\mathbf{S})^2 |\mathbf{p}_{V^\perp\cap \alpha_1} (\beta (L))|^2\big)\nonumber\\
&\leq C \sum_{i\leq -1-\log_2(\varrho)} \sum_{L\in \mathcal{G}_i\cap \mathscr{F}} 2^{3i/2-mi-7i/4} \int_{R(L)} \frac{\dist^2 (q, \beta (L) + \mathbf{S})}{|q-\beta (L)|^{m+7/4}}\, d\|T\| (q)\nonumber\\
&\quad\quad + C (\hat{\mathbf{E}} (T, \mathbf{S}, \Bbf_1) + \mathbf{A}^2) \sum_{i\leq -1-\log_2(\varrho)} 2^{3i/2-mi} \# \mathcal{G}_i\nonumber\\
&\leq C  (\hat{\mathbf{E}} (T, \mathbf{S}, \Bbf_1) + \mathbf{A}^2) \sum_{i\leq -1-\log_2(\varrho)}
( 2^{-i/4-mi}+2^{3i/2-m i}) \#\mathcal{G}_i\, .\nonumber
\end{align}
Since the cardinality $\# \mathcal{G}_i$ of $\mathcal{G}_i$ is bounded by $C 2^{(m-2) i}$, it follows that 
\begin{equation} \label{e:no-concentration-far-from-V}
\int_{\Bbf_r\setminus B_\varrho (V)} \frac{\dist^2 (q, \mathbf{S})}{\dist (q, V)^{3/2}}\, d\|T\| (q)
\leq C (\hat{\mathbf{E}} (T, \mathbf{S}, \Bbf_1) + \mathbf{A}^2)\, . 
\end{equation}
Clearly \eqref{e:no-concentration-close-to-V} and \eqref{e:no-concentration-far-from-V} imply \eqref{e:no-concentration}.

\medskip

We next use the same family $\mathscr{F}$ of cubes to prove the remaining estimates. Fix any cube $L\in \Fscr$ and recall Proposition \ref{p:refined}, but rather than approximating over $\mathbf{S}$, approximate the current over the shifted cone $\mathbf{S}+ \beta (L)$. We can indeed do this by Proposition \ref{p:HS-3}, which moreover yields
\begin{equation}\label{e:shifted-excess}
\hat{\mathbf{E}} (T, \mathbf{S}+\beta (L), \Bbf_{C 2^{-\ell (L)}} (\beta (L)))\leq C 2^{\ell (L)/4} (\hat{\mathbf{E}} (T, \mathbf{S}, \Bbf_1) + \mathbf{A}^2)\, .
\end{equation}
Note that we can apply \eqref{e:HS-3} provided $C2^{-\ell(L)} \leq r$, where $r=r(Q,m,n,\bar{n})$ is the radius from Proposition \ref{p:HS-3}, i.e. provided $\ell(L)$ is larger than a constant depending on $Q,m,n,\bar{n}$. For the cubes $L$ with $\ell(L)$ smaller than this constant, the inequality above follows more easily as we have an upper bound on $\ell(L)$).

We then gain an approximation $\tilde{u}_{L,i} : \mathbf{p}_{V^\perp\cap \alpha_i} (\beta (L))+ \lambda L_i \to \mathcal{A}_{Q_i} (\alpha_i^\perp)$ which satisfies the estimate 
\begin{align*}
\|D\tilde{u}_{L,i}\|^2_{L^2} &\leq C 2^{-m \ell (L)} (\hat{\mathbf{E}} (T, \mathbf{S}+\beta (L), \Bbf_{C 2^{-\ell (L)}} (\beta (L)))+2^{-2\ell (L)}\mathbf{A}^2) \\
&\leq C 2^{-m \ell (L) +\ell (L)/4} (\hat{\mathbf{E}} (T, \mathbf{S}, \Bbf_1) + \mathbf{A}^2)\, .
\end{align*}

Similarly, Proposition \ref{p:refined}, \eqref{e:HS-4} and \eqref{e:shifted-excess} also give the estimate
\[
\|\tilde{u}_{L,i} \ominus (\mathbf{p}_{\alpha_i}^\perp (\beta (L)))\|_{L^\infty} \leq C 2^{- 7\ell (L)/4} (\hat{\mathbf{E}} (T, \mathbf{S}, \Bbf_1) + \mathbf{A}^2)
\]
We can now consider the map $\hat{u}_{L,i}\coloneqq \tilde{u}_{L,i} \ominus (\mathbf{p}_{\alpha_i}^\perp (\beta (L)))$ over the domain $\mathbf{p}_{V^\perp\cap \alpha_i} (\beta (L)) + \lambda L_i$, which we can assume contains $L_i$, provided that $\eps$ is sufficiently small. 
Observe that we can use this new map in the algorithm within Proposition \ref{p:coherent} in place of $u_{L,i}$ to construct a coherent approximation from $\tilde{u}_{L,i}$, since this new map satisfies the same estimates as above. We can in turn use this new coherent approximation to construct a coherent transversal approximation as in  Proposition \ref{p:transversal}.  In particular, if we keep denoting these improved approximations by $u_i$ and $w_i$ without relabeling, they will satisfy the estimates
\begin{align}
\int_{L_i} \frac{|Du_i|^2}{\dist (z,V)^{3/2}}\, dz &\leq C 2^{-m \ell (L) + 7\ell (L)/4} (\hat{\mathbf{E}} (T, \mathbf{S}, \Bbf_1) + \mathbf{A}^2)\\
\int_{L_1} \frac{|Dw_i|^2}{\dist (z,V)^{3/2}}\, dz &\leq C 2^{-m \ell (L) + 7\ell (L)/4} (\hat{\mathbf{E}} (T, \mathbf{S}, \Bbf_1) + \mathbf{A}^2)\\
\int_{L_i} \frac{|u_i\ominus (\mathbf{p}_{\alpha_i}^\perp (\beta (L)))|^2}{\dist (z,V)^{7/2}}\, dz &\leq C 2^{-m \ell (L) + 7\ell (L)/4} (\hat{\mathbf{E}} (T, \mathbf{S}, \Bbf_1) + \mathbf{A}^2)
\end{align}
\begin{align}
\int_{L_1} \frac{|w_i\ominus (\mathbf{p}_{\alpha_1}^\perp (\beta (L)) + A_i (\mathbf{p}_{V^\perp\cap \alpha_1} (\beta (L))))|^2}{\dist (z,V)^{7/2}}\, dz &\leq C 2^{-m \ell (L) + 7\ell (L)/4} (\hat{\mathbf{E}} (T, \mathbf{S}, \Bbf_1) + \mathbf{A}^2)\, .
\end{align}
We now sum the latter estimates over $L\in \mathscr{F}\cap \mathcal{G}_i$ for $i\leq -\log_2(\varrho)-1$. Considering that the cardinality of $\mathcal{G}_i$ is $C2^{(m-2)i}$ we get 
\eqref{e:no-concentration-2}--\eqref{e:no-concentration-5}.
\end{proof}

\subsection{Blow-ups}
We are now ready to define blow-up sequences for the contradiction argument in the two cases of Proposition \ref{p:decay-collapsed} and Proposition \ref{p:decay-noncollapsed}. In both cases we have a sequence of currents $T_k$, cones $\mathbf{S}_k$, and manifolds $\Sigma_k$ for which the assumption of their respective proposition fails, with a sequence of thresholds $\eps_c = 1/k$ in the case of Proposition \ref{p:decay-collapsed} and with $\eps_{nc} = 1/k$ in the case of Proposition \ref{p:decay-noncollapsed} (we stress that $\eps^\star_c$ is a fixed number in the case of Proposition \ref{p:decay-noncollapsed} and currently unrelated to $\eps_c$!). In particular, we know that (cf. Remark \ref{r:A-excess})
\begin{equation}\label{e:blow-up-1}
\lim_{k\uparrow \infty} \Big(\frac{\mathbf{A}_k^2}{\mathbb{E} (T_k, \mathbf{S}_k, \Bbf_1)} + \frac{\mathbb{E} (T_k, \mathbf{S}_k, \Bbf_1)}{\boldsymbol{\sigma} (\mathbf{S}_k)^2}\Big)=0\, 
\end{equation}
while, denoting by $\beta^k (L)$ the nails of Definition \ref{d:nails} for $T_k$, we have 
\begin{equation}\label{e:nailing-it}
\lim_{k\uparrow\infty} |\beta^k (L) - y_L| = 0 \qquad \forall L \in \mathcal{G}\, .
\end{equation}
Moreover, upon extracting a subsequence, we can assume that $\mathbf{S}_k$ consists of a fixed number $N(k)=N \in \{2, \ldots, Q\}$ of distinct planes, which we denote by $\alpha^k_1, \ldots, \alpha^k_N$. Applying suitable rotations, we can assume that $T_0 \Sigma_k$, $V (\mathbf{S}_k)$, and $\alpha^k_1$ are independent of $k$. In particular, we will often use $V$ in place of $V (\mathbf{S}_k)$ and $\alpha_1$ in place of $\alpha^k_1$. 

The distinction between the case of Proposition \ref{p:decay-collapsed} and Proposition \ref{p:decay-noncollapsed} is that
\begin{itemize}
\item[(C)] In the collapsed case (Proposition \ref{p:decay-collapsed}) we have $\boldsymbol{\mu} (\mathbf{S}_k)\to 0$;
\item[(NC)] In the non-collapsed case (Proposition \ref{p:decay-noncollapsed}) we have $\liminf_k \boldsymbol{\mu} (\mathbf{S}_k)\geq \varepsilon^\star_{c}>0$, for some $\eps^\star_c > 0$.
\end{itemize}
In particular, in the non-collapsed case we can assume in addition that: 
\begin{itemize}
\item[(NC1)] Up to relabelling, $\dist (\alpha^k_1 \cap \Bbf_1, \alpha^k_N\cap \Bbf_1) = \boldsymbol{\mu} (\mathbf{S}_k)$;
\item[(NC2)] Each sequence $\alpha^k_i$ converges (in Hausdorff distance) to a plane $\alpha_i$, and $\alpha_N$ is necessarily distinct from $\alpha_1$ (but some of the other planes $\alpha_i$ could coincide).
\end{itemize}
In the collapsed case we consider the transversal coherent approximations $w^k_i$, while in the non-collapsed case we consider the orthogonal coherent approximations $u^k_i$. In both cases we assume that their domains of definition ${\rm Dom}\, (w^k_i)$ and ${\rm Dom}\, (u^k_i)$ contain, respectively, $(\Bbf_r \cap \alpha_1)\setminus B_{1/k} (V)$ and $(\Bbf_r \cap \alpha^k_i) \setminus B_{1/k} (V)$. We further consider the rescaled functions
\begin{align}
\bar{u}^k_i &:= \frac{u^k_i}{\sqrt{\mathbb{E} (T_k, \mathbf{S}_k, \Bbf_1)}}\\
\bar{w}^k_i &:= \frac{w^k_i}{\sqrt{\mathbb{E} (T_k, \mathbf{S}_k, \Bbf_1)}}\, .
\end{align}
Observe that we have uniform estimates for the $W^{1,2}$ norms of $\bar{u}_i^k$ and $\bar{w}^i_k$ (from Proposition \ref{p:no-concentration}, which are even stronger as they have a weight). The maps $w^k_i$ are also defined over the same plane $\alpha_1$. For the maps $u^k_i$ there is the annoying technicality that they are not. In the latter case we refer to Remark \ref{r:blow-up} and assume they are, up to composition with a small rotation, defined on the same planes $\alpha_i$.

In particular, up to extraction of a subsequence we can assume that the $\bar{u}^k_i$ and $\bar{w}^k_i$ converge strongly in $W^{1,2}$ locally away from $V$ and strongly in $L^2$ on the entirety of $B_r$ (due to the non-concentration estimates from Proposition \ref{p:no-concentration}) to $W^{1,2}$ maps $\bar{u}_i$ and $\bar{w}_i$ defined on $(\Bbf_r \cap \alpha_i)\setminus V$ and $(\Bbf_r\cap \alpha_1)\setminus V$ respectively (and taking values in $\Acal_Q(\alpha_i^\perp)$ and $\Acal_Q(\alpha_1^\perp)$ respectively). However, since $V$ is an $(m-2)$-dimensional subspace, which has zero $2$-capacity, the maps $\bar{u}_i$ and $\bar{w}_i$ belong indeed to $W^{1,2} (\alpha_i\cap \Bbf_r)$ and $W^{1,2} (\alpha_1\cap \Bbf_r)$ respectively. In the collapsed case we also introduce the linear maps $A^k_i:\alpha_1\to \alpha_1^\perp$ whose graphs describe the planes $\alpha^k_i$, and can consider their rescalings
\begin{equation}
\bar{A}^k_i := \frac{A^k_i}{\boldsymbol{\mu} (\mathbf{S}_k)}\, .    
\end{equation}
Moreover, up to extraction of a further subsequence, we assume that $\bar{A}^k_i$ converges to some linear map $\bar{A}_i$. 

We are now ready to state our main proposition concerning the properties of the limiting maps $\bar{u}_i$ and $\bar{w}_i$.

\begin{proposition}\label{p:blowup2}
Let $\bar{u}_i^k$, $\bar{w}_i^k$ and their respective limits $\bar{u}_i$, $\bar{w}_i$ be as described above. The following holds in the non-collapsed case:
\begin{itemize}
    \item[(a)] Each $\bar{u}_i$ is Dir-minimizing on $\Bbf_r\cap \alpha_i$, for $r$ as in Proposition \ref{p:no-concentration}.
    \item[(b)] $\bar{u}_i = Q_i \llbracket \boldsymbol{\eta}\circ \bar{u}_i\rrbracket$ on $V$.
    \item[(c)]$\bar{u}_i(0) = Q_i\llbracket\boldsymbol{\eta}\circ \bar{u}_i(0)\rrbracket =Q_i\llbracket 0 \rrbracket$ and the frequency obeys $I_{0,\bar{u}_i} (0)\geq 1$. Moreover, we have 
    $I_{y,\bar{u}_i \ominus \boldsymbol{\eta}\circ \bar{u}_i} (0)\geq 1$ for all $y\in V\cap \Bbf_r$.
    \item[(d)] There is a smooth function $\beta: V \to V^\perp$ such that $\boldsymbol{\eta}\circ \bar{u}_i = \mathbf{p}_{\alpha_i}^\perp (\beta)$ on $V$ for every $i$. (Note that with a slight abuse of notation we are using the same letter identifying the nails of Definition \ref{d:nails}. In fact this function $\beta$ is the trace of the limit of suitable normalizations of the nails.)
\end{itemize}
The following holds in the collapsed case:
\begin{itemize}
    \item[(e)] Each $\bar{w}_i$ is Dir-minimizing on $\Bbf_r\cap \alpha_1$.
    \item[(f)] $\bar{w}_i = Q_i \llbracket \boldsymbol{\eta}\circ \bar{w}_i \rrbracket$ on $V$.
    \item[(g)] $\bar{w}_i(0)= Q_i\llbracket\boldsymbol{\eta}\circ \bar{u}_i(0)\rrbracket =Q_i\llbracket 0 \rrbracket$ and the frequency  $I_{0,\bar{w}_i} (0)\geq 1$.
    Moreover, we have $I_{y,\bar{w}_i \ominus \boldsymbol{\eta}\circ \bar{w}_i} (0) \geq 1$ for all $y\in V\cap \Bbf_r$;
    \item[(h)] There are smooth function $\beta^\perp: V \to \alpha_1^\perp$ and $\beta_\parallel: V \to V^\perp\cap \alpha_1$ such that $\boldsymbol{\eta}\circ \bar{w}_i = \beta^\perp + \bar{A}_i (\beta_\parallel)$ on $V$ for all $i$. 
\end{itemize}
The functions $\bar{u}_i$ and $\beta$ also enjoy the estimates \begin{equation}\label{e:est-u}
\|\bar{u}_i\|_{W^{1,2}}+\|\beta\|_{C^2} \leq C
\end{equation}
for some constant $C$ which depends only on $Q, m,n, \bar n$ and $\varepsilon^\star_{c}$. The functions $\bar{w}_i$, $\beta^\perp$, and $\beta_\parallel$ enjoy the estimates \begin{equation}\label{e:est-w}
\|\bar{w}_i\|_{W^{1,2}}+\|\beta^\perp\|_{C^2} + \|\beta_\parallel\|_{C^2} \leq C
\end{equation}
for some constant $C$ which depends only on $Q,m,n,\bar n$.
\end{proposition}

\begin{proof}
We start with (a). Here we first appeal to Proposition \ref{p:first-blow-up} to prove the Dir-minimality of $\bar{u}_i$ away from the spine $V$. Observe that, if, along the sequence $\bar{u}_i^k$, the double-sided excess $\mathbb{E}(T_k,\Sbf_k,\Bbf_1)$ stays comparable to $\hat{\Ebf}(T_k,\Sbf_k,\Bbf_1)$, this claim follows from Proposition \ref{p:first-blow-up}(iii) because $\mathbf{A}^2$ is infinitesimal compared to $\mathbb{E}(T_k,\Sbf_k,\Bbf_1)$. The alternative left is that (up to subsequence) $\hat{\Ebf}(T_k,\Sbf_k,\Bbf_1)/\mathbb{E}(T_k,\Sbf_k,\Bbf_1)$ is infinitesimal. In this case, by Proposition \ref{p:first-blow-up}(ii), $D \bar{u}_i$ vanishes. 
Having established the Dir-minimality away from the spine, Proposition \ref{p:remove-spine-2} allows us to conclude that it is Dir-minimizing on any $\Omega\subset\subset \Bbf_r$; this proves (a). 

Now let us show (b). We next observe that, by \eqref{e:HS-4} (in which one can clearly replace $\alpha_1$ by any other plane of $\Sbf$), $|\mathbf{p}_{\alpha^k_i}^\perp (\beta^k (L))| \leq C \mathbb{E} (T_k, \mathbf{S}_k, \Bbf_1)^{1/2}$ and, because of the lower bound on $\mathbf{\mu} (\mathbf{S}_k)$,
$|\mathbf{p}_{V^\perp \cap \alpha^k_i} (\beta^k (L))| \leq C \mathbb{E} (T_k, \mathbf{S}_k, \Bbf_1)^{1/2}$, for each $i=1,\dotsc,N$. In particular we easily conclude that 
\[
|\mathbf{p}_{V}^\perp (\beta^k (L))|\leq C \mathbb{E} (T_k, \mathbf{S}_k, \Bbf_1)^{1/2}
\]
for a constant $C= C(Q,m,n,\bar{n},\varepsilon^\star_{c})$. Thus we can assume that for every $L$, we can find $\bar{\beta} (L)$ such that 
\[
\lim_{k\to \infty} \frac{\mathbf{p}_{V}^\perp (\beta^k (L))}{\mathbb{E} (T_k, \mathbf{S}_k, \Bbf_1)^{1/2}} = \bar\beta (L)\, .
\]
It is notationally convenient to think of $\bar\beta$ as a piecewise constant function which is defined over $\Bbf_r$ as being identically equal to $\bar\beta (L)$ on the set $R(L)$ (note that this is well-defined away from $V$ and the overlaps of the regions $R(L)$, but this set has measure 0). In that way we can divide by $\mathbb{E}(T_k,\Sbf_k,\Bbf_1)$ in \eqref{e:no-concentration-3} and pass to the limit to conclude that
\begin{equation}\label{e:beta-bar}
\int_{\Bbf_r\cap \alpha_i} \frac{|\bar{u}_i (z) \ominus \mathbf{p}_{\alpha_i}^\perp (\bar\beta) (z)|^2}{\dist (z, V)^{7/2}}\, dz < \infty\, .
\end{equation}
But in particular, since $|\bar{u}_i (z) \ominus (\boldsymbol{\eta}\circ \bar{u}_i) (z)|\leq |\bar{u}_i (z) \ominus \mathbf{p}_{\alpha_i}^\perp (\bar\beta) (z)|$ (this is just the trivial fact that the barycenter minimizes this quantity), we get 
\begin{equation}\label{e:average-u}
\int_{\Bbf_r\cap \alpha_i} \frac{|\bar{u}_i (z) \ominus \boldsymbol{\eta} \circ \bar{u}_i (z)|^2}{\dist (z, V)^{7/2}}\, dz<\infty\, .
\end{equation}
Given that the numerator is continuous, it must be identically $0$ for $z\in V\cap \Bbf_r$, otherwise the integral would diverge (recall that $V$ has codimension $2$ and thus the integral of $\dist (z,V)^{-s}$ diverges for every $s\geq 2$). In particular we have shown (b). 

As for (c), recall first that we have the estimate of Corollary \ref{c:HS-patch}. Once again choosing $\kappa = \frac{1}{4}$, this is easily seen to imply the estimate
\[
\int_{\Bbf_r\cap \alpha_i} \frac{|\bar u_i (z)|^2}{|z|^{m+7/4}}\, dz < \infty\, .
\]
Arguing as above, this clearly implies that $|\bar u_i (0)|=0$ and hence the first claim of point (c). We next argue for the second part of the claim; here we resort to the full version of Corollary \ref{c:HS-patch} for any $\kappa\in (0,m+2)$. Observe that it suffices to prove the latter claim $I_{y,\bar{u}_i \ominus \boldsymbol{\eta}\circ \bar{u}_i} (0)\geq 1$ for all $y\in V\cap \Bbf_r$. Indeed, the fact that $I_{0,\bar{u}_i}(0)\geq 1$ follows immediately from this at $y=0$, combined with the property $\boldsymbol{\eta}\circ \bar{u}_i(0) = 0$ that we have just proved. Fix a point $y\in V \cap \Bbf_r$ and for each $k\in \N$, consider a point $q_k\in \Bbf_{\eps_k} (y)$ with $\Theta (T_k, q_k)\geq Q$ and fix any $\rho \in (0,r)$. It follows from Proposition \ref{p:HS-3} that, if $p\in {\rm gr}\, (u_i^k)\cap \spt (T_k)$ and $|\mathbf{p}_{\alpha_i} (p)| \geq \rho$, then $\dist (p, q_k)+\Sbf_k)=\dist (p,q_k+\alpha^k_i)$ for all $k$ sufficiently large (depending on $\rho$ and $\kappa$) and in particular, if $K_{k,i}\subset\alpha_i^k$ denotes the set over which the graph of $u_i^k$ coincides with the current $T$, then Proposition \ref{p:HS-3} 
\[
\int_{(\Bbf_r(y)\setminus \Bbf_\rho (y))\cap K_{k,i}} \frac{|u^k_i (z) \ominus \mathbf{p}_{\alpha^k_i}^\perp (q_k)|^2}{|z-y|^{m+2-\kappa}}\, dz
\leq C(\kappa) \Ebb (T_k, \Sbf_k, \Bbf_1)\, ,
\]
for every positive $\kappa$, and for $k$ sufficiently large, depending on $\rho$, $\kappa$.

Again recall that $|u^k_i\ominus \boldsymbol{\eta} \circ u^k_i|\leq |u^k_i \ominus \mathbf{p}_{\alpha^k_i}^\perp (q_k)|$, and so in particular we conclude that
\[
\int_{(\Bbf_r (y)\setminus \Bbf_\rho(y))\cap K_{k,i}} \frac{|u^k_i (z) \ominus \boldsymbol{\eta}\circ u^k_i (z)|^2}{|z-y|^{m+2-\kappa}}\, dz
\leq C(\kappa) \Ebb (T_k, \Sbf_k, \Bbf_1)\, .
\]
Dividing by $\Ebb (T_k, \Sbf_k, \Bbf_1)$, taking the limiting in $k$ and noting that the measure of $(\Bbf_r\setminus \Bbf_\rho )\cap (\alpha^k_i \setminus K_{k,i})$ converges to $0$ (cf. \eqref{e:coherent-3}), we arrive at 
\[
\int_{(\Bbf_r (y)\setminus \Bbf_\rho (y)) \cap \alpha_i} \frac{|\bar u_i (z)\ominus \boldsymbol{\eta}\circ \bar u_i (z)|^2}{|z-y|^{m+2-\kappa}}\, dz \leq C (\kappa)\, .
\]
On the other hand letting $\rho\downarrow 0$ we then get 
\begin{equation}\label{e:finite-for-kappa}
\int_{\Bbf_r (y)\cap \alpha_i } \frac{|\bar u_i (z)\ominus \boldsymbol{\eta}\circ \bar u_i (z)|^2}{|z-y|^{m+2-\kappa}}\, dz \leq C (\kappa)
\end{equation}
Now let $\gamma := I_{y,\bar{u}_i\ominus \boldsymbol{\eta}\circ \bar u_i} (0)$. It follows from \cite{DLS_MAMS}*{Corollary 3.18} that for every $\gamma'> \gamma$ there is a positive constant $C$ (depending on $\gamma'$ and $\bar{u}_i$) such that 
\[
\int_{\partial \Bbf_\sigma (y) \cap \alpha_1} |\bar{u}_i\ominus \boldsymbol{\eta}\circ \bar{u}_i|^2 \geq C \sigma^{m-1+2\gamma'}\, ,
\]
for every $\sigma \in (0,r)$. Combined with \eqref{e:finite-for-kappa}, we conclude that
\[
\int_0^r \sigma^{2\gamma'+\kappa -3}\, d\sigma < \infty
\]
which implies that $2\gamma' + \kappa > 2$. Letting $\kappa\downarrow 0$ and $\gamma'\downarrow \gamma =I_{y,\bar{u}_i\ominus \boldsymbol{\eta}\circ \bar u_i} (0)$ we then conclude that $I_{y,\bar{u}_i\ominus \boldsymbol{\eta}\circ \bar u_i} (0) \geq 1$. Note that one could alternatively conclude that $I_{y,\bar{u}_i\ominus \boldsymbol{\eta}\circ \bar u_i} (0) \geq 1$ via the Hardt-Simon inequality, as done in \cite{Simon_cylindrical}.

We next come to point (d). Observe that if $\zeta$ is in $V^\perp$ then
\begin{equation}\label{e:yet-again-lin-alg}
|\zeta| \leq C (|\mathbf{p}_{\alpha_1}^\perp (\zeta)| + |\mathbf{p}_{\alpha_N}^\perp (\zeta)|)\, ,
\end{equation}
(where the constant $C$ in particular depends on $\varepsilon^\star_c$). 
We next observe that we have (from \eqref{e:beta-bar}, \eqref{e:average-u})
\begin{equation}\label{e:beta-bar-u-average}
\int_{\Bbf_r\cap \alpha_i} \frac{|\mathbf{p}_{\alpha_i}^\perp (\bar\beta) - \boldsymbol{\eta}\circ \bar{u}_i|^2}{\dist (z,V)^{7/2}}\, dz < \infty\, .
\end{equation}
For each $\rho \in (0,r)$ consider the functions defined on $V$ which result averaging $\bar\beta$ and $\boldsymbol{\eta}\circ \bar{u}_i$ respectively, over $2$-dimensional disks in $\alpha_i$ (for any $i\in \{1,\dotsc,N\}$) of radius $\rho$ perpendicular to $V$ and centered at $y \in V$:
\begin{align*}\label{e:averages}
\bar\beta_\rho (y) &= \frac{1}{\pi \rho^2} \int_{\Bbf_\rho (y) \cap (y+V^\perp) \cap \alpha_i} \bar\beta\, , \\
g_\rho (y) &= \frac{1}{\pi \rho^2} \int_{\Bbf_\rho (y) \cap (y+V^\perp) \cap \alpha_i} \boldsymbol{\eta} \circ \bar{u}_i\, .
\end{align*}
We clearly have, from \eqref{e:beta-bar-u-average}, 
\[
\int_{V\cap \Bbf_{\sqrt{r^2-\rho^2}}} |\mathbf{p}_{\alpha_i}^\perp (\bar\beta_\rho) - g_\rho|^2 \leq C \rho^{3/2}\, .
\]
In particular $\mathbf{p}_{\alpha_i}^\perp (\bar\beta_\rho)$ converges, as $\rho\downarrow 0$, in $L^2 (V \cap\Bbf_r)$ to $\lim_{\rho\downarrow 0}g_\rho = \boldsymbol{\eta} \circ \bar{u}_i$. But then for any given $\rho\in(0,r)$,
\[
\int_{V\cap \Bbf_{\sqrt{r^2-\rho^2}}} |\mathbf{p}_{\alpha_i}^\perp (\bar\beta_\rho) - \mathbf{p}_{\alpha_i}^\perp (\bar\beta_{\rho'})|^2 \leq C \rho^{3/2} + \int_{V\cap \Bbf_{\sqrt{r^2-\rho^2}}} |g_\rho - g_{\rho'}|^2\qquad \forall \rho'<\rho
\]
Hence, recalling that $\bar\beta$ takes values in $V^\perp$, applying \eqref{e:yet-again-lin-alg} and the above inequality with $i=1,N$, we conclude
\[
\int_{V\cap \Bbf_{\sqrt{r^2-\rho^2}}} |\bar\beta_\rho - \bar\beta_{\rho'}|^2 \leq C \rho^{3/2} + \int_{V\cap \Bbf_{\sqrt{r^2-\rho^2}}} |g_\rho - g_{\rho'}|^2 \qquad \forall \rho'<\rho\, .
\]
In particular $\{\bar\beta_\rho\}_{\rho>0}$ converges strongly in $L^2 (V\cap \Bbf_r)$ as $\rho\downarrow 0$ to some function $\beta\in L^2(V\cap \Bbf_r)$ and clearly $\mathbf{p}_{\alpha_i}^\perp (\beta) = \boldsymbol{\eta}\circ \bar{u}_i$ for each $i=1,\dotsc,N$. The smoothness of $\beta$ follows easily from the smoothness of the $\boldsymbol{\eta}\circ\bar{u}_i$, which in turn is obvious from the fact that they are the traces on $V$ of classical harmonic functions. This proves (d).

Coming to the estimate \eqref{e:est-u} in the non-collapsed case, by construction (namely, the $L^2$ non-concentration estimates \eqref{e:no-concentration-2}, \eqref{e:no-concentration-3}), $\|\bar u_i\|_{W^{1,2}} \leq C$ for some constant $C = C(Q,m,n,\bar{n})$, which in fact does not depend on $\eps_c^\star$. Next, as $\mathbf{p}_{\alpha_i}^\perp(\beta) = \boldsymbol{\eta}\circ\bar{u}_i$ and the cone $\Sbf$ is well-separated (cf. \eqref{e:yet-again-lin-alg}), $\beta$ is determined as a linear combination of the $\mathbf{p}_{\alpha_i}^\perp (\beta)$ for $i=1,N$, which in turn are the traces of $\boldsymbol{\eta}\circ \bar{u}_i$. The latter are harmonic functions and they enjoy the same $W^{1,2}$ bound on $\bar{u}_i$. Hence, the $C^2$ estimate follows from the classical theory of harmonic functions.

\medskip
The conclusion (e) is more subtle than (a). As above, we wish to argue that $\bar{w}_i$ is Dir-minimizing on every $\Omega\subset\subset \Bbf_r\cap \alpha_1\setminus V$ and invoke Proposition \ref{p:remove-spine-2} to argue that therefore it is Dir-minimizing on $\Bbf_r\cap \alpha_1$. First of all we assume, without loss of generality, that the double-sided excess $\mathbb{E} (T_k, \mathbf{S}_k, \Bbf_1)$ and the one sided excess $\hat{\Ebf} (T_k, \mathbf{S}_k, \Bbf_1)$ are comparable, since the alternative would be that the latter is infinitesimal compared to the former, and in that case $\bar{w}_i$ would vanish. Then, we notice that in order to argue for the minimality in $\Omega\subset\subset \Bbf_r\cap \alpha_1\setminus V$ we cannot invoke directly the argument of \cite{DLS14Lp}*{Theorem 2.6}. However, since the mismatch in mass between the current $T$ and the graphs of the multi-valued maps $\sum_i w^k_i\oplus A^k_i$ is controlled by $E_k^{1+\gamma} := \mathbb{E} (T_k, \mathbf{S}_k, \Bbf_1)^{1+\gamma}$, the area-minimizing property of $T$ and the arguments in \cite{DLS14Lp}*{Section 5} shows the following minimality property for the map $w^k_i \oplus A^k_i$:
\begin{itemize}
\item[(Min)] If $\tilde{w}^k_i$ is a Lipschitz map which coincides with $w^k_i$ outside of $\Omega$ and its Lipschitz constant is controlled by $1$, then
\begin{equation}\label{e:almost-minimality}
\|\mathbf{G}_{\tilde{w}^k_i\oplus A^k_i}\| (\Omega \times \alpha_1^\perp) 
\geq \|\mathbf{G}_{w^k_i\oplus A^k_i}\| (\Omega \times \alpha_1^\perp) - C E_k^{1+\gamma}\, .
\end{equation}
\end{itemize}
Our aim is to show that, if $\bar{w}_i$ is not Dir-minimizing, then \eqref{e:almost-minimality} is violated for $k$ sufficiently large. Assuming it is not Dir-minimizing, the Lipschitz truncation and ``cut-and-paste'' arguments of \cite{DLS14Lp}*{Theorem 2.6} show the existence of a sequence $\tilde{w}^k_i$ of multi-valued maps such that 
\begin{itemize}
    \item[(i)] $\tilde{w}^k_i = w^k_i$ outside $\Omega$;
    \item[(ii)] $\Lip (\tilde{w}^k_i) \leq C E_k^\gamma$;
    \item[(iii)] The Dirichlet energy of $\tilde{w}^k_i$ has the following gain:
    \begin{equation}\label{e:gain}
    \int_\Omega |D\tilde{w}^k_i|^2 \leq \int_\Omega |Dw^k_i|^2 - \vartheta E_k\, ,
    \end{equation}
    for some positive $\vartheta>0$ independent of $k$.
\end{itemize}
We now wish to show that (i)-(iii) contradict \eqref{e:almost-minimality}. We let $\mathcal{A} (Du)$ be the area integrand for the graph of a single-valued function $u$. More precisely, if we denote by $\mathcal{M} (B)$ the set of all $k\times k$ minors of the matrix $B$ with $k\geq 2$, then\footnote{Note that, when $u$ is a scalar function, $Du$ only has $1\times 1$ minors, namely $\mathcal{M} (Du)$ is empty, and we would use the convention that the formula reduces to $\mathcal{A} (B) = \sqrt{1+|B|^2}$. However in our case $u$ is necessarily vector-valued.}
\[
\mathcal{A} (B) = \sqrt{1+ |B|^2 + \sum_{M\in \mathcal{M} (B)} \det (M)^2}\, .
\]
We next use the notation $w^k_i = \sum_j \llbracket w^k_{i,j}\rrbracket$ and $\tilde{w}^k_i = \sum_j \llbracket \tilde{w}^k_{i,j}\rrbracket$. Using the area formula in \cite{DLS_multiple_valued}*{Corollary 1.11} we can compute
\[
\|\mathbf{G}_{\tilde{w}^k_i\oplus A^k_i}\| (\Omega \times \alpha_1^\perp) =
\int_\Omega \sum_j \mathcal{A} (A^k_i + D\tilde{w}^k_{i,j})
\]
and
\[
\|\mathbf{G}_{w^k_i\oplus A^k_i}\| (\Omega \times \alpha_1^\perp) =
\int_\Omega \sum_j \mathcal{A} (A^k_i + Dw^k_{i,j})\, .
\]
In particular, defining 
\[
\Delta := \|\mathbf{G}_{\tilde{w}^k_i\oplus A^k_i}\| (\Omega \times \alpha_1^\perp)-\|\mathbf{G}_{w^k_i\oplus A^k_i}\| (\Omega \times \alpha_1^\perp)\, ,
\]
we arrive at the expression
\begin{equation}\label{e:mismatch}
\Delta = \int_\Omega \left(\sum_j \mathcal{A} (A^k_i + D\tilde{w}^k_{i,j}) - \sum_j \mathcal{A} (A^k_i + Dw^k_{i,j})\right)\, .
\end{equation}
We now wish to make a Taylor expansion of the function $\mathcal{A}$ at the constant $A^k_i$. In particular we can write
\[
\mathcal{A} (A^k_i + B) = \mathcal{A} (A^k_i) + \mathcal{L}_{k,i} (B) + \mathcal{Q}_{k,i} (B) + O (|B|^3)
\]
where $\mathcal{L}_{k,i}$ are appropriate linear functions and $\mathcal{Q}_{k,i}$ are appropriate quadratic forms. 
However note that at $A^k_i=0$ the quadratic form would be $\frac{|B|^2}{2}$, and given that $|A^k_i|\leq \boldsymbol{\mu} (\mathbf{S}_k)$ we can further write 
\[
\mathcal{A} (A^k_i + B) = \mathcal{A} (A^k_i) + \mathcal{L}^k_i (B) + \frac{|B|^2}{2} + O ((\boldsymbol{\mu} (\mathbf{S}_k) + |B|)|B|^2)\, .
\]
Now, inserting the latter expansion in our expression \eqref{e:mismatch} and using that $\|D\tilde{w}^k_i\|_\infty + \|D w^k_i\|_\infty \leq CE_k^\gamma$ and $\|D\tilde{w}^k_i\|_{L^2} + \|Dw^k_i\|_{L^2}\leq CE_k$ we arrive at
\begin{align*}
\Delta &= \int_\Omega \left(\sum_j \left(\mathcal{L}^k_i (D\tilde{w}^k_{i,j}) + \frac{|D\tilde{w}^k_{i,j}|^2}{2}\right) - \sum_j \left(\mathcal{L}^k_i (D w^k_{i,j}) + \frac{|D w^k_{i,j}|^2}{2}\right)\right) \\
&\qquad +O ((\boldsymbol{\mu} (\mathbf{S}_k)+E_k^\gamma)E_k)\\
&= Q_i \underbrace{\int_\Omega \mathcal{L}^k_i (D (\boldsymbol{\eta}\circ \tilde{w}^k_i - \boldsymbol{\eta}\circ w^k_i))}_{=:\, \text{(L)}}
+ \frac{1}{2} \underbrace{\int_\Omega (|D\tilde{w}^k_i|^2 - |Dw^k_i|^2)}_{=:(\mathscr{Q})} + O ((\boldsymbol{\mu} (\mathbf{S}_k)+E_k^\gamma)E_k)\, ,
\end{align*}
where we have used the linearity of $\mathcal{L}^k_i$. Now observe that the function 
\[
f_k := \boldsymbol{\eta}\circ \tilde{w}^k_i - \boldsymbol{\eta}\circ w^k_i
\]
is a single-valued Lipschitz function that vanishes on $\partial \Omega$. On the other hand $\mathcal{L}^k_i$ is a linear function. In particular, if we write $f_k = (f_k^1, \ldots, f_k^n)$ for the components of the functions $f_k$, there are vectors $v_k^1, \ldots v_k^n\in \mathbb{R}^n$ (determined by $\mathcal{L}^k_i$) such that  
\begin{align*}
\text{(L)} &= \sum_l \int_\Omega v_k^l \cdot \nabla f_k^l =
\sum_l \int_\Omega {\rm div}\, (v_k^l f_k^l) = \sum_l \int_{\partial \Omega} f_k^l \nu \cdot v_k^l = 0\, ,
\end{align*}
where $\nu$ denotes the unit normal determined by the Stokes' orientation of $\partial\Omega$. Moreover, by \eqref{e:gain} we have
\[
(\mathscr{Q}) \leq - \vartheta E_k\, .
\]
We can thus write 
\[
\Delta \leq - \frac{\vartheta}{2} E_k + O ((\boldsymbol{\mu} (\mathbf{S}_k)+E_k^\gamma)E_k)\, .
\]
Since $E_k+ \boldsymbol{\mu} (\mathbf{S}_k)\to 0$, the latter clearly contradicts \eqref{e:almost-minimality}. This concludes the proof of (e). 

\medskip

The argument for (f) is entirely analogous to the argument for (b), where we use \eqref{e:no-concentration-5} in place of \eqref{e:no-concentration-3}: indeed \eqref{e:no-concentration-5} leads to the conclusion that 
\[
\int_{\Bbf_r\cap \alpha_1} \frac{|\bar{w}_i \ominus \boldsymbol{\eta}\circ \bar{w}_i|^2}{\dist (z, V)^{7/2}} < \infty\, ,
\]
which in turn clearly implies (f) using the continuity of $\bar{w}_i$. The argument for (g) is entirely analogous to the argument for (c): using Corollary \ref{c:HS-patch} with $\kappa=\frac{1}{4}$ we derive
\[
    \int_{\Bbf_r\cap \alpha_1} \frac{|\bar w_i (z)|^2}{|z|^{m+7/4}}\, dz < \infty\, ,
\]
which yields $\bar{w}_i(0)=\boldsymbol{\eta}\circ\bar{w}_i(0)=0$, while Proposition \ref{p:HS-3} gives
\begin{equation}\label{e:avg-free-w}
\int_{\Bbf_r(y)\cap \alpha_1} \frac{|\bar{w}_i \ominus \boldsymbol{\eta}\circ \bar{w}_i (z)|^2}{|z-y|^{m+2-\kappa}}\, dz \leq C (\kappa) 
\end{equation}
for $y\in V\cap \Bbf_r$. We then use this 
in place of \eqref{e:finite-for-kappa} to conclude the desired lower bound on the frequency.

\medskip

We finally come to (h). Recall the notation $\beta^k (L)$ for the nail of $L$ when the current is $T_k$. Recall that from \eqref{e:HS-4} that
\begin{align*}
|\mathbf{p}^\perp_{\alpha_1} (\beta^k (L))|& \leq C \mathbb{E} (T_k, \mathbf{S}_k, \Bbf_1)^{1/2}\, ,\\
\boldsymbol{\mu} (\mathbf{S}_k)|\mathbf{p}_{V^\perp\cap \alpha_1} (\beta^k (L))| &\leq C \mathbb{E} (T_k, \mathbf{S}_k, \Bbf_1)^{1/2}\, .
\end{align*}
Arguing as in (d), we assume, upon extraction of a subsequence, that there are $\bar \beta^\perp$ and $\bar \beta^\parallel$ such that 
\begin{align*}
\frac{\mathbf{p}^\perp_{\alpha_1} (\beta^k (L))}{\mathbb{E} (T_k, \mathbf{S}_k, \Bbf_1)^{1/2}} &\to \bar \beta^\perp (L)\\
\frac{\boldsymbol{\mu} (\mathbf{S}_k) \mathbf{p}_{V^\perp\cap\alpha_1} (\beta^k (L))}{\mathbb{E} (T_k, \mathbf{S}_k, \Bbf_1)^{1/2}} &\to \bar \beta_\parallel (L)\, .
\end{align*}
In analogy with the argument for (d) we consider both functions as defined over $\Bbf_r\cap \alpha_1$. Since $\bar{A}^k_i = \boldsymbol{\mu} (\mathbf{S}_k)^{-1} A^k_i$, by passing into the limit in \eqref{e:no-concentration-5} we get to, for each $i=1,\dotsc,N$
\begin{equation}\label{e:nailing-it-2}
\int_{\Bbf_r\cap \alpha_1} \frac{|\bar{w}_i \ominus (\bar\beta^\perp + \bar{A}_i (\bar \beta_\parallel))|^2}{\dist (z,V)^{7/2}}\, dz < \infty\, .
\end{equation}
Combining with \eqref{e:avg-free-w}, in turn these estimates imply
\begin{equation}\label{e:nailing-it-3}
\int_{\Bbf_r\cap \alpha_1} \frac{|\boldsymbol{\eta} \circ \bar{w}_i  - (\bar\beta^\perp + \bar{A}_i (\bar \beta_\parallel))|^2}{\dist (z,V)^{7/2}}\, dz < \infty\, .
\end{equation}
As $\bar{A}_1 =0$, we immediately conclude that $\bar{\beta}^\perp = \boldsymbol{\eta}\circ\bar{w}_1$ on $V\cap \Bbf_r$, and by the triangle inequality that 
\begin{equation}\label{e:nailing-it-4}
\int_{\Bbf_r\cap \alpha_1} \frac{|\boldsymbol{\eta} \circ \bar{w}_i  - (\boldsymbol{\eta} \circ \bar{w}_1 + \bar{A}_i (\bar \beta_\parallel))|^2}{\dist (z,V)^{7/2}}\, dz < \infty\, .
\end{equation}
Next, notice that $\sum_i \llbracket \bar{A}_i\rrbracket$ is Dir-minimizing (indeed, this will simply be the usual blow-up of $T_k$ relative to $\alpha_1$). Moreover, because $\max_i |A^k_i|\geq C^{-1} \boldsymbol{\mu} (\mathbf{S}_k)$, we have that $\max_i |\bar{A}_i| \geq C^{-1}>0$ for some constant $C = C(Q,m,n,\bar{n})$; we also know that this maximum is achieved by $\bar{A}_N$. Since $\bar{A}_1 =0$, the map $\llbracket \bar{A}_N\rrbracket + \llbracket 0 \rrbracket$ is Dir minimizing and hence, by subtracting the average and rescaling by a factor $2$, so is 
$\llbracket \bar{A}_N\rrbracket + \llbracket -\bar{A}_N \rrbracket$. In particular we can apply Lemma \ref{l:quasi-conformal} and conclude that, if we let $W$ be the image of $\bar{A}_N$,  $\bar{A}_N : V^\perp \to W$ is invertible and its inverse $B$ satisfies $|B|\leq C$ for some $C = C(Q,m,n,\bar{n})$.

We therefore have the identity $\bar \beta_\parallel = B (\bar{A}_N (\bar \beta_\parallel))$. For $y\in V\cap\Bbf_r$, define now the functions $(\bar\beta_\parallel)_\rho$ and $h_\rho$ as follows, analogously to those the proof of (d): for $y\in V\cap\Bbf_r$,
\begin{align}\label{e:averages-2}
(\bar\beta_\parallel)_\rho (y) &= \frac{1}{\pi \rho^2} \int_{\Bbf_\rho (y) \cap (y+V^\perp) \cap \alpha_1} \bar\beta_\parallel\, , \\
h_\rho(y) & = \frac{1}{\pi \rho^2} \int_{\Bbf_\rho (y) \cap (y+V^\perp) \cap \alpha_1} (\boldsymbol{\eta}\circ\bar{w}_i - \boldsymbol{\eta}\circ\bar{w}_1)\, .
\end{align}
Arguing as in the proof of (d) we use \eqref{e:nailing-it-4} to conclude that $\bar{A}_i ((\bar\beta_\parallel)_\rho)$ converges in $L^2 (V\cap\Bbf_r)$ to the function $\boldsymbol{\eta}\circ \bar{w}_i - \boldsymbol{\eta}\circ \bar{w}_1 = \lim_{\rho\downarrow 0} h_\rho$, which in particular gives us the conclusion that $\boldsymbol{\eta}\circ \bar{w}_N - \boldsymbol{\eta}\circ \bar{w}_1$ is in the image of $\bar{A}_N$. Hence, if we define $\beta_\parallel := B (\boldsymbol{\eta}\circ \bar{w}_N - \boldsymbol{\eta}\circ \bar{w}_1)$ and $\beta^\perp:= \boldsymbol{\eta}\circ\bar{w}_1$, we see immediately that we have the identity
\[
\boldsymbol{\eta}\circ \bar{w}_i = \boldsymbol{\eta}\circ \bar{w}_1 + \bar{A}_i (\beta_\parallel) = \beta^\perp + \bar{A}_i (\beta_\parallel)\, .
\]
In order to conclude the proof we need to show the desired estimates \eqref{e:est-w} on $\|\bar{w}_i\|_{W^{1,2}}$, $\|\beta^\perp\|_{C^2}$, and $\|\beta_\parallel\|_{C^2}$. The first is obvious because $\|\bar{w}^k_i\|_{W^{1,2}}$ is bounded by a universal constant. The second is also obvious because $\beta^\perp$ is the restriction to $V$ of the harmonic function $\boldsymbol{\eta}\circ \bar{w}_1$ whose $W^{1,2}$ norm is controlled by $\|\bar{w}_1\|_{W^{1,2}}$. Finally, the estimate on $\|\beta_\parallel\|_{C^2}$ follows from the same argument because $\beta_\parallel = B(\boldsymbol{\eta}\circ \bar{w}_N - \boldsymbol{\eta}\circ \bar{w}_1)$ and the norm of the linear map $B$ is bounded by a universal constant. Thus, the proof of Proposition \ref{p:blowup2} is complete.
\end{proof}

\subsection{Final argument}

In this section we are going to show Proposition \ref{p:decay-collapsed} and Proposition \ref{p:decay-noncollapsed}. We fix the decay scale $\varsigma_1$ and we will show that this will be reached at a certain radius, $r_c$ or $r_{nc}$, whether we are in the collapsed or non-collapsed setting, respectively, via a contradiction argument. We start with Proposition \ref{p:decay-collapsed}; we fix a contradiction sequence $T_k$, $\mathbf{S}_k$, and $\Sigma_k$ as in the previous section and use Proposition \ref{p:blowup2}(e)-(h) to extract the blow-up limits $\bar{A}_i$, $\bar{w}_i$, $\beta^\perp$, and $\beta_\parallel$. As before, without loss of generality we have rotated so that the planes $\alpha^k_1$ all coincide with the same plane $\alpha_1$. We next observe that, by (e)-(h), the functions $\beta^\perp$ and $\beta_\parallel$ both equal $0$ at the origin. We therefore linearize them and let 
$\gamma^\perp$ and $\gamma_\parallel$ be their respective linearizations at $0$. Observe that the $C^2$-regularity of $\beta^\perp$ and $\beta_\parallel$ guarantees that
\begin{align}
|\beta^\perp (y) - \gamma^\perp (y)|&\leq C |y|^2\\
|\beta_\parallel (y) - \gamma_\parallel (y)| & \leq C |y|^2\, .
\end{align}
The constant $C$ depends only on the $C^2$ norm of $\beta^\perp$ and $\beta_\parallel$, which in turn is bounded by a constant depending only upon on $Q, m, n, \bar{n}$, by Proposition \ref{p:blowup2}. Observe that $\gamma_\parallel: V \to V^\perp\cap \alpha_1$ and let $\gamma_\parallel^T : V^\perp\cap \alpha_1 \to V$ be its transpose. We build a skew-symmetric map of $\alpha_1$ onto itself by mapping
\[
V\oplus (V^\perp\cap\alpha_1) \ni y+x \mapsto \gamma_\parallel (y) - \gamma_\parallel^T (x)\, .
\]
This skew-symmetric map generates a one-parameter family $R [t]$ of rotations of $\alpha_1$, which we may extend to all of $\R^{m+\bar{n}}$ by setting it to be the identity on $\alpha_1^\perp$ and extended linearly. We next define the rotations
\[
R_k := R \left[\frac{\mathbb{E} (T_k, \mathbf{S}_k, \Bbf_1)^{1/2}}{\boldsymbol{\mu} (\mathbf{S}_k)}\right]
\]
and observe that these rotations map $\alpha_1$ and $\alpha_1^\perp$ onto themselves. 

The rotated cones $\mathbf{S}'_k := R_k (\mathbf{S}_k)$ are thus a first step towards the cones which will have the desired decay at the radius $r_c$.
Next consider the Dir-minimizing map $\bar{w}_i \ominus \boldsymbol{\eta}\circ \bar{w}_i$ and the (single-valued) harmonic function $\zeta_i := \boldsymbol{\eta}\circ \bar{w}_i - \gamma^\perp - \bar{A}_i (\gamma_\parallel)$, where the latter two linear maps are extended in the $V^\perp\cap \alpha_1$ directions as constant (in particular, being linear, they are harmonic). To the map $\bar{w}_i \ominus \boldsymbol{\eta}\circ \bar{w}_i$ we apply Theorem \ref{t:decay} (namely \eqref{e:decay-2}, which we can do by Proposition \ref{p:blowup2}(g)): for a fixed $\delta$, which will be chosen later, we find a radius $\bar{r} = \bar{r}(Q,m,n,\bar{n},\delta)$ such that for every $\rho<\bar{r}$ we can find a $1$-homogeneous Dir-minimizer $h_{i, \rho}\in \mathscr{L}_1$ with the property that
\[
\int_{\Bbf_\rho\cap \alpha_1} \mathcal{G} (\bar{w}_i \ominus \boldsymbol{\eta}\circ \bar{w}_i, h_{i, \rho})^2 \leq \delta \rho^{m+2}\, 
\]
(indeed, $\int_{\Bbf_1\cap\alpha_1}|\bar{w}_i|^2\leq 1$ by construction). As for the classical harmonic part $\zeta_i$, since $D\zeta_i(0) = 0$, we find a linear map $\xi_i$ (namely, the linearization of $\zeta_i$) which vanishes on $V$ such that 
\[
|\zeta_i (z)-\xi_i (z)|\leq C |z|^2\, .
\]
We fix now a radius $r_c<\bar{r}$. We are now ready to define a new sequence of cones $\mathbf{S}''_k$. We take the linear functions
$A^k_i$ whose graphs over $\alpha_1$ give the planes $\alpha^k_i$, hence the linear functions $\xi_i$, and construct the maps 
\[
A^k_i + \mathbb{E} (T_k, \mathbf{S}_k, \Bbf_1)^{1/2} \xi_i\, ,
\]
and then we add to them the multi-valued functions $\mathbb{E} (T_k, \mathbf{S}_k, \Bbf_1)^{1/2} h_{i, \rho}$, and compose the resulting $Q$-valued function with $R_k^{-1}$. The formula for this $Q$-valued function over $\alpha_1$ is thus given by
\[
\sum_i (\mathbb{E} (T_k, \mathbf{S}_k, \Bbf_1)^{1/2} h_{i, \rho} \oplus (A^k_i + \mathbb{E} (T_k, \mathbf{S}_k, \Bbf_1)^{1/2} (\xi_i+\gamma^\perp))) \circ R_k^{-1}\, .
\]
By construction, since the support of the graph of any element $h\in\Lscr_1$ with $h(0) = Q\llbracket 0 \rrbracket$ lies in $\Cscr(Q, 0)$, the graphs of these functions give new cones $\mathbf{S}''_k$ which belong to $\mathscr{C} (Q, 0)$. From the estimates that we have, we can check that, if $r_c = \rho \leq \bar{r}$ 
\[
\lim_{k\to \infty} \frac{\mathbb{E} (T_k, \mathbf{S}_k'', \Bbf_{r_c})}{\mathbb{E} (T_k, \mathbf{S}_k, \Bbf_1)} 
\leq C \delta + C r_c^2\, ,
\]
where $C$ is just a geometric constant. In particular, we choose $\delta$ sufficiently small so that $C\delta \leq \frac{\varsigma_1}{2}$, which in turn fixes $\bar r$, and hence we fix $r_c \leq \bar{r}$ so that $C r_c^2 \leq \frac{\varsigma_1}{4}$. With this choice we conclude that 
\[
\lim_{k\to \infty} \frac{\mathbb{E} (T_k, \mathbf{S}_k'', \Bbf_{r_c})}{\mathbb{E} (T_k, \mathbf{S}_k, \Bbf_1)} \leq \frac{3\varsigma_1}{4}\, .
\]
Now, with this particular choice of $r_c$, which depends only on $\varsigma_1$, we get for $k$ large enough a contradiction to the absence of decay with factor $\varsigma_1$. This proves Proposition \ref{p:decay-collapsed}.

\medskip

The proof of Proposition \ref{p:decay-noncollapsed} works in a very similar way. We again assume that $\varsigma_1$ and $\varepsilon_{c}^\star$ are given, that $r_{nc}$ is fixed, and that there is absence of decay by $\varsigma_1$ for sequences $T_k$, $\Sigma_k$, and $\mathbf{S}_k$. We then apply Proposition \ref{p:blowup2} and get the maps $\bar{u}_i$ and $\beta$. Arguing as above, we linearize the map $\beta$ (which vanishes at $0$) to a map $\gamma$. Again we note that 
\[
|\beta (y) - \gamma (y)|\leq C |y|^2\, .
\]
However, this time the constant $C$ depends on $\varepsilon^\star_{c}$ as well as the $C^2$ norm of $\beta$. First of all we consider $\gamma$ as a map from $V$ to $V^\perp$, we let $\gamma^T: V^\perp \to V$ be its transpose, we again build the skew-symmetric map
\[
V \oplus V^\perp \ni y+x \mapsto \gamma (y) - \gamma^T (x)
\]
and hence we consider the one-parameter family of rotations $R[t]$ generated by it. We then introduce the multi-valued functions $\bar{u}_i \ominus \boldsymbol{\eta}\circ \bar{u}_i$ and the harmonic functions $\boldsymbol{\eta}\circ \bar{u}_i - \mathbf{p}_{\alpha_i}^\perp (\gamma)$, where we assume that $\gamma$ is extended in the $V^\perp$ directions as a constant map. 

As above, we fix $\delta>0$ (whose choice will be specified later) and we appeal to Theorem \ref{t:decay} to find a threshold $\bar r =\bar{r}(Q,m,n,\bar{n},\delta)>0$ with the property that, for every $\rho< \bar{r}$ we can find a $1$-homogeneous map $h_{i, \rho} \in \mathscr{L}_1$ (again, this map lies in $\mathscr{L}_1$ due to \eqref{e:decay-2}, which is applicable due to Proposition \ref{p:blowup2}(c)) with the property that
\[
\int_{\Bbf_\rho\cap \alpha_i} \mathcal{G} (\bar{u}_i \ominus \boldsymbol{\eta}\circ \bar{u}_i, h_{i, \rho})^2 \leq \delta \rho^{m+2}\, .
\]
Likewise we find a linear map $\xi_i$ which vanishes on $V$ and such that 
\[
|(\boldsymbol{\eta}\circ \bar{u}_i - \mathbf{p}_{\alpha_i}^\perp (\gamma)) (z) - \xi_i (z)|\leq C |z|^2\, ,
\]
where the constant $C$ depends again on $\varepsilon_{c}^\star$. 

We are now ready to find the desired new cones. First of all we define
\[
R_k := R [ \mathbb{E} (T_k, \mathbf{S}_k, \Bbf_1)^{1/2}]
\]
and we consider the first adjustment to the cones as
\[
\mathbf{S}'_k := R_k (\mathbf{S}_k)\, .
\]
Next recall that along the sequence we are fixing $\alpha^k_1$ to be always the same plane (by applying a suitable rotation), and we are assuming that $\alpha^k_i$ converges to $\alpha_i$. Moreover, for each $k$ and $i\neq 1$, we fix a rotation $O_{k,i}$ which maps the approximating planes $\alpha^k_i$ onto $\alpha_i$ (where we use Lemma \ref{l:rotations} to determine the $O_{k,i}$). In particular, we now consider the maps $\mathbb{E} (T_k, \mathbf{S}_k, \Bbf_1)^{1/2} (h_{i, \rho}\oplus \xi_i)$ over $\alpha_i$, compose them with 
$O_{k,i}^{-1} \circ R_k^{-1}$ on the right and with $R_k \circ O_{k,i}$ on the left.

 We thus find the following multi-valued maps over $(\alpha^k_i)' = R_k (\alpha^k_i)$, namely
 \[
 R_k \circ O_{k,i}\circ (\mathbb{E} (T_k, \mathbf{S}_k, \Bbf_1)^{1/2} (h_{i, \rho}\oplus \xi_i)) \circ O_{k,i}^{-1} \circ R_k^{-1}\, .
 \]
 The graphs of these maps give the new cone $\mathbf{S}''_k$, which are easily seen to belong to $\mathscr{C} (Q, 0)$. We now arrive at the same conclusion of the argument for Proposition \ref{p:decay-collapsed}, namely, under the assumption that $\rho= r_{nc} \leq \bar{r}$,
\[
\lim_{k\to \infty} \frac{\mathbb{E} (T_k, \mathbf{S}_k'', \Bbf_{r_{nc}})}{\mathbb{E} (T_k, \mathbf{S}_k, \Bbf_1)} 
\leq C \delta + C r_{nc}^2\, .
\]
The only difference is that in this case the constants $C$ depend upon $\varepsilon^\star_{c}$ as well. However, since the $\varepsilon^\star_c$ and $\varsigma_1$ are both fixed, the constants above are also fixed, and we just choose $\delta$ so that $C\delta \leq \frac{\varsigma_1}{2}$. This in turn fixes $\bar{r}$, and we can further choose $r_{nc}\leq \bar{r}$ so that $C r_{nc}^2 \leq \frac{\varsigma_1}{4}$. As in the proof of Proposition \ref{p:decay-collapsed} we conclude that 
\[
\lim_{k\to \infty} \frac{\mathbb{E} (T_k, \mathbf{S}_k'', \Bbf_{r_{nc}})}{\mathbb{E} (T_k, \mathbf{S}_k, \Bbf_1)} 
\leq \frac{3\varsigma_1}{4}\, .
\]
Since $r_{nc}$ is now fixed, this now contradicts the assumption that there was no decay by a factor $\varsigma_1$ at that radius for any of the currents $T_k$. Thus the proof of Proposition \ref{p:decay-noncollapsed} is complete.
\qed

\medskip

To summarize, we have now established Theorem \ref{c:decay}.

\part{Uniqueness of Tangent Cones and Rectifiability}\label{p:rect}

\section{Uniqueness and Rectifiability}

\subsection{Rectifiability}

In this section we will use Theorem \ref{c:decay} to prove Theorem \ref{con:stronger}(i), namely that the set $\flatS_{Q,1}$ has $\Hcal^{m-2}$-measure zero. Before proceeding with the proof that Theorem \ref{c:decay} implies Theorem \ref{con:stronger}(i), we begin with some preliminaries.

First of all, we would like to show that the cones in the class $\Cscr(Q,0)\setminus\Pscr(0)$ are the ``prevalent" fine blow-ups appearing in the compactness procedure in \cite{DLSk1}. This can be thought of as the analogue of \cite{Simon_cylindrical}*{Lemma 2.4} and \cite{KW}*{Lemma 4.3} for the present setting, and is the key (and, in fact, the only) ingredient needed from the analysis in \cite{DLSk1}.

\begin{lemma}\label{l:dimension-drop}
     For each $\eps \in (0,1]$, the following holds. Suppose that $T$ and $\Sigma$ are as in Assumption \ref{a:main}. Then, for each $p \in \flatS_{Q,1} (T) \cap\Bbf_1$ there exists $\bar\rho=\bar\rho(p,\eps)>0$ such that the following dichotomy holds for each $\rho\in (0,\bar\rho]$: 
    \begin{itemize}
        \item[(a)] There exists $\Sbf\in \Cscr(Q,p)\setminus\Pscr(p)$ with
        \[
            (\rho\Abf)^2 + \Ebb(T,\Sbf,\Bbf_\rho(p)) \leq \eps^2 \Ebf^p(T, \Bbf_\rho(p));
        \]
        \item[(b)] There exists an $(m-3)$-dimensional affine subspace $V\subset T_p\Sigma$ (depending on $\rho$) such that
        \[
            \flatS_{Q,1}\cap {\Bbf}_\rho(p) \subset \{q:\dist(q,V)< \eps\rho \}.
        \]
    \end{itemize}
\end{lemma}

\begin{remark}
    The subspace $V$ in alternative (b) of Lemma \ref{l:dimension-drop} arises as either
    \begin{itemize}
        \item the spine of a non-flat area-minimizing cone;
        \item the set of multiplicity $Q$ points of a 1-homogeneous $Q$-valued Dir-minimizer arising as a coarse blow-up at $p$ (cf. Proposition \ref{p:coarse-blow-ups}).
    \end{itemize}
\end{remark}

\begin{proof}[Proof of Lemma \ref{l:dimension-drop}]
	We argue by contradiction. Namely, suppose that there exists some $\eps \in (0,1]$ for which the following holds. There exist $T$, $\Sigma$, a point $p \in \flatS_{Q,1}(T)\cap\Bbf_1$, and a corresponding sequence of scales $\bar{\rho}_k \downarrow 0$ such that both (a) and (b) fail on balls $\Bbf_{\bar\rho_k}(p)$, for this choice of $\eps$. 
 
    Up to extracting a further subsequence, we have two possibilities. Either $T_{p, \bar\rho_k}$ converges to a flat tangent cone in $\Bbf_{6\sqrt{m}}$, namely $Q \llbracket \pi \rrbracket$ for some $m$-dimensional plane $\pi$, or it converges to a non-flat tangent cone $T_\infty$. In the latter case we let $V$ be the spine of $T_\infty$, and note that either $V$ is $(m-2)$-dimensional, and hence its support is a cone $\Sbf\in \mathscr{C} (Q,0)\setminus \mathscr{P}(0)$, or the dimension of $V$ is at most $m-3$. On the other hand, in the latter case $\{\Theta (T_k, \cdot) \geq Q\}\cap \overline{\Bbf}_1$ converges in the sense of Hausdorff to a subset of $\{\Theta (T_\infty, \cdot)\geq Q\}$ (by upper semi-continuity of the density) and the latter set is in fact the spine $V$. Hence, in the latter case, alternative (b) of the lemma holds for all sufficiently large $k$, giving a contradiction. In the former case clearly alternative (a) holds for all sufficiently large $k$, again giving the desired contradiction (note that since $\lim_{k\to\infty}\Ebf^p(T_{p,\bar\rho_k},\Bbf_{6\sqrt{m}}) > 0$ in this case, we clearly have $(\bar\rho_k\Abf)^2\leq \eps^2 \Ebf^p(T_{p,\bar\rho_k},\Bbf_{6\sqrt{m}})$ for all $k$ sufficiently large).
    
    So we are left with contradicting the case where $T_{p,\bar{\rho}_k}$ converges to a flat tangent cone; in this situation $\Ebf^p(T,\Bbf_{6\sqrt{m}\bar\rho_k}(p)) \downarrow 0$. Proposition \ref{p:coarse-blow-ups} then tells us that 
    \begin{equation}\label{e:A-decay}
        \Ebf^p(T,\Bbf_{6\sqrt{m}\bar\rho_k}(p))^{-1} \cdot \rho_k^2 \Abf_k^2\to 0.
    \end{equation}
    Now let $\pi_k$ be a plane such that $\Ebf^p(T,\Bbf_{6\sqrt{m}\bar\rho_k}(p))= \hat{\Ebf} (T, \pi_k,\Bbf_{6\sqrt{m}\bar \rho_k}(p))$ and, without loss of generality, we can rotate to assume that $\pi_k$ is a fixed plane $\pi$ for all $k$. We now consider a coarse blow-up $f$ as defined in \cite{DLSk1}; we know that $f$ is non-trivial and 1-homogeneous by Proposition \ref{p:coarse-blow-ups}. We then have two possibilities: either $f$ is translation invariant in $m-2$ independent directions, or its spine has dimension at most $m-3$. In the former case, notice that in combination with the fact that $I_{0,f}(0) = 1$ (as $f$ is $1$-homogeneous), the support of the (multi-valued) graph of $f$ is a superposition of planes $\Sbf\in \Cscr(Q,0)\setminus \Pscr(0)$. In light of the estimates in \cite{DLS14Lp}*{Theorem 2.4}, combined with \eqref{e:A-decay} and the strong $L^2$-convergence of the normalizations of the Lipschitz approximations of $T_{p,\bar\rho_k}\mres \Bbf_{6\sqrt{m}}$ to $f$, we contradict (a).

    Otherwise, consider the set $Z$ of points $z$ which are limits of $\mathbf{p}_{\pi} (p_k)$ with $\Theta (T_{p, \bar \rho_k}, p_k) = Q$ and $|p_k|\leq 1$. By \cite{DLS14Lp}*{Theorem 2.7} $f (z) = Q \llbracket \boldsymbol{\eta}\circ f (z) \rrbracket$ for any such point $z$ (this also follows by the Hardt--Simon inequality). Moreover, since by Proposition \ref{p:coarse-blow-ups} we know that $\boldsymbol{\eta}\circ f \equiv 0$, we in fact have $f(z) = Q \llbracket 0 \rrbracket$ at such $z$. By the $1$-homogeneity of $f$ and the upper semi-continuity of the frequency function, we know that $I_{z, f} (0) \leq 1$. On the other hand by the Hardt-Simon inequality, as in \cite{DLSk1}*{Section 3}, $I_{z,f} (0) \geq 1$. Therefore, $I_{z,f} (0) =1$ and hence $f$ is translation invariant along any $z\in Z$. In particular $Z\subset V$, which, being at most $(m-3)$-dimensional, would imply that (b) would hold for $\rho = \bar{\rho}_k$ when $k$ is large enough. This contradiction therefore proves the result. 
\end{proof}

It will become convenient to subdivide $\flatS_{Q,1} (T)\cap\Bbf_1$ as follows.

\begin{definition}\label{d:pieces}
    Fix $\eps^\dagger > 0$. Suppose that $T$, $\Sigma$ are as in Assumption \ref{a:main}. For $r>0$, we let $\flatS_{Q,1,r} (T)$ denote the set of all points $p\in \flatS_{Q,1} (T)\cap \Bbf_1$ for which the conclusions of Lemma \ref{l:dimension-drop} hold for all scales $\rho \in (0,r]$ with $\eps^\dagger$ in place of $\eps$, and moreover $\|T\|(\Bbf_\rho (p)) \leq (Q+\frac{1}{4}) \omega_m \rho^m$ for all $\rho \in (0, r]$.
\end{definition}

 Notice that by Lemma \ref{l:dimension-drop} we may write 
 \begin{equation}\label{e:decomposition}
 \flatS_{Q,1}(T)\cap \Bbf_1 = \bigcup_{k} \flatS_{Q,1,2^{-k}} (T)\cap\Bbf_1\, .
 \end{equation}

The fact that $\iota_{0,r} (\flatS_{Q,1,r} (T)) \subset \flatS_{Q,1,1} (T_{0, r})$ for any $r>0$, combined with a translation, together imply that it suffices to prove Theorem \ref{con:stronger}(i) for $\flatS_{Q,1,1}(T)$.

We will now proceed to show that a set satisfying a number of properties at every scale around every point is, up to a $\mathcal{H}^k$-negligible set, a countable union of $k$-dimensional $C^{1,\alpha}$ graphs. This is only one possible abstract formulation of what the arguments used by Simon in the proof of \cite{Simon_cylindrical}*{Theorem 1} imply for general sets. 

\begin{assumption}\label{a:Simon}
$k\in \N$ and $\delta,\alpha, \varepsilon > 0$ are fixed. $E$ is a Borel subset of $\Bbf_2 \subset \Rbb^{m+n}$ such that for every $p\in E\cap\Bbf_1$ and every $r \in (0, 1-|q|]$ there is a choice of a $k$-dimensional affine subspace $V(p, r)$ satisfying the following properties: 
\begin{itemize}
    \item[(1)] $E\cap \Bbf_r (p) \subset B_{\varepsilon r} (V(p, r))\, . $
\item[(2)] if $0\leq a < b\leq 1-|p|$ are radii for which the condition
\begin{equation}\label{e:ng}
V (p, r) \cap \Bbf_{r/2} \subset B_{\delta r} (E)
\end{equation}
holds for every $r \in (a,b]$, then the following estimate holds for every $a<s<r\leq b$:
\begin{equation}\label{e:C1alpha}
|\mathbf{p}_{V (p,r)}- \mathbf{p}_{V (p, s)}| \leq \varepsilon \left(\frac{r}{b}\right)^\alpha\, .
\end{equation}
\end{itemize}
\end{assumption}

We now state the result concerning how a set satisfying Assumption \ref{a:Simon} decomposes into a union of a $\Hcal^k$-negligible set and a countable union of $k$-dimensional $C^{1,\alpha}$ submanifolds:

\begin{theorem}\label{t:Simon}
Let $k\in \Nbb$ and $\alpha, \delta>0$ be fixed numbers. There exists $\varepsilon_s = \varepsilon_s (k,m,n, \delta)>0$ such that, if $E$ satisfies Assumption \ref{a:Simon} with $\eps\leq\eps_s$, then $E\cap \Bbf_{1/2}$ can be decomposed as a disjoint union $\tilde{E}\cup R$ where 
\begin{itemize}
\item[(i)] $\mathcal{H}^k (\tilde{E}) =0$;
\item[(ii)] $\tilde{E}$ consists of all points $p\in E\cap \Bbf_{1/2}$ for which there is a sequence $\rho_k\downarrow 0$ violating condition \eqref{e:ng};
\item[(iii)] $R$ is contained in a countable union of $C^{1,\alpha}$ graphs (each defined on an open cube in a $k$-dimensional affine subspace of $\R^{m+n}$).
\end{itemize}  
\end{theorem}

{We omit the proof of Theorem \ref{t:Simon}, since it is a simple consequence of the arguments given in \cite{Simon_cylindrical}.} We next demonstrate how Theorem \ref{con:stronger}(i) follows from it. In order to do this, we require the following lemma (in place of \cite{Simon_cylindrical}*{Lemma 1}), which is a simple consequence of Theorem \ref{c:decay}. It verifies that not only does the ratio of double-sided excess and planar excess under the assumptions of Theorem \ref{c:decay}, but in fact so does the maximum of this ratio and the ratio between $r^2 \Abf^2$ and the planar excess.

\begin{lemma}\label{c:decay-10-1}
	There are positive constants $\varepsilon_{\mathrm{f}}= \eps_{\mathrm{f}} (Q,m,n, \bar{n})$, $\theta = \theta (Q,m,n,\bar{n})\leq\frac{1}{2}$, and $C=C(Q,m,n,\bar{n})$ with the following property. Let $T$ and $\Sigma$ be as in Assumption \ref{a:main} and $p\in \flatS_{Q,1,1} (T)$. Assume $0<r\leq 2-|p|$ and 
	\begin{itemize}
		\item $\Sbf\in \mathscr{C} (Q, p)\setminus \Pscr(p)$ satisfies $\Ebb (T, \Sbf, \Bbf_r(p)) = \inf \{\mathbb{E} (T, \bar{\Sbf}, \Bbf_r(p)) : \bar{\Sbf}\in \mathscr{C} (Q, p)\}$;
		\item $\max \{\eps_{\mathrm{f}}^{-2} \Abf^2 r^2, \Ebb (T, \Sbf, \Bbf_r(p))\} \leq \varepsilon_{\mathrm{f}}^2 \Ebf^p (T, \Bbf_r(p))$;
		\item $\Bbf_{\varepsilon_{\mathrm{f}} r} (q) \cap \flatS_{Q,1,1}(T) \neq \emptyset$ for every $q\in V (\Sbf)\cap \Bbf_{r/4}(p)$. 
	\end{itemize}
	Then there is a cone $\Sbf^\prime\in\mathscr{C}(Q,p)\setminus \Pscr(p)$ such that 
	\begin{equation}\label{e:decay-10-i}
	\frac{\max \{\eps_{\mathrm{f}}^{-2} \Abf^2\theta^2 r^2, \Ebb (T, \Sbf^\prime, \Bbf_{\theta r}(p))\}}{\Ebf^p (T, \Bbf_{\theta r}(p))}
	\leq \frac{1}{4} \frac{\max \{\eps_{\mathrm{f}}^{-2} \Abf^2 r^2, \Ebb (T, \Sbf, \Bbf_r(p))\}}{\Ebf^p (T, \Bbf_r(p))}\, ,
	\end{equation}
    and
    \begin{equation}\label{e:decay-10-ii}
        \dist^2 (V (\Sbf)\cap \Bbf_r(p), V (\Sbf')\cap \Bbf_r(p)) \leq C \frac{\max\{\eps_{\mathrm{f}}^{-2}\mathbf{A}^2r^2, \mathbb{E} (T, \Sbf, \Bbf_r(p))}{\mathbf{E}^p (T, \Bbf_r(p))}\, .
    \end{equation}
\end{lemma}
\begin{proof}
	Let $\varsigma = \frac{1}{8}$ in Theorem \ref{c:decay}, and let $\eps_{\mathrm{f}}\coloneqq \eps_0$ and $\theta\coloneqq r_0$ be as in Theorem \ref{c:decay} for this choice of $\varsigma$. The decay \eqref{e:decay-10-i} is obvious by Theorem \ref{c:decay} if $\varepsilon_{\mathrm{f}}^{-2} \Abf^2 r^2 \leq \Ebb (T, \Sbf, \Bbf_r)$. In the other case we apply the continuity argument in Step 1 of Lemma \ref{l:a->bcd} to find a cone $\Sbf_e$ such that $\varepsilon_{\mathrm{f}}^{-2} \Abf r^2 \leq \Ebb (T, \Sbf_e, \Bbf_r) \leq 2 \varepsilon_{\mathrm{f}}^{-2} \Abf^2 r^2$ and obeying the conditions therein, at which point we may once again apply Theorem \ref{c:decay}, now with $\varsigma=\frac{1}{16}$, and again choose $\eps_{\mathrm{f}}\coloneqq \eps_0$ and $\theta\coloneqq r_0$ to be as in Theorem \ref{c:decay} for this choice of $\varsigma$. The conclusion \eqref{e:decay-10-ii} is an immediate consequence of Theorem \ref{c:decay}(d), \eqref{e:decay-10-i}, and the reasoning above.
\end{proof}

We will also require the following lemma, which gives control on the tilting between different $k$-dimensional subspaces that a given set $E$ satisfying Assumption \ref{a:Simon} is bilaterally close to. It is analogous to \cite{DavidSemmes-Asterisque}*{Lemma 5.13}. We thus refer the reader to the proof therein, and do not include it here.
\begin{lemma}\label{l:tilt}
	Let $k\in \Nbb$ and $\alpha, \delta>0$ be fixed numbers. There are positive constants $C_0 = C_0 (m,n,k)$, $\delta_0=\delta_0 (m,n,k)$, and $\varepsilon_0 = \varepsilon_0 (m,n,k)$ such that the following holds. Assume $E\subset \mathbb R^{m+n}$, $p\in E$, and $r>0$ are radii such that Assumption \ref{a:Simon}(1) and \eqref{e:ng} hold for some $V = V (p, r)$ and $\varepsilon \leq \varepsilon_0$ and $\delta\leq \delta_0$. Assume moreover that $V'$ is another $k$-dimensional affine subspace for which 
	\[
	E \cap \Bbf_r (p) \subset B_{\varepsilon r} (V')\, .
	\]
	Then
	\begin{equation}\label{e:tilt}
		|\mathbf{p}_V - \mathbf{p}_{V'}|\leq C_0 \varepsilon\, 
	\end{equation}
	and, if we set $\delta' := 2\delta + C_0 \varepsilon$,
	\begin{equation}\label{e:tilt-no-gap}
		V'\cap \Bbf_{r/2} (p) \subset B_{\delta' r} (E)\, .
	\end{equation}
\end{lemma}

We are now in a position to conclude the proof of the fact that $\flatS_{Q,1}(T)$ is $\Hcal^{m-2}$-null.

\begin{proof}[Proof of Theorem \ref{con:stronger}(i)]
    Let $E= \flatS_{Q,1,1}(T)$, let $k=m-2$, and in Definition \ref{d:pieces} and fix $\eps=\min\{\eps_s,\eps_{\mathrm{f}}\}$, {where $\eps_{\mathrm{f}}$ is the constant of Lemma \ref{c:decay-10-1} and $\eps_s$ is the constant of Theorem \ref{t:Simon}} Choose $\eps^\dagger = \eps$. We will proceed to verify that Assumption \ref{a:Simon} holds for this choice of $E$ and $k$, and with a choice of $\delta=\delta(\eps)$ sufficiently small. First of all, notice that Lemma \ref{l:dimension-drop} implies that property (1) of Assumption \ref{a:Simon} holds for each $p\in E\cap \Bbf_1$, with $V(p,r)$ defined to be the spine of $\Sbf$ (defined therein, which implicitly depends on $p$ and $r$) if alternative (a) of Lemma \ref{l:dimension-drop} holds in $\Bbf_r(p)$, and any $(m-2)$-dimensional affine subspace containing $V$ of alternative (b) of Lemma \ref{l:dimension-drop} if that holds in $\Bbf_r(p)$ instead. To see that (2) of Assumption \ref{a:Simon} holds, we proceed as follows. Suppose that \eqref{e:ng} holds for a given point $p\in E\cap\Bbf_1$ and for all scales $r\in (a,b]$, for a given pair of radii $a<b$ as in (2). Then for all such scales $r$, the alternative (a) of Lemma \ref{l:dimension-drop} must hold in $\Bbf_r(p)$, and in addition, Lemma \ref{l:tilt} tells us that provided that $\delta$ is chosen sufficiently small (depending on $\eps_{\mathrm{f}}$), all of the hypotheses of Lemma \ref{c:decay-10-1} {hold with the ball $\Bbf_{r/2} (p)$ replacing $\Bbf_r (p)$}. Now fix any such $r$ and find $j\in \N$ such that $r\in (\theta^j b,\theta^{j-1} b]$. Applying Lemma \ref{c:decay-10-1} $j$ times successively, starting from $\Bbf_b(p)$, yields
    \[
    	\dist(V(\Sbf_{j-1})\cap\Bbf_b(p),V(\Sbf_j)\cap\Bbf_b(p))\leq C 2^{-j} \eps_{\mathrm{f}} \leq \eps \left(\frac{r}{b}\right)^\alpha,
    \]
    for an appropriate choice of $\alpha=\alpha(Q,m,n,\bar{n},\theta)$, where $\Sbf_j,\Sbf_{j-1}$ are the cones $\Sbf,\Sbf'$ in Lemma \ref{c:decay-10-1} when $r$ is replaced by $\theta^{j-1} r$. This verifies that property (2) of Assumption \ref{a:Simon} holds for this choice of $E$, and thus allows us to apply Theorem \ref{t:Simon} with the above choice of parameters. As observed above for each point $p\in R$ all of the hypotheses of Lemma \ref{c:decay-10-1} are satisfied in $\Bbf_r(p)$ for each sufficiently small scale, so one may iteratively apply this lemma to deduce that any tangent cone at each such point $p$ will be supported in a unique element $\Sbf\in\Cscr(Q,p)\setminus\Pscr(p)$. This, however, is in contradiction with the fact that $p\in \flatS_{Q,1,1}(T)$. Therefore, $R=\emptyset$ here. Combined with \eqref{e:decomposition}, this completes the proof.
\end{proof}

\subsection{Uniqueness of tangent cones}

Having proved the rectifiability, we now turn to the conclusion of Theorem \ref{con:stronger}, namely the $\Hcal^{m-2}$-a.e. uniqueness of tangent cones. Recalling \cite{DLSk1}*{Theorem 2.10}, we know that the tangent cone is unique at every flat singular point in $\flatS_Q(T)\setminus \flatS_{Q,1}(T)$, and given the previous section, it is therefore unique at $\mathcal{H}^{m-2}$-a.e. point in $\flatS_Q(T)$. Since $Q$ is arbitrary, it remains to show that the tangent cone is unique at $\Hcal^{m-2}$-a.e. point $p\in \mathcal{S}^{(m-2)}$. Although countable $(m-2)$-rectifiability of $\mathcal{S}^{(m-2)}\setminus\Scal^{(m-3)}$ follows from \cite{NV_varifolds}, we will achieve this independently as an additional consequence of the arguments in this section, together with the $\Hcal^{m-2}$-a.e. uniqueness of tangent cones, following an argument analogous to that in the previous section. First of all recall that every point $p\in \mathcal{S}^{(m-2)}\setminus \mathcal{S}^{(m-3)}$ has integer density and that $\mathcal{S}^{(m-3)}$ has Hausdorff dimension at most $m-3$. Then we have the following counterpart of Lemma \ref{l:dimension-drop} above.

\begin{lemma}\label{l:dimension-drop-2}
     For each $\eps \in (0,1]$, the following holds. Suppose that $T$ and $\Sigma$ are as in Assumption \ref{a:main}. For each $p \in \mathcal{S}^{(m-2)}\setminus \mathcal{S}^{(m-3)} \cap\Bbf_1$ with $\Theta (T, p) = Q$, there exists $\bar\rho=\bar\rho(p,\eps)>0$ such that the following dichotomy holds for each $\rho\in (0,\bar\rho]$: 
    \begin{itemize}
        \item[(a)] There exists $\Sbf\in \Cscr(Q,p)\setminus\Pscr(p)$ and
        \[
            (\rho\Abf)^2 +\Ebb(T,\Sbf,\Bbf_\rho(p)) \leq \eps^2 \Ebf^p(T, \Bbf_\rho(p));
        \]
        \item[(b)] There exists an $(m-3)$-dimensional affine subspace $V\subset T_p\Sigma$ (depending on $\rho$) such that
        \[
            \{\Theta(T,\cdot)\geq Q\}\cap\bar{\Bbf}_\rho(p) \subset \{q:\dist(q,V)< \eps\rho \}.
        \]
    \end{itemize}
\end{lemma}

\begin{proof}
The proof is by contradiction and is much simpler than the one of Lemma \ref{l:dimension-drop} since in this case 
\[
\liminf_{\rho\downarrow 0} \Ebf^p(T, \Bbf_\rho(p)) > 0\, .
\]
Therefore, any tangent cone to $T$ at $p$ either has an $(m-2)$-dimensional spine, and hence its support is an element of $\mathscr{C} (Q, p)\setminus \mathscr{P}(p)$, or it has an $(m-3)$-dimensional spine, leading to the desired contradiction in either case. 
\end{proof}

We now subdivide $\mathcal{S}^{(m-2)}\setminus \mathcal{S}^{(m-3)}$ into countably many pieces (analogously to the subdivision of $\flatS_{Q,1}(T)$ carried out in Definition \ref{d:pieces}) in the following way:

\begin{definition}\label{d:pieces-2}
    Let $\eps^\dagger > 0$ and $\delta>0$ be fixed. Suppose that $T$, $\Sigma$ are as in Assumption \ref{a:main}. For $r > 0$ we let $\mathcal{S}_{Q, \delta, r} (T)$ be the set of points $q\in (\mathcal{S}^{(m-2)}\setminus \mathcal{S}^{(m-3)})\cap\Bbf_1$ of density $\Theta (q, T) = Q$ such that 
    \begin{itemize}
        \item The dichotomy of Lemma \ref{l:dimension-drop-2} applies at every scale $\rho\in (0,r]$ with $\eps=\eps^\dagger$;
        \item $\Ebf^p (T, \Bbf_\rho) \geq \delta$ for every $\rho\in(0,r]$.
    \end{itemize}
\end{definition}

First observe that we have the decomposition
\[
\mathcal{S}^{(m-2)}\setminus \mathcal{S}^{(m-3)} = \bigcup_{Q, j, \ell} \mathcal{S}_{Q,1/j, 1/\ell}\, .
\]
Thus, in order to conclude the $(m-2)$-rectifiability and $\Hcal^{m-2}$-a.e. uniqueness of tangent cones, it suffices to prove that each piece $\mathcal{S}_{Q,1/j, 1/\ell}$ in the above decomposition is rectifiable and the tangent cone is unique at $\mathcal{H}^{m-2}$-a.e. point $q\in\mathcal{S}_{Q,1/j, 1/\ell}$. By scaling we assume $\ell =1$, and we may without loss of generality further assume that $j=1$, in order to simplify notation.

The proof of the $(m-2)$-rectifiability of $\mathcal{S}^{(m-2)}\setminus \mathcal{S}^{(m-3)}$ and uniqueness of tangent cones at $\Hcal^{m-2}$-a.e. point $p \in \Scal^{(m-2)}$ 
{will once again be concluded from Theorem \ref{t:Simon}. In order to do this, we first require the following analogue of Lemma \ref{c:decay-10-1}, which is proved in exactly the same way.

\begin{lemma}\label{c:decay-11-1}
There are positive constants $\varepsilon_{\mathrm{\mathrm{nf}}} = \eps_{\mathrm{nf}} (Q,m,n, \bar{n})$, $\theta = \theta (Q,m,n,\bar{n})\leq \frac{1}{2}$ with the following property. Let $T$ and $\Sigma$ be as in Assumption \ref{a:main} and let $p\in \mathcal{S}_{Q,1,1} (T)$. Assume $0<r\leq 2-|p|$ and 
\begin{itemize}
    \item $\Sbf\in \mathscr{C} (Q, p)\setminus \Pscr(p)$ satisfies $\Ebb (T, \Sbf, \Bbf_r (p)) = \inf \{\mathbb{E} (T, \bar{\Sbf}, \Bbf_r(p)) : \bar{\Sbf}\in \mathscr{C} (Q, p)\}$;
    \item $\max \{\eps_{\mathrm{nf}}^{-2} r^2 \Abf^2, \Ebb (T, \Sbf, \Bbf_r(p))\} \leq \varepsilon_{\mathrm{nf}}^2 \Ebf^p (T, \Bbf_r(p)))$;
    \item $\Bbf_\varepsilon (q) \cap \mathcal{S}_{Q,1,1}(T) \neq \emptyset$ for every $q\in V (\Sbf)\cap \Bbf_{r/4}(p)$. 
\end{itemize}
Then there is a $\Sbf^\prime\in\mathscr{C}(Q,0)$ such that 
\[
\frac{\max \{\eps_{\mathrm{nf}}^{-2} \theta^2 r^2 \Abf^2, \Ebb (T, \Sbf^\prime, \Bbf_{\theta r})\}}{\Ebf^p (T, \Bbf_{\theta r})}
\leq \frac{1}{4} \frac{\max \{\eps_{\mathrm{nf}}^{-2} r^2 \Abf^2, \Ebb (T, \Sbf, \Bbf_r)\}}{\Ebf^p (T, \Bbf_r)}\, ,
\]
and
\[
    \dist^2 (V (\Sbf)\cap \Bbf_r(p), V (\Sbf')\cap \Bbf_r(p)) \leq C \frac{\max\{\eps_{\mathrm{f}}^{-2}\mathbf{A}^2r^2, \mathbb{E} (T, \Sbf, \Bbf_r(p))}{\mathbf{E}^p (T, \Bbf_r(p))}\, .
\]
\end{lemma}

\begin{proof}[Proof of Theorem \ref{con:stronger}(ii) and rectifiability of $\mathcal{S}^{(m-2)}\setminus \mathcal{S}^{(m-3)}$]
    Let $E= \Scal_{Q,1,1}(T)$, let $k=m-2$, and in Definition \ref{d:pieces} and fix $\eps=\min\{\eps_s,\eps_{\mathrm{nf}}\}$, where $\eps_{\mathrm{nf}}$ is the constant in Lemma \ref{c:decay-11-1} and $\eps_s$ is the constant in Theorem \ref{t:Simon}. Choose $\eps^\dagger = \eps$. Arguing exactly as in the proof of Theorem \ref{con:stronger}(i), with Lemma \ref{l:dimension-drop-2} and Lemma \ref{c:decay-11-1} applied in place of Lemma \ref{l:dimension-drop} and \ref{c:decay-10-1} respectively, we verify that Assumption \ref{a:Simon} holds for this choice of $E$ and $k$, and with a choice of $\delta=\delta(\eps)$ sufficiently small. We may thus apply Theorem \ref{t:Simon} with the above choice of parameters. At each point $p\in R$, all of the hypotheses of Lemma \ref{c:decay-11-1} are satisfied in $\Bbf_r(p)$ for each sufficiently small scale $r>0$. This in turn implies that the tangent cone is supported in a unique element $\Sbf\in \Cscr(Q,p)\setminus\Pscr(p)$ at each such point, and the supports of the rescalings $T_{p,r}$ decay towards $\Sbf$ with a power law for all $r$ sufficiently small. But, because of this decay and the graphical approximation in Proposition \ref{p:refined} applied to $T_{p,r}$ and this unique $\Sbf$, we can select $r$ sufficiently small to see that the origin is in fact in the closure of the outer region. Conclusion (i) of Proposition \ref{p:refined} ensures that the multiplicities assigned to the planes in $\Sbf\cap \Bbf_\rho (0)$ with $\rho\leq r$ do not change with the radius. Thus, not only is the support of the tangent cone unique, but so is the tangent cone itself. Combined with \eqref{e:decomposition}, this completes the proof.
\end{proof}}

\appendix

\section{Proofs of combinatorial results in Section \ref{s:comb-lemmas}}\label{ap:combinatorics}

\subsection{Proof of Lemma \ref{l:clusters-1}}
First of all, note that if $N=2$ the conclusion of the lemma is trivially true for $P'=P$, so we may assume $N\geq 3$. 
Then we define inductively the sets $P_j$, with $j= \{0, \ldots, N-2\}$ and the elements $p_j$ in the following way. $P_0=P$ and, for every $j\geq 1$, we select $p_j\in P_{j-1}$ with the property that 
\begin{align*}
M_{j-1} := &\max \{|p-q|: p,q\in P_{j-1}\} = \max \{|p-q|: p,q\in P_{j-1}\setminus \{p_j\}\}\\
m_{j-1} := &\min \{|p-q|:p,q\in P_{j-1}, \ p\neq q\} \leq \min \{|p-q|: p,q\in P_{j-1}\setminus \{p_j\}, \ p\neq q\}\, 
\end{align*}
Since the cardinality of $P_{j-1}$ is always at least three it is clear that such an element exists: we first select $p,q$ such that $|p-q|=\min \{|a-b|:a\neq b\in P_{j-1}\}$ and then $p',q'$ such that $|p'-q'|=\max \{|a-b|:a,b\in P_{j-1}\}$. If the pairs $(p,q)$ and $(p',q')$ are the same, then obviously all points in $P_{j-1}$ are equidistant and we can define $p_j$ to be any of them. Otherwise the two pairs have at most one element in common, and we define $p_j$ to be an element in $\{q,p\}\setminus \{q',p'\}$. We then let $P_j \coloneqq P_{j-1}\setminus \{p_j\} = P\setminus\{p_1,\dots,p_j\}$. By construction, we have
\begin{equation}\label{e:m_j}
    m_{j-1} = \dist(p_j,P_{j-1}\setminus\{p_j\}) = \dist(p_j,P_{j})\, ,
\end{equation}\
while at the same time $M_j =M_0$ and $m_j \geq m_{j-1}$. Observe in particular that the requirements (iii) is satisfied if we terminate this process at any stage $P_J$.

Now let $\lambda= \lambda(\bar\delta)\geq 1$ be a large number whose choice will be specified later and let $J$ be the smallest number $j$ such that $M_j\leq \lambda^{N-2-j} m_j$. Since $P_{N-2}$ consists of two elements, $M_{N-2}= m_{N-2}$ and thus such a number $J$ exists. Obviously $P'=P_J$ satisfies (i) with $C= \lambda^{N-2}$.  We claim that for (ii) to hold we just need $\lambda$ to be large enough.
Note that
\[
\min \{|p-q|:p,q\in P', \ p\neq q\} = m_J\geq \lambda^{-(N-2-J)} M_J\, ,
\]
while, for every $p\in \{p_1, \ldots, p_J\}$, by our definition of $J$ and by \eqref{e:m_j}, we have 
\begin{align*}
\dist (p, P') &\leq \sum_{i=1}^J \dist (p_{i}, P_{i}) = \sum_{i=1}^J m_{i-1} \leq M_J \lambda^{-(N-2-J)} \sum_{k\geq 0} \lambda^{-k}\\
&= \frac{M_J \lambda^{-(N-2-J)}}{\lambda -1}
\leq \frac{m_J}{\lambda-1}\, .
\end{align*}
In particular, it suffices to choose $\frac{1}{\lambda-1}= \bar\delta$.

\subsection{Proof of Lemma \ref{l:clusters-2}}
If $\varepsilon \leq \delta \min \{|q_1-q_2|: q_1, q_2\in P,\, q_1\neq q_2\}$, we then set $\tilde P =P$. Otherwise, pick a pair of points $p_0, p_1 \in P$ not satisfying this property and set $P_1:= P\setminus \{p_1\}$. By construction, $\dist (p_1, P_1)\leq \delta^{-1} \varepsilon$. If $N=2$, $\tilde P := P_1$ contains a single element and hence (ii) holds. 
On the other hand, if $N\geq 3$ and $\delta^{-1} \varepsilon \leq \delta \min \{|q_1-q_2|:q_1,q_2\in P_1, \ q_1\neq q_2\}$, we define $\tilde P = P_1$, which satisfies (ii) because $\dist (p_1, P_1) \leq \delta^{-1} \varepsilon$.  Otherwise we discard yet another point $p_2$ from $P$ in the same way as above, with $P_1$ in place of $P$, and set $P_2 := P_1\setminus \{p_2\}$. This time, we notice that 
\[
\dist (p_2, P_2) \leq \delta^{-2} \varepsilon\, ,
\] 
while
\[
\dist (p_1, P_2) \leq \dist (p_1, P_1) + \dist (p_2, P_2) \leq (\delta^{-1}+\delta^{-2}) \varepsilon\, .
\]
Iterating this procedure, we generate a family of sets $P_j$, $j=1,\dots,J$. We stop if either $P_J$ is a singleton, or if 
 \[
\delta^{-1} (1+\delta^{-1})^{J-1}
\varepsilon \leq \delta \min \{|q_1-q_2|: q_1, q_2\in P_J, \ q_1\neq q_2\} \, .
\]
For every point $p\in P$, since $\dist(p_j,P_j) \leq \delta^{-2}(1+\delta^{-1})^{j-2}$ for each $j\geq 2$, we therefore inductively have
\[
\dist (p, P_J) \leq \sum_{j=1}^J \dist(p_j,P_j) \leq \delta^{-1}(\delta^{-1}+1)^{J-1} \varepsilon\, .
\]
In particular, when we stop the set $\tilde P= P_J$ satisfies (ii). 
Since $P_{N-1}$ would necessarily be a singleton, we must stop at a $J\leq N=1$, which makes the estimate (i) trivially true.

\subsection{Proof of Lemma \ref{l:clusters-3}}
When $N=2$ and $P=\{p_1,p_2\}$, we can clearly just let $P_j:=\{p_j\}$ for $j=1,2$. We now argue by induction on $N$. Fix $N\geq 3$, and suppose the conclusion of the lemma holds for $N'\leq N-1$. Select $p\in P$ such that $\max \{|q-q'|:q,q'\in P\setminus \{p\}\} = M = \max \{|q-q'|: q,q'\in P\}$. By our inductive hypothesis, we can then decompose $P\setminus \{p\}$ into $P'_1\cup P'_2$ such that 
\[
    \min \{|p_1-p_2|: p_1\in P'_1, p_2\in P'_2\} \geq \frac{M}{2^{N-3}}\, .
\] 
Let $p_i\in P'_i$ be such that $\dist (p, p_i)= \dist (p, P'_i)$ for $i=1,2$. Since $\dist (p_1, p_2)\geq \frac{M}{2^{N-3}}$, at least one among $\dist (p_i,p)$ has to be $\geq \frac{M}{2^{N-2}}$. Assuming, upon relabeling, that the latter happens for $i=1$, we set $P_1=P'_1$ and $P_2 = P'_2\cup \{p\}$, proving the conclusion of the Lemma.

\bibliographystyle{alpha}
\bibliography{references}

\end{document}